\documentclass[11pt, article]{amsart}
\usepackage{tikz-cd}
\usepackage{relsize}
\usepackage{amsmath,amsfonts,amssymb,amsthm,mathtools,xparse}
\usepackage[abbrev,nobysame]{amsrefs}
\usepackage{mathrsfs}

\usepackage{hyperref}
\usepackage{accents}
\usepackage{xcolor}
\usepackage{soul}
\usepackage{calligra}
\usepackage{enumitem}

\usepackage[T1]{fontenc}
\usepackage[ttdefault=true]{AnonymousPro}

\setlist[enumerate]{wide=\parindent}

\usepackage{bbm}

 \allowdisplaybreaks

\DeclareMathAlphabet{\mathcalligra}{T1}{calligra}{m}{n}

\DeclareMathOperator{\Res}{Res}

\DeclareMathOperator{\Hom}{Hom}

\DeclareMathOperator{\Sing}{Sing}

\DeclareMathOperator{\Aut}{Aut}
\DeclareMathOperator{\Sign}{Sign}

\DeclareMathOperator{\End}{End}

\DeclareMathOperator{\obj}{obj}

\DeclareMathOperator{\ev}{ev}
\DeclareMathOperator{\res}{res}
\DeclareMathOperator{\rank}{rank}

\begin{document}

\newtheorem{thm}{Theorem}[section]
\newtheorem{prop}[thm]{Proposition}
\newtheorem{coro}[thm]{Corollary}
\newtheorem{conj}[thm]{Conjecture}
\newtheorem{example}[thm]{Example}
\newtheorem{lem}[thm]{Lemma}

\newtheorem{hy}[thm]{Hypothesis}
\newtheorem*{acks}{Acknowledgements}
\theoremstyle{definition}
\newtheorem{de}[thm]{Definition}
\newtheorem{ex}[thm]{Example}
\newtheorem{rem}[thm]{Remark}
\newtheorem{draft}[thm]{[DRAFT]}

\newtheorem{convention}[thm]{Convention}


\newcommand{\C}{{\mathbb{C}}}
\newcommand{\Z}{{\mathbb{Z}}}
\newcommand{\N}{{\mathbb{N}}}
\newcommand{\te}[1]{\mbox{#1}}
\newcommand{\set}[2]{{
    \left\{
        {#1}
    \,\middle\vert\,
        {#2}
    \right\}
}}

\newcommand{\choice}[2]{{
\left[
\begin{array}{c}
{#1}\\{#2}
\end{array}
\right]
}}
\def \<{{\langle}}
\def \>{{\rangle}}

\def \:{\mathopen{\overset{\circ}{
    \mathsmaller{\mathsmaller{\circ}}}
    }}
\def \;{\mathclose{\overset{\circ}{\mathsmaller{\mathsmaller{\circ}}}}}

\newcommand{\overit}[2]{{
    \mathop{{#1}}\limits^{{#2}}
}}
\newcommand{\belowit}[2]{{
    \mathop{{#1}}\limits_{{#2}}
}}

\newcommand{\wt}[1]{\widetilde{#1}}

\newcommand{\wh}[1]{\widehat{#1}}

\newcommand{\wck}[1]{\reallywidecheck{#1}}

\newcommand{\qqq}{{\mathbbm q}}

\newlength{\dhatheight}
\newcommand{\dwidehat}[1]{%
    \settoheight{\dhatheight}{\ensuremath{\widehat{#1}}}%
    \addtolength{\dhatheight}{-0.45ex}%
    \widehat{\vphantom{\rule{1pt}{\dhatheight}}%
    \smash{\widehat{#1}}}}
\newcommand{\dhat}[1]{%
    \settoheight{\dhatheight}{\ensuremath{\hat{#1}}}%
    \addtolength{\dhatheight}{-0.35ex}%
    \hat{\vphantom{\rule{1pt}{\dhatheight}}%
    \smash{\hat{#1}}}}

\newcommand{\dwh}[1]{\dwidehat{#1}}

\newcommand{\dis}{\displaystyle}

\newcommand{\pd}[1]{\frac{\partial}{\partial {#1}}}

\newcommand{\pdiff}[2]{\frac{\partial^{#2}}{\partial #1^{#2}}}

\newcommand{\qb}[2]{\genfrac{[}{]}{0pt}{}{#1}{#2}}

\newcommand{\vvp}[1]{\< #1 \>_{\,\,\,\mathllap{\wp}}}

\newcommand{\vvinf}[1]{\< #1 \>_{\,\,\,\,\,\mathllap{\infty}}}


\newcommand{\g}{{\mathfrak g}}
\newcommand{\ff}{{\mathfrak f}}
\newcommand{\f}{\ff}
\newcommand{\gc}{{\bar{\g'}}}
\newcommand{\h}{{\mathfrak h}}
\newcommand{\cent}{{\mathfrak c}}
\newcommand{\notc}{{\not c}}
\newcommand{\Loop}{{\mathcal L}}
\newcommand{\G}{{\mathcal G}}
\newcommand{\D}{\mathcal D}
\newcommand{\T}{\mathcal T}
\newcommand{\Free}{\mathcal F}
\newcommand{\Cfk}{\mathcal C}
\newcommand{\nil}{\mathfrak n}
\newcommand{\al}{\alpha}
\newcommand{\be}{\beta}
\newcommand{\beck}{\be^\vee}
\newcommand{\ssl}{{\mathfrak{sl}}}
\newcommand{\id}{\te{id}}
\newcommand{\rtu}{{\xi}}
\newcommand{\period}{{N}}
\newcommand{\half}{{\frac{1}{2}}}
\newcommand{\reciprocal}[1]{{\frac{1}{#1}}}
\newcommand{\inverse}{^{-1}}
\newcommand{\inv}{\inverse}
\newcommand{\SumInZm}[2]{\sum\limits_{{#1}\in\Z_{#2}}}
\newcommand{\uce}{{\mathfrak{uce}}}
\newcommand{\Rcat}{\mathcal R}
\newcommand{\cS}{{\mathcal{S}}}
\newcommand{\Ures}{\U^{\res}}

\newcommand{\psc}[2]{\vphantom{#1}^{\mathrlap{#2}} #1}

\newcommand{\biot}[2]{\,\vphantom{\ot}_{\mathrlap{#1}} \ot_{#2}}

\newcommand{\Apsc}[1]{A}

\newcommand{\Bpsc}[1]{B}

\newcommand{\tpsc}[2]{\wt{#1}}


\newcommand{\E}{{\mathcal{E}}}
\newcommand{\F}{{\mathcal{F}}}

\newcommand{\Etopo}{{\mathcal{E}_{\te{topo}}}}

\newcommand{\Ye}{{\mathcal{Y}_\E}}

\newcommand{\rh}{{{\bf h}}}
\newcommand{\rp}{{{\bf p}}}
\newcommand{\rrho}{{{\pmb \varrho}}}
\newcommand{\ral}{{{\pmb \al}}}

\newcommand{\comp}{{\mathfrak{comp}}}
\newcommand{\ctimes}{{\widehat{\boxtimes}}}
\newcommand{\ptimes}{{\widehat{\otimes}}}
\newcommand{\ptimeslt}{{
{}_{\te{t}}\ptimes
}}
\newcommand{\ptimesrt}{{\ot_{\te{t}} }}
\newcommand{\ttp}[1]{{
    {}_{{#1}}\ptimes
}}
\newcommand{\bigptimes}{{\widehat{\bigotimes}}}
\newcommand{\bigptimeslt}{{
{}_{\te{t}}\bigptimes
}}
\newcommand{\bigptimesrt}{{\bigptimes_{\te{t}} }}
\newcommand{\bigttp}[1]{{
    {}_{{#1}}\bigptimes
}}

\newcommand{\ot}{\otimes}
\newcommand{\Ot}{\bigotimes}
\newcommand{\bt}{\boxtimes}

\newcommand{\affva}[1]{V_{\wh\g}\left(#1,0\right)}
\newcommand{\saffva}[1]{L_{\wh\g}\left(#1,0\right)}
\newcommand{\saffmod}[1]{L_{\wh\g}\left(#1\right)}

\newcommand{\bd}[1]{{\boldsymbol #1}}

\newcommand{\otcopies}[2]{\belowit{\underbrace{{#1}\ot \cdots \ot {#1}}}{{#2}\te{-times}}}

\newcommand{\wtotcopies}[3]{\belowit{\underbrace{{#1}\wh\ot_{#2} \cdots \wh\ot_{#2} {#1}}}{{#3}\te{-times}}}


\newcommand{\tar}{{\mathcal{DY}}_0\left(\mathfrak{gl}_{\ell+1}\right)}
\newcommand{\U}{{\mathcal{U}}}
\newcommand{\htar}{\mathcal{DY}_\hbar\left(A\right)}
\newcommand{\hhtar}{\widetilde{\mathcal{DY}}_\hbar\left(A\right)}
\newcommand{\htarz}{\mathcal{DY}_0\left(\mathfrak{gl}_{\ell+1}\right)}
\newcommand{\hhtarz}{\widetilde{\mathcal{DY}}_0\left(A\right)}
\newcommand{\qhei}{\U_\hbar\left(\hat{\h}\right)}
\newcommand{\n}{{\mathfrak{n}}}
\newcommand{\vac}{{{\mathbbm 1}}}


\makeatletter
\@addtoreset{equation}{section}
\def\theequation{\thesection.\arabic{equation}}
\makeatother \makeatletter

\title{Quantum affine vertex algebra at root of unity}

\author{Fei Kong$^1$}
\email{kongmath@hunnu.edu.cn}
\address{Key Laboratory of Computing and Stochastic Mathematics (Ministry of Education), School of Mathematics and Statistics, Hunan Normal University, Changsha, China 410081}
\thanks{$^1$Partially supported by the NSF of China  (No.12371027 and No.12471029).}


\begin{abstract}
Let $\mathfrak g$ be a finite simple Lie algebra, and let $r$ denote the ratio of the square length of long roots to that of short roots.
Let $\wp>2r$ be an integer and $\zeta$ a primitive $\wp$-th root of unity.
Denote by $\mathcal U_\zeta(\widehat{\mathfrak g})$ the Lusztig big quantum affine algebra at root of unity defined by divided powers.
In this paper, we establish a current algebra presentation of $\mathcal U_\zeta(\widehat{\mathfrak g})$.
Based on this presentation, we construct a  $\mathbb Z_\wp$-module quantum vertex algebras $V_{\wp,\tau}^\ell(\mathfrak g)$ for each integer $\ell$.
Moreover, we establish a fully faithful functor from the category of smooth weighted $\mathcal U_\zeta(\widehat{\mathfrak g})$-modules of level $\ell$ to the category of $(\Z_\wp,\chi_\phi)$-equivariant $\phi$-coordinated quasi-modules of $V_{\wp,\tau}^\ell(\mathfrak g)$,
where $\chi_\phi:\Z_\wp\to\C^\times$ is the group homomorphism defined by $s\mapsto \zeta^s$.
We also determine the image of this functor.
The structure $V_{\wp,\tau}^\ell(\mathfrak g)$ is substantially different from that of affine vertex algebras.
We realize $V_{\wp,\tau}^\ell(\mathfrak g)$ as a deformation of a simpler quantum vertex algebra $V_{\wp,\varepsilon}^\ell(\mathfrak g)$ by using vertex bialgebras,
and decompose $V_{\wp,\varepsilon}^\ell(\mathfrak g)$ into a Heisenberg vertex algebra and a more interesting quantum vertex algebra determined by a quiver.
\end{abstract}


\keywords{Quantum vertex algebras, quantum affine vertex algebras, quantum affine algebra at root of unity, equivariant $\phi$-coordinated quasi-modules}

\subjclass[2020]{17B69}
%
\maketitle


\section{Introduction}

Quantum affine algebras admit two parallel realizations, distinguished by the deformation parameter and the choice of base ring: the formal deformation quantum affine algebra $\U_\hbar(\wh\g)$ defined over the ring $\C[[\hbar]]$ of formal power series, and the complex-parametric quantum affine algebra $\U_\zeta(\wh\g)$
in which the quantum parameter $\zeta$ is a nonzero complex number.
Based on the Drinfeld presentation \cite{Dr-new}, Frenkel and Jing constructed vertex representations
for simply-laced untwisted quantum affine algebras in \cite{FJ-vr-qaffine}, and formulated a fundamental problem of
developing certain ``quantum vertex algebra theory'' associated to quantum affine algebras, in parallel with the connection between affine Kac-Moody Lie algebras and vertex algebras.

As one of the fundamental works, Etingof and Kazhdan (\cite{EK-qva}) developed a theory of quantum vertex operator algebras in the sense of formal deformations of vertex algebras.
The most visible difference between these and vertex algebras is that the usual locality is replaced by the $S$-locality. Such $S$-locality is controlled by a rational quantum Yang-Baxter operator.
Partly motivated by the work of Etingof and Kazhdan, H. Li
conducted a series of studies.
While vertex algebras are analogues of commutative associative algebras, H. Li introduced the notion of
nonlocal vertex algebras \cite{Li-nonlocal}, which are analogues of noncommutative associative algebras. A nonlocal vertex algebra is a weak quantum vertex algebra \cite{Li-nonlocal} if it satisfies the $S$-locality. In addition, it becomes a quantum vertex algebra \cite{Li-nonlocal} if the $S$-locality is controlled by a rational quantum Yang-Baxter operator.
Mainly in order to associate quantum vertex algebras to quantum affine algebras, a theory of $\phi$-coordinated quasi modules was developed in \cite{Li-phi-coor, Li-G-phi}.
The $\hbar$-adic counterparts of these notions were introduced in \cite{Li-h-adic}.
In this framework, a quantum vertex operator algebra in sense of Etingof-Kazhdan is an $\hbar$-adic quantum vertex algebra whose classical limit is a vertex algebra.

In the very paper \cite{EK-qva}, Etingof and Kazhdan constructed quantum vertex operator algebras as formal deformations of
the universal affine vertex algebras of type $A$, by using the $R$-matrix type relations given in \cite{RS-RTT}.
Butorac, Jing and Ko\v{z}i\'{c} (\cite{BJK-qva-BCD}) extended Etingof-Kazhdan's construction
to type $B$, $C$ and $D$ rational $R$-matrices.
The quantum vertex operator algebras associated with trigonometric $R$-matrices of type $A$, $B$, $C$ and $D$ were constructed in \cite{Kozic-qva-tri-A, K-qva-phi-mod-BCD}.
In \cite{JKLT-Defom-va}, we developed a method to construct quantum vertex operator algebras by using vertex bialgebras.
By using this method, we constructed the quantum lattice vertex algebras, which is a family of quantum vertex operator algebras as deformations of lattice vertex algebras.
Moreover, based on the Drinfeld's quantum affinization construction (\cite{Dr-new,J-KM,Naka-quiver}), in \cite{K-Quantum-aff-va} we constructed the quantum affine vertex algebras $V_{\hat\g,\hbar}(\ell,0)$ and $L_{\hat\g,\hbar}(\ell,0)$ for all symmetric Kac-Moody Lie algebras $\g$.
When $\g$ is of finite type, we proved that $L_{\hat\g,\hbar}(\ell,0)/\hbar L_{\hat\g,\hbar}(\ell,0)$
is isomorphic to the simple affine vertex algebra $L_{\hat\g}(\ell,0)$ of positive level $\ell$.

Quantum groups at roots of unity exhibit substantial differences in their structure and representation theory compared to formal deformation quantum groups, which were first systematically studied by Lusztig in \cite{Luztig-quantum-book}.
In contrast to the semisimple representation category of generic quantum groups, the category of finite-dimensional modules over a quantum group at a root of unity is nonsemisimple, featuring abundant projective modules, various types of blocks, and nontrivial extension groups.
Although vertex operator algebras and quantum groups originally arose from distinct research contexts, the representation categories of vertex operator algebras are connected to the module category of quantum groups at root of unity \cite{KL-tensor-struct-qaff-1,KL-tensor-struct-qaff-2,
KL-tensor-struct-qaff-3,KL-tensor-struct-qaff-4,
finkelberg1996equivalence, zhang2008vertex, huang2017applicability,ACK-braidedTC,CR-unrolledquantumgps,
CRR-KL-corresp,ACK-wt-aff-small-quantum-gps}.
The classification of finite-dimensional irreducible representations of quantum affine algebras at root of unity was given in \cite{CP-qaff-rut, FM-rtu}.
And a family of infinite-dimensional irreducible representations was constructed by using vertex representations in \cite{CJ-vr-rtu}.

Let $\zeta$ be a primitive $\wp$-th root of unity.
In this paper, the quantum affine algebra $\U_\zeta(\wh\g)$ at a root of unity refers to the Lusztig big quantum group defined by divided powers. Our aim is to construct the $\Z_\wp$-module quantum vertex algebra $V_{\wp,\tau}^\ell(\g)$ associated with $\U_\zeta(\wh\g)$ at a $\wp$-th primitive root of unity and an integer level $\ell$. Furthermore, we establish a correspondence between smooth weighted $\U_\zeta(\wh\g)$-modules of level $\ell$ and equivariant $\phi$-coordinated quasi-modules for $V_{\wp,\tau}^\ell(\g)$.

Compared with the formal deformation version studied in \cite{K-Quantum-aff-va}, two additional obstacles must be overcome.
First, to apply the conceptual construction developed in \cite{Li-G-phi}, a current algebra presentation of $\U_\zeta(\wh\g)$ is required.
Let $\U_q(\wh\g)$ denote the quantum affine algebra over $\C(q)$, where $q$ is a formal variable.
The algebra $\U_\zeta(\wh\g)$ is realized as a quotient of the $\C[q,q^{-1}]$-subalgebra of $\U_q(\wh\g)$ generated by divided powers.
Using operator product expansions, we establish such a current algebra presentation in Section \ref{sec:qaff-rtu}.
At this point, the second obstacle emerges.
Among the defining relations, relation \eqref{zeta-Psi-def2}:
\begin{align}
\tag{$\zeta$10}\Psi_i^\pm(z)
    =\zeta_i^{\mp 2H_i(0)}\exp\left(\sum_{m<0}(\zeta_i^m-\zeta_i^{-m})H_i(m)\frac{z^{-m}}{\mp m}\right)
    \exp\left(\sum_{m>0}(\zeta_i^m-\zeta_i^{-m})H_i(m)\frac{z^{-m}}{\mp m}\right).
\end{align}
cannot be realized as an element in a quantum vertex algebra.
More precisely, for an object
$(W,H_i(z),\Psi_i^\pm(z),X_i^\pm(z))$ in $\mathcal R_\zeta^\ell(\wh\g)$ (see Theorem \ref{thm:presentation}), the set
\begin{align*}
  U_W=\set{H_i(\zeta^m z),\,\Psi_i^\pm(\zeta^m z),\,X_i^\pm(\zeta^m z)}{i\in I,\,m\in\Z_\wp}
\end{align*}
generates a $\Z_\wp$-module nonlocal vertex algebra $\<U_W\>_\phi$ for the associate $\phi(x,z)=xe^z$ (see Theorems \ref{thm:nonlocal-VA-gen} and \ref{thm:G-mod-nonlocal-VA-gen}).
In the formal deformation setting, a similar relation can be expressed as follows (see \cite[(6.21), Lemma 8.7]{K-Quantum-aff-va}):
\begin{align*}
  &\Psi_i^+(z)=\left(\frac{F(r_i+r\ell)}{F(r_i-r\ell)}\right)^\half
  \exp\left( \left(-q^{-r\ell z\pd z}F\left(z\pd z\right)H_i(z)\right)_{-1}^\phi \right)1_W,\\
  &\quad\te{where }F(z)=(q^{z}-q^{-z})/z\in \hbar \C[z^2][[\hbar]]\,\,\te{and}\,\,q=\exp(\hbar).
\end{align*}
However, this expression does not converge when $\hbar\to\log \zeta$.
Consequently, we require an additional family of generators $\xi_{i,m}^{1^\pm}$ corresponding to $\Psi_i^\pm(\zeta^m z)$,
along with relation \eqref{eq:va-rel-der} corresponding to the following relation, which approximates \eqref{zeta-Psi-def2}:
\begin{align*}
  &z \pd z \Psi_i^\pm(z)
  =\pm \:(H_i(\zeta_i\inv z)-H_i(\zeta_i z))\Psi_i^\pm(z)\;.
\end{align*}
The relations \eqref{tau1}-\eqref{eq:va-e+e-},
\eqref{eq:va-rel-varphi+-}-\eqref{eq:va-rel-rtu} reflect the rest defining relations.
We construct a $\Z_\wp$-module quantum vertex algebra $V_{\wp,\tau}^\ell(\g)$ corresponding to these relations \eqref{tau1}-\eqref{eq:va-rel-rtu} (see Section \ref{subsec:construct-V}).
Let $(W,Y_W^\phi)$ be an equivariant $\phi$-coordinated quasi module for $V_{\wp,\tau}^\ell(\g)$.
Relation \eqref{eq:va-rel-der} ensures that the differences between $\Psi_i^\pm(z)$ and $Y_W^\phi(\xi_{i,0}^{1^\pm},z)$ ($i\in I$) are controlled.
In fact, they generate a commutative group within the centroid of $Y_W^\phi(V_{\wp,\tau}^\ell(\g),z)$ (see Proposition \ref{prop:mho} for details).

In the formal deformation setting, the quantum affine vertex algebras $V_{\hat\g,\hbar}(\ell,0)$ and $L_{\hat\g,\hbar}(\ell,0)$ are formal deformations of affine vertex algebras.
In contrast, the structure of $V_{\wp,\tau}^\ell(\g)$ exhibits substantially different from that of affine vertex algebras.
Not only the additional family of generators $\xi_{i,m}^{1^\pm}$ corresponding to $\Psi_i^\pm(\zeta^mz)$,
but the theory of equivariant $\phi$-coordinated quasi-modules requires multiple times more generators than that of affine vertex algebras.
We give a rough description of the structure of $V_{\wp,\tau}^\ell(\g)$ in Section \ref{sec:struct}.
The definition of $V_{\wp,\tau}^\ell(\g)$ depends on a parameter $\tau$ lying in an abelian group $\mathfrak T\subset \C[[z]]^{9\wp^2\rank\g}$ with identity element denoted by $\varepsilon$.
We realize $V_{\wp,\tau}^\ell(\g)$ as a deformation of $V_{\wp,\varepsilon}^\ell(\g)$ by using a vertex bialgebra $H$, and furthermore identify $V_{\wp,\tau}^\ell(\g)$ as a quantum vertex subalgebra inside the smash product $V_{\wp,\varepsilon}^\ell(\g)\sharp H$.
Finally, we decompose
$V_{\wp,\varepsilon}^\ell(\g)$ into two quantum vertex algebras: a Heisenberg vertex algebra and a more interesting one, a quantum vertex algebra determined by a quiver (see Definition \ref{de:qva-quiver}).

This paper is organized as follows:
In Section \ref{sec:qva}, we recall the notions of quantum vertex algebras and equivariant $\phi$-coordinated quasi-modules, along with the theory of twistors and their associated deformations as developed in \cite{Sun-Twistor}. We then refine and extend the deformation approach introduced in \cite{JKLT-Defom-va} by using the language of twistors.
In Section \ref{sec:qaff-rtu}, we obtain a current algebra presentation of $\U_\zeta(\wh\g)$ by using vertex operator expansion.
In Section \ref{sec:construct}, we construct the $\Z_\wp$-module quantum vertex algebra $V_{\wp,\tau}^\ell(\g)$ associated to $\U_\zeta(\wh\g)$.
In Section \ref{sec:phi-mod}, we establish a fully faithful functor from the category of smooth weighted $\U_\zeta(\wh\g)$-modules of level integer $\ell$
to the category of equivariant $\phi$-coordinated quasi-modules of $V_{\wp,\tau}^\ell(\g)$ and determine the image of this functor.
Finally, in Section \ref{sec:struct}, we realize $V_{\wp,\tau}^\ell(\g)$ as a deformation of $V_{\wp,\varepsilon}^\ell(\g)$ by using vertex bialgebras
and we decompose $V_{\wp,\varepsilon}^\ell(\g)$ into a Heisenberg vertex algebra and a quantum vertex algebra determined by a quiver.

Throughout this paper, we denote the set of nonnegative integer and  positive integers by $\N$ and $\Z_+$, respectively.
For any ring $R$, we denote the set of invertible elements by $R^\times$.

\section{Quantum vertex algebras}\label{sec:qva}

In this section, we recall the notions of nonlocal vertex algebras, quantum vertex algebras, their $(G,\chi)$-module structures, and $(G,\chi_\phi)$-equivariant $\phi$-coordinated quasi-modules, together with their conceptual constructions. We also review the theory of twistors and the associated deformation of quantum vertex algebras. Finally, we refine and extend the deformation approach introduced in \cite{JKLT-Defom-va} by employing the language of twistors.

\subsection{Equivariant $\phi$-coordinated quasi-modules}\label{subsec:equiv-phi-mod}

In this subsection, we recall the notions of $(G,\chi)$-module nonlocal vertex algebras and $(G,\chi_\phi)$-equivariant $\phi$-coordinated quasi-modules, together with their conceptual constructions. We also present a result that allows one to compute relations in $(G,\chi_\phi)$-equivariant $\phi$-coordinated quasi-modules from relations in $(G,\chi)$-module nonlocal vertex algebras.

Let $W$ be a vector space over $\C$, let $k\in \Z_+$ and let $z_1,\dots,z_k$ be some mutually commutative formal variables.
Define
\begin{align*}
  \E^{(k)}(W;z_1,\dots,z_k)=\Hom(W,W((z_1,\dots,z_k))).
\end{align*}
We will also write $\E^{(k)}(W)=\E^{(k)}(W;z_1,\dots,z_k)$ and $\E(W)=\E^{(1)}(W)$ for convenience.

For a subset $\Gamma$ of $\C^\times$, we let
$\C_\Gamma[z]$ be the monoid generated by polynomials of the form $z-\al$, where $\al\in \Gamma$.
Then an ordered sequence $(a_1(z),\dots, a_k)$ in $\E(W)$ is called \emph{$\Gamma$-compatible} if there exists $f(z_1,z_2)\in \C_\Gamma[z_1/z_2]$
such that
\begin{align*}
  &\left(\prod_{1\le i<j\le k}f(z_i,z_j)\right)a_1(z_1)a_2(z_2)\cdots a_k(z_k)\in\E^{(k)}(W).
\end{align*}
In addition, this sequence is called \emph{compatible} if $\Gamma=\{1\}$, and is called \emph{quasi-compatible}
if $\Gamma=\C^\times$.
Moreover, a subset $U\subset \E(W)$ is called ($\Gamma$-, quasi-)compatible if every ordered sequence is ($\Gamma$-, quasi-)compatible.

Let $\phi(z,z_1)\ne z$ be an associate of the additive formal group law $F_a(z_1,z_2)=z_1+z_2$.
To be more precise, $z\ne\phi(z,z_1)\in \C((z))[[z_1]]$ such that
\begin{align*}
  \phi(z,0)=z,\quad\te{and}\quad \phi(\phi(z,z_1),z_2)=\phi(z,z_1+z_2).
\end{align*}
It was proved in \cite{Li-phi-coor} that
\begin{align}\label{eq:phi-exp}
  \phi(z,z_1)=\exp\left(z_1p(z)\pd{z}\right)z\quad\te{for some }0\ne p(z)\in \C((z)).
\end{align}
For a quasi-compatible ordered pair $(a(z),b(z))$ such that
\begin{align*}
  f(z_1,z_2)a(z_1)b(z_2)\in\E^{(2)}(W)\quad\te{for some }0\ne f(z_1,z_2)\in \C[z_1,z_2],
\end{align*}
a vertex operator map $Y_\E^\phi$ was defined in \cite{Li-phi-coor}:
\begin{align}\label{eq:def-Y-E}
  &Y_\E^\phi(a(z),z_0)b(z)=\sum_{n\in\Z}a(z)_n^\phi b(z)z_0^{-n-1}\\
  &\quad\nonumber=\iota_{z,z_0}f(\phi(z,z_0),z)\inv \left.\left( f(z_1,z)a(z_1)b(z) \right)\right|_{z_1=\phi(z,z_0)}.
\end{align}
The formal group law $F_a$ is obvious an associate of $F_a$. We denote the corresponding vertex operator map $Y_\E^{F_a}$ by $Y_\E$ for short.

The following two notions were introduced in \cite{Li-nonlocal}:
\begin{de}
A \emph{nonlocal vertex algebra} is a vector space $V$ equipped a distinguished vector $\vac$ and a linear map
$Y(\cdot,z):V\to \E(V)$ such that
\begin{align}\label{eq:vacuum-prop}
    Y(\vac,z)=1_V,\quad Y(u,z)\vac\in V[[z]]\quad \te{and}\quad \lim_{z\to 0}Y(u,z)\vac=u,
\end{align}
and that $\set{Y(u,z)}{u\in V}$ is compatible and
\begin{align}\label{eq:weak-asso}
    Y_\E(Y(u,z),z_0)Y(v,z)=Y(Y(u,z_0)v,z)\quad \te{for } u,v\in V.
\end{align}
\end{de}

\begin{rem}
We note that
\begin{align*}
    u\mapsto \lim_{z\to 0}\pd z Y(u,z)\vac
\end{align*}
defines a derivation of $V$.
This derivation is called the \emph{canonical derivation} and is denoted by $\partial$ in this paper.
\end{rem}

\begin{rem}
For a vector space $V$, and $A(z_1,z_2),B(z_1,z_2)\in \Hom(V,V((z_1))((z_2)))$,
we set
\begin{align*}
  A(z_1,z_2)\sim B(z_2,z_1),\quad\te{if }
  (z_1-z_2)^kA(z_1,z_2)=(z_1-z_2)^kB(z_2,z_1)\,\,\te{for some }k\in\N.
\end{align*}
\end{rem}

\begin{de}
A nonlocal vertex algebra $V$ is called a \emph{weak quantum vertex algebra}, if for each $u,v\in V$, there exist $\sum v_i\ot u_i\ot f_i(z)\in V\ot V\ot \C((z))$ such that
\begin{align*}
  Y(u,z_1)Y(v,z_2)\sim\sum \iota_{z_2,z_1}f_i(z_2-z_1)Y(v_i,z_2)Y(u_i,z_1).
\end{align*}
\end{de}

%

The following notion was given in \cite{Li-phi-coor}.
\begin{de}
Let $V$ be a nonlocal vertex algebra, and let $\phi(z,z_1)\ne z$ be an associate.
A \emph{$\phi$-coordinated (quasi) $V$-module} is a vector space $W$ equipped with a linear map $Y_W^\phi(\cdot,z):V\to \E(W)$ such that $Y_W^\phi(\vac,z)=1_W$, the subset $\{Y_W^\phi(v,z)\mid v\in V\}$ is (quasi-)compatible, and
\begin{align}\label{eq:phi-mod-asso}
  &Y_\E^\phi(Y_W^\phi(u,z),z_0)Y_W^\phi(v,z)=Y_W^\phi(Y(u,z_0)v,z)\quad \te{for }u,v\in V.
\end{align}
Furthermore, $W$ is called a \emph{(quasi) module} if $\phi=F_a$, and in this case, we denote the module map $Y_W^\phi$ by $Y_W$.
\end{de}

The following two notions were introduced in \cite{JKLT-G-phi-mod}:
\begin{de}
Let $G$ be a group with a linear character $\chi:G\to \C^\times$.
A \emph{$(G,\chi)$-module nonlocal vertex algebra} is a nonlocal vertex algebra $V$ equipped with a group homomorphism $R:G\to GL(V)$ such that $R(g)\vac=\vac$ and
\begin{align*}
  &R(g)Y(u,z)R(g)\inv =Y(R(g)u,\chi(g)z)\quad\te{for } g\in G,\, u\in V.
\end{align*}
Let $(V',R')$ be another $(G,\chi)$-module nonlocal vertex algebra. A nonlocal vertex algebra homomorphism $f:V\to V'$ is called a {\em $(G,\chi)$-module nonlocal vertex algebra homomorphism} if
\begin{align*}
  R'(g)\circ f=f\circ R(g)\quad \te{for }g\in G.
\end{align*}
\end{de}

\begin{rem}
Let $\chi$ be a trivial character, and let $(V,R)$ be a $G$-module nonlocal vertex algebra.
Then $R$ is a group homomorphism from $G$ to the automorphism group $\Aut(V)$ of $V$.
In this case, we call $(V,R)$ a \emph{$G$-module nonlocal vertex algebra}.
\end{rem}

\begin{de}
Let $V$ be a $(G,\chi)$-module nonlocal vertex algebra,
and let $\chi_\phi:G\to \C^\times$ be another linear character of $G$.
A \emph{$(G,\chi_\phi)$-equivariant $\phi$-coordinated quasi $V$-module}
is a $\phi$-coordinated quasi $V$-module $(W,Y_W^\phi)$
such that $\{Y_W^\phi(u,z)\mid u\in V\}$ is $\chi_\phi(G)$-compatible and
\begin{align}\label{eq:phi-mod-equiv}
  Y_W^\phi(R(g)u,z)=Y_W^\phi(u,\chi_\phi(g)\inv z)\quad \te{for }g\in G,\, u\in V.
\end{align}
\end{de}

\begin{rem}
Assume that there is $u\in V$ such that $\displaystyle\frac{d}{dz}Y_W^\phi(u,z)\ne 0$.
It was proved in \cite{JKLT-G-phi-mod} that
\begin{align}\label{eq:char-phi-compatible}
  \phi(z,\chi(g)z_0)=\chi_\phi(g)\phi(\chi_\phi(g)\inv z,z_0)\quad\te{for }g\in G.
\end{align}
In view of \eqref{eq:phi-exp}, the equation \eqref{eq:char-phi-compatible} is equivalent to
\begin{align}\label{eq:char-phi-compatible-alt}
  p(\chi_\phi(g)z)=\chi(g)\inv \chi_\phi(g)p(z)\quad\te{for }g\in G.
\end{align}
\end{rem}

The following was given in \cite[Theorem 4.8]{Li-phi-coor}:
\begin{thm}\label{thm:nonlocal-VA-gen}
Let $W$ be a vector space and let $U$ be a (quasi-)compatible subset of $\E(W)$.
For an associate $\phi(z,z_1)\ne z$, there exists a unique minimal vertex algebra $(\<U\>_\phi,Y_\E^\phi,1_W)$, such that
$U\subset \<U\>_\phi$.
Moreover, $W$ is a faithful $\phi$-coordinated (quasi) $\<U\>_\phi$-module with module map
\begin{align*}
    Y_W^\phi(a(z),z_0)=a(z_0)\quad\te{for }a(z)\in \<U\>_\phi.
\end{align*}
\end{thm}

\begin{rem}
If $\phi=F_a$, we write $\<U\>=\<U\>_{F_a}$ for short.
\end{rem}

\begin{rem}\label{rem:wquantum-VA-gen}
A subset $U$ of $\E(W)$ is called \emph{$S$-local} if for each $u(z),v(z)\in U$, there exists $\sum v_i(z)\ot u_i(z)\ot f_i(z_0)\in U\ot U\ot \C((z_0))$ such that
\begin{align*}
  (z_1-z_2)^ku(z_1)v(z_2)=(z_1-z_2)^k\sum \iota_{z_2,z_1}f_i(z_2-z_1) v_i(z_2)u_i(z_1)\quad\te{for some }k\in\N.
\end{align*}
Then $U$ is compatible. Moreover, it is proved in \cite[Theorem 5.8]{Li-nonlocal} that $\<U\>$ becomes a weak quantum vertex algebra.
\end{rem}

Let $G$ be a subset group of $\C^\times$.
It is clear that the inclusion is a linear character of $G$, and is denoted by $\chi_\phi$.
Let $\chi$ be another linear character of $G$ such that the condition \eqref{eq:char-phi-compatible} holds.
For each $g\in G$, one has a linear automorphism $R(g)$ on $\E(W)$ defined by
\begin{align*}
    &R(g)(a(z))=a(g\inv z).
\end{align*}
The following was given in \cite[Theorem 3.11]{JKLT-G-phi-mod}.

\begin{thm}\label{thm:G-mod-nonlocal-VA-gen}
Let $W$ be a vector space and let $U$ be a $G$-invariant quasi compatible subset of $\E(W)$.
Then $(\<U\>_\phi, R)$ is a $(G,\chi)$-module nonlocal vertex algebra,
and $(W,Y_W^\phi)$ is a $(G,\chi_\phi)$-equivariant $\phi$-coordinated quasi $\<U\>_\phi$-module.
\end{thm}

In the rest of this subsection, we work with a fixed special associate of $F_a(z_1,z_2)$, namely
\begin{align*}
  \phi(z,z_1)=ze^{z_1}=\exp\left(z_1 z\pd z\right)z,
\end{align*}
and consider a group $G$ together with two linear characters
\begin{align*}
  \chi,\chi_\phi:G\to \C^\times
\end{align*}
satisfying condition \eqref{eq:char-phi-compatible}.
This condition forces $\chi$ to be the trivial character.
The following result extends \cite[Lemma 7.1]{K-Quantum-aff-va} to $(G,\chi_\phi)$-equivariant $\phi$-coordinated quasi modules.

\begin{lem}\label{lem:phi-mod-rel-pre}
Let $V$ be a $G$-module nonlocal vertex algebra, and let $(W,Y_W^\phi,R)$ be a $(G,\chi_\phi)$-equivariant quasi-module of $V$. Let
\begin{align*}
  \sum_{i}\al_i\ot \beta_i\ot f_i(z),\quad
  \sum_{j}\mu_j\ot \nu_j\ot g_j(z)\in V\ot V\ot \C(z),
\end{align*}
such that for any $a\in G$,
\begin{align}\label{eq:phi-mod-rel-pre-1}
  &\sum_{i}\iota_{z_1,z_2}f_i(\chi_\phi(a)\inv e^{z_1-z_2})Y(R(a)\al_i,z_1)Y(\beta_i,z_2)\\
  =&\nonumber\sum_{j}\iota_{z_2,z_1}g_j(\chi_\phi(a)\inv e^{z_1-z_2})
  Y(\mu_j,z_2)Y(R(a)\nu_j,z_1).
\end{align}
Then
\begin{align}\label{eq:phi-mod-rel-pre-2}
  &\sum_{i}\iota_{z_1,z_2}f_i(z_1/z_2)Y_W^\phi(\al_i,z_1)Y_W^\phi(\beta_i,z_2)
  =\sum_{j}\iota_{z_2,z_1}g_j(z_1/z_2)Y_W^\phi(\mu_j,z_2)Y_W^\phi(\nu_j,z_1).
\end{align}
\end{lem}

\begin{proof}
Let $n$ be a positive integer such that $(i,j\ge 1,a\in G)$
\begin{align*}
  &z^n f_i(\chi_\phi(a)\inv z),\quad z^ng_j(\chi_\phi(a)\inv z)\in \C[[z]],\quad
  z^nY(R(a)\al_i,z)\beta_i,\quad z^nY(R(a)\mu_j,z)\nu_j\in V[[z]].
\end{align*}
By letting the both hand sides of \eqref{eq:phi-mod-rel-pre-1} act on $\vac$ and multiplying $(z_1-z_2)^{2n}$, we get that
\begin{align*}
  &\sum_{i}\big((z_1-z_2)^n\iota_{z_1,z_2}f_i( e^{z_1-z_2})\big)
  \big((z_1-z_2)^nY(\al_i,z_1-z_2)\beta_i)\big)\\
  =&\sum_{j}\big((z_1-z_2)^n\iota_{z_2,z_1}g_j( e^{z_1-z_2})\big)
  e^{(z_1-z_2)\partial}\big((z_1-z_2)^nY(\mu_j,z_2-z_1)\nu_j\big).
\end{align*}
Taking $z_0=0$, we get that
\begin{align*}
  z_1^{2n}\sum_{i}f_i( e^{z_1})Y(\al_i,z_1)\beta_i=z_1^{2n}
  \sum_{j}g_i( e^{z_1})e^{z_1\partial}Y(\mu_j,-z_1)\nu_j.
\end{align*}
Since $z_1$ is invertible, we get that
\begin{align}\label{eq:phi-mod-rel-pre-3}
  &\sum_{i}f_i( e^{z})Y(\al_i,z)\beta_i=
  \sum_{j}g_i( e^{z})e^{z\partial}Y(\mu_j,-z)\nu_j.
\end{align}
Let $h(z)\in \C_{\chi_\phi(G)}[z]$, such that $(i,j\ge 1)$
\begin{align*}
  &h(z_2/z_1)f_i(z_1/z_2),\quad h(z_2/z_1)g_j(z_1/z_2)\in \C((z_1,z_2)),\\
  &h(z_2/z_1)Y_W^\phi(\al_i,z_1)Y_W^\phi(\beta_i,z_2),\quad
  h(z_2/z_1)Y_W^\phi(\mu_j,z_2)Y_W^\phi(\nu_j,z_1)\in \E^{(2)}(W).
\end{align*}
Then for any $a\in G$, we have that
\begin{align}\label{eq:phi-mod-rel-pre-4}
  &h(\chi_\phi(a)z_2/z_1)Y_W^\phi(R(a)\al_i,z_1)Y_W^\phi(\beta_i,z_2)\\
  =&h\left(z_2/(\chi_\phi(a)\inv z_1)\right)
  Y_W^\phi(\al_i,\chi_\phi(a)\inv z_1)Y_W^\phi(\beta_i,z_2)\in \E^{(2)}(W).\nonumber
\end{align}
From the weak associativity of $\phi$-coordinated quasi-modules \eqref{eq:phi-mod-asso} and \eqref{eq:phi-mod-rel-pre-4}, we get that
\begin{align}
  &\label{eq:phi-mod-rel-pre-5}
  \left.\left(h(\chi_\phi(a)z/z_1)Y_W^\phi(R(a)\al_i,z_1)Y_W^\phi(\beta_i,z)\right)\right|_{z_1=\phi(z,z_0)}\\
  &\nonumber\quad=h(\chi_\phi(a)e^{-z_0}) Y_W^\phi(Y(R(a)\al_i,z_0)\beta_i,z),\\
  &\left.\left( h(z/z_1)Y_W^\phi(\mu_j,z)Y_W^\phi(\nu_j,z_1) \right)\right|_{z=\phi(z_1,-z_0)}
  =h(e^{-z_0})Y_W^\phi(Y(\mu_j,-z_0)\nu_j,z_1).\nonumber
\end{align}
Combining the second equation with \cite[Remark 2.8]{Li-phi-coor}, we get that
\begin{align*}
  &\left.\left( h(z/z_1)Y_W^\phi(\mu_j,z)Y_W^\phi(\nu_j,z_1) \right)\right|_{z_1=\phi(z,z_0)}\nonumber\\
  =&\left.\left( \left.\left( h(z/z_1)Y_W^\phi(\mu_j,z)Y_W^\phi(\nu_j,z_1) \right)\right|_{z=\phi(z_1,-z_0)} \right)\right|_{z_1=\phi(z,z_0)}\nonumber\\
  =&\left.\left( h(e^{-z_0})Y_W^\phi(Y(\mu_j,-z_0)\nu_j,z_1) \right)\right|_{z_1=\phi(z,z_0)}\nonumber\\
  =&h(e^{-z_0})Y_W^\phi(Y(\mu_j,-z_0)\nu_j,\phi(z,z_0))\nonumber\\
  =&h(e^{-z_0})Y_W^\phi(e^{z_0\partial}Y(\mu_j,-z_0)\nu_j,z),
\end{align*}
where the last equation follows from \cite[Lemma 3.7]{Li-phi-coor}.
Combining this with \eqref{eq:phi-mod-rel-pre-3} and \eqref{eq:phi-mod-rel-pre-5}, we get
\begin{align*}
  &\left.\left(h(z/z_1)^2\sum_{i}\iota_{z_1,z}f_i(z_1/z)Y_W^\phi(\al_i,z_1)Y_W^\phi(\beta_i,z)\right)\right|_{z_1=\phi(z,z_0)}\\
  =&\left.\left(\sum_{i}\big(h(z/z_1)\iota_{z_1,z}f_i(z_1/z)\big)\big(h(z/z_1)Y_W^\phi(\al_i,z_1)Y_W^\phi(\beta_i,z)\big)\right)\right|_{z_1=\phi(z,z_0)}\\
  =&h(e^{-z_0})^2\sum_{i}f_i(e^{z_0}) Y_W^\phi(Y(\al_i,z_0)\beta_i,z)\\
  =&h(e^{-z_0})^2\sum_{j}g_i(e^{z_0})Y_W^\phi(e^{z_0\partial}Y(\mu_j,-z_0)\nu_j,z)\\
  =&\left.\left(\sum_{j}\big(h(z/z_1)\iota_{z,z_1}g_j(z_1/z)\big)
  \big(h(z/z_1)Y_W^\phi(\mu_j,z)Y_W^\phi(\nu_j,z_1)\big)\right)\right|_{z_1=\phi(z,z_0)}\\
  =&\left.\left(h(z/z_1)^2\sum_{j}\iota_{z,z_1}g_j(z_1/z)Y_W^\phi(\mu_j,z)Y_W^\phi(\nu_j,z_1)\right)\right|_{z_1=\phi(z,z_0)}.
\end{align*}
Combining this with the following fact
\begin{align*}
  &h(z_2/z_1)^2\sum_{i}\iota_{z_1,z_2}f_i(z_1/z_2)Y_W^\phi(\al_i,z_1)Y_W^\phi(\beta_i,z_2)\in \E^{(2)}(W),\\
  &h(z_2/z_1)^2\sum_{j}\iota_{z_2,z_1}g_j(z_1/z_2)Y_W^\phi(\mu_j,z_2)Y_W^\phi(\nu_j,z_1)\in \E^{(2)}(W),
\end{align*}
we get from \cite[Remark 2.8]{Li-phi-coor} that
\begin{align*}
  &h(z_2/z_1)^2\sum_{i}\iota_{z_1,z_2}f_i(z_1/z_2)Y_W^\phi(\al_i,z_1)Y_W^\phi(\beta_i,z_2)\\
  =&h(z_2/z_1)^2\sum_{j}\iota_{z_2,z_1}g_j(z_1/z_2)Y_W^\phi(\mu_j,z_2)Y_W^\phi(\nu_j,z_1).\nonumber
\end{align*}
It follows that
\begin{align*}
  &\sum_{i}\iota_{z_1,z_2}f_i(z_1/z_2)Y_W^\phi(\al_i,z_1)Y_W^\phi(\beta_i,z_2)
  -\sum_{j}\iota_{z_2,z_1}g_j(z_1/z_2)Y_W^\phi(\mu_j,z_2)Y_W^\phi(\nu_j,z_1)\\
  &\quad\nonumber=
  \left((\iota_{z_1,z_2}-\iota_{z_2,z_1})h(z_2/z_1)^{-2}\right)A(z_1,z_2),
\end{align*}
where
\begin{align*}
  A(z_1,z_2)= h(z_2/z_1)^2  \sum_{i}f_i(z_1/z)Y_W^\phi(\al_i, z_1)Y_W^\phi(\beta_i,z_2).
\end{align*}
Let $c_1,\dots,c_s\in \chi_\phi(G)$ and let positive integers $n_1,\dots, n_s$ such that
\begin{align*}
  h(z)^2=(1-c_1z)^{n_1}\cdots (1-c_sz)^{n_s}.
\end{align*}
By using \cite[Lemma 2.5]{Li-new-construction} and \cite[Remark 4.26]{JKLT-Defom-va}, we get that
\begin{align*}
  &\left((\iota_{z_1,z_2}-\iota_{z_2,z_1})h(z_2/z_1)^{-2}\right)
  =\sum_{t=1}^s\sum_{p=0}^{n_t-1}C_{t,p}(z_2)
  \frac{1}{p!}\left(z_2\pd{z_2}\right)^p\delta\left(\frac{c_tz_2}{z_1}\right),
\end{align*}
where $C_{t,p}(z)\in\C((z))$.
To prove this lemma, it suffices to show that
\begin{align*}
  A(z_1,z_2)\frac{1}{p!}\left(z_2\pd{z_2}\right)^p\delta\left(\frac{c_tz_2}{z_1}\right)
  =0\quad\te{for all }1\le t\le s,\,0\le p<n_t.
\end{align*}
This reduce to proving that
\begin{align}\label{eq:phi-mod-rel-pre-6}
  \lim_{z_1\to c_tz}\frac{1}{p!}\left(z_1\pd{z_1}\right)^pA(z_1,z)=0\quad\te{for all }1\le t\le s,\,0\le p<n_t.
\end{align}

For any $1\le t\le s$, let $a\in G$ such that $\chi_\phi(a)=c_t\inv$.
Then we have that
\begin{align}
  &A(z_1,z)|_{z_1=c_t\phi(z,z_0)}
  =A(c_t z_1,z)|_{z_1=\phi(z,z_0)}\nonumber\\
  =&\left.\left(h(c_t\inv z/z_1)^2\sum_{i}
  f_i(c_t z_1/z)Y_W^\phi(R(a)\al_i,
  z_1)Y_W^\phi(\beta_i,z)\right)\right|_{z_1=\phi(z,z_0)}\nonumber\\
  =&\left.\left(\sum_{i}
  \big(h(c_t\inv z/z_1)f_i(c_t z_1/z)\big)
  \big(h(c_t\inv z/z_1)Y_W^\phi(R(a)\al_i,
  z_1)Y_W^\phi(\beta_i,z)\big)\right)\right|_{z_1=\phi(z,z_0)}\nonumber\\
  =&\sum_{i}h(c_t\inv e^{-z_0})f_i(c_t e^{z_0})
  h(c_t\inv e^{-z_0})Y_W^\phi(Y(R(a)\al_i,z_0)\beta_i,z)\nonumber\\
  =&h(c_t\inv e^{-z_0})^2Y_W^\phi\left(\sum_{i}f_i(c_t e^{z_0})
  Y(R(a)\al_i,z_0)\beta_i,z\right),\nonumber%
\end{align}
where the second equation follows from \eqref{eq:phi-mod-equiv} and
the forth equation follows from \eqref{eq:phi-mod-rel-pre-5}.
It follows that for each $0\le p<n_t$, we have that
\begin{align*}
  &\lim_{z_1\to c_tz}\frac{1}{p!}\left(z_1\pd{z_1}\right)^pA(z_1,z)\nonumber\\
  =&
  \Res_{z_0}z_0^{-p-1}\left.\left(
  h(z/z_1)^2
  \sum_{i}f_i(z_1/z)Y_W^\phi(\al_i, z_1)Y_W^\phi(\beta_i,z)\right)\right|_{z_1=c_t\phi(z,z_0)}\nonumber\\
  =&
  \Res_{z_0}z_0^{-p-1}h(c_t\inv e^{-z_0})^2Y_W^\phi\left(\sum_{i}f_i(c_t e^{z_0})
  Y(R(a)\al_i,z_0)\beta_i,z\right).
\end{align*}
Observe that
\begin{align*}
  z_0^{-p-1}h(c_t\inv e^{-z_0})=(z_0^{-p-1}(1-e^{-z_0})^{n_t})
  \prod_{u\ne t}(1-c_uc_t\inv e^{-z_0})^{n_u}\in C[[z_0]],\quad\te{since }p<n_t.
\end{align*}
Thus, in order to prove \eqref{eq:phi-mod-rel-pre-6}, it suffices to show
\begin{align}\label{eq:phi-mod-rel-pre-7}
  \sum_{i}f_i(c_t e^{z_0})
  Y(R(a)\al_i,z_0)\beta_i\in V[[z_0]].
\end{align}

From \eqref{eq:phi-mod-rel-pre-1}, we have that
\begin{align}\label{eq:phi-mod-rel-pre-8}
  &\sum_{i}\iota_{z_1,z_2}f_i(c_t e^{z_1-z_2})
  Y(R(a)\al_i,z_1)Y(\beta_i,z_2)\in\E^{(2)}(W).
\end{align}
Then
\begin{align*}
  &z_0^{2n}\left.\left(\sum_{i}\iota_{z_1,z_2}f_i(c_t e^{z_1-z_2})
  Y(R(a)\al_i,z_1)Y(\beta_i,z_2)\right)\right|_{z_1=z_2+z_0}\\
  =&\left.\left(\sum_{i}\big((z_1-z_2)^nf_i(c_t e^{z_1-z_2})\big)\big((z_1-z_2)^nY(R(a)\al_i,z_1)Y(\beta_i,z_2)\big)\right)\right|_{z_1=z_2+z_0}\\
  =&\sum_{i}z_0^nf_i(c_t e^{z_0})z_0^n Y_\E(Y(R(a)\al_i,z_2),z_0)Y(\beta_i,z_2)\\
  =&\sum_{i}z_0^nf_i(c_t e^{z_0})z_0^n
  Y(Y(R(a)\al_i,z_0)\beta_i,z_2),
\end{align*}
where the first equation follows from \eqref{eq:phi-mod-rel-pre-8},
the second relation follows from the definition of $Y_\E$ (see \eqref{eq:def-Y-E}), and the last equation follows from \eqref{eq:weak-asso}.
Since $z_0$ is invertible, we get
\begin{align}
  &\sum_{i}f_i(c_t e^{z_0})Y(Y(R(a)\al_i,z_0)\beta_i,z_2)\nonumber\\
  =&\left.\left(\sum_{i}\iota_{z_1,z_2}f_i(c_t e^{z_1-z_2})Y(R(a)\al_i,z_1)Y(\beta_i,z_2)\right)\right|_{z_1=z_2+z_0}
  \label{eq:phi-mod-rel-pre-9}
  \quad\in \Hom(V,V((z_2))[[z_0]]).
\end{align}
By letting the both hand sides of \eqref{eq:phi-mod-rel-pre-9} act on $\vac$, we get that
\begin{align*}
  &\sum_{i}f_i(c_t e^{z_0})Y(Y(R(a)\al_i,z_0)\beta_i,z_2)\vac\nonumber\\
  =&\left.\left(\sum_{i}f_i(c_t e^{z_1-z_2})Y(R(a)\al_i,z_1)Y(\beta_i,z_2)\vac\right)\right|_{z_1=z_2+z_0}\in V[[z_0,z_2]].\nonumber
\end{align*}
Taking $z_2=0$, we get that
\begin{align*}
  &\sum_{i}f_i(c_t e^{z_0})Y(R(a)\al_i,z_0)\beta_i\in V[[z_0]],
\end{align*}
which proves \eqref{eq:phi-mod-rel-pre-7}.
\end{proof}

%

\begin{prop}\label{prop:phi-mod-rel-f}
Let $V$ be a $G$-module nonlocal vertex algebra, and let $(W,Y_W^\phi,R)$ be a $(G,\chi_\phi)$-equivariant quasi-module of $V$. Suppose that $\chi_\phi$ is injective. Let
\begin{align*}
  &\sum_{i\ge 1}\al_i\ot \beta_i\ot f_i(z),\quad \sum_{j\ge 1}\mu_j\ot \nu_j\ot g_j(z)\in V\ot V\ot \C(z),
  \quad \sum_{k\ge 0}\sum_{a\in G}\gamma_{k,a}\in V,
\end{align*}
such that for any $a\in G$,
\begin{align}\label{eq:phi-mod-rel-f-1}
  &\sum_{i\ge 1}\iota_{z_1,z_2}f_i(\chi_\phi(a)\inv e^{z_1-z_2})
  Y(R(a)\al_i,z_1)Y(\beta_i,z_2)\\
  &\qquad\nonumber
  -\sum_{j\ge 1}\iota_{z_2,z_1}g_j(\chi_\phi(a)\inv e^{z_1-z_2})
  Y(\mu_j,z_2)Y(R(a)\nu_j,z_1)\\
  &\quad\nonumber=
  \sum_{k\ge 0}Y(\gamma_{a,k},z_2)\frac{1}{k!}\pdiff{z_2}{k}
    z_1\inv\delta\left(\frac{z_2}{z_1}\right).
\end{align}
Then
\begin{align}\label{eq:phi-mod-rel-f-2}
  &\sum_{i\ge 1}\iota_{z_1,z_2}f_i(z_1/z_2)Y_W^\phi(\al_i,z_1)Y_W^\phi(\beta_i,z_2)\\
  &\qquad\nonumber
  -\sum_{j\ge 1}\iota_{z_2,z_1}g_j(z_1/z_2)Y_W^\phi(\mu_j,z_2)Y_W^\phi(\nu_j,z_1)\\
  &\quad\nonumber=
  \sum_{k\ge 0}\sum_{a\in G}\gamma_{a,k}(z_2)\frac{1}{k!}\left(z_2\pd{z_2}\right)^k
  \delta\left(\chi_\phi(a)\inv\frac{z_2}{z_1}\right).
\end{align}
\end{prop}

\begin{proof}
For any integer $k\ge 0$ and $a\in G$, we set $\beta_{a,-k}=\mu_{a,-k}=-\gamma_{a,k}$, $\al_{a,-k}=\nu_{a,-k}=\vac$ and
\begin{align*}
  &f_{a,-k}(z)=g_{a,-k}(z)=\frac{1}{2k!}\left(-z\pd z\right)^k\frac{\chi_\phi(a)z+1}{\chi_\phi(a)z-1}\in\C(z).
\end{align*}
For $i\ge 1$ and $a\in G$, we set
\begin{align*}
  &\al_{a,i}=\begin{cases}
              \al_i, & \mbox{if }a=1,\\
              0,&\mbox{if }a\ne 1,
            \end{cases}
  \quad \beta_{a,i}=\begin{cases}
              \beta_i, & \mbox{if }a=1,\\
              0,&\mbox{if }a\ne 1,
            \end{cases}
  \quad f_{a,i}(z)=\begin{cases}
              f_i(z), & \mbox{if }a=1,\\
              0,&\mbox{if }a\ne 1,
            \end{cases}
  \\
  &\mu_{a,j}=\begin{cases}
              \mu_j, & \mbox{if }a=1,\\
              0,&\mbox{if }a\ne 1,
            \end{cases}
  \quad \nu_{a,j}=\begin{cases}
              \nu_j, & \mbox{if }a=1,\\
              0,&\mbox{if }a\ne 1,
            \end{cases}
  \quad g_{a,j}(z)=\begin{cases}
              g_j(z), & \mbox{if }a=1,\\
              0,&\mbox{if }a\ne 1.
            \end{cases}
\end{align*}
It is straightforward to verify that
\begin{align*}
  &\iota_{z_1,z_2}f_{a,-k}(z_1/z_2)-\iota_{z_2,z_1}g_{a,-k}(z_1/z_2)
  =\frac{1}{k!}\left(z_2\pd{z_2}\right)^k\delta\left(\chi_\phi(a)\inv \frac{z_2}{z_1}\right),\\
  &\iota_{z_1,z_2}f_{b,-k}(\chi_\phi(a)\inv e^{z_1-z_2})
  -\iota_{z_2,z_1}g_{b,-k}(\chi_\phi(a)\inv e^{z_1-z_2})
  =\delta_{a,b}\frac{1}{k!}\pdiff{z_2}{k}z_1\inv\delta\left(\frac{z_2}{z_1}\right).
\end{align*}
Then the equation \eqref{eq:phi-mod-rel-f-1} is equivalent to
\begin{align*}
  &\sum_{(b,i)\in G\times \Z}\iota_{z_1,z_2}f_{b,i}(\chi_\phi(a)\inv e^{z_1-z_2})Y(R(a)\al_{b,i},z_1)Y(\beta_{b,i},z_2)\\
  =&\sum_{(b,j)\in G\times \Z}\iota_{z_2,z_1}g_{b,j}(\chi_\phi(a)\inv e^{z_1-z_2})Y(\mu_{b,j},z_2)Y(R(a)\nu_{b,j},z_1)\quad\te{for all }a\in G,
\end{align*}
and the equation \eqref{eq:phi-mod-rel-f-2} is equivalent to
\begin{align*}
  &\sum_{(a,i)\in G\times \Z}\iota_{z_1,z_2}f_{a,i}(z_1/z_2)
  Y_W^\phi(\al_{a,i},z_1)Y_W^\phi(\beta_{a,i},z_2)\\
  =&\sum_{(a,j)\in G\times \Z}\iota_{z_2,z_1}g_{a,j}(z_1/z_2)
  Y_W^\phi(\mu_{a,j},z_2)Y_W^\phi(\nu_{a,j},z_1).
\end{align*}
Therefore, this proposition follows immediate from Lemma \ref{lem:phi-mod-rel-pre}.
\end{proof}

\subsection{Quantum vertex algebras and twistors}\label{subsec:qva-and-deform}

In this section, we recall the notion of quantum vertex algebras, the concepts of twisting operators and twisting tensor products introduced in \cite{LS-twisted-tensor}, as well as the notion of twistors and the deformation of nonlocal vertex algebras using twistors developed in \cite{Sun-Twistor}.
We also give a necessary condition for a deformation of a quantum vertex algebra by a twistor to remain a quantum vertex algebra.
Finally, we construct a special twistor of the tensor product quantum vertex algebra $U\otimes V$ using a twisting operator for the ordered pair $(U,V)$. This twistor will be used in the decomposition of quantum vertex algebras in Section \ref{subsec:structure}.

\begin{de}
Let $U,V$ be two nonlocal vertex algebras.
A \emph{twisting operator} $S(z)$ of the ordered pair $(U,V)$ is a linear map $S(z):U\ot V\to U\ot V\ot\C((z))$ such that
\begin{align}
  &S(z)(v\ot \vac)=v\ot \vac\quad
   S(z)(\vac\ot u)=\vac\ot u\quad\te{for }u,v\in V,\label{eq:qyb-vac}\tag{QYB1}\\
  &S(z_1)Y_U^{12}(z_2)
  =Y_U^{12}(z_2)S^{23}(z_1)S^{13}(z_1+z_2),
  \label{eq:qyb-hex1}\tag{QYB2}\\
  &S(z_1)Y_V^{23}(z_2)
  =Y_V^{23}(z_2)S^{12}(z_1-z_2)S^{13}(z_1),
  \label{eq:qyb-hex2}\tag{QYB3}
\end{align}
\end{de}

\begin{rem}
    \emph{
    Given a twisting operator $S(z)$ for the ordered pair $(U,V)$, $S(z)\sigma$ is a twisting operator defined in \cite[Definition 2.2]{LS-twisted-tensor}, where $\sigma:U\ot V\to V\ot U$ is the flip map.
    Conversely, let $R(z)$ be a twisting operator for the ordered pair $(U,V)$ in the sense of \cite[Definition 2.2]{LS-twisted-tensor}.
    Then $R(z)\sigma$ is a twisting operator for $(U,V)$.
    }
\end{rem}

\begin{rem}
The relations \eqref{eq:qyb-vac}, \eqref{eq:qyb-hex1} and \eqref{eq:qyb-hex2} imply the following relation
\begin{align}
  &[\partial\ot 1,S(z)]=-\frac{d}{dz}S(z),
  \quad [1\ot\partial, S(z)]=\frac{d}{dz}S(z).\label{eq:qyb-der-shift}\tag{QYB4}
\end{align}
\end{rem}

\begin{rem}\label{rem:twisting-op-inv}
Let $S(z)$ be a twisting operator for the ordered pair $(U,V)$, which is invertible as an $\C((z))$-module map on $U\ot V\ot \C((z))$, then $S^{21}(-z)\inv$ is a twisting operator for the ordered pair $(V,U)$.
\end{rem}

\begin{rem}\label{rem:prod-twisting-op}
Let $U,V$ be two nonlocal vertex algebra and let $S(z),T(z)$ be two twisting operators for $(U,V)$ such that
\begin{align}\label{eq:qyb-com1}
  &S^{12}(z_1)T^{13}(z_2)=T^{13}(z_2)S^{12}(z_1),\quad
  S^{23}(z_1)T^{13}(z_2)=T^{13}(z_2)S^{23}(z_1).
\end{align}
Then $S(z)T(z)$ is also a twisting operator for $(U,V)$.
\end{rem}

The following result was proved in \cite[Theorem 2.4]{LS-twisted-tensor}.

\begin{thm}\label{thm:twisted-tensor-prod}
Let $U$ and $V$ be nonlocal vertex algebras and let $S(z)$ be a twisting operator for the pair $(U,V)$.
Define
\begin{align*}
  Y_S(z)=\big(Y(z)\ot Y(z)\big)S^{23}(-z)\sigma^{23}.
\end{align*}
Then $(U\ot V,Y_S,\vac\ot\vac)$ carries the structure of a nonlocal vertex algebra, which contains $U$ and $V$ canonically as nonlocal vertex subalgebras.
Moreover, we denote this nonlocal vertex algebra by $U\ot_S V$.
\end{thm}

\begin{de}
A \emph{quantum vertex algebra} \cite{Li-nonlocal} is a nonlocal vertex algebra $V$ equipped with a twisting operator $S(z)$ for $(V,V)$ (called a \emph{quantum Yang-Baxter operator}) such that
\begin{align}
  &\exists k\in\N,\,\,(z_1-z_2)^kY(u,z_1)Y(v,z_2)=(z_1-z_2)^k Y^{12}(z_2)Y^{23}(z_2)S^{12}(z_2-z_1)(v\ot u),\label{eq:qyb-locallity}\tag{QYB5}\\
  &S(z)S^{21}(-z)=1,\quad S^{12}(z_1)S^{13}(z_1+z_2)S^{23}(z_2)
  =S^{23}(z_2)S^{13}(z_1+z_2)S^{12}(z_1).\label{eq:qyb-unitary-qyb}\tag{QYB6}
\end{align}
\end{de}

The following notion was given in \cite{Sun-Twistor}.

\begin{de}\label{de:twistor}
Let $V$ be a nonlocal vertex algebra.
A \emph{twistor} $T(z)$ of $V$ is a twisting operator for $(V,V)$, such that
\begin{align}\label{eq:twistor}
  T^{12}(z_1)T^{23}(z_2)=T^{23}(z_2)T^{12}(z_1).
\end{align}
In addition, we say $T(z)$ \emph{invertible} if it is invertible as a $\C((z))$-module map on $V\ot V\ot \C((z))$.
\end{de}

\begin{rem}
This definition is slightly different from the origin one given by Sun \cite{Sun-Twistor}.
Let $T(z)$ be a linear from $V\ot V\to V\ot V\ot \C((z))$.
Then $T(z)$ is a twistor if and only if $T^{21}(-z)$ is a twistor in the sense of \cite{Sun-Twistor}.
Moreover, if $T(z)$ is invertible, then $T(z)$ is a twistor if and only if $T(z)\inv$ is a twistor in the sense of \cite{Sun-Twistor}.
\end{rem}

\begin{rem}\label{rem:twistor-inv}
Let $T(z)$ be an invertible twistor.
Then $T^{21}(-z)\inv$ is also a twistor.
\end{rem}

\begin{rem}\label{rem:prod-twistor}
Let $V$ be a nonlocal vertex algebra and let $S(z),T(z)$ be two twistors of $V$, which satisfies the conditions \eqref{eq:qyb-com1} and
\begin{align}\label{eq:twistor-com}
  S^{12}(z_1)T^{23}(z_2)=T^{23}(z_2)S^{12}(z_1).
\end{align}
Then $S(z)T(z)$ is also a twistor of $V$.
\end{rem}

The following two results were proved in \cite[Theorem 3.2,Proposition 3.5]{Sun-Twistor}.

\begin{thm}\label{thm:nonlocal-VA-twistor-deform}
Let $(V,Y,\vac)$ be a nonlocal vertex algebra and let $T(z)$ be a twistor. Define
\begin{align*}
  &\mathfrak D_T(Y)(z)=Y(z)T^{21}(-z).
\end{align*}
Then $(V,\mathfrak D_T(Y),\vac)$ is also a nonlocal vertex algebra, which is denoted by $\mathfrak D_T(V)$.
\end{thm}

%
%
%

\begin{prop}\label{prop:prod-twistor}
Let $V$ be a nonlocal vertex algebra and let $S(z),T(z)$ be two twistors of $V$, which satisfies the conditions \eqref{eq:qyb-com1} and
\eqref{eq:twistor-com}.
Then $T(z)$ is a twistor of $\mathfrak D_S(V)$, and
\begin{align*}
  \mathfrak D_T\left(\mathfrak D_S(V)\right)=\mathfrak D_{ST}(V).
\end{align*}
\end{prop}

We rewrite \cite[Theorem 3.6]{JKLT-Defom-va} by using the theory of twistors.

\begin{thm}\label{thm:qva-twistor}
Let $V$ be a quantum vertex algebra with quantum Yang-Baxter operator $S(z)$, and let $T(z)$ be an invertible twistor of $V$ such that the relation \eqref{eq:qyb-com1} holds for $(T(z),T(z))$.
Suppose that the relations \eqref{eq:qyb-com1} and \eqref{eq:twistor-com} hold for $(T(z),S(z))$.
Then $\mathfrak D_T(V)$ is a quantum vertex algebra with quantum Yang-Baxter operator
\begin{align}
  &T^{21}(-z)\inv S(z)T(z).
\end{align}
\end{thm}

\begin{proof}
Set $\bar T(z)=T^{21}(-z)\inv$, and set $S_T(z)=T^{21}(-z)\inv S(z)T(z)$. From Remark \ref{rem:twistor-inv}, we get that $\bar T(z)$ is a also twistor of $V$.
From \eqref{eq:qyb-com1} and \eqref{eq:twistor-com}, we have that
\begin{align*}
  &\bar T^{12}(z_1)S^{13}(z_2)=S^{13}(z_2)\bar T^{12}(z_1),\quad S^{23}(z_1)\bar T^{13}(z_2)=\bar T^{13}(z_2)S^{23}(z_1),\\
  &\bar T^{12}(z_1)T^{13}(z_2)=T^{13}(z_2)\bar T^{12}(z_1),\quad T^{23}(z_1)\bar T^{13}(z_2)=\bar T^{13}(z_2)T^{23}(z_1).
\end{align*}
Then we get from Remark \ref{rem:prod-twistor} that $\bar T(z)S(z)T(z)$ is a twistor of $V$.
From \eqref{eq:qyb-hex1}, we get that
\begin{align*}
  &T^{21}(-z_1)Y^{23}(z_2)=\sigma^{12}T^{12}(-z_1)\sigma^{12}Y^{23}(z_2)\\
  =&\sigma^{12}T^{12}(-z_1)Y^{12}(z_2)\sigma^{23}\sigma^{12}\\
  =&\sigma^{12}Y^{12}(z_2)T^{23}(-z_1)T^{13}(-z_1+z_2)\sigma^{23}\sigma^{12}\\
  =&Y^{23}(z_2)\sigma^{12}\sigma^{23}T^{23}(-z_1)T^{13}(-z_1+z_2)\sigma^{23}\sigma^{12}\\
  =&Y^{23}(z_2)T^{31}(-z_1)T^{21}(-z_1+z_2).
\end{align*}
Combining this with \eqref{eq:qyb-com1} and \eqref{eq:twistor-com}, we get that
\begin{align*}
  &Y_T^{12}(z_1)Y_T^{23}(z_2)=Y^{12}(z_1)T^{21}(-z_1)Y^{23}(z_2)T^{32}(-z_2)\\
  =&Y^{12}(z_1)Y^{23}(z_2)T^{31}(-z_1)T^{21}(-z_1+z_2)T^{32}(-z_2),\\
  &Y_T^{12}(z_2)Y_T^{23}(z_1)S_T(z_2-z_1)\sigma^{12}
  =Y^{12}(z_2)Y^{23}(z_1)T^{31}(-z_2)T^{21}(-z_2+z_1)T^{32}(-z_1)\\
  &\times T^{21}(-z_2+z_1)\inv S^{12}(z_2-z_1)T^{12}(z_2-z_1)\sigma^{12}\\
  =&Y^{12}(z_2)Y^{23}(z_1)S^{12}(z_2-z_1)\sigma^{12}T^{31}(-z_1)
  T^{21}(z_2-z_1)T^{32}(-z_2).
\end{align*}
Note that $T(z)$ maps $V\ot V$ into $V\ot V\ot\C((z))$.
We get from \eqref{eq:qyb-locallity} that for each $u,v\in V$, there exists $k\in\N$, such that
\begin{align*}
  &(z_1-z_2)^kY_T(u,z_1)Y_T(v,z_2)w\\
  =&(z_1-z_2)^kY_T^{12}(z_1)Y_T^{23}(z_2)(u\ot v\ot w)\\
  =&(z_1-z_2)^kY^{12}(z_1)Y^{23}(z_2)T^{31}(-z_1)T^{21}(-z_1+z_2)T^{32}(-z_2)(u\ot v\ot w)\\
  =&(z_1-z_2)^kY^{12}(z_2)Y^{23}(z_1)S^{12}(z_2-z_1)\sigma^{12}T^{31}(-z_1)
  T^{21}(z_2-z_1)T^{32}(-z_2)(u\ot v\ot w)\\
  =&(z_1-z_2)^kY_T^{12}(z_2)Y_T^{23}(z_1)T^{21}(-z_2+z_1)\inv S^{12}(z_2-z_1)T^{12}(z_2-z_1)(v\ot u\ot w)\\
  =&(z_1-z_2)^kY_T^{12}(z_2)Y_T^{23}(z_1)S_T^{12}(z_2-z_1)(v\ot u\ot w)
\end{align*}
for any $w\in V$. It shows that $(\mathfrak D_T(V),T^{21}(-z)\inv S(z)T(z))$ satisfies the relation \eqref{eq:qyb-locallity}.
From \eqref{eq:qyb-unitary-qyb}, \eqref{eq:qyb-com1} and \eqref{eq:twistor-com}, we get that
\begin{align*}
  &S_T^{21}(-z) S_T(z)=T(z)\inv S^{21}(-z) T^{21}(-z)T^{21}(-z)\inv S(z)T(z)=1,
\end{align*}
and
\begin{align*}
  &S_T^{12}(z_1)S_T^{13}(z_1+z_2)S_T^{23}(z_2)\\
  =&T^{21}(-z_1)\inv S^{12}(z_1)T^{12}(z_1)T^{31}(-z_1-z_2)\inv S^{13}(z_1+z_2)T^{13}(z_1+z_2)\\
  &\times T^{32}(-z_2)\inv S^{23}(z_2)T^{23}(z_2)\\
  =&T^{21}(-z_1)\inv T^{31}(-z_1-z_2)\inv T^{32}(-z_2)\inv S^{12}(z_1)S^{13}(z_1+z_2)S^{23}(z_2)\\
  &\times T^{12}(z_1)T^{13}(z_1+z_2)T^{23}(z_2)\\
  =&T^{32}(-z_2)\inv T^{31}(-z_1-z_2)\inv T^{21}(-z_1)\inv S^{23}(z_2)S^{13}(z_1+z_2)S^{12}(z_1)\\
  &\times T^{23}(z_2)T^{13}(z_1+z_2)T^{12}(z_1)\\
  =&T^{32}(-z_2)\inv S^{23}(z_2)T^{23}(z_2)T^{31}(-z_1-z_2)\inv S^{13}(z_1+z_2)T^{13}(z_1+z_2)\\
  &\times T^{21}(-z_1)\inv S^{12}(z_1)T^{12}(z_1)\\
  =&S_T^{23}(z_2)S_T^{13}(z_1+z_2)S_T^{12}(z_1).
\end{align*}
This proves the relations \eqref{eq:qyb-unitary-qyb}.
Therefore, $\mathfrak D_T(V)$ becomes a quantum vertex algebra with quantum Yang-Baxter operator $T^{21}(-z)\inv S(z)T(z)$.
\end{proof}

The twisted tensor product can be viewed as the deformation by a twistor (\cite[Lemma 4.10]{Sun-Twistor}).
\begin{lem}\label{lem:twistor-twisting-op}
Let $S(z)$ be a twisting operator for the pair $(U,V)$.
Then $T(z)=S^{14}(z)$ is a twistor of the tensor product nonlocal vertex algebra $U\ot V$ and $\mathfrak D_T(U\ot V)=U\ot_SV$.
\end{lem}

We introduce the another twistor of the tensor product nonlocal vertex algebra $U\ot V$
that will be used in Section \ref{subsec:structure}.

\begin{lem}\label{lem:double-twisted}
Let $U$ and $V$ be nonlocal vertex algebras and let $S_1(z)$, $S_2(z)$ be two twisting operators of the pair $(U,V)$.
Suppose that the relation \eqref{eq:qyb-com1} holds for $(S_1(z),S_2(z))$ and $(S_2(z),S_2(z))$.
Then $T(z)=S_1^{14}(z)S_2^{32}(-z)$ is a twistor of the tensor product nonlocal vertex algebra $U\ot V$.
\end{lem}

\begin{proof}
Set $T_1(z)=S_1^{14}(z)$ and $T_2(z)=S_2^{32}(-z)$.
From Lemma \ref{lem:twistor-twisting-op}, we have that $T_1(z)$ is a twistor of $U\ot V$.
It is straightforward to verify that $T_2(z)$ is also a twistor of $U\ot V$ and \eqref{eq:qyb-com1}, \eqref{eq:twistor-com} hold for $(T_1(z),T_2(z))$.
Then Remark \ref{rem:prod-twistor} yields the lemma.
\end{proof}
%

\begin{rem}\label{rem:double-twisted}
Let $T(z)$ be the twistor obtained in Lemma \ref{lem:double-twisted}. Then
$\mathfrak D_T(U\ot V)$ contains $U$ and $V$ canonically as nonlocal vertex subalgebras.
\end{rem}
%

\begin{prop}\label{prop:qva-double-twisted}
Let $U$ and $V$ be two quantum vertex algebras with quantum Yang-Baxter operators $S_U(z)$ and $S_V(z)$, respectively.
Let $S_1(z)$, $S_2(z)$ be two invertible twisting operators of the pair $(U,V)$, such that the relation \eqref{eq:qyb-com1} holds for $(S_i(z),S_j(z))$ for $i,j\in \{1,2\}$,
and
\begin{align}
  &S_i^{12}(z_1)S_U^{13}(z_2)=S_U^{13}(z_2)S_i^{12}(z_1),\quad
   S_i^{13}(z_1)S_V^{23}(z_2)=S_V^{23}(z_2)S_i^{13}(z_1),
   \label{eq:qva-double-cond1}\\
  &S_i^{12}(z_1)S_V^{23}(z_2)=S_V^{23}(z_2)S_i^{12}(z_1),\quad
   S_U^{12}(z_1)S_i^{23}(z_2)=S_2^{23}(z_2)S_i^{12}(z_1)
   \label{eq:qva-double-cond2}
\end{align}
for $i=1,2$.
Denote by $T(z)=S_1^{14}(z)S_2^{32}(-z)$ the twistor obtained in Lemma \ref{lem:double-twisted}.
Then $\mathfrak D_T(U\ot V)$ is a quantum vertex algebra with quantum Yang-Baxter operator
\begin{align*}
  &S_U^{13}(z)S_V^{24}(z)S_2^{14}(z)\inv  S_1^{14}(z)S_1^{32}(-z)\inv S_2^{32}(-z).
\end{align*}
\end{prop}

\begin{proof}
Let $T(z)=S_1^{14}(z)S_2^{32}(-z)$ be the twistor of $U\ot V$ given in Lemma \ref{lem:double-twisted}.
It is straightforward to verify that the relation \eqref{eq:qyb-com1} holds for $(T(z),T(z))$ and $(T(z),S_U^{13}(z)S_V^{24}(z))$, and the relation \eqref{eq:twistor-com} holds for $(T(z),S_U^{13}(z)S_V^{24}(z))$.
Then Theorem \ref{thm:qva-twistor} yields the result.
\end{proof}

\subsection{Deform by vertex bialgebras}\label{subsec:deform-vertex-bialgebra}

In \cite{JKLT-Defom-va}, the authors developed a method for obtaining quantum vertex algebras as deformations of vertex algebras via vertex bialgebras. In this subsection, we refine and extend this method using the language of twistors, with the aim of constructing quantum vertex algebras deformed from general quantum vertex algebras by using of vertex bialgebras.
Recalling the notions of vertex bialgebras and smash products given in \cite{Li-smash}.
\begin{de}
An \emph{(nonlocal) vertex bialgebra} is a (nonlocal) vertex algebra $V$ equipped with a classical coalgebra structure $(\Delta,\varepsilon)$ such that the coproduct $\Delta:V\to V\ot V$ and the counit $\varepsilon:V\to\C$ are homomorphisms of (nonlocal) vertex algebras.
\end{de}

\begin{rem}\label{rem:bialg-der}
Let $(H,\Delta,\varepsilon)$ be a bialgebra over $\C$ equipped with a derivation $\partial$.
Then $H$ is a nonlocal vertex bialgebra with vacuum $1$ and vertex operator map defined by
\begin{align*}
  Y(a,z)b=\left(e^{z\partial}a\right)b\quad\te{for } a,b\in H.
\end{align*}
We denote this nonlocal vertex bialgebra by $(H,\partial,\Delta,\varepsilon)$.
\end{rem}

\begin{de}
Let $(H,\Delta,\varepsilon)$ be a nonlocal vertex bialgebra.
A (left) \emph{$H$-module (nonlocal) vertex algebra} (\cite{Li-smash}) is a nonlocal vertex algebra $V$ equipped with a module structure $\tau(\cdot,z)$ on $V$ for $H$ viewed as a nonlocal vertex algebra such that
\begin{align}
  &\tau(h,z)v\in V\ot \C((z)),\qquad
  \tau(h,z)\vac=\varepsilon(h)\vac,
  \label{eq:mod-va-for-vertex-bialg1-2}\\
  &
  \tau(h,z_1)Y(u,z_2)v=\sum Y(\tau(h_{(1)},z_1-z_2)u,z_2)\tau(h_{(2)},z_1)v
  \label{eq:mod-va-for-vertex-bialg3}
\end{align}
for $h\in H$, $u,v\in V$, where $\Delta(h)=\sum h_{(1)}\ot h_{(2)}$ is the coproduct in the Sweedler notation.
\end{de}

\begin{thm}\cite[Theorem 4.9]{Li-smash}
Let $(H,\Delta,\varepsilon)$ be a nonlocal vertex bialgebra and let $(V,\tau)$ be an $H$-module nonlocal vertex algebra.
Set $V\sharp H=V\ot H$ as a vector space.
For $u,v\in V$, $h,h'\in H$ define
\begin{align*}
  Y^\sharp (u\ot h,z)(v\ot h')=\sum Y(u,z)\tau(h_{(1)},z)v\ot Y(h_{(2)},z)h',
\end{align*}
where $\Delta(h)=\sum h_{(1)}\ot h_{(2)}$.
Then $(V\sharp H,Y^\sharp, \vac\ot \vac)$ carries the structure of a nonlocal vertex algebra, which contains $V$ and $H$ canonically as nonlocal vertex subalgebras such that
\begin{align*}
  Y^\sharp(h,z_1)Y^\sharp(u,z_2)=\sum Y^\sharp (\tau(h_{(1)},z_1-z_2)u,z_2)Y^\sharp(h_{(2)},z_1)\quad\te{for }u\in V,\,h\in H.
\end{align*}
\end{thm}

Recall the following notion introduced in \cite{JKLT-Defom-va}.

\begin{de}
Let $(H,\Delta,\varepsilon)$ be a nonlocal vertex bialgebra.
A \emph{(left) $H$-comodule nonlocal vertex algebra} (\cite{JKLT-Defom-va}) is a nonlocal vertex algebra $V$ equipped with a homomorphism
$\rho:V\to V\ot H$ of nonlocal vertex algebras such that
\begin{align}
  (\rho\ot 1)\rho=\sigma^{23}(1\ot \Delta)\rho,\quad (1\ot \varepsilon)\rho=\te{Id}_V.
\end{align}
\end{de}

Now, we fix a nonlocal vertex bialgebra $H$.
For any vector space $U$, we note that $$\Hom(H,\Hom(U,U\ot\C((z))))$$ is a unital associative algebra,
where the multiplication is defined by
\begin{align*}
  (f\ast g)(h,z)u=\sum f(h_{(2)},z)g(h_{(1)},z)u
\end{align*}
for $f,g\in\Hom(H,\Hom(U,U\ot\C((z))))$,
and the unit $\varepsilon$ defined by
\begin{align*}
  \varepsilon(h,z)u=\varepsilon(h)u\quad\te{for }h\in H,\,u\in U.
\end{align*}
For $f,g\in\Hom(H,\Hom(U,U\ot\C((z))))$, we say that $f$ and $g$ commute if
\begin{align*}
  [f(h,z_1),g(k,z_2)]=0\quad \te{for }h,k\in H.
\end{align*}

\begin{lem}\label{lem:T-tau-mult}
Let $(V,\rho)$ be a left $H$-comodule nonlocal vertex algebra,
and let $U$ be a vector space.
For each $\tau\in \Hom(H,\Hom(U,U\ot\C((z))))$,
define
\begin{align}\label{eq:T-tau}
  T_\tau(z):&U\ot V\longrightarrow U\ot V\ot \C((z))\\
  &\nonumber u\ot v\mapsto \sum \tau(v_{(2)},-z)u\ot v_{(1)}
\end{align}
where $\rho(v)=\sum v_{(1)}\ot v_{(2)}\in V\ot H$.
Then $T_\varepsilon(z)=1_U\ot 1_V$ and
\begin{align}\label{eq:T-tau-mult}
  &T_{\tau_1\ast\tau_2}(z)=T_{\tau_1}(z)T_{\tau_2}(z).
\end{align}
Moreover, for $\tau_1,\tau_2\in \Hom(H,\Hom(U,U\ot \C((z))))$ we have that
\begin{align}
  &T_{\tau_1}^{23}(z_1)T_{\tau_2}^{13}(z_2)
  =T_{\tau_2}^{13}(z_2)T_{\tau_1}^{23}(z_1)\quad\te{on }U\ot U\ot V,\quad\te{if $H$ is cocommutative},\label{eq:T-tau-com2}\\
  &T_{\tau_1}^{12}(z_1)T_{\tau_2}^{13}(z_2)
  =T_{\tau_2}^{13}(z_2)T_{\tau_1}^{12}(z_1)\quad\te{on }U\ot V\ot V,\quad \te{if $\tau_1$ commutes with $\tau_2$}.
  \label{eq:T-tau-com3}
\end{align}
\end{lem}

\begin{proof}
The first statement is clear.
For $\sigma\in S_k$, we write
\begin{align*}
  \sigma=\left(\genfrac{}{}{0pt}{0}{1}{\sigma(1)}
  \genfrac{}{}{0pt}{0}{2}{\sigma(2)}
  \genfrac{}{}{0pt}{0}{\cdots}{\cdots}
  \genfrac{}{}{0pt}{0}{k}{\sigma(k)}\right).
\end{align*}
The relation \eqref{eq:T-tau-mult} follows from the identities below
\begin{align*}
  &T_{\tau_1\ast\tau_2}(z)=\tau_1^{12}(-z)\tau_2^{23}(-z)\binom{1234}{3412}\sigma^{34}\Delta^3\rho^2,\\
  &T_{\tau_1}(z)T_{\tau_2}(z)=\tau_1^{12}(-z)\tau_2^{23}(-z)\binom{1234}{3412}\rho^2\rho^2,
\end{align*}
and the fact that $\sigma^{34}\Delta^3\rho^2=\rho^2\rho^2$.
Note that
\begin{align}
  &T_{\tau_1}^{23}(z_1)T_{\tau_2}^{13}(z_2)=\tau_2^{12}(-z_2)\tau_1^{34}(-z_1)\binom{12345}{24531}\rho^3\rho^3,\label{eq:deforming-triple-twistor-1}\\
  &T_{\tau_1}^{12}(z_1)T_{\tau_2}^{13}(z_2)
  =\tau_1^{12}(-z_1)\tau_2^{23}(-z_2)\binom{12345}{34152}\rho^2\rho^3,
  \label{eq:deforming-triple-twistor-2}\\
  &T_{\tau_2}^{13}(z_2)T_{\tau_1}^{23}(z_1)=
  {\tau_2}^{12}(-z_2){\tau_1}^{34}(-z_1)\binom{12345}{24513}
  \rho^3\rho^3,\label{eq:deforming-triple-twistor-4}\\
  &T_{\tau_2}^{13}(z_2)T_{\tau_1}^{12}(z_1)=
  {\tau_2}^{12}(-z_2){\tau_1}^{23}(-z_1)\binom{12345}{34251}
  \rho^3\rho^3.\label{eq:deforming-triple-twistor-5}
\end{align}
Then the relation \eqref{eq:T-tau-com2} follows from \eqref{eq:deforming-triple-twistor-1}, \eqref{eq:deforming-triple-twistor-4} and
the fact
\begin{align*}
  &\rho^3\rho^3=\sigma^{45}\rho^3\rho^3,\quad\te{since $H$ is cocommutative},
\end{align*}
and the relation \eqref{eq:T-tau-com3} follows from \eqref{eq:deforming-triple-twistor-2}, \eqref{eq:deforming-triple-twistor-5} and the fact
\begin{align*}
  &\tau_1^{12}(-z_1)\tau_2^{23}(-z_2)
  =\tau_2^{12}(-z_2)\tau_1^{23}(-z_1)\sigma^{12},\quad\te{since $\tau_1$ commutes with $\tau_2$}.
\end{align*}
\end{proof}

We furthermore fix a left $H$-comodule nonlocal vertex algebra $(V,\rho)$.
It is straightforward to verify the following result.
\begin{lem}\label{lem:def-twisting-op}
Let $(U,\tau)$ be an $H$-module nonlocal vertex algebra.
Then $T_\tau(z)$ is a twisting operator for $(U,V)$.
\end{lem}

\begin{lem}\label{lem:T-S-com}
Suppose further that $V$ is a quantum vertex algebra with quantum Yang-Baxter operator $S_V(z)$, such that
\begin{align}\label{eq:tau-rho-S-com2}
  (\rho\ot 1)S(z)=S^{13}(z)(\rho\ot 1),\quad
  (1\ot \rho)S(z)=S^{12}(z)(1\ot \rho).
\end{align}
Let $U$ be a quantum vertex algebra with quantum Yang-Baxter operator $S_U(z)$, and let $\tau$ be an $H$-module nonlocal vertex algebra structure on $U$ such that
\begin{align}
  &(\tau(h,z_1)\ot 1)S_U(z_2)=S_U(z_2)(\tau(h,z_1)\ot 1),\label{eq:tau-rho-S-com1-1}\\
  &(1\ot \tau(h,z_1))S_U(z_2)=S_U(z_2)(1\ot \tau(h,z_1)).
  \label{eq:tau-rho-S-com1-2}
\end{align}
Then we have that
\begin{align*}
  &T_\tau^{12}(z_1)S_V^{23}(z_2)
  =S_V^{23}(z_2)T_\tau^{12}(z_1),\quad
   T_\tau^{12}(z_1)S_U^{13}(z_2)
  =S_U^{13}(z_2)T_\tau^{12}(z_1),\\
  &T_\tau^{13}(z_1)S_V^{23}(z_2)
  =S_V^{23}(z_2)T_\tau^{12}(z_1),\quad
   S_U^{12}(z_1)T_\tau^{23}(z_2)
  =T_\tau^{23}(z_1)S_U^{12}(z_1).
\end{align*}
\end{lem}

\begin{proof}
Note that
\begin{align*}
  &T_\tau^{12}(z_1)S_V^{23}(z_2)
  =S_V^{23}(z_2)\tau^{12}(-z_1)\binom{1234}{2314}\rho^2
  =S_V^{23}(z_2)T_\tau^{12}(z_1),\\
  &T_\tau^{12}(z_1)S_U^{13}(z_2)
  =S_U^{13}(z_2)\tau^{12}(-z_1)\binom{1234}{2314}\rho^2
  =S_U^{13}(z_2)T_\tau^{12}(z_1),\\
  &T_\tau^{13}(z_1)S_V^{23}(z_2)
  =S_V^{23}(z_2)\tau^{12}(-z_1)\binom{1234}{2341}\rho^3
  =S_V^{23}(z_2)T_\tau^{12}(z_1),\\
  &S_U^{12}(z_1)T_\tau^{23}(z_2)
  =\tau^{23}(-z_2)S_U^{13}(z_2)\binom{1234}{1342}\rho^3
  =T_\tau^{23}(z_1)S_U^{12}(z_1),
\end{align*}
which completes the proof of lemma.
\end{proof}

%

Recall the following notion introduced in \cite{JKLT-Defom-va}.

\begin{de}\label{de:deform-triple}
Let $V$ be a nonlocal vertex algebra.
A \emph{deforming triple} for $V$ is a triple $(H,\rho,\tau)$,
where $H$ is a nonlocal vertex bialgebra,
$(V,\rho)$ is a left $H$-comodule nonlocal vertex algebra and $(V,\tau)$ is an $H$-module nonlocal vertex algebra, such that
\begin{align}\label{eq:rho-tau-compatible}
  \rho(\tau(h,z)v)=(\tau(h,z)\ot 1)\rho(v)\quad \te{for }h\in H,\,v\in V.
\end{align}
\end{de}



Let $\mathfrak L_H^\rho(V)$ denote the set of all $H$-module structures on $V$ for which $(H,\rho,\tau)$ forms a deforming triple.

\begin{lem}\label{lem:T-tau-com}
Let $(U,\tau_1)$ be an $H$-module nonlocal vertex algebra and let $\tau_2\in \mathfrak L_H^\rho(V)$.
Then
\begin{align}\label{eq:T-tau-com1}
  T_{\tau_1}^{12}(z_1)T_{\tau_2}^{23}(z_2)
  =T_{\tau_2}^{23}(z_2)T_{\tau_1}^{12}(z_1)
  \quad\te{on }U\ot V\ot V.
\end{align}
\end{lem}

\begin{proof}
The lemma follows from the following two equations
\begin{align*}
  &T_{\tau_1}^{12}(z_1)T_{\tau_2}^{23}(z_2)=\tau_1^{12}(-z_1)\tau_2^{34}(-z_2)\binom{12345}{24153}\rho^2\rho^3,
  \\
  &T_{\tau_2}^{23}(z_2)T_{\tau_1}^{12}(z_1)={\tau_1}^{12}(-z_1){\tau_2}^{34}(-z_2)\binom{12345}{24153}
  \rho^2\rho^3.
\end{align*}
\end{proof}

\begin{prop}\label{prop:deforming-triple-twistor}
For $\tau\in\mathfrak L_H^\rho(V)$,
$T_\tau(z)$ is a twistor of $V$.
Moreover, if $H$ is cocommutative and $\tau$ commutes with itself,
then \eqref{eq:qyb-com1} holds for $(T_\tau(z),T_\tau(z))$.
\end{prop}

\begin{proof}
Lemmas \ref{lem:def-twisting-op} and \ref{lem:T-tau-com} proves the first statement, and
the moreover statement follows immediate from Lemmas \ref{lem:T-tau-mult}.
\end{proof}

As an immediate consequence of Theorem \ref{thm:nonlocal-VA-twistor-deform}, Lemma \ref{lem:T-tau-mult}
and Proposition \ref{prop:deforming-triple-twistor}, we have the following result.
\begin{coro}\label{coro:D-T-tau-Y}
For $\tau\in \mathfrak L_H^\rho(V)$, we have that
\begin{align*}
  \mathfrak D_{T_\tau}(Y)(u,z)=\sum Y(u_{(1)},z)\tau(u_{(2)},z),
  \quad\te{where}\,\,\rho(u)=\sum u_{(1)}\ot u_{(2)}\in V\ot H.
\end{align*}
\end{coro}

The following result was given in \cite[Theorem 2.25, Proposition 2,26]{JKLT-Defom-va}.

\begin{prop}\label{prop:L-H-rho-V-comosition-pre}
Suppose that $H$ is cocommutative.
For $\tau\in\mathfrak L_H^\rho(V)$, $(\mathfrak D_{T_\tau}(V), \rho)$ is an $H$-comodule nonlocal vertex algebra.
Moreover, $\rho:\mathfrak L_H^\rho(V)\to V\sharp H$ is a nonlocal vertex algebra homomorphism.
\end{prop}

The following result was given in \cite[Proposition 3.3]{JKLT-Defom-va}.

\begin{prop}\label{prop:L-H-rho-V-compostition}
Suppose that $H$ is cocommutative.
Let $\tau$ and $\tau'$ be commuting elements in $\mathfrak L_H^\rho(V)$.
Then $\tau\ast\tau'\in\mathfrak L_H^\rho(V)$
and $\tau\ast\tau'=\tau'\ast\tau$.
Moreover, $\tau\in \mathfrak L_H^\rho\left(\mathfrak D_{T_{\tau'}}(V)\right)$.
\end{prop}

The following result was given in \cite[Proposition 3.4]{JKLT-Defom-va},
and is a consequence of Propositions \ref{prop:prod-twistor},
\ref{prop:deforming-triple-twistor},
\ref{prop:L-H-rho-V-compostition}
and Lemma \ref{lem:T-tau-mult}.

\begin{prop}\label{prop:L-H-rho-V-compostition2}
Suppose that $H$ is cocommutative.
Let $\tau$ and $\tau'$ be commuting elements in $\mathfrak L_H^\rho(V)$.
Then
\begin{align}
  \mathfrak D_{T_\tau}\left(\mathfrak D_{T_{\tau'}}(V)\right)=\mathfrak D_{T_{\tau\ast\tau'}}(V).
\end{align}
\end{prop}

An element $\tau\in\mathfrak L_H^\rho(V)$ is called \emph{invertible} if there exists $\tau\inv\in \mathfrak L_H^\rho(V)$, such that $\tau\ast \tau\inv=\varepsilon$.
The following result is a generalization of \cite[Theorem 3.6]{JKLT-Defom-va}.

\begin{thm}\label{thm:qva}
Let $V$ be a quantum vertex algebra with quantum Yang-Baxter operator $S(z)$.
Suppose that $(H,\rho,\tau)$ is a deforming triple of $V$, such that $H$ is cocommutative, $\tau$ is commutes with itself and $\tau$ is invertible in $\mathfrak L_H^\rho(V)$.
Suppose further that $S(z)$, $\rho$ and $\tau$ satisfy the relations \eqref{eq:tau-rho-S-com2}, \eqref{eq:tau-rho-S-com1-1} and \eqref{eq:tau-rho-S-com1-2}.
Then $\mathfrak D_{T_\tau}(V)$ is a quantum vertex algebra with quantum Yang-Baxter operator $S_T(z)$ defined by
\begin{align*}
  S_T(z)=T_\tau^{21}(-z)\inv S(z) T_\tau(z).
\end{align*}
\end{thm}

\begin{proof}
Applying Proposition \ref{prop:deforming-triple-twistor},
we get a twistor $T_\tau(z)$ such that the relation \eqref{eq:qyb-com1} holds for $(T_\tau(z),T_\tau(z))$.
Since $\tau\ast\tau\inv=\varepsilon$, we get from Lemma \ref{lem:T-tau-mult} that
\begin{align*}
  T_\tau(z)T_{\tau\inv}(z)=T_\varepsilon(z)=1_{V\ot V},
\end{align*}
which proves that $T_\tau(z)$ is invertible.
Lemma \ref{lem:T-S-com} proves that the relations \eqref{eq:qyb-com1} and \eqref{eq:twistor-com} hold for $(S(z),T_\tau(z))$.
Then Theorem \ref{thm:qva-twistor} yields the theorem.
\end{proof}

Using deforming triples, we obtain two twisting operators satisfying the hypotheses of Proposition \ref{prop:qva-double-twisted}, as described in the following proposition.

\begin{prop}\label{prop:double-twisted-vba}
Let $U$ and $V$ be two quantum vertex algebras with quantum Yang-Baxter operators $S_U(z)$ and $S_V(z)$, respectively.
Let $H$ be a cocommutative nonlocal vertex bialgebra,
let $\rho$ be an $H$-comodule nonlocal vertex algebra structure on $V$, and let $\tau_1$, $\tau_2$ be two invertible $H$-module nonlocal vertex algebra structures on $U$, such that $\tau_i$ commutes with $\tau_j$ for $i,j=1,2$.
Suppose that the relations \eqref{eq:tau-rho-S-com1-1}, \eqref{eq:tau-rho-S-com1-2} hold for $\tau_i$ and $S_U(z)$ ($i=1,2$)
and the relation \eqref{eq:tau-rho-S-com2} holds for $\rho$ and $S_V(z)$.
Then the 
two linear maps $T_{\tau_1}(z)$ and $T_{\tau_2}(z)$ constructed via \eqref{eq:T-tau} are two invertible twisting operators for $(U,V)$ satisfying the hypotheses of Proposition \ref{prop:qva-double-twisted}.
\end{prop}

\begin{proof}
Applying Lemma \ref{lem:def-twisting-op}, we get that $T_{\tau_1}(z)$ and $T_{\tau_2}(z)$ are two twisting operators for $(U,V)$.
Note that both $\tau_1$ and $\tau_2$ are invertible.
We get from Lemma \ref{lem:T-tau-mult} that
\begin{align*}
  T_{\tau_i}(z)T_{\tau_i\inv}(z)=T_{\tau_i\ast\tau_i\inv}(z)
  =T_{\varepsilon}(z)=1_U\ot 1_V\quad\te{for }i=1,2,
\end{align*}
which shows that both $T_{\tau_1}(z)$ and $T_{\tau_2}(z)$ are invertible.
Since $H$ is cocommutative and $\tau_i$ commutes with $\tau_j$ ($i,j=1,2$), we get from Lemma \ref{lem:T-tau-mult} that the relation \eqref{eq:qyb-com1} holds for $(T_{\tau_i}(z),T_{\tau_j}(z))$ ($i,j=1,2$).
From Lemma \ref{lem:T-S-com}, we have that the relations \eqref{eq:qva-double-cond1} and \eqref{eq:qva-double-cond2} hold for $T_{\tau_1}(z)$, $T_{\tau_2}(z)$, $S_U(z)$ and $S_V(z)$.
\end{proof}

\section{Quantum affine algebras at root of unity}\label{sec:qaff-rtu}

Let $I=\{1,2,\dots,n\}$,
let $A=(a_{ij})_{i,j\in I}$ be a Cartan matrix, and let $\g=\g(A)$ be the corresponding finite dimensional simple Lie algebra over $\C$.
We fix a realization $(\h,\Pi,\Pi^\vee)$ of $\g$, where $\h$ is a Cartan subalgebra of $\g$, $\Pi=\set{\al_i}{i\in I}\subset\h^\ast$ is the set of simple roots and
$\Pi^\vee=\set{h_i}{i\in I}\subset\h$ is the set of simple coroots.
Define
\begin{align}
    I_L=\set{i\in I}{\al_i\,\,\te{is a long root}},\quad
    I_S=\set{i\in I}{\al_i\,\,\te{is a short root}},
\end{align}
and define
\begin{align}\label{eq:def-r-i}
    r=\begin{cases}
        1,&\mbox{if $\g$ is simply-laced},\\
        2,&\mbox{if $\g$ is of type $B_n,C_n,F_4$},\\
        3,&\mbox{if $\g$ is of type $G_2$},
    \end{cases}
    \quad r_i=\begin{cases}
        1,&\mbox{if }i\in I_S,\\
        r,&\mbox{if }i\in I_L.
    \end{cases}
\end{align}
Let $\wh I=\{0\}\cup I$, and let $\wh A=(a_{ij})_{i,j\in \wh I}$ be the generalized Cartan matrix of the untwisted affine Lie algebra $\wh\g=\g\ot\C[t,t\inv]\oplus \C c$ of $\g$.
Let $r_0=r$.
Then we get from \cite[(6.22)]{Kac-book} and \cite[Table Aff]{Kac-book} that
\begin{align}\label{eq:sym}
    r_ia_{ij}=r_ja_{ji}\quad\te{for }i,j\in \wh I.
\end{align}
Denote by $\theta^\vee$ the coroot of $\g$ of maximal height, and set $h_0=c-\theta^\vee$.
Then $\wh\Pi^\vee=\{h_0\}\uplus \Pi^\vee$ forms a set of simple coroots of $\wh\g$.
Define a bilinear map $\<\cdot,\cdot\>$ on $\h\oplus\C c$ by
\begin{align}
    \<h_i,h_j\>=ra_{ij}/r_j\quad\te{for }i,j\in \wh I.
\end{align}
Let $\Lambda$ be the weight lattice of $\g$, that is,
\begin{align}
  \Lambda=\set{\lambda\in \h^\ast}{\lambda(h_i)\in\Z,\,\,i\in I}.
\end{align}

Let $q$ be an indeterminate, let $\C(q)$ be the field of rational functions of $q$ with complex coefficients, and let $\C[q,q\inv]$ be the ring of Laurent polynomials.
For $r,n\in\N$, $n\ge r$, and an invertible element $a\in \C[q,q\inv]$, we define
\begin{align*}
  [n]_a=\frac{a^n-a^{-n}}{a-a\inv},\quad [n]_a!=[n]_a[n-1]_a\cdots [2]_a[1]_a,\quad \qb{n}{r}_a=\frac{[n]_a!}{[r]_a![n-r]_a!}.
\end{align*}
For each real coroot $h$ of $\wh\g$, we set
\begin{align*}
  q_h=q^{2r/\<h,h\>}.
\end{align*}
For convenience, we set
\begin{align*}
  q_i=q_{h_i}\quad\te{for }i\in \wh I.
\end{align*}
Then it is easy to see that $q_i=q^{r_i}$.
Let $\U_q(\wh\g)$ be the associative algebra over $\C(q)$ which is generated by elements $x_i^\pm, k_i^{\pm 1}$ ($i\in \wh I$) with the following defining relations:
\begin{align*}
  &k_ik_i\inv=k_i\inv k_i=1,\quad k_ik_j=k_jk_i,\\
  &k_ix_j^\pm k_i\inv=q_i^{\pm a_{ij}}x_j^\pm,\quad
   [x_i^+,x_j^-]=\delta_{ij}\frac{k_i-k_i\inv}{q_i-q_i\inv},\\
  &\sum_{k=0}^{1-a_{ij}}(-1)^k\qb{1-a_{ij}}{k}_{q_i}(x_i^\pm)^kx_j^\pm (x_i^\pm)^{1-a_{ij}-k}=0\quad \te{if }i\ne j.
\end{align*}
View $\U_q(\wh\g)$ as an $\C[q,q\inv]$-algebra. Let $\U_q^{\res}(\wh\g)$ be the $\C[q,q\inv]$-subalgebra of $\U_q(\wh\g)$ generated by
\begin{align*}
  &(x_i^\pm)^{(k)}:=(x_i^\pm)^k/[k]_{q_i}!,\quad k_i^{\pm 1}\quad \te{for }i\in\wh I,\quad k\in\Z_+.
\end{align*}
For each invertible element $a$ in an arbitrary associative algebra over $\C(q)$, set
\begin{align*}
  &\qb{a}{0}_q=1,\quad
  \qb{a}{m}_q=\prod_{s=1}^{m}\frac{aq^{1-s}-a\inv q^{s-1}}{q^s-q^{-s}}\quad\te{for }m\in\Z_+.
\end{align*}
Then
\begin{align*}
  \qb{k_i}{m}_{q_i}\in \U_q^{\res}(\wh\g)\quad\te{for }i\in \wh I,\,m\in\N.
\end{align*}
Note that $\displaystyle\qb{k_i}{m}_{q_i}$ is denoted by $\displaystyle\qb{k_i;0}{m}$ in \cite{CP-qaff-rut}.

\begin{lem}\label{lem:qb-mult}
Let $a,b$ be two invertible elements in an arbitrary associative algebra over $\C(q)$, and let $m\in\N$.
Suppose $ab=ba$. Then
\begin{align*}
  \qb{ab}{m}_q=\sum_{k=0}^ma^{m-k}b^{-k}\qb{b}{m-k}_q\qb{a}{k}_q.
\end{align*}
\end{lem}

\begin{proof}
Recall the following formula
\begin{align*}
  \prod_{s=0}^{m-1}(1+q^{2s}z)=\sum_{s=0}^mq^{s(m-1)}\qb{m}{s}_qz^s.
\end{align*}
Then we have
\begin{align*}
  &\qb{a}{m}_q=\frac{1}{(q-q\inv)^m}
  \sum_{k=0}^m(-1)^kq^{-m(m-1)/2}q^{k(m-1)}
  \frac{a^{m-2k}}{[k]_q![m-k]_q!}.
\end{align*}
Then
\begin{align*}
  &\sum_{k=0}^ma^{m-k}b^{-k}\qb{b}{m-k}_q\qb{a}{k}_q\\
  =&\frac{1}{(q-q\inv)^m}
  \sum_{k=0}^m\sum_{s=0}^k\sum_{t=0}^{m-k}
  (-1)^{s+t} q^{-k(k-1)/2}q^{-(m-k)(m-k-1)/2}q^{s(k-1)}
  q^{t(m-k-1)}\\
  &\quad\times
  \frac{a^{m-2s}b^{m-2k-2t}}{[s]_q![k-s]_q![t]_q![m-k-t]_q!}\\
  =&\frac{1}{(q-q\inv)^m}
  \sum_{s=0}^m
  (-1)^sq^{-m(m-1)/2}q^{s(m-1)}\frac{(ab)^{m-2s}}{[s]_q![m-s]_q!}\\
  &\quad\times
  \sum_{k=s}^m\sum_{t=0}^{m-k}
  q^{(k+t-s)(m-k-t)}
  \qb{m-s}{k+t-s}_q b^{2s-2k-2t}
  (-1)^tq^{t(k+t-s-1)}\qb{k+t-s}{t}_q\\
  =&\frac{1}{(q-q\inv)^m}
  \sum_{s=0}^m
  (-1)^sq^{-m(m-1)/2}q^{s(m-1)}\frac{(ab)^{m-2s}}{[s]_q![m-s]_q!}\\
  &\quad\times
  \sum_{k=0}^{m-s}q^{k(m-s)}\qb{m-s}{k}_q b^{-2k}
  \sum_{t=0}^k q^{t(k-1)}\qb{k}{t}_q (-1)^t\\
  =&\frac{1}{(q-q\inv)^m}
  \sum_{s=0}^m
  (-1)^sq^{-m(m-1)/2}q^{s(m-1)}\frac{(ab)^{m-2s}}{[s]_q![m-s]_q!}\\
  &\quad\times
  \sum_{k=0}^{m-s}q^{k(m-s)}\qb{m-s}{k}_q b^{-2k}
  \prod_{t=0}^{k-1}(1-q^{2t})\\
  =&\frac{1}{(q-q\inv)^m}
  \sum_{s=0}^m
  (-1)^sq^{-m(m-1)/2}q^{s(m-1)}\frac{(ab)^{m-2s}}{[s]_q![m-s]_q!}
  =\qb{ab}{m},
\end{align*}
which completes the proof.
\end{proof}

Let $T_i$ ($i\in \wh I$) be the $\C(q)$-algebra automorphisms of $\U_q(\wh\g)$ defined by Lusztig \cite{Luztig-quantum-book}:
\begin{align*}
  &T_i(x_i^+)=-x_i^-k_i,\quad T_i(x_i^-)=-k_i\inv x_i^+,\quad T_i(k_j)=k_i^{-a_{ij}}k_j,\\
  &T_i(x_j^+)=\sum_{s=0}^{-a_{ij}}(-1)^{s-a_{ij}}q_i^{-s}
  (x_i^+)^{(-a_{ij}-s)}
  x_j^+(x_i^+)^{(s)}\quad\te{if }i\ne j,\\
  &T_i(x_j^-)=\sum_{s=0}^{-a_{ij}}(-1)^{s-a_{ij}}q_i^s(x_i^-)^{(s)}
  x_j^-(x_i^-)^{(-a_{ij}-s)}\quad\te{if }i\ne j.
\end{align*}
In fact \cite{Luztig-quantum-book}, $T_i$ preserves $\U_q^{\res}(\wh\g)$.
Denote by $\Delta^\times$ the set of all real coroots of $\wh\g$.
For each $h=\sum_{i\in \wh I}a_i h_i\in \oplus_{i\in \wh I}\Z h_i\in\Delta^\times$ of $\wh\g$, we set
\begin{align*}
  k_h^{\pm 1}=\prod_{i\in \wh I}k_i^{\pm a_i}.
\end{align*}
Then for $h\in\Delta^{\vee,\times}$, $i\in \wh I$ and $m\in\N$, we have
\begin{align*}
  T_i(k_h^{\pm 1})=k_{s_i(h)}^{\pm 1},\,\,
  T_i\left(\qb{k_h}{m}_{q_h}\right)=\qb{k_{s_i(h)}}{m}_{q_h},
  \quad\te{where }
  s_i:\Delta^{\vee,\times}\to \Delta^{\vee,\times}:\,\,h\mapsto h-\al_i(h)h_i.
\end{align*}
Note that each $s_i$ is bijective preserving the bilinear form $\<\cdot,\cdot\>$, and for each $h\in\Delta^{\vee,\times}$, there exist $i_1,\dots,i_n\in \wh I$, such that
$s_{i_1}\circ s_{i_2}\circ \cdots \circ s_{i_n}(h)\in \wh I$.
It follows that
\begin{align*}
  k_h^{\pm 1},\quad \qb{k_h;n}{m}_{q_h}\in \U_q^{\res}(\wh\g)
  \quad\te{for }h\in \Delta^\times,\,\,n\in\Z,\,\,m\in\Z_+.
\end{align*}

We also need another realization of $\U_q(\wh\g)$,
due to \cite{Dr-new, beck}.
\begin{thm}
$\U_q(\wh\g)$
is the $\C(q)$-algebra generated by central elements $\gamma^{\pm 1}$ and
\begin{eqnarray}\label{eq:tqagenerators}
\set{ k_i^{\pm 1},\, h_i(m),\  x^\pm_i(n)}
{
   i\in I, \, m\in\Z\setminus\{0\},\,  n\in\Z
},
\end{eqnarray}
subject to the relations in terms of the generating functions
\begin{eqnarray*}
 &&\psi_i^\pm(z)=\sum_{\pm m\ge 0}\psi_i^\pm(m)z^{-m} =k_i^{\pm 1}\exp
    \left(
        \pm (q_i-q_i\inv)\sum\limits_{\pm m> 0}h_i(m)z^{-m}
    \right),\\
 &&x^\pm_i(z)=\sum\limits_{m\in \Z} x^\pm_i(m)z^{-m},
\end{eqnarray*}
The relations are ($i,j\in I$):
%
%
\begin{align}
\label{Q0}\tag{Q0}&\gamma\gamma\inv=\gamma\inv\gamma=1,\quad k_i^{\pm 1} k_i^{\mp 1}=1,\quad k_i^{\pm 1}h_j(m)=h_j(m)k_i^{\pm 1},\\
\label{Q1}\tag{Q1}&  \psi_i^\pm(z_1)\psi_j^\pm(z_2)=\psi_j^\pm(z_2)\psi_i^\pm(z_1) ,\\
\label{Q2}\tag{Q2}& \psi^+_i(z_1)\psi^-_j(z_2)=\psi^-_j(z_2)\psi^+_i(z_1)
    \iota_{z_1,z_2}g_{ij}(z_2/z_1)\inverse g_{ij}(\gamma^{-2}z_2/z_1),\\
\label{Q3}\tag{Q3}& k_ix_j^\pm(z)k_i\inv=q^{\pm r_ia_{ij}}x_j^\pm(z),\\
\label{Q4}\tag{Q4}&\psi^+_i(z_1)x^\pm_j(z_2)=x^\pm_j(z_2)\psi^+_i(z_1)
    \iota_{z_1,z_2}g_{ij}(\gamma^{-1} z_2/z_1)^{\pm 1},\\
\label{Q5}\tag{Q5}&\psi^-_i(z_1)x^\pm_j(z_2)=x^\pm_j(z_2)\psi^-_i(z_1)
    \iota_{z_2,z_1}g_{ji}(\gamma^{\mp 1}z_1/z_2)^{\mp 1},\\
\label{Q6}\tag{Q6}&[x_i^+(z_1),x_j^-(z_2)]
=\frac{\delta_{ij}}{q_i-q_i\inv}
    \Bigg(
        \psi_i^+(z_2\gamma^{-1} )\delta\left(
            \frac{z_2}{z_1}
        \right)
        -
        \psi_i^-(z_2\gamma^{-1})\delta\left(
            \frac{z_2 \gamma^{-2} }{z_1}
        \right)
    \Bigg),\\
\label{Q7}\tag{Q7}&(z_1-q_i^{\pm a_{ij}}z_2)x^\pm_i(z_1)x^\pm_j(z_2)=
    ( q_i^{\pm a_{ij}}z_1-z_2)x^\pm_j(z_2)x^\pm_i(z_1),\\
\label{Q8}\tag{Q8}&\sum_{\sigma\in S_{{m}_{ij}}}\sum_{k=0}^{{m}_{ij}}
    (-1)^k \qb {{m}_{ij}}{k}_{q_i}x_i^\pm(z_{\sigma(1)})\cdots x_i^\pm(z_{\sigma(k)}) x_j^\pm(w)\\
&\nonumber\quad\quad\times
       x_i^\pm(z_{\sigma(k+1)})\cdots x_i^\pm(z_{\sigma({m}_{ij})})
\ =0,\quad
    \te{if}\ \ {a}_{ij}\le 0,
\end{align}
where $m_{ij}=1-a_{ij}$, $g_{ij}(z)=(q_i^{a_{ij}}-z)/(1-q_i^{a_{ij}}z)$,
and the map $\iota_{z_1,z_2}$ is defined as in \cite[\S 3.1]{fhl}.
\end{thm}

Let $\U_q(\wh\g)^\pm$ be the subalgebra of $\U_q(\wh\g)$ generated by $x_i^\pm(n)$, $i\in I$, $n\in \Z$, and let
$\U_q(\wh\g)^0$ be the subalgebra of $\U_q(\wh\g)$ generated by $h_i(m)$, $k_i^{\pm 1}$, $\gamma^{\pm 1}$, $i\in I$, $0\ne m\in \Z$.
Then $\U_q(\wh\g)^0$ (resp. $\U_q(\wh\g)^\pm$) is isomorphic to the $\C(q)$-algebra generated by $h_i(m)$, $k_i^{\pm 1}$, $\gamma^{\pm 1}$
(resp. $x_i^\pm(n)$) subject to relations \eqref{Q0}, \eqref{Q1}, \eqref{Q2} (resp. \eqref{Q7}, \eqref{Q8}).
Moreover, the multiplication gives an isomorphism of vector spaces
\begin{align*}
  \U_q(\wh\g)\cong \U_q(\wh\g)^-\ot_{\C(q)} \U_q(\wh\g)^0\ot_{\C(q)} \U_q(\wh\g)^+.
\end{align*}



Denoted by $\U_q^{\res}(\wh \g)^\pm$ the $\C[q,q\inv]$-subalgebra of $\U_q(\wh\g)$ generated by $(x_i^\pm(n))^{(m)}$, $i\in I$, $n\in \Z$,
$m\in \Z_+$.
Denoted by $\U_q^{\res}(\wh\g)^0$ the $\C[q,q\inv]$-subalgebra of $\U_q(\wh\g)$ generated by $k_i$, $\displaystyle\qb{k_i}{m}_{q_i}$ for $i\in \wh I$, $m\in\Z_+$, and $\wt h_i(n)$ for $i\in I$, $0\ne n\in\Z$, where $$\wt h_i(n)=n h_i(n)/[n]_{q_i}.$$

\begin{rem}
For each $i\in I$, we have that
\begin{align}
  \label{Q9}\tag{Q9}\psi_i^\pm(z)=k_i^{\pm 1}
  \exp\left( \pm\sum_{\pm m>0}(q_i^m-q_i^{-m})\wt h_i(m)\frac{z^{-m}}{m} \right).
\end{align}
\end{rem}

The following result was proved in \cite[Proposition 6.1]{CP-qaff-rut}.

\begin{prop}\label{prop:res-form}
We have $\U_q^{\res}(\wh\g)^\pm\subset \U_q^{\res}(\wh\g)$, $\U_q^{\res}(\wh\g)^0\subset \U_q^{\res}(\wh\g)$.
The algebra $\U_q^{\res}(\wh\g)$ is generated by the subalgebras $\U_q^{\res}(\wh\g)^+$, $\U_q^{\res}(\wh\g)^-$, $\gamma^{\pm 1}$ and $k_i^{\pm 1}$, $i\in I$.
Moreover, the multiplication gives an isomorphism of $\C[q,q\inv]$-modules
\begin{align*}
  &\U_q^{\res}(\wh\g)\cong \U_q^{\res}(\wh\g)^-\ot_{\C[q,q\inv]} \U_q^{\res}(\wh\g)^0\ot_{\C[q,q\inv]} \U_q^{\res}(\wh\g)^+.
\end{align*}
\end{prop}

Let $\wp\in \Z$ such that $\wp>2r$, and let $\ell$ be an integer.
Let $\zeta$ be a $\wp$-th primitive root of unity.
For $i\in \wh I$, we set $\zeta_i=\zeta^{r_i}$ and $\wp_i=\wp/\gcd(\wp,2r_i)$.
Denote by $\C_\zeta$ the one-dimensional $\C[q,q\inv]$-module by letting $q\mapsto \zeta$.
Define
\begin{align}
  \U_\zeta(\wh\g)=\U_q^{\res}(\wh\g)\ot_{\C[q,q\inv]}\C_\zeta.
\end{align}

Note that $k_0^{\pm 1}$ can also be written as follows:
\begin{align*}
  &k_0^{\pm 1}=\gamma^{\pm 1} k_{\theta^\vee}^{\mp 1}.
\end{align*}
Since $\<\theta^\vee,\theta^\vee\>=\<h_0,h_0\>$, Lemma \ref{lem:qb-mult} yields the following result.
\begin{lem}
For each $m\in\N$,
\begin{align*}
  \qb{\gamma}{m}_{q_0}\in \U_q^{\res}(\wh\g).
\end{align*}
\end{lem}

A $\U_\zeta(\wh\g)$-module $W$ is called of \emph{level $\ell$} if
\begin{align*}
  \gamma=\zeta_0^{\ell}\quad\te{and}\quad \qb{\gamma}{\wp_0}_{q_0}=\qb{\ell}{\wp_0}_{\zeta_0}\quad\te{on }W,
\end{align*}
is called \emph{smooth} if for each $w\in W$, $\exists N\in\N$ such that $\wt h_i(n)w=0=\psi_i^+(n)w=x_i^\pm(n)w$ for $n>N$,
and is called \emph{weighted} if
  $W=\oplus_{\lambda\in \Lambda}W_\lambda$, where
      \begin{align*}
      W_\lambda=\set{w\in W}{  k_iw=\zeta_i^{\lambda(h_i)}w,\,
        \qb{k_i}{\wp_i}_{q_i}w=\qb{\lambda(h_i)}{\wp_i}_{\zeta_i}w,
        \,\,i\in I}.
      \end{align*}
%
%

Let $\mathcal R_\zeta^\ell(\wh\g)$ be the category of smooth $\U_\zeta(\wh\g)$-modules of level $\ell$.
The main purpose of the rest of this section is to give another description of the object in $\mathcal R_\zeta^\ell(\wh\g)$.
Define
\begin{align}
  &\vvp{\,\cdot\,}:\Z[q,q\inv]\longrightarrow \Z\quad\te{by}\quad
  \sum_{n\in\Z}a_nq^n\mapsto \sum_{n\in\Z}a_{n\wp},\\
  &\vvinf{\,\cdot\,}:\Z[q,q\inv]\longrightarrow \Z\quad\te{by}\quad
  \sum_{n\in\Z}a_nq^n\mapsto a_0.
\end{align}

\begin{de}
Let $\mathcal R_\zeta'^\ell(\wh\g)$ be the category, whose objects are vector spaces over $\C$ equipped with fields
$(i\in I)$
\begin{align*}
  H_i(z)=\sum_{n\in\Z}H_i(n)z^{-n},\quad\Psi_i^\pm(z)=\sum_{n\in\Z}\Psi_i^\pm(n)z^{-n},\quad X_i^\pm(z)=\sum_{n\in\Z}X_i^\pm(n)z^{-n}\in \E(W),
\end{align*}
satisfying the relations
\begin{align}
  &\label{zeta1}\tag{$\zeta$1}[H_i(z_1),H_j(z_2)]\\
  &\quad\nonumber
  =\sum_{s\in \Z_\wp}\left(
  \vvp{ [a_{ij}]_{q_i}[r\ell/r_j]_{q_j}
  q^{-r\ell-s}}\iota_{z_1,z_2}
  -
  \vvp{ [a_{ij}]_{q_i}[r\ell/r_j]_{q_j}
  q^{r\ell-s}}\iota_{z_2,z_1}\right)
  \frac{\zeta^sz_2/z_1}{(1-\zeta^sz_2/z_1)^2},\\
  &\label{zeta2}\tag{$\zeta$2}[H_i(z_1),\Psi_j^\pm(z_2)]
  =\pm \Psi_j^\pm(z_2)
  \\
  &\quad\nonumber\times
  \sum_{s\in\Z_\wp}\left(
  \vvp{ [a_{ij}]_{q_i}(q^{-2r\ell}-1)q^{-s} }
  \iota_{z_1,z_2}
  -\vvp{ [a_{ij}]_{q_i}(1-q^{2r\ell})q^{-s} }
  \iota_{z_2,z_1}
  \right)
  \frac{1+\zeta^sz_2/z_1}{2-2\zeta^sz_2/z_1},\\
  &\label{zeta3}\tag{$\zeta$3}
  \iota_{z_1,z_2}\prod_{s\in\Z_\wp}(1-\zeta^sz_2/z_1)^{
  \epsilon_1\epsilon_2\vvp{
    (q_i^{a_{ij}}-q_i^{-a_{ij}})(1-q^{-2r\ell})q^{-s}
  }}
  \Psi_i^{\epsilon_1}(z_1)\Psi_j^{\epsilon_2}(z_2)\\
  &\quad\nonumber=
  \iota_{z_2,z_1}\prod_{s\in\Z_\wp}(1-\zeta^sz_2/z_1)^{
  \epsilon_1\epsilon_2\vvp{
    (q_i^{a_{ij}}-q_i^{-a_{ij}})(q^{2r\ell}-1)q^{-s}
  }}
  \Psi_j^{\epsilon_2}(z_2)\Psi_i^{\epsilon_1}(z_1),\\
  &\label{zeta4}\tag{$\zeta$4}
  \lim_{z_1\to z_2}\prod_{s\in\Z_\wp}(1-\zeta^sz_2/z_1)^{
  \vvp{
    (q_i^2-q_i^{-2})(q^{-2r\ell}-1)q^{-s}
  }}
  \Psi_i^\pm(z_1)\Psi_i^\mp(z_2)=1,\\
  &\label{zeta5}\tag{$\zeta$5}[H_i(z_1),X_j^\pm(z_2)]=\pm X_j^\pm(z_2)\\
  &\quad\nonumber\times
  \sum_{s\in\Z_\wp}
  \left(
  \vvp{[a_{ij}]_{q_i}q^{-r\ell-s}}\iota_{z_1,z_2}
  -\vvp{[a_{ij}]_{q_i}q^{r\ell-s}}\iota_{z_2,z_1}
  \right)
  \frac{1+\zeta^sz_2/z_1}{2-2\zeta^sz_2/z_1},\\
  &\label{zeta6}\tag{$\zeta$6}
  \iota_{z_1,z_2}\prod_{s\in\Z_\wp}(1-\zeta^sz_2/z_1)^{-\epsilon_1\epsilon_2\vvp{(q_i^{a_{ij}}-q_i^{-a_{ij}})q^{-r\ell-s}}}
  \Psi_i^{\epsilon_1}(z_1) X_j^{\epsilon_2}(z_2)\\
  &\quad\nonumber=
  \iota_{z_2,z_1}\prod_{s\in\Z_\wp}(1-\zeta^sz_2/z_1)^{-\epsilon_1\epsilon_2
  \vvp{(q_i^{a_{ij}}-q_i^{-a_{ij}})q^{r\ell-s}}}
  X_j^{\epsilon_2}(z_2)\Psi_i^{\epsilon_1}(z_1),\\
  &\label{zeta7}\tag{$\zeta$7}(1-\zeta_i^{a_{ij}}z_2/z_1)X_i^\pm(z_1)X_j^\pm(z_2)
    =\zeta_i^{a_{ij}}(1-\zeta_i^{-a_{ij}}z_2/z_1)X_j^\pm(z_2)X_i^\pm(z_1),\\
  &\label{zeta8}\tag{$\zeta$8}X_i^+(z_1)X_j^-(z_2)
    -\zeta_i^{-a_{ij}}\iota_{z_2,z_1}
  \frac{1-\zeta_i^{a_{ij}}z_2/z_1}{1-\zeta_i^{-a_{ij}}z_2/z_1}
  X_j^-(z_2)X_i^+(z_1)\\
  &\quad\nonumber=\delta_{ij} \left(\delta\left(\frac{z_2}{z_1}\right)-\Psi_j^+(z_2\zeta^{-r\ell}) \delta\left(\frac{z_2\zeta^{-2r\ell}}{z_1}\right)\right),\\
  &\label{zeta-weight-mod}\tag{$\zeta$9}\te{for each $i\in I$, $H_i(0)$ acts semi-simply on $W$ with eigenvalues in $\Z$},\\
  &\label{zeta-Psi-def2}\tag{$\zeta$10}\Psi_i^\pm(z)
    =\zeta_i^{\mp 2H_i(0)}\exp\left(\sum_{m<0}(\zeta_i^m-\zeta_i^{-m})H_i(m)\frac{z^{-m}}{\mp m}\right)
    \exp\left(\sum_{m>0}(\zeta_i^m-\zeta_i^{-m})H_i(m)\frac{z^{-m}}{\mp m}\right).
\end{align}
\end{de}

\begin{rem}
The relations \eqref{zeta2}, \eqref{zeta3}, \eqref{zeta4} and \eqref{zeta6} follows from the relations \eqref{zeta1}, \eqref{zeta5} and \eqref{zeta-Psi-def2}.
\end{rem}

Let $\phi(x,z)=xe^z=\exp(zx\pd x)x$ be an associate of $F_a(x,y)=x+y$, and let $\Z_\wp=\Z/\wp\Z$.
Define 
$\chi_\phi:\Z_\wp\to \C^\times$ by letting $s\mapsto\zeta^{-s}$.
For $(W,H_i(z),\Psi_i^\pm(z),X_i^\pm(z))\in \obj\mathcal R_\zeta'^\ell(\wh\g)$, we note that
\begin{align}\label{eq:U-W}
  U_W=\set{H_i(\zeta^s z),\Psi_i^\pm(\zeta^s z),X_i^\pm(\zeta^s z)}{i\in I,\,s\in\Z_\wp}
\end{align}
is a $\chi_\phi(\Z_\wp)$-compatible subset of $\E(W)$, and is $\chi_\phi(\Z_\wp)$-invariant.
From Theorem \ref{thm:G-mod-nonlocal-VA-gen} we have the following result.

\begin{prop}\label{prop:qaff-mod-to-module-nva}
$(\<U_W\>_\phi,R)$ is a $\Z_\wp$-module nonlocal vertex algebra with vacuum $1_W$, vertex operator map $Y_\E^\phi$ and group homomorphism
$R:\Z_\wp\to \End_\C(\<U_W\>_\phi)$ defined by $R(s)a(z)=a(\zeta^s z)$,
and $W$ is a $(\Z_\wp,\chi_\phi)$-equivariant $\phi$-coordinated quasi $\<U_W\>_\phi$-module with module action
$Y_W^\phi(a(z),z_0)=a(z_0)$.
\end{prop}

\begin{thm}\label{thm:presentation}
The category $\mathcal R_\zeta^\ell(\wh\g)$ is isomorphic to the full subcategory of $\mathcal R_\zeta'^\ell(\wh\g)$, whose objects
$(W,H_i(z),\Psi_i^\pm(z),X_i^\pm(z))$
satisfy the following additional conditions
\begin{align}
  &\label{zeta-serre}\tag{$\zeta$11}X_i^\pm(\zeta_i^{-a_{ij}}z)_0^\phi X_i^\pm(\zeta_i^{-2-a_{ij}}z)_0^\phi \cdots
  X_i^\pm(\zeta_i^{a_{ij}}z)_0^\phi X_j^\pm(z)=0
  \quad\te{for }i,j\in I, \,\,\te{with }a_{ij}\le 0,\\
  &\label{zeta-res-form}\tag{$\zeta$12}X_i^\pm(\zeta_i^{2\wp_i-2} z)_0^\phi X_i^\pm(\zeta_i^{2\wp_i-4} z)_0^\phi \cdots X_i^\pm(\zeta_i^2 z)_0^\phi X_i^\pm(z)=0\quad\te{for }i\in I.
\end{align}
\end{thm}

\subsection{Proof of Theorem \ref{thm:presentation}}

Let $\U_q^\ell(\wh\g)$ be the quotient algebra of $\U_q(\wh\g)$ modulo the ideal generated by $\gamma-q_0^\ell$ and $\U_q^{\ell,\res}(\wh\g)$ be the quotient algebra of $\U_q^{\res}(\wh\g)$ modulo the ideal generated by
\begin{align*}
  \gamma-q_0^\ell\quad\te{and}\quad\qb{\gamma}{\wp_0}_{q_0}-\qb{\ell}{\wp_0}_{\zeta_0}.
\end{align*}
We first rewrite the defining relations of $\U_q^\ell(\wh\g)$ in terms of $\wt h_i(m)$.
\begin{lem}\label{lem:wth-psi-x-rels}
Define $\wt h_i'(z)=\sum_{0\ne n\in\Z}\wt h_i(n)z^{-n}$. Then the following relations hold on both $\U_q^\ell(\wh\g)$ and $\U_q^{\ell,\res}(\wh\g)$:
\begin{align}
  &\tag{Q1-2$'$}\label{Q1-2'}[\wt h_i'(z_1),\wt h_j'(z_2)]
  =\sum_{s\in\Z}\vvinf{[a_{ij}]_{q_i}[r\ell/r_j]_{q_j}q^{-r\ell-s}}
  \left(\frac{q^sz_2/z_1}{(1-q^sz_2/z_1)^2}
  -\frac{q^sz_1/z_2}{(1-q^sz_1/z_2)^2}\right),\\
  &[\wt h_i'(z_1),\psi_j^-(z_2)]=\psi_j^-(z_2)
  \sum_{s\in\Z}\vvinf{
    [a_{ij}]_{q_i}(q^{-2r\ell}-1)q^{-s}
  } \frac{q^sz_2/z_1}{1-q^sz_2/z_1},\nonumber\\
  &[\wt h_i'(z_1),\psi_j^+(z_2)]=\psi_j^+(z_2)
  \sum_{s\in\Z}\vvinf{[a_{ij}]_{q_i}(q^{-2r\ell}-1)q^{-s}}
  \frac{q^sz_1/z_2}{1-q^sz_1/z_2},\nonumber\\
  &\tag{Q4-5$'$}\label{Q4-5'}[\wt h_i'(z_1),x_j^\pm(z_2)]=\pm x_j^\pm(z_2)\\
  &\quad\times\nonumber
  \sum_{s\in\Z}\left(\vvinf{[a_{ij}]_{q_i}q^{-r\ell-s}}
  \frac{q^sz_2/z_1}{1-q^sz_2/z_1}
  +
  \vvinf{[a_{ij}]_{q_i}q^{\mp r\ell-s}}\frac{q^sz_1/z_2}{1-q^sz_1/z_2}\right).
\end{align}
Conversely, the relations \eqref{Q0}, \eqref{Q3}, \eqref{Q9}, \eqref{Q1-2'} and \eqref{Q4-5'} yield the relations \eqref{Q1}, \eqref{Q2}, \eqref{Q4} and \eqref{Q5}.
\end{lem}

\begin{proof}
From \eqref{Q0}, \eqref{Q1} and \eqref{Q2}, we get that for any $m,n\in\Z_+$,
\begin{align*}
  &[h_i(m),h_j(-n)]=\frac{\delta_{m,n}}{(q_i-q_i\inv)(q_j-q_j\inv)}\frac{1}{n}
  \left( (q_i^{na_{ij}}-q_i^{-na_{ij}})(1-q^{-2nr\ell}) \right)\\
  &\quad=\delta_{m,n}\frac{1}{n}[m]_{q_i}[n]_{q_j}[a_{ij}]_{q_i^n}[r\ell/r_j]_{q_j^n}q^{-r\ell n}.
\end{align*}
Set
\begin{align*}
  &\wt h_i(z)^\pm=\sum_{n=1}^\infty \wt h_i(\pm n)z^{\mp n}.
\end{align*}
Then we get the following relations
\begin{align}
  &\sum_{m,n=1}^\infty\left[\wt h_i(m),\wt h_j(-n)\right]z_1^{-m}z_2^n=\sum_{m,n=1}^\infty \frac{mn}{[m]_{q_i}[n]_{q_j}}[ h_i(m), h_j(-n)]z_1^{-m}z_2^n\nonumber\\
  &\quad=\sum_{n=1}^\infty n[a_{ij}]_{q_i^n}[r\ell/r_j]_{q_j^n}q^{-r\ell n}\frac{z_2^n}{z_1^n}
  =\sum_{s\in\Z}\vvinf{[a_{ij}]_{q_i}[r\ell/r_j]_{q_j}q^{-r\ell-s}}
  \frac{q^sz_2/z_1}{(1-q^sz_2/z_1)^2},
  \label{eq:wth-wth-rel}\\
  &\sum_{m,n=1}^\infty \left[\wt h_i(m),(q_j\inv-q_j)h_j(-n)\right]z_1^{-m}z_2^n=\sum_{m,n=1}^\infty \frac{m(q_j\inv-q_j)}{[m]_{q_i}}
  [h_i(m),h_j(-n)]z_1^{-m}z_2^n\nonumber\\
  &\quad=\sum_{n=1}^\infty [a_{ij}]_{q_i^n}(q^{-2nr\ell}-1)\frac{z_2^n}{z_1^n}
  =\sum_{s\in\Z}\vvinf{[a_{ij}]_{q_i}(q^{-2r\ell}-1)q^{-s}}
  \frac{q^sz_2/z_1}{1-q^sz_2/z_1},
  \label{eq:wth-psi-rel+}\\
  &\sum_{m,n=1}^\infty \left[\wt h_i(-m),(q_j-q_j\inv)h_j(n)\right]z_1^m z_2^{-n}=\sum_{m,n=1}^\infty \frac{m(q_j-q_j\inv)}{[m]_{q_i}}
  [h_i(-m),h_j(n)]z_1^m z_2^{-n}\nonumber\\
  &\quad=\sum_{n=1}^\infty [a_{ij}]_{q_i^n}(q^{-2r\ell n}-1)\frac{z_1^n}{z_2^n}
  =\sum_{s\in\Z}\vvinf{[a_{ij}]_{q_i}(q^{-2r\ell}-1)q^{-s}}
  \frac{q^sz_1/z_2}{1-q^sz_1/z_2}.
  \label{eq:wth-psi-rel-}
\end{align}
The first equation follows from \eqref{eq:wth-wth-rel}, the second equation follows from \eqref{eq:wth-psi-rel+} and the third equation follows from \eqref{eq:wth-psi-rel-}.
From \eqref{Q3}, \eqref{Q4} and \eqref{Q5}, we get the following relations
\begin{align*}
  &[h_i(m),x_j^\pm(z)]=\pm x_j^\pm(z)z^m \frac{1}{m}[m a_{ij}]_{q_i}q^{-mr\ell},\\
  &[h_i(-m),x_j^\pm(z)]=\pm x_j^\pm(z)z^{-m}\frac 1m q^{\mp mr\ell}[ma_{ij}]_{q_i}\quad\te{for }m\in\Z_+.
\end{align*}
Then the last equation follows from the identities below
\begin{align*}
  &\sum_{m=1}^\infty \left[\wt h_i(m),x_j^\pm(z_2)\right]z_1^{-m}
  =\sum_{m=1}^\infty \left[\frac{m h_i(m)}{[m]_{q_i}},x_j^\pm(z_2)\right]z_1^{-m}\\
  =&\pm x_j^\pm(z_2)\sum_{m=1}^\infty [ a_{ij}]_{q_i^{m}}q^{-mr\ell}\frac{z_2^m}{z_1^m}
  =\pm x_j^\pm(z_2)
  \sum_{s\in\Z}\vvinf{[a_{ij}]_{q_i}q^{-r\ell-s}}
  \frac{q^sz_2/z_1}{1-q^sz_2/z_1},\\
  &\sum_{m=1}^\infty \left[\wt h_i(-m),x_j^\pm(z_2)\right]z_1^{m}
  =\sum_{m=1}^\infty \left[\frac{m h_i(-m)}{[m]_{q_i}},x_j^\pm(z_2)\right]z_1^{m}\\
  =&\pm x_j^\pm(z_2)\sum_{m=1}^\infty [a_{ij}]_{q_i^{m}}q^{\mp mr\ell}\frac{z_1^m}{z_2^m}
  =\pm x_j^\pm(z_2)
  \sum_{s\in\Z}\vvinf{[a_{ij}]_{q_i}q^{\mp r\ell-s}}
  \frac{q^sz_1/z_2}{1-q^sz_1/z_2}.
\end{align*}
The converse statement follows similarly.
\end{proof}

In order to prove that every smooth weighted $\U_\zeta^\ell(\wh\g)$-module of level $\ell$ is an object of $\mathcal R_\zeta'^\ell(\wh\g)$, we need to construct the elements corresponding to $H_i(0)$ ($i\in I$) inside $\U_\zeta(\wh\g)$.
For $k\in\Z_+$, let
\begin{align*}
  \Phi_k(z)=\prod_{\substack{1\le a\le k\\ \gcd(a,k)=1}}(x-e^{2\pi \sqrt{-1}a/k})
\end{align*}
be the $k$-th cyclotomic polynomial.
We note that
\begin{align*}
  &z^k-1=\prod_{\substack{1\le d\le k\\ d\,|\,k}}\Phi_d(z),\quad \te{and }
  \Phi_d(\zeta_i^2)\ne 0\quad\te{for }d<\wp_i.
\end{align*}
For $0\le a<k$, define $C_{k,a}(z)$ as follows
\begin{align*}
  &w^k-1-\prod_{a=0}^{k-1}(w-z^a)=\sum_{a=0}^{k-1}C'_{k,a}(z)w^a,\quad C_{k,a}(z)=C'_{k,a}(z)/\Phi_k(z).
\end{align*}
Since $C'_{k,a}(e^{2\pi\sqrt{-1}/k})=0$ and $\Phi_k(z)$ in irreducible in $\Z[z]$, we have that $\Phi_k(z)\,|\,C'_{k,a}(z)$ and hence
\begin{align*}
  C_{k,a}(z)\in\Z[z].
\end{align*}
The following elements in $\U_q^{\res}(\wh\g)^0$ were introduced in \cite{Len-unrolled}:
\begin{align*}
  &\wt h_i'(0)=q_i^{-\wp_i}k_i^{\wp_i}\qb{k_i}{\wp_i}_{q_i}
  \prod_{a=1}^{\wp_i-1}(q_i^{2a}-1)\prod_{\substack{1\le d<\wp_i\\ d\,|\,\wp_i}}\Phi_d(q_i^2)+\sum_{a=0}^{\wp_i-1}C_{\wp_i,a}(q_i^2)k_i^{2a},\\
  &\wt h_i(0)=\wt h_i'(0)\prod_{\substack{1\le d<\wp_i\\ d\,|\,\wp_i}}\Phi_d(\zeta_i^2)\inv \quad\te{for }i\in I.
\end{align*}

\begin{lem}\label{lem:h-rel}
For each $i,j\in I$, the following relation holds in $\U_q^{\res}(\wh\g)$:
\begin{align}\label{eq:h-rel-1}
  &\wt h_i(0)x_j^\pm(z)-q_i^{\pm 2\wp_ia_{ij}}x_j^\pm(z)\wt h_i(0)\nonumber\\
  =&\pm\Sign(a_{ij})x_j^\pm(z)\sum_{a=0}^{|a_{ij}|-1}q_i^{2\wp_i a+(\pm a_{ij}-|a_{ij}|)\wp_i}\prod_{\substack{1\le d<\wp_i\\ d\,|\,\wp_i}}\left(\Phi_d(q_i^2)\Phi_d(\zeta_i^2)\inv\right).
\end{align}
Moreover, the following relation holds in $\U_\zeta(\wh\g)$
\begin{align}\label{eq:h-rel-2}
  [\wt h_i(0),x_j^\pm(z)]=\pm a_{ij}x_j^\pm(z).
\end{align}
\end{lem}

\begin{proof}
The following identity in $\U_q(\wh\g)$ was given in \cite{Len-unrolled}:
\begin{align}\label{eq:wth-0}
  \wt h'_i(0)
  =&\frac{k_i^{2\wp_i}-1}{\Phi_{\wp_i}(q_i^2)}.
\end{align}
Then the following relation holds in $\U_q(\wh\g)$
\begin{align*}
  &\wt h'_i(0)x_j^\pm(z)- q_i^{\pm 2\wp_i a_{ij}} x_j^\pm(z)\wt h'_i(0)\\
  =&\Phi_{\wp_i}(q_i^2)\inv\left( q_i^{\pm 2\wp_ia_{ij}}x_j^\pm(z)k_i^{\wp_i}-x_j^\pm(z)-q_i^{\pm 2\wp_i a_{ij}} x_j^\pm(z)k_i^{2\wp_i}+q_i^{\pm 2\wp_i a_{ij}} x_j^\pm(z) \right)\\
  =&x_j^\pm(z)\frac{q_i^{\pm 2\wp_ia_{ij}}-1}{\Phi_{\wp_i}(q_i^2)}
  =x_j^\pm(z)\frac{q_i^{\pm 2\wp_ia_{ij}}-1}{q_i^{2\wp_i}-1}\prod_{\substack{1\le d<\wp_i\\ d\,|\,\wp_i}}\Phi_d(q_i^2)\\
  =&\pm\Sign(a_{ij})x_j^\pm(z)\sum_{a=0}^{|a_{ij}|-1}q_i^{2\wp_i a+(\pm a_{ij}-|a_{ij}|)/2}\prod_{\substack{1\le d<\wp_i\\ d\,|\,\wp_i}}\Phi_d(q_i^2).
\end{align*}
It shows that the relation \eqref{eq:h-rel-1} holds in $\U_q(\wh\g)$, and hence in $\U_q^{\res}(\wh\g)$,
since all the elements involved lie in $\U_q^{\res}(\wh\g)$.
The moreover statements follows from the fact that $\zeta_i^{2\wp_i}=1$ and $\Phi_d(\zeta_i^2)\ne 0$ for all $1\le d<\wp_i$.
\end{proof}

\begin{lem}\label{lem:h-eigenvalues}
Let $W\in\obj \mathcal R_\zeta^\ell(\wh\g)$, let $\lambda\in\Lambda$ and let $w\in W_\lambda$. Then
\begin{align*}
  \wt h_i(0)w=\lambda(h_i)w\quad\te{for }i\in I.
\end{align*}
\end{lem}

\begin{proof}
From the definition of $\wt h_i(0)$ we have that
\begin{align*}
  &\wt h_i(0)w\\
  =&w\zeta_i^{-\wp_i}\zeta_i^{\wp_i\lambda(h_i)}\qb{\lambda(h_i)}{\wp_i}_{\zeta_i}\prod_{a=1}^{\wp_i-1}(\zeta_i^{2a}-1)
  +w\sum_{a=0}^{\wp_i-1}C_{\wp_i,a}(\zeta_i^2)\zeta_i^{2a\lambda(h_i)}
  \prod_{\substack{1\le d<\wp_i\\ d\,|\,\wp_i}}\Phi_d(\zeta_i^2)\inv\\
  =&w\lim_{q\mapsto \zeta} (q_i^{2\wp_i}-1)\inv \left( q_i^{\wp_i\lambda(h_i)-\wp_i} \qb{\lambda(h_i)}{\wp_i}_{q_i}\prod_{a=1}^{\wp_i}(q_i^{2a}-1)+\sum_{a=0}^{\wp_i-1}C_{\wp_i,a}'(q_i^2)q_i^{2a\lambda(h_i)} \right)\\
  =&w\lim_{q\mapsto \zeta} (q_i^{2\wp_i}-1)\inv \left( q_i^{\wp_i\lambda(h_i)-\wp_i} \qb{\lambda(h_i)}{\wp_i}_{q_i}\prod_{a=1}^{\wp_i}(q_i^{2a}-1)+q_i^{2\wp_i\lambda(h_i)}-1-\prod_{a=0}^{\wp_i-1}(q_i^{2\lambda(h_i)}-q_i^{2a}) \right)\\
  =&w\lim_{q\mapsto \zeta} (q_i^{2\wp_i}-1)\inv (q_i^{2\wp_i\lambda(h_i)}-1)
  =\lambda(h_i)w.
\end{align*}
We complete the proof of lemma.
\end{proof}

\begin{rem}\label{rem:level}
Let $W\in\obj \mathcal R_\zeta^\ell(\wh\g)$.
Similarly to the definition of $\wt h_i(0)$, we define
\begin{align*}
  &\wt \gamma'=q_0^{-\wp_0}\gamma^{\wp_0}\qb{\gamma}{\wp_0}_{q_0}
  \prod_{a=1}^{\wp_0-1}(q_0^{2a}-1)\prod_{\substack{1\le d<\wp_0\\ d\,|\,\wp_0}}\Phi_d(q_0^2)+\sum_{a=0}^{\wp_0-1}
  C_{\wp_0,a}(q_0^2)\gamma^{2a},\\
  &\wt \gamma=\wt \gamma\prod_{\substack{1\le d<\wp_0\\ d\,|\,\wp_0}}\Phi_d(\zeta_0^2)\inv.
\end{align*}
By an argument similar to the proof of Lemma \ref{lem:h-eigenvalues}, we have $\wt \gamma w=\ell w$.
Therefore, $\ell$ is uniquely determined by the action of $\gamma$ and $\displaystyle\qb{\gamma}{\wp_0}_{q_0}$.
\end{rem}

%
%
%
%

Combining Lemma \ref{lem:wth-psi-x-rels} and \eqref{eq:h-rel-2}, we immediately get the following result.

\begin{lem}\label{lem:wth-wth-x-rels}
Define $\wt h_i(z)=\wt h_i'(z)+\wt h_i(0)$.
Then the following relations hold on every $W\in\obj\mathcal R_\zeta^\ell(\wh\g)$:
\begin{align*}
  &[\wt h_i(z_1),\wt h_j(z_2)]=
  \sum_{s\in\Z_\wp}
  \vvp{[a_{ij}]_{q_i} [r\ell/r_j]_{q_j}q^{-r\ell-s} }
  \left( \frac{\zeta^sz_2/z_1}{(1-\zeta^sz_2/z_1)^2}
    -\frac{\zeta^sz_1/z_2}{(1-\zeta^sz_1/z_2)^2}
  \right),\\
  &[\wt h_i(z_1),\psi_j^-(z_2)^{\pm 1}\psi_j^+(z_2)^{\mp 1}]=\pm\psi_j^-(z_2)^{\pm 1}\psi_j^+(z_2)^{\mp 1} \\
  &\qquad\qquad\times  \sum_{s\in\Z_\wp}
    \vvp{[a_{ij}]_{q_i}(q^{-2r\ell}-1) q^{-s} }\left( \frac{1+\zeta^sz_2/z_1}{2-2\zeta^sz_2/z_1}
        -\frac{1+\zeta^sz_1/z_2}{2-2\zeta^sz_1/z_2}
    \right),\\
  &[\wt h_i(z_1),x_j^\pm(z_2)]=\pm x_j^\pm(z_2)\\
  &\qquad\qquad\times\sum_{s\in\Z_\wp}\left(\vvp{
    [a_{ij}]_{q_i}q^{-r\ell-s}
  }\frac{1+\zeta^sz_2/z_1}{2-2\zeta^sz_2/z_1}
  +\vvp{[a_{ij}]_{q_i}q^{\mp r\ell-s}}\frac{1+\zeta^sz_1/z_2}{2-2\zeta^sz_1/z_2}
  \right).
\end{align*}
\end{lem}

Combining Lemmas \ref{lem:h-eigenvalues}, \ref{lem:wth-wth-x-rels}, and \eqref{Q0}-\eqref{Q7}, we immediately have the following result.
\begin{prop}\label{prop:U-zeta-M}
Let $W\in \obj\mathcal R_\zeta^\ell(\wh\g)$. Then $W$ is an object of $\mathcal R_\zeta'^\ell(\wh\g)$ with
\begin{align*}
  &H_i(z)=\wt h_i(z),\quad \Psi_i^\pm(z)=\psi_i^-(z)^{\pm 1}\psi_i^+(z)^{\mp 1},\quad X_i^+(z)=x_i^+(z),\\ &X_i^-(z)=(\zeta_i-\zeta_i\inv)x_i^-(z)\psi_i^+(\zeta^{-r\ell}z)\inv,\quad \te{for }i\in I.
\end{align*}
\end{prop}

Denote by $\mathcal A$ the localization of $\C[q,q\inv]$ at the ideal generated by $(q^\wp-1)$,
and define $\U_{\mathcal A}'^\ell(\wh\g)$ to be the $\mathcal A$-algebra generated by
\begin{align*}
  &k_i^{\pm 1},\quad \qb{k_i}{\wp_i}_{q_i},\quad \wt h_i(m),\quad x_i^\pm(n)\quad \te{for }i\in I,\,0\ne m\in\Z,\,n\in\Z.
\end{align*}
subject to relations \eqref{Q0}, \eqref{Q1-2'}, \eqref{Q3}, \eqref{Q4-5'}, \eqref{Q6}, \eqref{Q7} and \eqref{Q9}.
We define the notion of smooth weighted $\U_{\mathcal A}'^\ell(\wh\g)$ in the obvious way.
The following is the reverse of Proposition \ref{prop:U-zeta-M} that can be verified straightforwardly.

\begin{prop}\label{prop:U-zeta-M-reverse}
Let $(W,H_i(z),\Psi_i^\pm(z),X_i^\pm(z))$ be an object of $\mathcal R_\zeta'^\ell(\wh\g)$.
Then $W$ becomes a smooth weighted $\U_{\mathcal A}'^\ell(\wh\g)$-module with
\begin{align*}
  &k_i^{\pm 1}\mapsto \zeta_i^{\pm H_i(0)}, \quad \qb{k_i}{\wp_i}_{q_i}\mapsto \qb{H_i(0)}{\wp_i}_{\zeta_i}, \quad\wt h_i(m)\mapsto H_i(m),\\
  &\psi_i^\pm(z)\mapsto\zeta_i^{\pm H_i(0)}\exp\left( \sum_{\pm m>0}(\zeta_i^m-\zeta_i^{-m})H_i(m)z^{-m} \right),\\
  &x_i^+(z)\mapsto X_i^+(z),\quad x_i^-(z)\mapsto (\zeta_i-\zeta_i\inv)\inv X_i^-(z)\psi_i^+(\zeta^{-r\ell}z)
\end{align*}
for $i\in I$, $0\ne m\in\Z$.
\end{prop}

Before proving every $W\in\obj\mathcal R_\zeta^\ell(\wh\g)$ satisfying \eqref{zeta-serre} and \eqref{zeta-res-form}, we need the following algebras.
Define $\U_{\mathcal A}^\ell(\wh\g)$ to be the quotient algebra of $\U_{\mathcal A}'^\ell(\wh\g)$ modulo the relation \eqref{Q8},
define 
\begin{align*}
  \U_{\mathcal A}^{\ell,\res}(\wh\g)
  =\U_q^{\ell,\res}(\wh\g)\ot_{\C[q,q\inv]}\mathcal A,
\end{align*}
and define $\U_\zeta^\ell(\wh\g)$ to be the quotient algebra of $\U_\zeta(\wh\g)$ modulo the ideal generated by
\begin{align*}
  \gamma-\zeta_0^\ell\quad\te{and}\quad \qb{\gamma}{\wp_0}_{q_0}-\qb{\ell}{\wp_0}_{\zeta_0}.
\end{align*}

For each $N\in\N$, we denote by $L_N$ (resp. $L_N'$, $L_N^{\res}$, $L_N(\zeta)$) the left ideal of $\U_{\mathcal A}^\ell(\wh\g)$
(resp. $\U_{\mathcal A}'^\ell(\wh\g)$, $\U_{\mathcal A}^{\ell,\res}(\wh\g)$, $\U_\zeta^\ell(\wh\g)$) generated by
\begin{align*}
  \wt h_i(m),\quad \psi_i^+(m),\quad x_i^\pm(m)\quad\te{for }i\in I,\,m\ge N.
\end{align*}
If $0>N\in\Z$, we set the ideal $L_N$ (resp. $L_N'$, $L_N^{\res}$, $L_N(\zeta)$) to be the whole algebra.

\begin{prop}\label{prop:topo-alg}
For each $N\in\N$ and $n\in \Z$, there exists $M\in\N$ such that
\begin{align*}
  L_M a(n)\in L_N,\quad  L_M' a(n)\in L_N',\quad L_M^{\res} a(n)\in L_N^{\res},\quad L_M(\zeta) a(n)\in L_N(\zeta)
\end{align*}
for $i\in I$, where $a=\wt h_i'$, $\psi_i^\pm$, $x_i^\pm$.
Moreover, view $\C$ as a topological field under discrete topology. Then
$\U_{\mathcal A}^\ell(\wh\g)$, $\U_{\mathcal A}'^\ell(\wh\g)$, $\U_{\mathcal A}^{\ell,\res}(\wh\g)$ and $\U_\zeta^\ell(\wh\g)$ equipped with topological algebra structures such that $\set{L_N}{N\in\N}$, $\set{L_N'}{N\in\N}$, $\set{L_N^{\res}}{N\in\N}$ and $\set{L_N(\zeta)}{N\in\N}$ form local bases at $0$, respectively.
\end{prop}

\begin{proof}
From \eqref{Q1}, \eqref{Q2}, \eqref{Q4}-\eqref{Q7} and Lemma \ref{lem:wth-psi-x-rels}, we have that for any
\begin{align*}
  \al(z)=\sum_{n\in\Z}\al(n)z^{-n},\beta(z)=\sum_{n\in\Z}\beta(n)z^{-n}\in \set{\wt h_i'(z),\psi_i^\pm(z), x_i^\pm(z)}{i\in I},
\end{align*}
there exist
$0\ne f(z),g(z)\in \C[q,q\inv,z]$, such that $f(0),g(0)\in\C[q,q\inv]^\times$ and
\begin{align*}
  \iota_{z_1,z_2}f(z_2/z_1)\al(z_1)\beta(z_2)=\iota_{z_2,z_1}g(z_1/z_2)\beta(z_2)\al(z_1).
\end{align*}
For each $N\in N$, we denote by $\pi_N:\U_{\mathcal A}^\ell(\wh\g)\to \U_{\mathcal A}^\ell (\wh\g)/L_N$ the quotient $\U_{\mathcal A}^\ell(\wh\g)$-module map.
Then
\begin{align*}
  &\iota_{z_1,z_2}f(z_2/z_1)\pi_N\big(\al(z_1)\beta(z_2)\big)=\iota_{z_2,z_1}g(z_1/z_2)\pi_N\big(\beta(z_2)\al(z_1)\big)\in
  \left(\U_{\mathcal A}^\ell (\wh\g)/L_N\right)((z_1,z_2)).
\end{align*}
It follows that
\begin{align*}
  &\sum_{m,n\in\Z}\pi_N(\al(m)\beta(n))z_1^{-m}z_2^{-n}=\pi_N\big(\al(z_1)\beta(z_2)\big)\\
  =&\iota_{z_1,z_2}f(z_2/z_1)\inv \left(\iota_{z_2,z_1}g(z_1/z_2)\pi_N\big(\beta(z_2)\al(z_1)\big)\right)
  \in \left(\U_{\mathcal A}^\ell (\wh\g)/L_N\right)((z_1))((z_2)).
\end{align*}
It implies that for each $n\in\Z$, there exists $M\in\N$, such that
$\pi_N(\al(m)\beta(n))=0$ for all $m\ge M$.
Since $\set{\wt h_i(z),\psi_i^+(z), x_i^\pm(z)}{i\in I}$ is a finite set, we complete the proof of the first assertion.
The moreover statement follows immediate from the first one and the following fact
\begin{align*}
  L_N a\in L_N\quad \te{for }N\in\N,\,a=k_i^{\pm 1},\,\qb{k_i}{\wp_i}_{q_i}.
\end{align*}
The proofs for $\U_{\mathcal A}'^\ell(\wh\g)$, $\U_{\mathcal A}^{\ell,\res}(\wh\g)$ and $\U_\zeta^\ell(\wh\g)$ are similar.
\end{proof}

Note the following commutative diagram
\begin{align*}
  \begin{tikzcd}[column sep=2.5em,row sep=1.5em, ampersand replacement=\&]
  \U_{\mathcal A}'^\ell(\wh\g)/L_0'\ar[d] \& \U_{\mathcal A}'^\ell(\wh\g)/L_1'\ar[l]\ar[d] \& \U_{\mathcal A}'^\ell(\wh\g)/L_2'\ar[l]\ar[d] \&\cdots \ar[l]\\
  \U_{\mathcal A}^\ell(\wh\g)/L_0\ar[d] \& \U_{\mathcal A}^\ell(\wh\g)/L_1\ar[l]\ar[d] \& \U_{\mathcal A}^\ell(\wh\g)/L_2\ar[l]\ar[d] \&\cdots \ar[l]\\
  \U_{\mathcal A}^{\ell,\res}(\wh\g)/L_0^{\res}\ar[d] \& \U_{\mathcal A}^{\ell,\res}(\wh\g)/L_1^{\res}\ar[l]\ar[d] \& \U_{\mathcal A}^{\ell,\res}(\wh\g)/L_2^{\res}\ar[l]\ar[d] \&\cdots \ar[l]\\
  \U_\zeta^\ell(\wh\g)/L_0(\zeta) \& \U_\zeta^\ell(\wh\g)/L_1(\zeta)\ar[l] \& \U_\zeta^\ell(\wh\g)/L_2(\zeta)\ar[l] \&\cdots \ar[l]
  \end{tikzcd}
\end{align*}
%
Let
\begin{align*}
  &\wh\U_{\mathcal A}'^\ell(\wh\g)=\varprojlim_{N\ge 0} \U_{\mathcal A}'^\ell (\wh\g)/L_N',\quad
  \wh\U_{\mathcal A}^\ell(\wh\g)=\varprojlim_{N\ge 0} \U_{\mathcal A}^\ell (\wh\g)/L_N,\\
  &\wh\U_{\mathcal A}^{\ell,\res}(\wh\g)=\varprojlim_{N\ge 0} \U_{\mathcal A}^{\ell,\res}(\wh\g)/L_N^{\res}, \quad
  \wh\U_\zeta^\ell(\wh\g)=\varprojlim_{N\ge 0}\U_\zeta^\ell(\wh\g)/L_N(\zeta)
\end{align*}
be the completions of the topological algebras $\U_{\mathcal A}'^\ell(\wh\g)$, $\U_{\mathcal A}^\ell(\wh\g)$, $\U_{\mathcal A}^{\ell,\res}(\wh\g)$ and $\U_\zeta^\ell(\wh\g)$, respectively.
The universal property of inverse limit deduces the following commutative diagram
\begin{align*}
  \begin{tikzcd}[column sep=2.5em,row sep=1.5em, ampersand replacement=\&]
    \U_{\mathcal A}'^\ell(\wh\g)\ar[r]\ar[d]\&
    \U_{\mathcal A}^\ell(\wh\g)\ar[r]\ar[d]\&
    \U_{\mathcal A}^{\ell,\res}(\wh\g)\ar[r]\ar[d]\&
    \U_\zeta^\ell(\wh\g)\ar[d]\\
    \wh\U_{\mathcal A}'^\ell(\wh\g)\ar[r]\&
    \wh\U_{\mathcal A}^\ell(\wh\g)\ar[r]\&
    \wh\U_{\mathcal A}^{\ell,\res}(\wh\g)\ar[r]\&
    \wh\U_\zeta^\ell(\wh\g).
  \end{tikzcd}
\end{align*}
All the maps above are continuous algebra homomorphisms.
For any $a\in \U_{\mathcal A}'^\ell(\wh\g)$, we still denoted by $a$ the images of the above maps for convenience.
Let $\wh L_N$, $\wh L_N'$, $\wh L_N^{\res}$ and $\wh L_N(\zeta)$ be the closures of $L_N$, $L_N'$, $L_N^{\res}$ and $L_N(\zeta)$, respectively. Then they are all left ideals and form local bases at $0$.

We define the notion of smooth modules of $U_{\mathcal A}'^\ell(\wh\g)$, $\U_{\mathcal A}^\ell(\wh\g)$ and $\U_{\mathcal A}^{\ell,\res}(\wh\g)$ similar to that of $\U_\zeta^\ell(\wh\g)$.
Then we immediately have the following result.
\begin{prop}\label{prop:smooth-continuous}
Every smooth module of $\U_{\mathcal A}'^\ell(\wh\g)$, $\U_{\mathcal A}^\ell(\wh\g)$, $\U_{\mathcal A}^{\ell,\res}(\wh\g)$ and $\U_\zeta^\ell(\wh\g)$ is naturally a continuous module of $\wh \U_{\mathcal A}'^\ell(\wh\g)$, $\wh\U_{\mathcal A}^\ell(\wh\g)$, $\wh\U_{\mathcal A}^{\ell,\res}(\wh\g)$ and $\wh\U_\zeta^\ell(\wh\g)$.
\end{prop}

Let $\Bbbk$ be a PID with a surjective ring homomorphism $\mathcal A\to \Bbbk$.
View $\Bbbk$ as a topological ring under the discrete topology.
Denote by $\mathscr L$ be the category of Hausdorff complete topological $\Bbbk$-modules, such that a family of submodules form a local basis at $0$.
For $U\in\mathscr L$ and $L$ be an open subspace of $U$,
the quotient map $U\to U/L$ induces the following map
\begin{align*}
  \wt \pi_L^{(k)}:U[[z_1^{\pm 1},\dots,z_k^{\pm 1}]]\to (U/L)[[z_1^{\pm 1},\dots,z_k^{\pm 1}]].
\end{align*}
Define $\wh\E^{(k)}(U;z_1,\dots,z_k)$ to be the set of elements $\al(z_1,\dots,z_k)\in \left(\End U\right)[[z_1^{\pm 1},\dots,z_k^{\pm 1}]]$
such that
\begin{align*}
  &\wt \pi_L^{(k)}(\al(z_1,\dots,z_k)w)\in (U/L)((z_1,\dots,z_k))\quad\te{for }w\in U,\,\,
  \te{and }L\,\,\te{be an open submodule}.
\end{align*}
We will also write $\wh\E^{(k)}(U)=\wh\E^{(k)}(U;z_1,\dots,z_k)$ and $\wh\E(U)=\wh\E^{(1)}(U)$ for convenience.
Moreover, if $U$ is discrete, then $U\in\mathscr L$ and $\wh\E^{(k)}(U)=\E^{(k)}(U)$.
In particular, if $U$ is a topological algebra with $1$, then we identical $\al(z_1,\dots,z_k)\in \wh\E^{(k)}(U)$ with
$\al(z_1,\dots,z_k)1\in U[[z_1^{\pm 1},\dots,z_k^{\pm 1}]]$.

We fix an invertible element $\qqq\in \Bbbk$, such that
$(\qqq^k-1)\inv\in\Bbbk$ for $0<k<\wp$.
For each $i\in I$, we set $\qqq_i=\qqq^{r_i}$.
We also fix a $W\in\mathscr L$ and fields
\begin{align*}
  e_i(z)=\sum_{n\in\Z}e_i(n)z^{-n}\in \wh\E(W)\quad\te{for }i\in I,
\end{align*}
satisfying the relation
\begin{align*}
  (z_1-\qqq_i^{a_{ij}}z_2)e_i(z_1)e_j(z_2)
  =(\qqq_i^{a_{ij}}z_1-z_2)e_j(z_2)e_i(z_1)\quad\te{for }i,j\in I.
\end{align*}

\begin{lem}\label{Q7ii}
For $i\in I$, the following relation holds:
\begin{align*}
  &(z_1-\qqq_i^{a_{ii}}z_2)(z_1-z_2)\inv e_i(z_1) e_i(z_2)\\
  =&(z_2-\qqq_i^{a_{ii}}z_1)(z_2-z_1)\inv e_i(z_2)e_i(z_1) \in\wh\E^{(2)}(W).
\end{align*}
\end{lem}

\begin{proof}
From \eqref{Q7}, we have that
\begin{align*}
  &X(z_1,z_2):=(z_1-q_i^{\epsilon a_{ii}})e_i(z_1)e_i(z_2)=(q_i^{\epsilon a_{ii}}z_1-z_2)e_i(z_2)e_i(z_1)\\
  &\quad=-X(z_2,z_1)\in\wh\E^{(2)}(W).
\end{align*}
Then $X(z,z)=0$. It follows that
\begin{align*}
  &(z_1-z_2)\inv X(z_1,z_2)-(z_2-z_1)\inv X(z_2,z_1)=\left((z_1-z_2)\inv+(z_2-z_1)\inv\right)X(z_1,z_2)\\
  &\quad=z_1\inv\delta\left(\frac{z_2}{z_1}\right)X(z_1,z_2)
  =z_1\inv\delta\left(\frac{z_2}{z_1}\right)X(z_2,z_2)=0.
\end{align*}
Combing this with the definition of $X(z_1,z_2)$, we complete the proof.
\end{proof}

Set
\begin{align}
  &f_{ij}(z_1,z_2,\qqq)=(1-\qqq_i^{ a_{ij}}z_2/z_1)(1-z_2/z_1)^{-\delta_{ij}}.
\end{align}
For $i_1,\dots,i_k$, we define
\begin{align}
  &\:e_{i_1}(z_1)\cdots e_{i_k}(z_k)\;=\prod_{1\le s<t\le k}f_{i_s,i_k}(z_s,z_t,\qqq) e_{i_1}(z_1)\dots e_{i_k}(z_k).
\end{align}
The following result follows immediate from \eqref{Q7} and Lemma \ref{Q7ii}.

\begin{lem}
For $i_1,\dots,i_k$, we have
\begin{align*}
  &\:e_{i_1}(z_1)\dots e_{i_k}(z_k)\;\in \wh\E^{(k)}(W).
\end{align*}
Moreover,
\begin{align*}
  &\:e_{i_{\sigma(1)}}(z_{\sigma(1)})\cdots e_{i_{\sigma(k)}}(z_{\sigma(k)})\;
  = \Bigg(\prod_{\substack{1\le s<t\le k\\ \sigma(s)>\sigma(t)}}C_{i_s,i_t}
    (z_{i_s}/z_{i_t})^{1-\delta_{i_s,i_t}}\Bigg) \:e_{i_1}(z_1)\cdots e_{i_k}(z_k)\;,
\end{align*}
where
\begin{align*}
  &C_{ij}=-(-1)^{\delta_{ij}}.
\end{align*}
\end{lem}

\begin{lem}\label{lem:no=phi-0}
Suppose that $\Bbbk=\C$ and $W$ is discrete.
Then the fields $\set{e_i(z)}{i\in I}$ is quasi-compatible.
Moreover, for each $i,j\in I$ with $a_{ij}\le 0$ and $-1\le k\in\Z$, we have
\begin{align}
  &e_i(\qqq_i^{a_{ij}+2(k-1)}z)_0^\phi e_i(\qqq_i^{a_{ij}+2(k-2)}z)_0^\phi\cdots e_i(\qqq_i^{a_{ij}}z)_0^\phi e_j(z)\label{eq:serre-phi-0}\\
  =&\:e_i(\qqq_i^{a_{ij}+2(k-1)}z)
  e_i(\qqq_i^{a_{ij}+2(k-2)}z)\cdots
  e_i(\qqq_i^{a_{ij}}z) e_j(z)\;.\nonumber
\end{align}
Furthermore, for each $i\in I$ and $k\in \N$, we have that
\begin{align}
  &e_i(\qqq_i^{2k}z)_0^\phi e_i(\qqq_i^{2(k-1)}z)_0^\phi
  \cdots e_i(\qqq_i^2z)_0^\phi e_i(z)\label{eq:res-form-phi-0}\\
  =&\nonumber
  (\qqq_i-\qqq_i\inv)^k
  \qqq_i^{-k(k+1)/2}[k]_{\qqq_i}!
  \:e_i(\qqq_i^{2k}z)
  e_i(\qqq_i^{2(k-1)}z)
  \cdots e_i(\qqq_i^2z) e_i(z)\;.
\end{align}
\end{lem}

\begin{proof}
We prove \eqref{eq:serre-phi-0} and \eqref{eq:res-form-phi-0} by using induction on $k$.
It is clear that \eqref{eq:serre-phi-0}
and \eqref{eq:res-form-phi-0} hold for $k=0$.
Suppose that \eqref{eq:serre-phi-0} and \eqref{eq:res-form-phi-0} hold for $k$.
Then
\begin{align*}
  &e_i(\qqq_i^{a_{ij}+2k}z)_0^\phi
  e_i(\qqq_i^{a_{ij}+2(k-1)}z)_0^\phi e_i(\qqq_i^{a_{ij}+2(k-2)}z)_0^\phi\cdots e_i(\qqq_i^{a_{ij}}z)_0^\phi e_j(z)\\
  =&\Res_{z_0}Y_\E^\phi(e_i(\qqq_i^{a_{ij}+2k}z),z_0)
  e_i(\qqq_i^{a_{ij}+2(k-1)}z)_0^\phi e_i(\qqq_i^{a_{ij}+2(k-2)}z)_0^\phi\cdots e_i(\qqq_i^{a_{ij}}z)_0^\phi e_j(z)\\
  =&\Res_{z_0}(1-e^{-z_0})\inv \:e_i(\qqq_i^{a_{ij}+2k}ze^{z_0})
  e_i(\qqq_i^{a_{ij}+2(k-1)}z)
  e_i(\qqq_i^{a_{ij}+2(k-2)}z)\cdots
  e_i(\qqq_i^{a_{ij}}z) e_j(z)\;\\
  =&\:e_i(\qqq_i^{a_{ij}+2k}z)
  e_i(\qqq_i^{a_{ij}+2(k-1)}z)
  e_i(\qqq_i^{a_{ij}+2(k-2)}z)\cdots
  e_i(\qqq_i^{a_{ij}}z) e_j(z)\;,
\end{align*}
which proves \eqref{eq:serre-phi-0} for $k+1$.
And
\begin{align*}
  &e_i(\qqq_i^{2(k+1)}z)_0^\phi
  e_i(\qqq_i^{2k}z)_0^\phi 
  \cdots e_i(\qqq_i^2z)_0^\phi e_i(z)\\
  =&\Res_{z_0}Y_\E^\phi(e_i(\qqq_i^{2(k+1)}z),z_0)
  e_i(\qqq_i^{2k}z)_0^\phi
  e_i(\qqq_i^{2(k-1)}z)_0^\phi
  \cdots e_i(\qqq_i^2z)_0^\phi e_i(z)\\
  =&\Res_{z_0}(\qqq_i-\qqq_i\inv)^k
  \qqq_i^{-k(k+1)/2}[k]_{\qqq_i}!
  Y_\E^\phi(e_i(\qqq_i^{2(k+1)}z),z_0)\\
  &\times
  \:e_i(\qqq_i^{2k}z)
  e_i(\qqq_i^{2(k-1)}z)
  \cdots e_i(\qqq_i^2z) e_i(z)\;\\
  =&(\qqq_i-\qqq_i\inv)^k
  \qqq_i^{-k(k+1)/2}[k]_{\qqq_i}!
  \Res_{z_0}\frac{1-\qqq_i^{-2k-2}e^{-z_0}}{1-e^{-z_0}}\\
  &\times \:
  e_i(\qqq_i^{2(k+1)}ze^{z_0})e_i(\qqq_i^{2k}z)
  e_i(\qqq_i^{2(k-1)}z)
  \cdots e_i(\qqq_i^2z) e_i(z)\;\\
  =&(\qqq_i-\qqq_i\inv)^{k+1}
  \qqq_i^{-(k+1)(k+2)/2}[k+1]_{\qqq_i}!\\
  &\times
  \:e_i(\qqq_i^{2(k+1)}z)e_i(\qqq_i^{2k}z)
  e_i(\qqq_i^{2(k-1)}z)
  \cdots e_i(\qqq_i^2z) e_i(z)\;,
\end{align*}
which proves \eqref{eq:res-form-phi-0} for $k+1$.
\end{proof}

We now prove that every $W\in\obj\mathcal R_\zeta^\ell(\wh\g)$ satisfying \eqref{zeta-serre}.
\begin{lem}\label{lem:xij-ind}
For $k\in\N$, we have that
\begin{align}
  &e_i(z_1)\:e_i(\qqq_i^{a_{ij}+ 2(k-1)}z_2)e_i(\qqq_i^{a_{ij}+ 2(k-2)}z_2)\cdots e_i(\qqq_i^{a_{ij}}z_2)e_j(z_2)\;\nonumber\\
  &+\frac{1-\qqq_i^{- a_{ij}+ 2}z_1/z_2}{1-\qqq_i^{- a_{ij}- 2(k-1)}z_1/z_2}
  \frac{1-\qqq_i^{a_{ij}}z_1/z_2}{\qqq_i^{a_{ij}+ 2k}-z_1/z_2}\frac{z_2}{z_1}\nonumber\\
  &\quad\times\:e_i(\qqq_i^{a_{ij}+ 2(k-1)}z_2)e_i(\qqq_i^{a_{ij}+ 2(k-2)}z_2)\cdots e_i(\qqq_i^{a_{ij}}z_2)e_j(z_2)\;
  e_i(z_1)\nonumber\\
  =&\:e_i(\qqq_i^{a_{ij}+ 2k}z_2)e_i(\qqq_i^{ a_{ij}+ 2(k-1)}z_2)\cdots e_i(\qqq_i^{a_{ij}}z_2)
  e_j(z_2)\;\delta\left(\frac{\qqq_i^{a_{ij}+ 2k}z_2}{z_1}\right).\label{eq:xij-ind}
\end{align}
Moreover, 
\begin{align}\label{eq:xij-k}
  \:e_i(\qqq_i^{a_{ij}+ 2(k-1)}z_2)e_i(\qqq_i^{ a_{ij}+ 2(k-2)}z_2)\cdots e_i(\qqq_i^{ a_{ij}}z_2)e_j(z_2)\;
  \in \wh\E(W).
\end{align}
\end{lem}

\begin{proof}
The prove of \eqref{eq:xij-ind} is straightforward.
We prove \eqref{eq:xij-k} by using induction on $k$.
It is clear that \eqref{eq:xij-k} holds for $k=0$.
Suppose it is true for $k$. Then the induction assumption and \eqref{eq:xij-ind} yield that
\begin{align*}
  &\:e_i(\qqq_i^{ a_{ij}+ 2k}z_2)e_i(\qqq_i^{ a_{ij}+ 2(k-1)}z_2)\cdots e_i(\qqq_i^{ a_{ij}}z_2)e_j(z_2)\;\\
  =&e_i(0)\:e_i(\qqq_i^{ a_{ij}+ 2(k-1)}z_2)e_i(\qqq_i^{ a_{ij}+ 2(k-2)}z_2)\cdots e_i(\qqq_i^{ a_{ij}}z_2)e_j(z_2)\;\\
  &+\Res_{z_1}z_1\inv \frac{1-\qqq_i^{- a_{ij}+ 2}z_1/z_2}{1-\qqq_i^{- a_{ij}- 2(k-1)}z_1/z_2}
  \frac{1-\qqq_i^{ a_{ij}}z_1/z_2}{\qqq_i^{ a_{ij}+ 2k}-z_1/z_2}\frac{z_2}{z_1}\nonumber\\
  &\quad\times\:e_i(\qqq_i^{ a_{ij}+ 2(k-1)}z_2)e_i(\qqq_i^{ a_{ij}+ 2(k-2)}z_2)\cdots e_i(\qqq_i^{ a_{ij}}z_2)e_j(z_2)\;
  e_i(z_1)\in\wh\E(W),
\end{align*}
which proves \eqref{eq:xij-k} for $k+1$.
%
%
\end{proof}

The following result rewrite the relation \eqref{Q8} into a normal ordering form, whose proof is analogue to but much simpler than that of \cite[Theorem 5.16]{CJKT-qeala-II-twisted-qaffinization}.
\begin{prop}\label{prop:serre}
The relation
\begin{align}\label{eq:serre}
  &\sum_{\sigma\in S_{m_{ij}}}\sum_{k=0}^{m_{ij}}(-1)^k\qb{m_{ij}}{k}_{\qqq_i}
  e_i(z_{\sigma(1)})\cdots e_i(z_{\sigma(k)})e_j(w)\\
  &\quad\times e_i(z_{\sigma(k+1)})\cdots e_i(z_{\sigma(m_{ij})})=0\quad \te{for }i,j\in I\,\,\te{with }a_{ij}\le 0\nonumber
\end{align}
is equivalent to the following relation
\begin{align*}
  \:e_i(\qqq_i^{-a_{ij}}z)e_i(\qqq_i^{-2-a_{ij}}z)\cdots e_i(\qqq_i^{a_{ij}}z)e_j(z)\;\quad \te{for }i,j\in I\,\,\te{with }a_{ij}\le 0.
\end{align*}
\end{prop}

For $i\in I$ and $k\in \N$, we set
\begin{align*}
  &e_{i,k}(z)=\sum_{n\in\Z}e_{i,k}(n)z^{-n}=\: e_i(\qqq_i^{2(k-1)}z)e_i(\qqq_i^{2(k-2)}z)\cdots e_i(z)\;.
\end{align*}
We now express $e_i(n)^k$ as a sum of products of $e_{i,s}(m)$ for $s\le k$.

\begin{lem}\label{lem:xk-ind}
For $i\in I$ and $k\in \N$, we have the following equations
\begin{align*}
  &e_i(z)e_{i,k}(w)-\frac{w-\qqq_i^{ 2}z}{w-\qqq_i^{- 2(k-1)}z} \frac{w-z}{\qqq_i^{ 2k}w-z}e_{i,k}(w)e_i(z)
  =
  (1-\qqq_i^{- 2k})e_{i,k+1}(w)\delta\left(\frac{\qqq_i^{ 2k}w}{z}\right),\\
  &e_{i,k}(z)e_i(w)-\frac{w-\qqq_i^{ 2k}z}{w-z}\frac{\qqq_i^{- 2(k-1)}w-z}{\qqq_i^{ 2}w-z}e_i(w)e_{i,k}(z)
  =(1-\qqq_i^{- 2k})e_{i,k+1}(w)\delta\left(\frac{\qqq_i^{ 2}w}{z}\right).
\end{align*}
Moreover, if $([m]_{\qqq_i}!)\inv\in\Bbbk$, then $e_{i,m+1}(z)\in\wh\E(W)$.
\end{lem}

\begin{proof}
The proofs of the first two relations are straightforward.
It is clear that $$e_{i,1}(z)=e_i(z)\in \wh\E(W).$$
Suppose that $e_{i,k}(z)\in \wh\E(W)$ and $k\le m$.
The assumption of $\qqq$ implies that $(\qqq_i-\qqq_i\inv)\inv\in \Bbbk$.
Then the hypotheses $([m]_{\qqq_i}!)\inv\in\Bbbk$
implies that $\prod_{s=1}^m(\qqq^{2s}-1)\inv\in \Bbbk$.
Under this, the first equation yields the following relation
\begin{align*}
  e_{i,k+1}(w)
  =&e_i(0)\frac{e_{i,k}(w)}{1-\qqq_i^{ 2k}}
  -\Res_zz\inv\frac{1-\qqq_i^{ 2}z/w}{1-\qqq_i^{- 2(k-1)}z/w}\frac{1-z/w}{\qqq_i^{ 2k}-z/w}
  \frac{e_{i,k}(w)}{1-\qqq_i^{ 2k}}e_i(z)\\
  =&e_i(0)\frac{e_{i,k}(w)}{1-\qqq_i^{ 2k}}-\qqq_i^{ 2k}\frac{e_{i,k}(w)}{1-\qqq_i^{ 2k}}e_i(0)\\
  &-\Res_zz\inv\frac{1+\qqq_i^{ 2k}-\qqq_i^{ 2}z/w-z/w}{(1-\qqq_i^{- 2k+ 2}z/w)(\qqq_i^{ 2k}-z/w)}e_{i,k}(w)e_i(z)
  \in\wh\E(W).
\end{align*}
By using induction on $k$, we complete the proof of the moreover statement.
\end{proof}

Let $J$ be an ordered finite set, and let $J_1,J_2,\dots,J_s$ be pair-wisely disjoint subsets of $J$ such that $\uplus_a J_a=J$.
We denote by $S_J$ the symmetric group on $J$ and
\begin{align*}
  &S_{J_1,\dots,J_s}=\set{\tau\in S_J}{\tau(u)<\tau(v),\,\,\te{for }u<v \in J_a\,\,\te{for some }a}.
\end{align*}
Suppose that each subset $J_a$ is non-empty.
For each $J_a$, we denote by $m(J_a)$ the maximal element in $J_a$.
Define $S^{J_1,\dots,J_s}$ to be the set of elements $\tau$ in $S_J$ such that
\begin{align*}
  \tau(m(J_1))>\tau(m(J_2))>\cdots \tau(m(J_s)),\,\,\te{and }\tau(m(J_a))\ge \tau(u),\,\,\te{for }u\in J_a,\,1\le a\le s.
\end{align*}
Let $K$ be a subset of $J$. We may view $S_K$ as a subgroup of $S_J$ such that for any $\sigma \in S_K$,
\begin{align*}
  \sigma(a)=a\quad \te{for }a\not\in K.
\end{align*}

Given two integers $a\le b$, we set $[a,b]=\set{k\in\Z}{a\le k\le b}$.
Let $k\in\Z_+$.
For convenience, we write $S_k=S_{[1,k]}$.
Set $\mathcal P_k=\set{(p_1,\dots,p_s)\in\Z_+^s}{\sum_{a=1}^s p_a=k}$.
Given $\overrightarrow{p}=(p_1,\dots,p_s)\in\mathcal P_k$, we write
\begin{align*}
  &p^t=\sum_{a=1}^t p_a\quad\te{for }0\le t\le s,\quad\te{and write }\\
  &S_{\overrightarrow{p}}=S_{[1,p^1],[p^1+1,p^2],\dots,[p^{s-1}+1,p^s]},\quad
  S^{\overrightarrow{p}}=S^{[1,p^1],[p^1+1,p^2],\dots,[p^{s-1}+1,p^s]},
\end{align*}
for convenience.

\begin{lem}\label{lem:sym-gps}
Let $k\in\Z_+$, $\overrightarrow{p}=(p_1,\dots,p_s)\in\mathcal P_k$, $\sigma\in S_k$.
For each pair $(\tau,\tau')$ such that
\begin{align}\label{eq:sym-gps-1}
  \tau\in S_{(p_1,k-p_1)},\quad \tau'\in S_{[p^1+1,p^2],\dots, [p^{s-1}+1,p^s]},\quad
  \sigma\tau(p_1)=k,\quad \sigma\tau\tau'\in S^{\overrightarrow{p}},
\end{align}
we have that $\tau\tau'\in S_{\overrightarrow{p}}\cap \sigma\inv S^{\overrightarrow{p}}$.
On the other hand, for each $\tau''\in S_{\overrightarrow{p}}\cap \sigma\inv S^{\overrightarrow{p}}$, there exists a unique pair
$(\tau,\tau')$ satisfying the condition \eqref{eq:sym-gps-1}.
\end{lem}

\begin{proof}
It is straightforwardly to check that $\tau\tau'\in S_{\overrightarrow{p}}$ and $\sigma\tau\tau'\in S^{\overrightarrow{p}}$.
On the other hand, let $\tau\in S_{(p_1,k-p_1)}$ such that
\begin{align}\label{eq:sym-gps-2}
  \tau(a)=\tau''(a)\quad\te{for }1\le a\le p_1.
\end{align}
Since $\sigma\tau''\in S^{\overrightarrow{p}}$, we get that $\sigma\tau''(p_1)$ is maximal in $[1,k]$.
So $\sigma\tau(p_1)=\sigma\tau''(p_1)=k$.

Set $\tau'=\tau\inv\tau''$. From \eqref{eq:sym-gps-2}, we have that $\tau'\in S_{[p^1+1,k]}$.
Moreover, let $1\le t\le s-1$ and let $p^t<a<b\le p^{t+1}$. Suppose $\tau'(a)>\tau'(b)$.
As $\tau'\in S_{[p^1+1,k]}$, $a,b>p^1$ and $\tau\in S_{(p_1,k-p_1)}$, we get
\begin{align*}
  \tau''(a)=\tau\tau'(a)>\tau\tau'(b)=\tau''(b).
\end{align*}
But $\tau''(a)<\tau''(b)$, since $\tau''\in S_{\overrightarrow{p}}$.
This contradiction shows that $\tau'\in S_{[p^1+1,p^2],\dots, [p^{s-1}+1,p^s]}$.
Furthermore, the definition of $\tau'$ shows that $\tau\tau'=\tau''\in\sigma\inv S^{\overrightarrow{p}}$.

Finally, let $(\bar \tau,\bar\tau')$ be another pair such that $\tau''=\bar\tau\bar\tau'$ and satisfying the condition \eqref{eq:sym-gps-1}.
Then for each $1\le a\le p^1$, we have that
\begin{align*}
  &\tau(a)=\tau''(\tau')\inv(a)=\tau''(a)=\tau''(\bar\tau')\inv(a)=\bar\tau(a).
\end{align*}
Since both $\tau$ and $\bar\tau$ lie in $S_{(p_1,k-p_1)}$, we get that $\tau=\bar\tau$. This also implies that $\tau'=\bar\tau'$.
\end{proof}

\begin{lem}\label{lem:xiii1}
Let $i\in I$ and $k\in\Z_+$.
Then we have the following relation
\begin{align*}
  &e_i(z_1)e_i(z_2)\cdots e_i(z_k)\\
  =&
  \sum_{1\le p_1\le k}\sum_{\substack{\sigma\in S_{(p_1,k-p_1)} \\ \sigma(p_1)=k }}e_{i,p_1}(z_k)
  e_i(z_{\sigma(p_1+1)})e_i(z_{\sigma(p_1+2)})\cdots e_i(z_{\sigma(k)})\\
  &\times \prod_{\substack{ 1\le s\le p_1<t\le k\\ \sigma(s)>\sigma(t) }}\frac{\qqq_i^{ 2}z_{\sigma(t)}-z_{\sigma(s)}}
    {z_{\sigma(t)}-\qqq_i^{ 2}z_{\sigma(s)}}
    \prod_{s=1}^{p_1-1}(1-\qqq_i^{- 2s})
    \prod_{s=1}^{p_1-1}\delta\left(\frac{\qqq_i^{ 2}z_{\sigma(s+1)}}{z_{\sigma(s)}}\right).
\end{align*}
\end{lem}

\begin{proof}
We prove this lemma by using induction on $k$.
It is clear for $k=1$.
Suppose that this lemma holds true for $k$.
We get from Lemma \ref{lem:xk-ind} that
\begin{align}\label{eq:xiii1-temp1}
  &e_i(z_0)e_i(z_1)\cdots e_i(z_k)\nonumber\\
  =&e_i(z_0)\sum_{1\le p_1\le k}\sum_{\substack{\sigma\in S_{(p_1,k-p_1)} \\ \sigma(p_1)=k }}e_{i,p_1}(z_k)
  e_i(z_{\sigma(p_1+1)})e_i(z_{\sigma(p_1+2)})\cdots e_i(z_{\sigma(k)})\nonumber\\
  &\quad\times \prod_{\substack{ 1\le s\le p_1<t\le k\\ \sigma(s)>\sigma(t) }}
  \frac{\qqq_i^{\pm 2}z_{\sigma(t)}-z_{\sigma(s)}}
    {z_{\sigma(t)}-\qqq_i^{\pm 2}z_{\sigma(s)}}
    \prod_{s=1}^{p_1-1}(1-\qqq_i^{- 2s})
    \prod_{s=1}^{p_1-1}
    \delta\left(\frac{\qqq_i^{ 2}z_{\sigma(s+1)}}{z_{\sigma(s)}}\right)\nonumber\\
  =&\sum_{1\le p_1\le k}\sum_{\substack{\sigma\in S_{(p_1,k-p_1)} \\ \sigma(p_1)=k }}e_{i,p_1}(z_k)
  e_i(z_0)e_i(z_{\sigma(p_1+1)})e_i(z_{\sigma(p_1+2)})\cdots e_i(z_{\sigma(k)})\\
  &\quad\times
  \frac{z_k-\qqq_i^{ 2}z_0}{z_k-\qqq_i^{- 2(p_1-1)}z_0}\frac{z_k-z_0}{\qqq_i^{ 2p_1}z_k-z_0}
  \prod_{\substack{ 1\le s\le p_1<t\le k\\ \sigma(s)>\sigma(t) }}\frac{\qqq_i^{ 2}z_{\sigma(t)}-z_{\sigma(s)}}
    {z_{\sigma(t)}-\qqq_i^{ 2}z_{\sigma(s)}}\nonumber\\
  &\quad\times  \prod_{s=1}^{p_1-1}(1-\qqq_i^{- 2s})
    \prod_{s=1}^{p_1-1}\delta\left(\frac{\qqq_i^{ 2}z_{\sigma(s+1)}}{z_{\sigma(s)}}\right)\nonumber\\
  &+\sum_{1\le p_1\le k}\sum_{\substack{\sigma\in S_{(p_1,k-p_1)} \\ \sigma(p_1)=k }}e_{i,p_1+1}(z_k)
  e_i(z_{\sigma(p_1+1)})e_i(z_{\sigma(p_1+2)})\cdots e_i(z_{\sigma(k)})\nonumber\\
  &\quad\times \prod_{\substack{ 1\le s\le p_1<t\le k\\ \sigma(s)>\sigma(t) }}\frac{q_i^{ 2}z_{\sigma(t)}-z_{\sigma(s)}}
    {z_{\sigma(t)}-\qqq_i^{ 2}z_{\sigma(s)}}
    \prod_{s=1}^{p_1}(1-\qqq_i^{- 2s})
    \delta\left(\frac{\qqq_i^{ 2p_1}z_k}{z_0}\right)
    \prod_{s=1}^{p_1-1}\delta\left(\frac{\qqq_i^{ 2}z_{\sigma(s+1)}}{z_{\sigma(s)}}\right).\nonumber
\end{align}
For $\sigma\in S_{(p_1,k-p_1)}$, we define $\sigma',\sigma''\in S_{[0,k]}$ such that
\begin{align*}
  &\sigma'(s)=\sigma(s)\quad \te{for }1\le s\le k,\quad\te{and }\sigma'(0)=0,\\
  &\sigma''(s)=\sigma(s+1)\quad\te{for }0\le s<p_1,\quad \sigma''(p_1)=0,\quad\te{and }\sigma''(s)=\sigma(s)\quad\te{for }p_1<s\le k.
\end{align*}
We have
\begin{align}\label{eq:xiii1-temp2}
  &\sum_{1\le p_1\le k}\sum_{\substack{\sigma\in S_{(p_1,k-p_1)} \\ \sigma(p_1)=k }}e_{i,p_1}(z_k)
  e_i(z_0)e_i(z_{\sigma(p_1+1)})e_i(z_{\sigma(p_1+2)})\cdots e_i(z_{\sigma(k)})\nonumber\\
  &\times
  \frac{z_k-\qqq_i^{\pm 2}z_0}{z_k-\qqq_i^{- 2(p_1-1)}z_0}\frac{z_k-z_0}{\qqq_i^{ 2p_1}z_k-z_0}
  \prod_{\substack{ 1\le s\le p_1<t\le k\\ \sigma(s)>\sigma(t) }}\frac{\qqq_i^{ 2}z_{\sigma(t)}-z_{\sigma(s)}}
    {z_{\sigma(t)}-\qqq_i^{ 2}z_{\sigma(s)}}\nonumber\\
  &\times  \prod_{s=1}^{p_1-1}(1-\qqq_i^{- 2s})
    \prod_{s=1}^{p_1-1}\delta\left(\frac{\qqq_i^{ 2}z_{\sigma(s+1)}}{z_{\sigma(s)}}\right)\nonumber\\
  =&\sum_{1\le p_1\le k}\sum_{\substack{\sigma\in S_{(p_1,k-p_1)} \\ \sigma(p_1)=k }}e_{i,p_1}(z_k)
  e_i(z_{\sigma''(p_1)})e_i(z_{\sigma''(p_1+1)})\cdots e_i(z_{\sigma''(k)})\nonumber\\
  &\times
  \frac{z_{\sigma''(p_1-1)}-\qqq_i^{ 2}z_{\sigma''(p_1)}}{z_{\sigma''(p_1-1)}-\qqq_i^{- 2(p_1-1)}z_{\sigma''(p_1)}}
  \frac{z_{\sigma''(p_1-1)}-z_{\sigma''(p_1)}}{\qqq_i^{ 2p_1}z_{\sigma''(p_1-1)}-z_{\sigma''(p_1)}}\nonumber\\
  &\times
  \prod_{\substack{ 0\le s< p_1\\  p_1<t\le k\\ \sigma''(s)>\sigma''(t) }}\frac{\qqq_i^{ 2}z_{\sigma''(t)}-z_{\sigma''(s)}}
    {z_{\sigma''(t)}-\qqq_i^{ 2}z_{\sigma''(s)}}
    \prod_{s=1}^{p_1-1}(1-\qqq_i^{- 2s})
    \prod_{s=0}^{p_1-2}\delta\left(\frac{\qqq_i^{ 2}z_{\sigma''(s+1)}}{z_{\sigma''(s)}}\right)\nonumber\\
  =&\sum_{1\le p_1\le k}\sum_{\substack{\sigma\in S_{(p_1,k-p_1)} \\ \sigma(p_1)=k }}e_{i,p_1}(z_k)
  e_i(z_{\sigma''(p_1)})e_i(z_{\sigma''(p_1+1)})\cdots e_i(z_{\sigma''(k)})\\
  &\times
  \prod_{\substack{ 0\le s< p_1\le t\le k\\ \sigma''(s)>\sigma''(t) }}\frac{\qqq_i^{ 2}z_{\sigma''(t)}-z_{\sigma''(s)}}
    {z_{\sigma''(t)}-\qqq_i^{\pm 2}z_{\sigma''(s)}}
    \prod_{s=1}^{p_1-1}(1-\qqq_i^{- 2s})
    \prod_{s=0}^{p_1-2}\delta\left(\frac{\qqq_i^{ 2}z_{\sigma''(s+1)}}{z_{\sigma''(s)}}\right),\nonumber
\end{align}
since
\begin{align*}
  &\prod_{\substack{ 0\le s< p_1\\ \sigma''(s)>\sigma''(p_1) }}\frac{\qqq_i^{ 2}z_{\sigma''(p_1)}-z_{\sigma''(s)}}
    {z_{\sigma''(p_1)}-\qqq_i^{\pm 2}z_{\sigma''(s)}}\prod_{s=0}^{p_1-2}
    \delta\left(\frac{\qqq_i^{ 2}z_{\sigma''(s+1)}}{z_{\sigma''(s)}}\right)\\
  =&\prod_{ 0\le s< p_1}\frac{\qqq_i^{ 2}z_0-z_{\sigma''(s)}}
    {z_0-\qqq_i^{ 2}z_{\sigma''(s)}}\prod_{s=0}^{p_1-2}
    \delta\left(\frac{\qqq_i^{ 2(p_1-1-s)}z_k}{z_{\sigma''(s)}}\right)\\
  =&\prod_{ 0\le s< p_1}\frac{\qqq_i^{ 2}z_0-\qqq_i^{ 2(p_1-1-s)}z_k}
    {z_0-\qqq_i^{ 2}\qqq_i^{ 2(p_1-1-s)}z_k}\prod_{s=0}^{p_1-2}
    \delta\left(\frac{\qqq_i^{ 2(p_1-1-s)}z_k}{z_{\sigma''(s)}}\right)\\
  =&\frac{z_k-\qqq_i^{ 2}z_0}{z_k-\qqq_i^{- 2(p_1-1)}z_0}\frac{z_k-z_0}{\qqq_i^{ 2p_1}z_k-z_0}
    \prod_{s=0}^{p_1-2}\delta\left(\frac{\qqq_i^{ 2(p_1-1-s)}z_k}{z_{\sigma''(s)}}\right)\\
  =&\frac{z_{\sigma''(p_1-1)}-\qqq_i^{ 2}z_{\sigma''(p_1)}}{z_{\sigma''(p_1-1)}-\qqq_i^{- 2(p_1-1)}z_{\sigma''(p_1)}}
  \frac{z_{\sigma''(p_1-1)}-z_{\sigma''(p_1)}}{\qqq_i^{ 2p_1}z_{\sigma''(p_1-1)}-z_{\sigma''(p_1)}}
  \prod_{s=0}^{p_1-2}\delta\left(\frac{\qqq_i^{ 2(p_1-1-s)}z_k}{z_{\sigma''(s)}}\right).
\end{align*}
And
\begin{align}\label{eq:xiii1-temp3}
  &\sum_{1\le p_1\le k}\sum_{\substack{\sigma\in S_{(p_1,k-p_1)} \\ \sigma(p_1)=k }}e_{i,p_1+1}(z_k)
  e_i(z_{\sigma(p_1+1)})e_i(z_{\sigma(p_1+2)})\cdots e_i(z_{\sigma(k)})\nonumber\\
  &\times \prod_{\substack{ 1\le s\le p_1<t\le k\\ \sigma(s)>\sigma(t) }}\frac{\qqq_i^{ 2}z_{\sigma(t)}-z_{\sigma(s)}}
    {z_{\sigma(t)}-\qqq_i^{ 2}z_{\sigma(s)}}
    \prod_{s=1}^{p_1}(1-\qqq_i^{- 2s})
    \delta\left(\frac{\qqq_i^{ 2p_1}z_k}{z_0}\right)
    \prod_{s=1}^{p_1-1}\delta\left(\frac{\qqq_i^{ 2}z_{\sigma(s+1)}}{z_{\sigma(s)}}\right)\nonumber\\
  =&
  \sum_{1\le p_1\le k}\sum_{\substack{\sigma\in S_{(p_1,k-p_1)} \\ \sigma(p_1)=k }}e_{i,p_1+1}(z_k)
  e_i(z_{\sigma'(p_1+1)})e_i(z_{\sigma'(p_1+2)})\cdots e_i(z_{\sigma'(k)})\\
  &\times \prod_{\substack{ 0\le s\le p_1<t\le k\\ \sigma'(s)>\sigma'(t) }}\frac{\qqq_i^{ 2}z_{\sigma'(t)}-z_{\sigma'(s)}}
    {z_{\sigma'(t)}-\qqq_i^{ 2}z_{\sigma'(s)}}
    \prod_{s=1}^{p_1}(1-\qqq_i^{- 2s})
    \prod_{s=0}^{p_1-1}\delta\left(\frac{\qqq_i^{ 2}z_{\sigma'(s+1)}}{z_{\sigma'(s)}}\right),\nonumber
\end{align}
since $\sigma'(0)=0$.
Notice that
\begin{align*}
  S_{[0,p_1],[p_1+1,k]}=\{\sigma'\,|\,\sigma\in S_{(p_1,k-p_1)}\}\uplus \{\sigma''\,|\,\sigma\in S_{(p_1+1,k-p_1-1)}\}.
\end{align*}
Combining this with \eqref{eq:xiii1-temp1}, \eqref{eq:xiii1-temp2} and \eqref{eq:xiii1-temp3}, we get that
\begin{align*}
  &e_i(z_0)e_i(z_1)\cdots e_i(z_k)\\
  =&\sum_{0\le p_1\le k}\sum_{\substack{\sigma\in S_{[0,p_1],[p_1+1,k]} \\ \sigma(p_1)=k }}e_{i,p_1}(z_k)
  e_i(z_{\sigma(p_1+1)})e_i(z_{\sigma(p_1+2)})\cdots e_i(z_{\sigma(k)})\\
  &\times \prod_{\substack{ 0\le s\le p_1<t\le k\\ \sigma(s)>\sigma(t) }}\frac{\qqq_i^{ 2}z_{\sigma(t)}-z_{\sigma(s)}}
    {z_{\sigma(t)}-\qqq_i^{ 2}z_{\sigma(s)}}
    \prod_{s=1}^{p_1}(1-\qqq_i^{- 2s})
    \prod_{s=0}^{p_1-1}\delta\left(\frac{\qqq_i^{ 2}z_{\sigma(s+1)}}{z_{\sigma(s)}}\right).
\end{align*}
By relabeling the symbols we see that the lemma holds true for $k+1$.
\end{proof}

\begin{lem}\label{lem:xiii2}
For each $i\in I$, $k\in \Z_+$ and $\sigma\in S_k$, the following relation holds
\begin{align*}
  &e_i(z_{\sigma(1)})e_i(z_{\sigma(2)})\cdots e_i(z_{\sigma(k)})\\
  =&\sum_{\overrightarrow p=(p_1,\dots,p_s)\in\mathcal P_k}\sum_{\tau\in S_{\overrightarrow p}\cap \sigma\inv S^{\overrightarrow p}}(-1)^{|\tau|}
  \prod_{1\le a<b\le k}\frac{z_{\sigma\tau (a)}-\qqq_i^{ 2}z_{\sigma\tau (b)}}
    { z_{\sigma(a)}-\qqq_i^{ 2}z_{\sigma(b)}}\\
  &\times e_{i,p_1}(z_{\sigma\tau(p^1)}) e_{i,p_2}(z_{\sigma\tau(p^2)}) \cdots e_{i,p_s}(z_{\sigma\tau(p^s)})\\
  &\times \qqq_i^{- \sum_{t=1}^s p_t(p_t-1)/2}(\qqq_i-\qqq_i^{ -1})^{k-s}\prod_{t=1}^s \left([p_t-1]_{\qqq_i}!\right)\\
  &\times \prod_{t=1}^{p^1-1}\delta\left( \frac{ \qqq_i^{ 2} z_{\sigma\tau(t+1)}}{ z_{\sigma\tau(t)} } \right)
  \prod_{t=p^1+1}^{p^2-1}\delta\left( \frac{ \qqq_i^{ 2} z_{\sigma\tau(t+1)}}{ z_{\sigma\tau(t)} } \right)
  \cdots
  \prod_{t=p^{s-1}+1}^{p^s-1}\delta\left( \frac{ \qqq_i^{ 2} z_{\sigma\tau(t+1)}}{ z_{\sigma\tau(t)} } \right).
\end{align*}
\end{lem}

\begin{proof}
We prove this lemma by using induction on $l(\sigma):=\set{(s,t)}{1\le s<t\le k,\,\sigma(s)<\sigma(t)}$.
It is clear that the lemma holds true if $l(\sigma)=0$.
Suppose $l(\sigma)>0$. Let $1\le p\le k$ such that $\sigma(p)=k$.
We get from Lemma \ref{lem:xiii1} that
\begin{align*}
  &e_i(z_{\sigma(1)})e_i(z_{\sigma(2)})\cdots e_i(z_{\sigma(k)})\\
  =&\sum_{1\le p_1\le p}\sum_{\substack{\tau\in S_{(p_1,p-p_1)} \\ \tau(p_1)=p }}e_{i,p_1}(z_k)
  e_i(z_{\sigma\tau(p_1+1)})e_i(z_{\sigma\tau(p_1+2)})\cdots e_i(z_{\sigma\tau(k)})\\
  &\times \prod_{\substack{ 1\le s\le p_1<t\le p\\ \tau(s)>\tau(t) }}\frac{\qqq_i^{ 2}z_{\sigma\tau(t)}-z_{\sigma\tau(s)}}
    {z_{\sigma\tau(t)}-\qqq_i^{ 2}z_{\sigma\tau(s)}}
    \prod_{s=1}^{p_1-1}(1-\qqq_i^{- 2s})
    \prod_{s=1}^{p_1-1}\delta\left(\frac{\qqq_i^{ 2}z_{\sigma\tau(s+1)}}{z_{\sigma\tau(s)}}\right),
\end{align*}
where $\tau\in S_{(p_1,k-p_1)}$ is viewed as an element in $S_k$ by letting $\tau(s)=s$ for $s>p$.
Notice that for each $\tau\in S_{(p_1,k-p_1)}$ such that $\tau(p_1)=p$, we have that
\begin{align*}
  \tau(s)=s\quad \te{for }s>p.
\end{align*}
Then $S_{(p_1,p-p_1)}=S_{(p_1,k-p_1)}$ viewed as subgroups of $S_k$.
It also implies that
\begin{align*}
  &\set{(s,t)}{1\le s\le p_1<t\le p,\,\, \tau(s)>\tau(t)}=\set{(s,t)}{1\le s<t\le k,\,\, \tau(s)>\tau(t)}.
\end{align*}
We need the following two facts
\begin{align*}
  & \prod_{\substack{ 1\le s\le p_1<t\le p\\ \tau(s)>\tau(t) }}\frac{\qqq_i^{ 2}z_{\sigma\tau(t)}-z_{\sigma\tau(s)}}
    {z_{\sigma\tau(t)}-\qqq_i^{ 2}z_{\sigma\tau(s)}}
  =\prod_{\substack{ 1\le s<t\le k\\ \tau(s)>\tau(t) }}\frac{\qqq_i^{ 2}z_{\sigma\tau(t)}-z_{\sigma\tau(s)}}
    {z_{\sigma\tau(t)}-\qqq_i^{ 2}z_{\sigma\tau(s)}}
  =(-1)^{|\tau|}
  \prod_{1\le a<b\le k}\frac{z_{\sigma\tau (a)}-\qqq_i^{ 2}z_{\sigma\tau (b)}}
    { z_{\sigma(a)}-\qqq_i^{ 2}z_{\sigma(b)}},
\end{align*}
and that
\begin{align*}
  &\set{(p_1,\tau)}{1\le p_1\le p,\,\,\tau\in S_{(p_1,k-p_1)},\,\,\tau(p_1)=p}\\
  =&\set{(p_1,\tau)}{1\le p_1\le k,\,\,\tau\in S_{(p_1,k-p_1)},\,\,\sigma\tau(p_1)=k}.
\end{align*}
Combining these with the induction assumption, we get that
\begin{align*}
  &e_i(z_{\sigma(1)})e_i(z_{\sigma(2)})\cdots e_i(z_{\sigma(k)})\\
  =&\sum_{1\le p_1\le k}\sum_{\substack{\tau\in S_{(p_1,k-p_1)} \\ \sigma\tau(p_1)=k }}e_{i,p_1}(z_k)
  e_i(z_{\sigma\tau(p_1+1)})e_i(z_{\sigma\tau(p_1+2)})\cdots e_i(z_{\sigma\tau(k)})\\
  &\times (-1)^{|\tau|}
  \prod_{1\le a<b\le k}\frac{z_{\sigma\tau (a)}-\qqq_i^{ 2}z_{\sigma\tau (b)}}
    { z_{\sigma(a)}-\qqq_i^{ 2}z_{\sigma(b)}}
    \prod_{s=1}^{p_1-1}(1-\qqq_i^{- 2s})
    \prod_{s=1}^{p_1-1}\delta\left(\frac{\qqq_i^{ 2}z_{\sigma\tau(s+1)}}{z_{\sigma\tau(s)}}\right)\\
  =&\sum_{1\le p_1\le k}\sum_{\substack{\tau\in S_{(p_1,k-p_1)} \\ \sigma\tau(p_1)=k }}
  \sum_{\substack{p_2,\dots,p_s\in\Z_+\\ p_1+p_2+\cdots +p_s=k}}\sum_{\substack{\tau'\in S_{[p^1+1,p^2],\cdots,[p^{s-1}+1,p^s]}\\ \sigma\tau\tau'\in S^{[p^1+1,p^2],\dots,[p^{s-1}+1,p^s]}}}(-1)^{|\tau\tau'|}\\
  &\times \prod_{1\le a<b\le k}\frac{z_{\sigma\tau (a)}-\qqq_i^{ 2}z_{\sigma\tau (b)}}
    { z_{\sigma(a)}-\qqq_i^{ 2}z_{\sigma(b)}}
    \prod_{p_1<a<b\le k}\frac{z_{\sigma\tau\tau'(a)}-\qqq_i^{ 2}z_{\sigma\tau\tau'(b)}}{z_{\sigma\tau(a)}-\qqq_i^{ 2}z_{\sigma\tau(b)}}\\
  &\times e_{i,p_1}(z_{\sigma\tau(p^1)})e_{i,p_2}(z_{\sigma\tau\tau'(p^2)})\cdots e_{i,p_s}(z_{\sigma\tau\tau'(p^s)})\\
  &\times \prod_{t=1}^{p_1-1}(1-\qqq_i^{- 2t})
    \prod_{t=1}^{p_1-1}\delta\left(\frac{\qqq_i^{ 2}z_{\sigma\tau(t+1)}}{z_{\sigma\tau(t)}}\right)\\
  &\times \qqq_i^{- \sum_{t=2}^sp_t(p_t-1)/2}(\qqq_i-\qqq_i^{- 1})^{k-s-1}\prod_{t=2}^s([p_t-1]_{\qqq_i}!)\\
  &\times
  \prod_{t=p^1+1}^{p^2-1}\delta\left(\frac{\qqq_i^{ 2}z_{\sigma\tau\tau'(t+1)}}{z_{\sigma\tau\tau'(t)}}\right)
  \prod_{t=p^2+1}^{p^3-1}\delta\left(\frac{\qqq_i^{ 2}z_{\sigma\tau\tau'(t+1)}}{z_{\sigma\tau\tau'(t)}}\right)
  \cdots
  \prod_{t=p^{s-1}+1}^{p^s-1}\delta\left(\frac{\qqq_i^{ 2}z_{\sigma\tau\tau'(t+1)}}{z_{\sigma\tau\tau'(t)}}\right)\\
  =&\sum_{\overrightarrow{p}=(p_1,\dots,p_s)\in\mathcal P_k}\sum_{\substack{\tau\in S_{(p_1,k-p_1)} \\ \sigma\tau(p_1)=k }}
  \sum_{\substack{\tau'\in S_{[p^1+1,p^2],\cdots,[p^{s-1}+1,p^s]}\\ \sigma\tau\tau'\in S^{[p^1+1,p^2],\dots,[p^{s-1}+1,p^s]}}}(-1)^{|\tau\tau'|}\\
  &\times \prod_{1\le a<b\le k}\frac{z_{\sigma\tau\tau'(a)}-\qqq_i^{ 2}z_{\sigma\tau\tau'(b)}}{z_{\sigma\tau(a)}-\qqq_i^{ 2}z_{\sigma\tau(b)}}
  e_{i,p_1}(z_{\sigma\tau\tau'(p^1)})
  e_{i,p_2}(z_{\sigma\tau\tau'(p^2)})\cdots e_{i,p_s}(z_{\sigma\tau\tau'(p^s)})
  \\
  &\times \qqq_i^{- \sum_{t=1}^sp_t(p_t-1)/2}(\qqq_i-\qqq_i^{- 1})^{k-s}\prod_{t=1}^s([p_t-1]_{\qqq_i}!)\\
  &\times
  \prod_{t=1}^{p^1-1}\delta\left(\frac{\qqq_i^{ 2}z_{\sigma\tau\tau'(t+1)}}{z_{\sigma\tau\tau'(t)}}\right)
  \prod_{t=p^1+1}^{p^2-1}\delta\left(\frac{\qqq_i^{ 2}z_{\sigma\tau\tau'(t+1)}}{z_{\sigma\tau\tau'(t)}}\right)
  \cdots
  \prod_{t=p^{s-1}+1}^{p^s-1}\delta\left(\frac{\qqq_i^{ 2}z_{\sigma\tau\tau'(t+1)}}{z_{\sigma\tau\tau'(t)}}\right),
\end{align*}
since
\begin{align*}
  &\tau'(a)=a\quad\te{for }1\le a\le p_1,\quad\te{and }\\
  &\prod_{1\le a<b\le k}\frac{z_{\sigma\tau (a)}-\qqq_i^{ 2}z_{\sigma\tau (b)}}
    { z_{\sigma(a)}-\qqq_i^{ 2}z_{\sigma(b)}}
    \prod_{p_1<a<b\le k}\frac{z_{\sigma\tau\tau'(a)}-\qqq_i^{ 2}z_{\sigma\tau\tau'(b)}}{z_{\sigma\tau(a)}-\qqq_i^{ 2}z_{\sigma\tau(b)}}\\
  =&\prod_{1\le a<b\le k}\frac{z_{\sigma\tau(a)}-\qqq_i^{ 2}z_{\sigma\tau(b)}}
    { z_{\sigma(a)}-\qqq_i^{ 2}z_{\sigma(b)}}
    \prod_{1\le a<b\le k}\frac{z_{\sigma\tau\tau'(a)}-\qqq_i^{ 2}z_{\sigma\tau\tau'(b)}}{z_{\sigma\tau(a)}-\qqq_i^{ 2}z_{\sigma\tau(b)}}\\
  =&\prod_{1\le a<b\le k}\frac{z_{\sigma\tau\tau' (a)}-\qqq_i^{ 2}z_{\sigma\tau\tau' (b)}}
    { z_{\sigma(a)}-\qqq_i^{ 2}z_{\sigma(b)}}.
\end{align*}
Combining this with Lemma \ref{lem:sym-gps}, we complete the proof of lemma.
\end{proof}

\begin{lem}\label{lem:xiii3}
Let $i\in I$, $k\in \Z_+$ and $n\in\Z$.
Suppose that
\begin{align*}
  \frac{e_{i,p}(z)}{[p]_{\qqq_i}}\in \wh\E(W)\quad\te{for any }p<k.
\end{align*}
Then
\begin{align*}
  (e_{i,k}(n))^{(k)}\in\End(W)\quad\te{if and only if}\quad \frac{e_{i,k}(nk)}{[k]_{\qqq_i}}\in \End(W),
\end{align*}
where $\End(W)$ denotes the set of all continuous $\mathcal A$-module maps on $W$,
and
\begin{align*}
  &(e_i(n))^{(k)}=(\qqq_i-\qqq_i^{- 1})^{k-1}\qqq_i^{- k(k-1)/2+ nk(k-1)}\frac{e_{i,k}(nk)}{k[k]_{\qqq_i}}\\
  &+\mathcal A\te{-linear combinations of }\frac{1}{k!}\frac{e_{i,p_1}(m_1)}{[p_1]_{\qqq_i}}\cdots \frac{e_{i,p_s}(m_s)}{[p_s]_{\qqq_i}}
  \,\,\te{with }p_1,\dots p_s<k.
\end{align*}
\end{lem}

%

\begin{proof}
From Lemma \ref{lem:xiii2}, we have that
\begin{align*}
  &\sum_{\sigma\in S_k}e_i(z_{\sigma(1)})e_i(z_{\sigma(2)})\cdots e_i(z_{\sigma(k)})\\
  =&\sum_{\overrightarrow p=(p_1,\dots,p_s)\in\mathcal P_k}\sum_{\sigma\in S_k}
  \sum_{\tau\in S_{\overrightarrow p}\cap \sigma\inv S^{\overrightarrow p}}
  (-1)^{|\tau|}
  \prod_{1\le a<b\le k}\frac{z_{\sigma\tau (a)}-\qqq_i^{ 2}z_{\sigma\tau (b)}}
    { z_{\sigma(a)}-\qqq_i^{ 2}z_{\sigma(b)}}\\
  &\times e_{i,p_1}(z_{\sigma\tau(p^1)}) e_{i,p_2}(z_{\sigma\tau(p^2)}) \cdots e_{i,p_s}(z_{\sigma\tau(p^s)})\\
  &\times \qqq_i^{- \sum_{t=1}^s p_t(p_t-1)/2}(\qqq_i-\qqq_i^{- 1})^{k-s}\prod_{t=1}^s \left([p_t-1]_{\qqq_i}!\right)\\
  &\times \prod_{t=1}^{p^1-1}\delta\left( \frac{ \qqq_i^{ 2} z_{\sigma\tau(t+1)}}{ z_{\sigma\tau(t)} } \right)
  \prod_{t=p^1+1}^{p^2-1}\delta\left( \frac{ \qqq_i^{ 2} z_{\sigma\tau(t+1)}}{ z_{\sigma\tau(t)} } \right)
  \cdots
  \prod_{t=p^{s-1}+1}^{p^s-1}\delta\left( \frac{ \qqq_i^{ 2} z_{\sigma\tau(t+1)}}{ z_{\sigma\tau(t)} } \right)\\
  =&\sum_{\overrightarrow p=(p_1,\dots,p_s)\in\mathcal P_k}\sum_{\tau\in S_{\overrightarrow{p}}}\sum_{\gamma\in S^{\overrightarrow{p}}}
  (-1)^{|\tau|}
  \prod_{1\le a<b\le k}\frac{z_{\gamma (a)}-\qqq_i^{ 2}z_{\gamma (b)}}
    { z_{\gamma\tau\inv(a)}-\qqq_i^{ 2}z_{\gamma\tau\inv(b)}}\\
  &\times e_{i,p_1}(z_{\gamma(p^1)}) e_{i,p_2}(z_{\gamma(p^2)}) \cdots e_{i,p_s}(z_{\gamma(p^s)})\\
  &\times \qqq_i^{- \sum_{t=1}^s p_t(p_t-1)/2}(\qqq_i-\qqq_i^{- 1})^{k-s}\prod_{t=1}^s \left([p_t-1]_{\qqq_i}!\right)\\
  &\times \prod_{t=1}^{p^1-1}\delta\left( \frac{ \qqq_i^{ 2} z_{\gamma(t+1)}}{ z_{\gamma(t)} } \right)
  \prod_{t=p^1+1}^{p^2-1}\delta\left( \frac{ \qqq_i^{ 2} z_{\gamma(t+1)}}{ z_{\gamma(t)} } \right)
  \cdots
  \prod_{t=p^{s-1}+1}^{p^s-1}\delta\left( \frac{ \qqq_i^{ 2} z_{\gamma(t+1)}}{ z_{\gamma(t)} } \right)\\
  =&\sum_{\overrightarrow p=(p_1,\dots,p_s)\in\mathcal P_k}\sum_{\gamma\in S^{\overrightarrow{p}}}\sum_{\tau\in S_{\overrightarrow{p}}}
  (-1)^{|\tau|}
  \prod_{1\le a<b\le k}\frac{z_{\gamma (a)}-\qqq_i^{ 2}z_{\gamma (b)}}
    { z_{\gamma\tau\inv(a)}-\qqq_i^{ 2}z_{\gamma\tau\inv(b)}}\\
  &\times e_{i,p_1}(z_{\gamma(p^1)}) e_{i,p_2}(z_{\gamma(p^2)}) \cdots e_{i,p_s}(z_{\gamma(p^s)})\\
  &\times \qqq_i^{- \sum_{t=1}^s p_t(p_t-1)/2}(\qqq_i-\qqq_i^{- 1})^{k-s}\prod_{t=1}^s \left([p_t-1]_{\qqq_i}!\right)\\
  &\times \prod_{t=1}^{p^1-1}\delta\left( \frac{ \qqq_i^{ 2} z_{\gamma(t+1)}}{ z_{\gamma(t)} } \right)
  \prod_{t=p^1+1}^{p^2-1}\delta\left( \frac{ \qqq_i^{ 2} z_{\gamma(t+1)}}{ z_{\gamma(t)} } \right)
  \cdots
  \prod_{t=p^{s-1}+1}^{p^s-1}\delta\left( \frac{ \qqq_i^{ 2} z_{\gamma(t+1)}}{ z_{\gamma(t)} } \right)\\
  =&\sum_{\overrightarrow p=(p_1,\dots,p_s)\in\mathcal P_k}\sum_{\gamma\in S^{\overrightarrow{p}}}\sum_{\tau\in S_{\overrightarrow{p}}}
  \prod_{\substack{1\le a<b\le k\\ \tau(a)>\tau(b)}}\frac{z_{\gamma (a)}-q_i^{\pm 2}z_{\gamma (b)}}
    {\qqq_i^{ 2} z_{\gamma(a)}-z_{\gamma(b)}}\\
  &\times e_{i,p_1}(z_{\gamma(p^1)}) e_{i,p_2}(z_{\gamma(p^2)}) \cdots e_{i,p_s}(z_{\gamma(p^s)})\\
  &\times \qqq_i^{- \sum_{t=1}^s p_t(p_t-1)/2}(\qqq_i-\qqq_i^{- 1})^{k-s}\prod_{t=1}^s \left([p_t-1]_{\qqq_i}!\right)\\
  &\times \prod_{t=1}^{p^1-1}\delta\left( \frac{ \qqq_i^{ 2} z_{\gamma(t+1)}}{ z_{\gamma(t)} } \right)
  \prod_{t=p^1+1}^{p^2-1}\delta\left( \frac{ \qqq_i^{ 2} z_{\gamma(t+1)}}{ z_{\gamma(t)} } \right)
  \cdots
  \prod_{t=p^{s-1}+1}^{p^s-1}\delta\left( \frac{ \qqq_i^{ 2} z_{\gamma(t+1)}}{ z_{\gamma(t)} } \right)\\
  =&\sum_{\overrightarrow p=(p_1,\dots,p_s)\in\mathcal P_k}\sum_{\gamma\in S^{\overrightarrow{p}}}\sum_{\tau\in S_{\overrightarrow{p}}}
  (-1)^{|\tau|}
  \prod_{1\le a<b\le k}\frac{\qqq_i^{ 2}z_{\gamma\tau\inv (a)}-z_{\gamma\tau\inv (b)}}
    {\qqq_i^{ 2} z_{\gamma(a)}-z_{\gamma(b)}}\\
  &\times e_{i,p_1}(z_{\gamma(p^1)}) e_{i,p_2}(z_{\gamma(p^2)}) \cdots e_{i,p_s}(z_{\gamma(p^s)})\\
  &\times \qqq_i^{- \sum_{t=1}^s p_t(p_t-1)/2}(\qqq_i-\qqq_i^{- 1})^{k-s}\prod_{t=1}^s \left([p_t-1]_{\qqq_i}!\right)\\
  &\times \prod_{t=1}^{p^1-1}\delta\left( \frac{ \qqq_i^{ 2} z_{\gamma(t+1)}}{ z_{\gamma(t)} } \right)
  \prod_{t=p^1+1}^{p^2-1}\delta\left( \frac{ \qqq_i^{ 2} z_{\gamma(t+1)}}{ z_{\gamma(t)} } \right)
  \cdots
  \prod_{t=p^{s-1}+1}^{p^s-1}\delta\left( \frac{ \qqq_i^{ 2} z_{\gamma(t+1)}}{ z_{\gamma(t)} } \right).
\end{align*}
Notice that
\begin{align*}
  &\prod_{1\le a<b\le k}(\qqq_i^{ 2}z_a-z_b)\inv \sum_{\sigma\in S_{\overrightarrow p}}(-1)^{|\sigma|}
  \prod_{1\le a<b\le k}(\qqq_i^{ 2}z_{\sigma\inv(a)}-z_{\sigma\inv(b)})\\
  &\times
  \prod_{t=1}^{p^1-1}\delta\left( \frac{ \qqq_i^{ 2} z_{t+1}}{ z_{t} } \right)
  \prod_{t=p^1+1}^{p^2-1}\delta\left( \frac{ \qqq_i^{ 2} z_{t+1}}{ z_{t} } \right)
  \cdots
  \prod_{t=p^{s-1}+1}^{p^s-1}\delta\left( \frac{ \qqq_i^{ 2} z_{t+1}}{ z_{t} } \right)\\
  =&\prod_{1\le a<b\le k}(\qqq_i^{ 2}z_a-z_b)\inv \sum_{\sigma\in S_{\overrightarrow p}}(-1)^{|\sigma|}
  \prod_{1\le \sigma (a)<\sigma(b)\le k}(\qqq_i^{ 2}z_a-z_b)\\
  &\times
  \prod_{t=1}^{p^1-1}\delta\left( \frac{ \qqq_i^{ 2} z_{t+1}}{ z_{t} } \right)
  \prod_{t=p^1+1}^{p^2-1}\delta\left( \frac{ \qqq_i^{ 2} z_{t+1}}{ z_{t} } \right)
  \cdots
  \prod_{t=p^{s-1}+1}^{p^s-1}\delta\left( \frac{ \qqq_i^{ 2} z_{t+1}}{ z_{t} } \right)\\
  =&\prod_{1\le a<b\le k}(z_a-\qqq_i^{- 2}z_b)\inv \sum_{\sigma\in S_k}(-1)^{|\sigma|}
  \prod_{1\le \sigma (a)<\sigma(b)\le k}(z_a-\qqq_i^{- 2}z_b)\\
  &\times
  \prod_{t=1}^{p^1-1}\delta\left( \frac{ \qqq_i^{ 2} z_{t+1}}{ z_{t} } \right)
  \prod_{t=p^1+1}^{p^2-1}\delta\left( \frac{ \qqq_i^{ 2} z_{t+1}}{ z_{t} } \right)
  \cdots
  \prod_{t=p^{s-1}+1}^{p^s-1}\delta\left( \frac{ \qqq_i^{ 2} z_{t+1}}{ z_{t} } \right)\\
  =&\qqq_i^{- k(k-1)/2}[k]_{\qqq_i}!\prod_{1\le a<b\le k}(z_a-\qqq_i^{- 2}z_b)\inv
  \prod_{1\le a<b\le k}(z_a-z_b)\\
  &\times
  \prod_{t=1}^{p^1-1}\delta\left( \frac{ \qqq_i^{ 2} z_{t+1}}{ z_{t} } \right)
  \prod_{t=p^1+1}^{p^2-1}\delta\left( \frac{ \qqq_i^{ 2} z_{t+1}}{ z_{t} } \right)
  \cdots
  \prod_{t=p^{s-1}+1}^{p^s-1}\delta\left( \frac{ \qqq_i^{ 2} z_{t+1}}{ z_{t} } \right)\\
  =&\qqq_i^{ k(k-1)/2}[k]_{\qqq_i}!\prod_{1\le a<b\le k}\frac{z_a-z_b}{ \qqq_i^{ 2} z_a- z_b}\\
  &\times
  \prod_{t=1}^{p^1-1}\delta\left( \frac{ \qqq_i^{ 2} z_{t+1}}{ z_{t} } \right)
  \prod_{t=p^1+1}^{p^2-1}\delta\left( \frac{ \qqq_i^{ 2} z_{t+1}}{ z_{t} } \right)
  \cdots
  \prod_{t=p^{s-1}+1}^{p^s-1}\delta\left( \frac{ \qqq_i^{ 2} z_{t+1}}{ z_{t} } \right)\\
  =&\qqq_i^{ k(k-1)/2 - \sum_{t=1}^s p_t(p_t-1)/2} \frac{ [k]_{\qqq_i}!}{\prod_{t=1}^s [p_t]_{\qqq_i}! }
  \prod_{1\le u<v\le s}\prod_{a=1}^{p_u}
  \prod_{b=1}^{p_v}
  \frac{\qqq_i^{ 2(p_u-a)} z_{p^u}- \qqq_i^{ 2(p_v-b)}z_{p^v} }{\qqq_i^{ 2(p_u-a)+ 2}z_{p^u}- \qqq_i^{ 2(p_v-b)}z_{p^v} }\\
  &\times
  \prod_{t=1}^{p^1-1}\delta\left( \frac{ \qqq_i^{ 2} z_{t+1}}{ z_{t} } \right)
  \prod_{t=p^1+1}^{p^2-1}\delta\left( \frac{ \qqq_i^{ 2} z_{t+1}}{ z_{t} } \right)
  \cdots
  \prod_{t=p^{s-1}+1}^{p^s-1}\delta\left( \frac{ \qqq_i^{ 2} z_{t+1}}{ z_{t} } \right)\\
  =&\qqq_i^{ k(k-1)/2 - \sum_{t=1}^s p_t(p_t-1)/2} \frac{ [k]_{\qqq_i}!}{\prod_{t=1}^s [p_t]_{\qqq_i}! }
  \prod_{1\le u<v\le s}\prod_{a=0}^{p_u-1}
  \prod_{b=0}^{p_v-1}
  \frac{\qqq_i^{ 2a} z_{p^u}- \qqq_i^{ 2b}z_{p^v} }{\qqq_i^{ 2a+ 2}z_{p^u}- \qqq_i^{ 2b}z_{p^v} }\\
  &\times
  \prod_{t=1}^{p^1-1}\delta\left( \frac{ \qqq_i^{ 2} z_{t+1}}{ z_{t} } \right)
  \prod_{t=p^1+1}^{p^2-1}\delta\left( \frac{ \qqq_i^{ 2} z_{t+1}}{ z_{t} } \right)
  \cdots
  \prod_{t=p^{s-1}+1}^{p^s-1}\delta\left( \frac{ \qqq_i^{ 2} z_{t+1}}{ z_{t} } \right)\\
  =&\qqq_i^{ k(k-1)/2 - \sum_{t=1}^s p_t(p_t-1)/2} \frac{ [k]_{\qqq_i}!}{\prod_{t=1}^s [p_t]_{\qqq_i}! }
  \prod_{1\le u<v\le s}\prod_{b=0}^{p_v-1}\frac{z_{p^u}-\qqq_i^{ 2b}z_{p^v}}{ \qqq_i^{ 2p_u}z_{p^u}-\qqq_i^{ 2b}z_{p^v}} \\
  &\times
  \prod_{t=1}^{p^1-1}\delta\left( \frac{ \qqq_i^{ 2} z_{t+1}}{ z_{t} } \right)
  \prod_{t=p^1+1}^{p^2-1}\delta\left( \frac{ \qqq_i^{ 2} z_{t+1}}{ z_{t} } \right)
  \cdots
  \prod_{t=p^{s-1}+1}^{p^s-1}\delta\left( \frac{ \qqq_i^{ 2} z_{t+1}}{ z_{t} } \right)\\
  =&\qqq_i^{ k(k-1)/2 - \sum_{t=1}^s p_t(p_t-1)/2} \frac{ [k]_{\qqq_i}!}{\prod_{t=1}^s [p_t]_{\qqq_i}! }\\
  &\times
  \prod_{1\le u<v\le s}\prod_{b=0}^{p_v-1}\left( \qqq_i^{- 2p_u}+(\qqq_i^{- 2p_u}-1)\sum_{n=1}^\infty \qqq_i^{ 2n(b-p_u)}(z_{p^v}/z_{p^u})^n  \right) \\
  &\times
  \prod_{t=1}^{p^1-1}\delta\left( \frac{ \qqq_i^{ 2} z_{t+1}}{ z_{t} } \right)
  \prod_{t=p^1+1}^{p^2-1}\delta\left( \frac{ \qqq_i^{ 2} z_{t+1}}{ z_{t} } \right)
  \cdots
  \prod_{t=p^{s-1}+1}^{p^s-1}\delta\left( \frac{ \qqq_i^{ 2} z_{t+1}}{ z_{t} } \right)\\
  =&\qqq_i^{ k(k-1)/2 - \sum_{t=1}^s p_t(p_t-1)/2}\qqq_i^{- 2\sum_{t=2}^sp^{t-1}p_t} \frac{ [k]_{\qqq_i}!}{\prod_{t=1}^s [p_t]_{\qqq_i}! }\\
  &\times
  \prod_{1\le u<v\le s}\prod_{b=0}^{p_v-1}\left( 1+(1-\qqq_i^{ 2p_u})\sum_{n=1}^\infty \qqq_i^{ 2n(b-p_u)}(z_{p^v}/z_{p^u})^n  \right) \\
  &\times
  \prod_{t=1}^{p^1-1}\delta\left( \frac{ \qqq_i^{ 2} z_{t+1}}{ z_{t} } \right)
  \prod_{t=p^1+1}^{p^2-1}\delta\left( \frac{ \qqq_i^{ 2} z_{t+1}}{ z_{t} } \right)
  \cdots
  \prod_{t=p^{s-1}+1}^{p^s-1}\delta\left( \frac{ \qqq_i^{ 2} z_{t+1}}{ z_{t} } \right).
\end{align*}
Combining these we get the following relation
\begin{align*}
  &\frac{1}{[k]_{\qqq_i}!}\sum_{\sigma\in S_k}e_i(z_{\sigma(1)})\cdots e_i(z_{\sigma(k)})\\
  =&\sum_{\overrightarrow p=(p_1,\dots,p_s)\in\mathcal P_k}\sum_{\gamma\in S^{\overrightarrow{p}}}
  \frac{e_{i,p_1}(z_{\gamma(p^1)})}{[p_1]_{\qqq_i}} \frac{e_{i,p_2}(z_{\gamma(p^2)})}{[p_2]_{\qqq_i}} \cdots \frac{e_{i,p_s}(z_{\gamma(p^s)})}{[p_s]_{\qqq_i}}\\
  &(\qqq_i-\qqq_i^{- 1})^{k-s}\qqq_i^{ k(k-1)/2 - \sum_{t=1}^s p_t(p_t-1)}\qqq_i^{- 2\sum_{t=2}^sp^{t-1}p_t}
  \\
  &\times \prod_{1\le u<v\le s}\prod_{b=0}^{p_v-1}\left( 1+(1-\qqq_i^{ 2p_u})\sum_{n=1}^\infty \qqq_i^{ 2n(b-p_u)}\frac{z_{\gamma(p^v)}^n}{z_{\gamma(p^u)}^n}  \right)\\
  &\times
  \prod_{t=1}^{p^1-1}\delta\left( \frac{ \qqq_i^{ 2} z_{\gamma(t+1)}}{ z_{\gamma(t)} } \right)
  \prod_{t=p^1+1}^{p^2-1}\delta\left( \frac{ \qqq_i^{ 2} z_{\gamma(t+1)}}{ z_{\gamma(t)} } \right)
  \cdots
  \prod_{t=p^{s-1}+1}^{p^s-1}\delta\left( \frac{ \qqq_i^{ 2} z_{\gamma(t+1)}}{ z_{\gamma(t)} } \right).
\end{align*}
By taking $\Res_{z_1,\dots,z_k}z_1^{n-1}\cdots z_k^{n-1}$ on the both hand sides of the relation above, we complete the proof of lemma.
\end{proof}

%

\begin{prop}\label{prop:xiii5}
The following relation holds in $\wh\U_\zeta^\ell(\wh\g)$
\begin{align*}
  &x_{i,\wp_i}^\pm(z)=0.
\end{align*}
\end{prop}

\begin{proof}
From the definition of $x_{i,\wp_i}^\pm(z)$, we have that
\begin{align*}
  &x_{i,\wp_i}^\pm(\zeta_i^{\pm 2}z)=\:x_i^\pm(\zeta_i^{\pm 2\wp_i}z)x_i^\pm(\zeta_i^{\pm 2(\wp_i-1)}z)\cdots x_i^\pm(\zeta_i^{\pm 2}z)\;\\
  =&\:x_i^\pm(z)x_i^\pm(\zeta_i^{\pm 2(\wp_i-1)}z)x_i^\pm(\zeta_i^{\pm 2(\wp_i-2)}z)\cdots x_i^\pm(\zeta_i^{\pm 2}z)\;\\
  =&\:x_i^\pm(\zeta_i^{\pm 2(\wp_i-1)}z)x_i^\pm(\zeta_i^{\pm 2(\wp_i-2)}z)\cdots x_i^\pm(z)\;
  =x_{i,\wp_i}^\pm(z).
\end{align*}
It follows that
\begin{align*}
  &x_{i,\wp_i}^\pm(n)=0\quad\te{for }\wp_i\not|\,n.
\end{align*}
Notice that $[\wp_i-1]_{q_i}!$ is invertible in $\mathcal A$.
From Lemma \ref{lem:xiii3}, we get that
\begin{align*}
  &\frac{x_{i,\wp_i}^\pm(n\wp_i)}{[\wp_i]_{q_i}}
  =\frac{x_{i,\wp_i}^\pm(n\wp_i)}{[\wp_i]_{q_i}}1
  \in \wh\U_{\mathcal A}^{\ell,\res}(\wh\g) \quad\te{for }n\in \Z.
\end{align*}
Since $[\wp_i]_{\zeta_i}=0$, we get that for any $n\in\Z$
\begin{align*}
  x_{i,\wp_i}^\pm(n\wp_i)=\ev_\zeta\left([\wp_i]_{q_i}\frac{x_{i,\wp_i}^\pm(n\wp_i)}{[\wp_i]_{q_i}}\right)
  =[\wp_i]_{\zeta_i}\ev_\zeta\left(\frac{x_{i,\wp_i}^\pm(n\wp_i)}{[\wp_i]_{q_i}}\right)=0.
\end{align*}
Therefore, we complete the proof.
\end{proof}

On the other hand, let $\overline \U_{\mathcal A}^\ell(\wh\g)$ be the quotient algebra of $\wh \U_{\mathcal A}^\ell(\wh\g)$ modulo the
closed ideal generated by $x_{i,\wp_i}^\pm(n)$ for $i\in I$ and $n\in\Z$.
Note that the map $\ev_\zeta:\wh\U_{\mathcal A}^\ell(\wh\g)\to \wh\U_\zeta^\ell(\wh\g)$ factor through $\overline \U_{\mathcal A}^\ell(\wh\g)$.

\begin{prop}\label{prop:xiii6}
For each $k\in\Z_+$, we have that
\begin{align*}
  &\frac{x_{i,k}^\pm(z)}{[k]_{q_i}}\in\overline \U_{\mathcal A}^\ell(\wh\g)[[z,z\inv]].
\end{align*}
Moreover, we have that
\begin{align*}
  (x_i^\pm(n)^{(k)}\in \overline\U_{\mathcal A}^\ell(\wh\g),\quad i\in I,\quad n\in\Z,\quad k\in\Z_+.
\end{align*}
\end{prop}

\begin{proof}
From Lemma \ref{lem:xk-ind}, we get that
\begin{align*}
  &x_{i,k}^\pm(z)\in \overline \U_{\mathcal A}^\ell(\wh\g)[[z,z\inv]]\quad\te{for }1\le k\le \wp_i,
  \quad\te{and hence}\\
  &\frac{x_{i,k}^\pm(z)}{[k]_{q_i}}\in \overline \U_{\mathcal A}^\ell(\wh\g)[[z,z\inv]]\quad\te{for }1\le k<\wp_i,
  \,\,\te{since }[k]_{q_i}\inv\in\mathcal A.
\end{align*}
We then prove the first statement by using induction on $k\ge \wp_i-1$.
Suppose that
\begin{align*}
  \frac{x_{i,k}^\pm(z)}{[k]_{q_i}}\in \overline \U_{\mathcal A}^\ell(\wh\g)[[z,z\inv]].
\end{align*}
From Lemma \ref{lem:xk-ind}, we get that
\begin{align*}
  &x_{i,k+1}^\pm(w)=(q_i^{\mp 1}-q_i^{\pm 1})q_i^{\mp k}x_i^\pm(0)\frac{x_{i,k}^\pm(w)}{[k]_{q_i}}
  -(q_i^{\mp 1}-q_i^{\pm 1})q_i^{\pm k}\frac{x_{i,k}^\pm(w)}{[k]_{q_i}}x_i^\pm(0)\\
  &\quad-\Res_zz\inv\frac{1+q_i^{\pm 2k}-q_i^{\pm 2}z/w-z/w}{(1-q_i^{\mp 2k\pm 2r_i}z/w)(q_i^{\pm 2k}-z/w)}x_{i,k}^\pm(w)x_i^\pm(z)
  \in \overline \U_{\mathcal A}^\ell(\wh\g)[[z,z\inv]].
\end{align*}
Since $x_{i,\wp_i}^\pm(z)=0$, we get that,
\begin{align*}
  &x_{i,k+1}^\pm(z)=\:x_i^\pm(q^{\pm 2k}z)x_i^\pm(q^{\pm 2(k-1)}z)\cdots x_i^\pm(q^{\pm 2\wp_i}z)x_{i,\wp_i}^\pm(z)\;=0.
\end{align*}
Combining this with Lemma \ref{lem:xk-ind}, we get that
\begin{align*}
  &x_i^\pm(z)x_{i,k}^\pm(w)=\frac{w-q_i^{\pm 2}z}{w-q_i^{\mp 2(k-1)}}\frac{w-z}{q_i^{\pm 2k}w-z}x_{i,k}^\pm(w)x_i^\pm(z)\in\wh\E^{(2)}(\overline \U_{\mathcal A}^\ell(\wh\g)),\\
  &x_{i,k}^\pm(z)x_i^\pm(w)=\frac{w-q_i^{\pm 2k}z}{w-z}\frac{q_i^{\mp 2(k-1)}w-z}{q_i^{\pm 2}w-z}x_i^\pm(w)x_{i,k}^\pm(z)\in\wh\E^{(2)}(\overline \U_{\mathcal A}^\ell(\wh\g)).
\end{align*}
Then
\begin{align*}
  &\:x_i^\pm(z)x_{i,k}^\pm(w)\;=\:x_{i,k}^\pm(w)x_i^\pm(z)\;
  =\frac{w-q_i^{\pm 2}z}{w-q_i^{\mp 2(k-1)}z}x_{i,k}^\pm(w)x_i^\pm(z).
\end{align*}
It follows that
\begin{align*}
  \frac{x_{i,k+1}^\pm(z)}{[k+1]_{q_i}}=&\frac{1}{[k+1]_{q_i}}\left.\left(\frac{z-q_i^{\pm 2}z_1}{z-q_i^{\mp 2(k-1)}z_1}x_{i,k}^\pm(z)x_i^\pm(z_1)\right)\right|_{z_1=q_i^{\pm 2k}z}\\
  =&q_i^{\pm k}x_{i,k}^\pm(z)x_i^\pm(q_i^{\pm 2k}z)\in\overline \U_{\mathcal A}^\ell(\wh\g)[[z,z\inv]],
\end{align*}
which proves the first statements for $k+1$.
Finally, the moreover statement follows from
the first statement and Lemma \ref{lem:xiii3}.
\end{proof}

%

\begin{proof}[{\bf Proof of Theorem \ref{thm:presentation}}]
Let $W\in \obj\mathcal R_\zeta^\ell(\wh\g)$. From Proposition \ref{prop:U-zeta-M}, we get that $W\in\obj\mathcal R_\zeta'^\ell(\wh\g)$ by letting
\begin{align*}
  &H_i(z)=\wt h_i(z),\quad \Psi_i^\pm(z)=\psi_i^-(z)^{\pm 1}\psi_i^+(z)^{\mp 1},\quad X_i^+(z)=x_i^+(z),\\ &X_i^-(z)=(\zeta_i-\zeta_i\inv)x_i^-(z)\psi_i^+(\zeta^{-r\ell}z)\inv,\quad \te{for }i\in I.
\end{align*}
Proposition \ref{prop:smooth-continuous} yields that
$W$ is a $\wh\U_\zeta^\ell(\wh\g)$-module.
Applying Proposition \ref{prop:serre} to \eqref{Q8}, we get that
\begin{align}\label{eq:pf-serre-1}
  \:x_i^\pm(\zeta_i^{\mp a_{ij}}z)x_i^\pm(\zeta_i^{\mp 2\mp a_{ij}}z)\cdots x_i^\pm(\zeta_i^{\pm a_{ij}}z)
  x_j^\pm(z)\;=0\quad\te{for }i,j\in I,\,\te{with }a_{ij}\le 0.
\end{align}\label{eq:pf-res-form-1}
By Proposition \ref{prop:xiii5}, we have that
\begin{align}
  \:x_i^\pm(\zeta_i^{\pm 2(\wp_i-1)}z)x_i^\pm(\zeta_i^{\pm 2(\wp_i-2)}z)
  \cdots x_i^\pm(z)\;=0\quad\te{for }i\in I.
\end{align}
Note that for $i_1,\dots,i_k\in I$, one has that
\begin{align}\label{eq:pf-no-alt}
  &\:X_{i_1}^-(z_1)X_{i_2}^-(z_2)\cdots X_{i_k}^-(z_k)\;\\
  =&(\zeta_i-\zeta_i\inv)^k\:x_{i_1}^\pm(z_1)\cdots x_{i_k}^\pm(z_k)\;\psi_{i_1}^+(z_1q^{-r\ell})\inv \cdots \psi_{i_k}^+(z_kq^{-r\ell})\inv.\nonumber
\end{align}
Then \eqref{eq:pf-serre-1} is equivalent to the following relation
\begin{align*}
  \:X_i^\pm(\zeta_i^{- a_{ij}}z)X_i^\pm(\zeta_i^{- 2- a_{ij}}z)\cdots X_i^\pm(\zeta_i^{a_{ij}}z)
  x_j^\pm(z)\;=0\quad\te{for }i,j\in I,\,\te{with }a_{ij}\le 0,
\end{align*}
which is equivalent to \eqref{zeta-serre}, since \eqref{eq:serre-phi-0}.
And \eqref{eq:pf-res-form-1} is equivalent to the following relation
\begin{align*}
  \:X_i^\pm(\zeta_i^{2(\wp_i-1)}z)X_i^\pm(\zeta_i^{2(\wp_i-2)}z)
  \cdots X_i^\pm(z)\;=0\quad\te{for }i\in I,
\end{align*}
which yields \eqref{zeta-res-form}.

On the other hand, let $(W,H_i(z),\Psi_i^\pm(z),X_i^\pm(z))\in\obj\mathcal R_\zeta'^\ell(\wh\g)$ satisfying the conditions \eqref{zeta-serre} and \eqref{zeta-res-form}.
Proposition \ref{prop:U-zeta-M-reverse} shows that $W$ becomes a smooth $\U_{\mathcal A(\wp)}'^\ell(\wh\g)$-module by letting
\begin{align*}
  &k_i^{\pm 1}\mapsto \zeta_i^{\pm H_i(0)}, \quad \qb{k_i}{\wp_i}_{q_i}\mapsto \qb{H_i(0)}{\wp_i}_{\zeta_i},
  \quad  \wt h_i(m)\mapsto H_i(m), \\
  &\psi_i^\pm(z)\mapsto\zeta_i^{\pm H_i(0)}\exp\left( \sum_{\pm m>0}(\zeta_i^m-\zeta_i^{-m})H_i(m)z^{-m} \right),\\
  &x_i^+(z)\mapsto X_i^+(z),\quad x_i^-(z)\mapsto (\zeta_i-\zeta_i\inv)\inv X_i^-(z)\psi_i^+(\zeta^{-r\ell}z)
\end{align*}
for $i\in I$, $0\ne m\in\Z$,
and hence a $\wh\U_{\mathcal A(\wp)}'^\ell(\wh\g)$-module,
since Proposition \ref{prop:smooth-continuous}.
Applying Lemma \ref{lem:no=phi-0} to \eqref{zeta-serre} and \eqref{zeta-res-form}, we get that
\begin{align*}
  &\:X_i^\pm(\zeta_i^{- a_{ij}}z)X_i^\pm(\zeta_i^{- 2- a_{ij}}z)\cdots X_i^\pm(\zeta_i^{a_{ij}}z)
  x_j^\pm(z)\;=0\quad\te{for }i,j\in I,\,\te{with }a_{ij}\le 0,\\
  &\:X_i^\pm(\zeta_i^{2(\wp_i-1)}z)X_i^\pm(\zeta_i^{2(\wp_i-2)}z)
  \cdots X_i^\pm(z)\;=0\quad\te{for }i\in I,
\end{align*}
where the second relation needs an additional fact
\begin{align*}
  (\zeta_i-\zeta_i\inv)^{\wp_i-1}\zeta_i^{-\wp_i(\wp_i-1)/2}
  [\wp_i-1]_{\zeta_i}\ne 0.
\end{align*}
Note that the relation \eqref{eq:pf-no-alt} still holds.
Then the relations \eqref{eq:pf-serre-1} and \eqref{eq:pf-res-form-1} hold on $W$.
Combining these with Proposition \ref{prop:serre}, we get that $W$ becomes a $\overline\U_{\mathcal A}^\ell(\wh\g)$-module.
By using Proposition \ref{prop:xiii6}, we get that
$W$ naturally becomes a smooth $\U_q^{\ell,\res}(\wh\g)$-module, and hence a smooth weighted $\U_\zeta(\wh\g)$-module of level $\ell$, as desired.
%
\end{proof}

\section{Construction of quantum vertex algebras}
\label{sec:construct}

In this section, we construct the desired $\Z_\wp$-module 
quantum vertex algebra corresponding to the category $\mathcal R_\zeta^\ell(\wh\g)$.

\subsection{General construction}\label{subsec:gen-construct}

This subsection is devoted to constructing a weak quantum vertex algebra from a category. More precisely, we prove the following result.

\begin{prop}\label{prop:V-M-pre}
Let $\mathcal M$ be a category whose objects are vector spaces $W$ equipped with fields
\begin{align*}
  \{a_0(z)=1_W\}\cup\set{a_i(z)}{i\in J}\subset \E(W),
\end{align*}
satisfying the relations below
\begin{align}\label{eq:V-M-pre-def-rels}
  &(z_1-z_2)^{M_{ij}}a_i(z_1)a_j(z_2)
  -(-z_2+z_1)^{M_{ij}}\sum_{s,t\in \{0\}\uplus J}f_{ij}^{st}(z_2-z_1)a_s(z_2)a_t(z_1)\\
  &\quad=\sum_{k=0}^{N_{ij}}c(i,j,k)a_{p(i,j,k)}(z_2)\frac{1}{k!}
  \pdiff{z_2}{k}z_1\inv\delta\left(\frac{z_2}{z_1}\right)
  \quad\te{for }i,j\in J,\nonumber
\end{align}
where $M_{ij},N_{ij}\in\Z$, $f_{ij}^{st}(z)\in\C((z))$ and
\begin{align*}
  p(i,j,-):\{0,1,\dots,N_{ij}\}\to \{0\}\uplus J,\quad c(i,j,-):\{0,1,\dots,N_{ij}\}\to \C
\end{align*}
are independent of the choice of $W$.
Then there exists
\begin{align*}
  (V,a_i(z))\in \obj\mathcal M,\quad \vac\in V,\quad\te{and}\quad
  a_i\in V,\quad i\in I,
\end{align*}
such that $V$ carries the structure of a weak quantum vertex algebra with vacuum vector $\vac$, and the vertex operator map $Y$ uniquely determined by
\begin{align*}
  Y(a_i,z)=a_i(z)\quad\te{for }i\in J.
\end{align*}
Moreover, for each $V$-module $(W,Y_W)$,
\begin{align*}
  (W,Y_W(a_i,z))\in \obj\mathcal M.
\end{align*}
Furthermore, for each $(W,a_i(z))\in \obj\mathcal M$,
there exists a $V$-module structure $Y_W$ determined by
\begin{align*}
  Y_W(a_i,z)=a_i(z)\quad\te{for }i\in J.
\end{align*}
\end{prop}

We first prove the following result.

\begin{lem}\label{lem:M-obj-V}
Let $\mathcal M$ be a category satisfying the hypotheses of Proposition \ref{prop:V-M-pre}.
Then there exists an object $(V,a_i(z))$ in $\mathcal M$
and a vector $\vac\in V$, satisfying the following two conditions:

(i) $a_i(z)\vac\in V[[z]]$ for all $i\in J$;

(ii) for each $(W,a_{i,W}(z))\in\obj\mathcal M$ equipped with a vector $v_+\in W$ satisfying the condition
\begin{align*}
  a_{i,W}(z)v_+\in W[[z]],
\end{align*}
there exists a unique linear map $\theta_W:V\to W$ such that
$\theta_W(\vac)=v_+$ and
\begin{align*}
  \theta_W(a_i(z)v)=a_{i,W}(z)\theta_W(v)\quad\te{for any }v\in W.
\end{align*}
\end{lem}

\begin{proof}
Let $\mathcal F$ be the forgetful functor from $\mathcal M$ to the category of vector spaces.
Define $\End_\C(\mathcal F)$ to be the algebra of endomorphisms of the functor $\mathcal F$.
For each $W\in\obj\mathcal M$, $\End_\C(W)$ is a topological algebra over $\C$ such that
\begin{align*}
  \set{(K)}{K\subset W,\,|K|<\infty}
\end{align*}
forms a local basis at $0$, where
\begin{align*}
  (K)=\set{f\in \End_\C(W)}{f(K)=0}.
\end{align*}
Equip $\End_\C(\mathcal F)$ with the coarsest topology such that for any $W\in\obj\mathcal M$ the canonical algebra epimorphism from $\End_\C(\mathcal F)$ to $\End_\C(W)$ is continuous.
For each $i\in \{0\}\uplus J$, $n\in\Z$, we define endomorphisms $a_i(n)$ of $\mathcal F$ as follows
\begin{align*}
  \sum_{n\in\Z}a_i(n).vz^{-n-1}=a_i(z)v\quad\te{where }v\in W,\,W\in \obj\mathcal M.
\end{align*}
It is easy to see that $a_0(n)=\delta_{n+1,0}1_{\mathcal F}$.
We denote by $\mathcal A$ the closed subalgebra of $\End_\C(\mathcal F)$ generated by $\set{a_i(n)}{i\in J,\,n\in\Z}$.
Let $\mathcal A_+$ be the minimal closed left ideal of $\mathcal A$ containing $a_i(n)$ $(i\in J,\,n\ge 0)$. Define
\begin{align*}
  V=\mathcal A/\mathcal A_+\quad\te{as a left module of }\mathcal A,
\end{align*}
and define $a_i(z)\in \End(V)[[z,z\inv]]$ by the very module action.
Set $\vac=1+\mathcal A_+$.
The construction of $V$ directly yields the two required conditions.

Observe that the fields $a_i(z)$ ($i\in J$) satisfy the defining relations \eqref{eq:V-M-pre-def-rels}. To complete the proof, it therefore suffices to show that $a_i(z)\in\mathcal{E}(V)$ for every $i\in J$.
Let $V'$ be the maximal subspace of $V$ consisting of vectors $v\in V$ such that
\begin{align*}
  a_i(z)v\in V((z))\quad \te{for all }i\in J.
\end{align*}
By the construction of $V$, $\vac$ lies in $V'$.
Now take arbitrary $i,j\in J$ and $v\in V'$. Using the relation \eqref{eq:V-M-pre-def-rels}, there exists $N\in\N$, such that
\begin{align*}
  (z_1-z_2)^Na_i(z_1)a_j(z_2)v=(z_1-z_2)^N\sum_{s,t\in \{0\}\uplus J}f_{ij}^{st}(z_2-z_1)a_s(z_2)a_t(z_1)v\in V((z_1,z_2)).
\end{align*}
It follows that
\begin{align*}
  a_i(z_1)a_j(z_2)v=(z_1-z_2)^{-N}\left((z_1-z_2)^Na_i(z_1)a_j(z_2)v\right)
  \in V((z_1))((z_2)).
\end{align*}
From this we deduce that $a_j(z)v\in V'((z))$, which shows that $V'$ is an $\mathcal{A}$-submodule of $V$. Since $V$ is generated by $\vac$ and $\vac\in V'$, we conclude that $V'=V$. Hence $a_i(z)\in\E(V)$ for all $i\in J$, as required.
\end{proof}

\begin{proof}[{\bf Proof of Proposition \ref{prop:V-M-pre}}]
For each $(W,a_i(z))\in\obj\mathcal M$, we set
\begin{align*}
  U_W=\{a_i(z) \mid i\in \{0\}\uplus J\}\subset\E(W).
\end{align*}
Note that $U_W$ is a $S$-local subset.
It follows from Theorem \ref{thm:nonlocal-VA-gen} and Remark \ref{rem:wquantum-VA-gen} that $U_W$ generates a weak quantum vertex algebra $(\<U_W\>,Y_\E,1_W)$ and $W$ becomes a faithful $\<U_W\>$-module with module map
$Y_W(a(z),z_0)=a(z_0)$ for $a(z)\in \<U_W\>$.
From \cite[Proposition 6.6]{Li-nonlocal}, we have that $\<U_W\>$ becomes an object in $\mathcal M$ with
fields $Y_\E(a_i(x),z)$.
From the vacuum property \eqref{eq:vacuum-prop}, we have that
\begin{align*}
  Y_\E(a_i(z),z_0)1_W\in \<U_W\>[[z_0]]\quad\te{for }i\in J.
\end{align*}
Using Lemma \ref{lem:M-obj-V}, we get a unique linear map $\theta_{\<U_W\>}:V\to \<U_W\>$, such that $\theta_{\<U_W\>}(\vac)=1_W$, and
\begin{align*}
  \theta_{\<U_W\>}(a_i(z)v)=Y_\E(a_i(z_1),z)\theta_{\<U_W\>}(v)
  \quad\te{for }i\in I,\,v\in V.
\end{align*}

Define $a_i=a_i(-1)\vac\in V$ for $i\in J$.
Recall from Lemma \ref{lem:M-obj-V} that $V\in\obj\mathcal M$.
Then
\begin{align*}
  \theta_{\<V\>}(a_i(z_0)v)=Y_\E(a_i(z),z_0)\theta_{\<V\>}(v)\quad\te{for }i\in J.
\end{align*}
Taking $v=\vac$, we get that
\begin{align*}
  &\theta_{V}(a_i(z_0)\vac)=Y_\E(a_i(z),z_0)\theta_{V}(\vac)=Y_\E(a_i(z),z_0)1_{V}=a_i(z+z_0).
\end{align*}
It follows that
\begin{align*}
  &\theta_{V}(a_i)=\Res_{z_0}z_0\inv \theta_{V}(a_i(z_0)\vac)=a_i(z).
\end{align*}
Applying \cite[Theorem 2.9]{Li-constructing}, the map $a_i\mapsto a_i(z)$ extends uniquely to a linear map $Y:V\to \E(V)$ such that
$(V, Y,\vac)$ carries the structure of a weak quantum vertex algebra.

By using \cite[Theorem 6.7]{Li-nonlocal}, we have that every $V$-module is an object in $\mathcal M$.
On the other hand, note that
\begin{align*}
  &\theta_{\<W\>}(Y(a_i,z_0)v)=\theta_{\<W\>}(a_i(z_0)v)
  =a_i(z_0)\theta_{\<W\>}(v)=Y_\E(a_i(z),z_0)\theta_{\<W\>}(v)
\end{align*}
for any $i\in J$. Hence, $\theta_{\<W\>}$ is a weak quantum vertex algebra homomorphism. Since $W$ is a faithful $\<U_W\>$-module,
we get that $W$ is a $V$-module.
\end{proof}

\begin{rem}
Note that for $f(z)\in\C((z))$, one has that
\begin{align*}
  f(z_1-z_2)-f(-z_2+z_1)=\sum_{k=0}^nc_k\frac{1}{k!}\pdiff{z_2}{k}
  z_1\inv\delta\left(\frac{z_2}{z_1}\right)
  \quad\te{for some }c_k\in \C,\,n\in\N.
\end{align*}
Then the relation
\begin{align*}
  [a_i(z_1),a_j(z_2)]=c^+f(z_1-z_2)-c^-f(-z_2+z_1)
\end{align*}
is equivalent to the following relation
\begin{align*}
  &a_i(z_1)a_j(z_2)-a_j(z_2)a_i(z_1)+ c^-f(-z_2+z_1)a_0(z_2)a_0(z_1)
  -c^+f(-z_2+z_1)a_0(z_2)a_0(z_1)\\
  &\qquad=\sum_{k=0}^nc_kc^+\frac{1}{k!}\pdiff{z_2}{k}
  z_1\inv\delta\left(\frac{z_2}{z_1}\right).
\end{align*}
And the relation
\begin{align*}
  [a_i(z_1),a_j(z_2)]=c^+a_j(z_2)f(z_1-z_2)-c^-a_j(z_2)f(-z_2+z_1)
\end{align*}
is equivalent to the following relation
\begin{align*}
  &a_i(z_1)a_j(z_2)-a_j(z_2)a_i(z_1)+c^-f(-z_2+z_1)a_j(z_2)a_0(z_1)
  -c^+f(-z_2+z_1)a_j(z_2)a_0(z_1)\\
  &\qquad=\sum_{k=0}^nc_kc^+a_j(z_2)\frac{1}{k!}\pdiff{z_2}{k}
  z_1\inv\delta\left(\frac{z_2}{z_1}\right).
\end{align*}
\end{rem}

\begin{prop}\label{prop:universal-nonlocal-va}
Let $(\bar V,Y,\vac)$ be a nonlocal vertex algebra.
Suppose that there exist $\bar a_i\in \bar V$ ($i\in J$), such that
\begin{align*}
  (\bar V,Y(\bar a_i,z))\in\obj\mathcal M.
\end{align*}
Then there exists a unique nonlocal vertex algebra homomorphism $V\to \bar V$ such that
$a_i\mapsto \bar a_i$.
\end{prop}

\begin{proof}
From Lemma \ref{lem:M-obj-V}, we get a unique linear map $\theta_{\bar V}:V\to \bar V$
such that $\theta_{\bar V}(\vac)=\vac$,
and
\begin{align*}
  \theta_{\bar V}(Y(a_i,z)v)=\theta_{\bar V}(a_i(z)v)
  =Y(\bar a_i,z)\theta_{\bar V}(v)\quad\te{for }i\in J,\,v\in V.
\end{align*}
Combining this with Proposition \ref{prop:V-M-pre}, $\bar V$ can be viewed as a $V$-module and $\theta_{\bar V}$ is a $V$-module map. Then
\begin{align}\label{eq:universal-nonlocal-va-theta-gens}
  &\theta_{\bar V}(a_i)=\lim_{z\to 0}\theta_{\bar V}(Y(a_i,z)\vac)=\lim_{z\to 0}Y(\bar a_i,z)\vac=\bar a_i\quad i\in J.
\end{align}
Let $U$ be the subspace of $V$ consisting of elements $u$ such that
\begin{align*}
  \theta_{\bar V}(Y(u,z)v)=Y(\theta_{\bar V}(u),z)\theta_{\bar V}(v)\quad\te{for any }v\in V.
\end{align*}
Since $\bar V$ being a $V$-module, we get from \eqref{eq:universal-nonlocal-va-theta-gens} that
\begin{align}\label{eq:universal-nonlocal-va-1}
  \set{a_i}{i\in J}\cup\{\vac\}\subset U.
\end{align}
For any $u,v\in U$, we get from the weak associativity \eqref{eq:weak-asso} that there exists $k\in\N$, such that
\begin{align*}
  z^kY(Y(u,z)v,y)w=&\left( (x-y)^k Y(u,x)Y(v,y)w \right)|_{x=y+z},\\
  z^kY(Y(\theta_{\bar V}(u),z)\theta_{\bar V}(v),y)\theta_{\bar V}(w)=&\left( (x-y)^k Y(\theta_{\bar V}(u),x)Y(\theta_{\bar V}(v),y)\theta_{\bar V}(w) \right)|_{x=y+z}
\end{align*}
Then
\begin{align*}
  \theta_{\bar V}(Y(Y(u,z)v,y)w)=&z^{-k}\left( (x-y)^k \theta_{\bar V}(Y(u,x)Y(v,y)w) \right)|_{x=y+z}\\
  =&z^{-k}\left( (x-y)^k Y(\theta_{\bar V}(u),x)Y(\theta_{\bar V}(v),y)\theta_{\bar V}(w) \right)|_{x=y+z},
\end{align*}
which implies that
\begin{align*}
  Y(u,z)v\in U((z)).
\end{align*}
It follows that $U$ is a subalgebra of $V$.
Combining this with \eqref{eq:universal-nonlocal-va-1}, we get that $U=V$.
Therefore, $\theta_{\bar V}$ is the nonlocal vertex algebra homomorphism as desired.
\end{proof}

%

%

\subsection{Free quantum vertex algebras}\label{subsec:free-qva}

This subsection is devoted to constructing a quantum vertex algebra from a category. More precisely, we prove the following result.

\begin{prop}\label{prop:V-M-quantumVA}
Let $\mathcal M$ be a category whose objects are vector spaces $W$ equipped with fields
\begin{align*}
  \set{a_i^0(z)}{i\in J^0}\uplus\set{a_i^\pm(z)}{i\in J^\pm= J_0^\pm\uplus J_1^\pm}\subset \E(W),
\end{align*}
satisfying the relations below
\begin{align}
  &[a_i^0(z_1),a_j^0(z_2)]
  =\vartheta_{ij}^{0,0}(z_1-z_2)-\vartheta_{ji}^{0,0}(z_2-z_1),
  \label{eq:vartheta1}\\
  &[a_i^0(z_1),a_j^\pm(z_2)]=\pm a_j^\pm(z_2)\left(\vartheta_{ij}^{0,\pm}(z_1-z_2)
    +\vartheta_{ji}^{\pm,0}(z_2-z_1)\right),\label{eq:vartheta2}\\
  &\vartheta_{ij}^{\epsilon_1,\epsilon_2}(z_1-z_2)
  a_i^{\epsilon_1}(z_1)a_j^{\epsilon_2}(z_2)
  =(-1)^{|a_i^{\epsilon_1}|\cdot |a_j^{\epsilon_2}|}
  \vartheta_{ji}^{\epsilon_2,\epsilon_1}(z_2-z_1)
  a_j^{\epsilon_2}(z_1)a_i^{\epsilon_1}(z_1).\label{eq:vartheta3}
\end{align}
where $J^a$ are countable sets,
$\vartheta_{ij}^{s,t}(z)\in\C((z))$ with
$\vartheta_{ij}^{\pm,\pm}(z),
\vartheta_{ij}^{\pm,\mp}(z)\ne 0$,
and
\begin{align*}
  |a_i^\pm|=
  \begin{cases}
    0, & \mbox{if }i\in J_0^\pm,\\
    1, & \mbox{if }i\in J_1^\pm
  \end{cases}
\end{align*}
are all independent of the choice of $W$.
Then the weak quantum vertex algebra $F(\mathcal M)$ corresponding to $\mathcal M$ (see Proposition \ref{prop:V-M-pre})
is a quantum vertex algebra with quantum Yang-Baxter operator
$S_\vartheta(z)$ determined by
\begin{align*}
  &S_\vartheta(z)(a_j^0\ot a_i^0)=a_j^0\ot a_i^0
  +\vac\ot\vac\ot (\vartheta_{ij}^{0,0}(-z)-\vartheta_{ji}^{0,0}(z)),\\
  &S_\vartheta(z)(a_j^0\ot a_i^\pm)=a_j^0\ot a_i^\pm\mp\vac\ot a_i^\pm\ot(\vartheta_{ij}^{\pm,0}(-z)
    +\vartheta_{ji}^{0,\pm}(z)),\\
  &S_\vartheta(z)(a_j^{\pm}\ot a_i^0)=a_j^\pm\ot a_i^0
  \pm a_j^\pm\ot\vac\ot( \vartheta_{ij}^{0,\pm}(-z)+\vartheta_{ji}^{\pm,0}(z) ),\\
  &S_\vartheta(z)(a_j^{\epsilon_1}\ot a_i^{\epsilon_2})
  =(-1)^{|a_i^{\epsilon_2}|\cdot |a_j^{\epsilon_1}|}a_j^{\epsilon_1}\ot a_i^{\epsilon_2}\ot
  \vartheta_{ij}^{\epsilon_2,\epsilon_1}(-z)\inv \vartheta_{ji}^{\epsilon_1,\epsilon_2}(z).
\end{align*}
\end{prop}

Throughout this subsection, we fix the countable sets $J^0$, $J_0^\pm$ and $J_1^\pm$ and denote by $\mathcal M_\vartheta$
the category dependent on the tuple of series $\vartheta=(\vartheta_{ij}^{st}(z)\,|\,i\in J^s,\,j\in J^t\,s,t\in\{0,\pm\})$.

We note that the set $\mathfrak V$ of tuples
\begin{align*}
  &\vartheta=(\vartheta_{ij}^{st}(z)\,|\,i\in J^s,\,j\in J^t\,s,t\in\{0,\pm\})\in \prod_{s,t\in\{0,\pm\}}\C((z))^{J_s\times J_t}
\end{align*}
such that $\vartheta_{ij}^{st}(z)\ne 0$ if $s,t\in\{\pm\}$,
carries an abelian group structure with multiplication
$\vartheta\ast\bar\vartheta$ defined by
\begin{align*}
  (\vartheta\ast\bar\vartheta)_{ij}^{st}(z)=
  \begin{cases}
    \vartheta_{ij}^{st}(z)+\bar\vartheta_{ij}^{st}(z), & \mbox{if } 0\in\{s,t\},\\
    \vartheta_{ij}^{st}(z)\bar\vartheta_{ij}^{st}(z), & \mbox{otherwise},
  \end{cases}
\end{align*}
the identity $\varepsilon$ defined by
\begin{align*}
  \varepsilon_{ij}^{st}(z)=
  \begin{cases}
    0, & \mbox{if }0\in\{s,t\},\\
    1, & \mbox{otherwise},
  \end{cases}
\end{align*}
and the inverse $\vartheta\inv$ defined by
\begin{align*}
  (\vartheta\inv)_{ij}^{st}(z)=
  \begin{cases}
    -\vartheta_{ij}^{st}(z), & \mbox{if }0\in \{s,t\},\\
    \vartheta_{ij}^{st}(z)\inv, & \mbox{otherwise}.
  \end{cases}
\end{align*}

We first prove Proposition \ref{prop:V-M-quantumVA} for $\mathcal M_\varepsilon$ by realizes $F(\mathcal M_\varepsilon)$ as a supercommutative algebra with a derivation.
Then we construct a deforming triple of $F(\mathcal M_\varepsilon)$, and realizes $F(\mathcal M_\vartheta)$ as a deformation of $F(\mathcal M_\varepsilon)$.
Finally, we prove Proposition \ref{prop:V-M-quantumVA} for
general $\mathcal M_\vartheta$ by using Theorem \ref{thm:qva-twistor}.

\begin{prop}\label{prop:V-M-sp-epsilon}
Proposition \ref{prop:V-M-quantumVA} holds for $\mathcal M_\varepsilon$.
\end{prop}

\begin{proof}
Let $F$ be a free supercommutative algebra generated by the set of even generators
\begin{align*}
  \set{\partial^n \bar a_i^0,\,\partial^n \bar a_j^+,\,\partial^n \bar a_k^-}{i\in J^0,\,j\in J_0^+,\,k\in J_0^-,\,n\in-\Z_+}
\end{align*}
and the set of odd generators
\begin{align*}
  \set{\partial^n \bar a_i^+,\,\partial^n \bar a_i^-}{i\in J_1^+,\,j\in J_1^-,\,n\in-\Z_+}.
\end{align*}
For a homogeneous element $v\in F$, we denote by $|v|$ the parity of $v$.
There is a derivation $\partial$ on $F$ defined by
\begin{align*}
  \partial (\partial^n \bar a_i^s)=\partial^{n+1} \bar a_i^s\quad \te{for }n\in-\Z_+,\,i\in J^s,\,s\in\{0,\pm\}.
\end{align*}
Then $(F,\partial)$ carries a quantum vertex algebra structure, with vacuum vector $1\ot 1$, vertex operator map $Y(a,z)b=(e^{z\partial }a)b$ and quantum Yang-Baxter operator defined by
\begin{align*}
  &S(z)(v\ot u)=(-1)^{|v|\cdot|u|}v\ot u\quad\te{for }u,v\in F\,\,\te{homogeneous}.
\end{align*}
Note that $F$ equipped with
\begin{align*}
  \bar a_i^s(z)=\sum_{n\ge 0}\frac{z^n}{n!}L_{\partial^n \bar a_i^s}\in \E(F)
\end{align*}
is an object in $\mathcal M_\varepsilon$, where
$L_u$ denotes the left multiplication of $u$.
By using Proposition \ref{prop:universal-nonlocal-va} we get a unique nonlocal vertex algebra homomorphism $f:F(\mathcal M_\varepsilon)\to F$ such that
\begin{align*}
  f(a_i^s)=\bar a_i^s\quad\te{for }i\in J^s,\,s\in\{0,\pm\}.
\end{align*}
Since $F$ is generated by $\set{\bar a_i^s}{i\in J^s,\,s\in\{0,\pm\}}$ as a nonlocal vertex algebra,
$f$ must be surjective.

On the other hand, from Proposition \ref{prop:V-M-pre}, we have that
\begin{align*}
  (F(\mathcal M_\varepsilon),Y(a_i^s,z))\in \obj\mathcal M_\varepsilon.
\end{align*}
Then $Y(a_i^s,z)$ ($i\in J^s$, $s\in\{0,\pm\}$)
are supercommutative.
Combining this with the vacuum property \eqref{eq:vacuum-prop}
and the fact that $F(\mathcal M_\varepsilon)$ is generated by
$\set{a_i^s}{i\in J^s,\,s\in\{0,\pm\}}$, we have that
\begin{align*}
  Y(a_i^s,z)\in \End(F(\mathcal M_\varepsilon))[[z]].
\end{align*}
It follows that
\begin{align*}
  \set{a_i^s(-n-1)}{i\in J^s,\,s\in\{0,\pm\},\,n\in\N}
  \subset \End(F(\mathcal M_\varepsilon)
\end{align*}
generates a supercommutative subalgebra.
Since $F$ is free, we get an algebra homomorphism $F\to \End(F(\mathcal M_\varepsilon))$, and hence a left
$F$-module structure on $F(\mathcal M_\varepsilon)$ defined by
\begin{align*}
  (\partial^n \bar a_i^s)v=n!a_i^s(-n-1)v.
\end{align*}
Moreover, $F(\mathcal M_\varepsilon)$ is generated by $\vac$.
Let $u\in \ker f$.
Write $u$ as a linear combination of elements of the following
form
\begin{align}\label{eq:V-M-sp-epsilon-temp}
  a_{i_1}^{s_1}(-n_1-1)\cdots a_{i_k}^{s_k}(-n_k-1)\vac.
\end{align}
Define $\bar u\in F$ by replacing \eqref{eq:V-M-sp-epsilon-temp}
with
\begin{align*}
  \frac{1}{n_1!}\partial^{n_1}a_{i_1}^{s_1}\cdots \frac{1}{n_k!}\partial^{n_k} a_{i_k}^{s_k}.
\end{align*}
It is obvious that $u=\bar u\vac$.
Since $f$ is a nonlocal vertex algebra homomorphism, we have that
$\bar u=f(u)$.
Consequently, $u=f(u)\vac=0$ as $u\in \ker f$.
Therefore, $f$ is also injective,
and hence a nonlocal vertex algebra isomorphism.
The quantum vertex algebra structure on $F$ therefore induces the required quantum Yang-Baxter operator on $F(\mathcal M_\varepsilon)$.
\end{proof}

Next, we construct the needed deforming triple.
The following result is a consequence of Proposition \ref{prop:universal-nonlocal-va}.

\begin{lem}\label{lem:map-from-V}
Let $U$ be a commutative vertex algebra, and let $\wh a_i^a\in U$, $i\in J^a$, $a\in\{0,\pm\}$.
Then there exists a unique nonlocal vertex algebra homomorphism $\rho:F(\mathcal M_\vartheta)\to F(\mathcal M_\vartheta)\ot U$
such that
\begin{align*}
  &\rho(a_i^0)=a_i^0\ot \vac+\vac\ot \wh a_i^0\quad\te{for }i\in J^0,\\
  &\rho(a_i^a)=a_i^a\ot \wh a_i^a\quad\te{for }i\in J^a,\, a\in\{\pm\}.
\end{align*}
\end{lem}

\begin{proof}
Denote by $Y_\ot$ the vertex operator map of the tensor product nonlocal vertex algebra $F(\mathcal M)\ot U$.
Since $U$ is commutative, the relations \eqref{eq:vartheta1}-\eqref{eq:vartheta3} hold with
\begin{align*}
  a_i^s(z)=Y_\ot(\rho(a_i^s),z)\quad\te{for }i\in J^s,\,s\in \{0,\pm\}.
\end{align*}
By using Proposition \ref{prop:universal-nonlocal-va}, we get the
nonlocal vertex algebra homomorphism as desired.
\end{proof}

Recall from \cite{Li-smash} that a \emph{pseudo-endomorphism} of a nonlocal vertex algebra $V$
is a linear map $A(z):V\to V\ot\C((z))$, such that
\begin{align*}
  A(z)\vac=\vac\ot 1,\quad
  A(z_1)Y(u,z_2)v=Y(A(z_1-z_2)u,z_2)v\quad\te{for } u,v\in V.
\end{align*}
And a \emph{pseudo-derivation} of $V$ is a linear map $D(z):V\to V\ot\C((z))$, such that
\begin{align*}
  [D(z_1),Y(u,z_2)]=Y(D(z_1-z_2)u,z_2)\quad\te{for } u\in V.
\end{align*}
The following result was given in \cite[Proposition 2,11]{Li-pseudo}.
\begin{prop}\label{prop:pseudo}
Let $V$ be a nonlocal vertex algebra,
and view $\C((z))$ as a vertex algebra with the vertex operator map
$$Y(f(z),z_0)g(z)=f(z-z_0)g(z).$$
Suppose that $A(z)$ is a nonlocal vertex algebra homomorphism from $V$ to the tensor product nonlocal vertex algebra $V\ot \C((z))$.
Then $A(z)$ is a pseudo-endomorphism of $V$.
Moreover, let $$D(z):V\to V\ot \C((z))$$ be a liner map.
We view $V\ot(\C[\delta]/\delta^2\C[\delta])$ as a nonlocal vertex algebra.
If $1+\delta D(z)$ is a $\C[\delta]$-linear pseudo-endomorphism of $V\ot(\C[\delta]/\delta^2\C[\delta])$,
then $D(z)$ is a pseudo-derivation of $V$.
\end{prop}

\begin{lem}\label{lem:pseudos}
Given $\vartheta,\bar\vartheta\in \mathfrak V$,
there exist a pseudo-derivation $\bar\vartheta_i^0(z)$ and pseudo-endomorphisms $\bar\vartheta_i^\pm(z)$
on $F(\mathcal M_\vartheta)$ defined by
\begin{align}
  &\bar\vartheta_i^0(z)a_j^0=\vac\ot \bar\vartheta_{ij}^{0,0}(z),\quad \bar\vartheta_i^0(z)a_j^{\pm}=\pm a_j^{\pm}\ot \bar\vartheta_{ij}^{0,\pm}(z),\label{eq:sigma-0}\\
  &\bar\vartheta_i^{\pm}(z)a_j^0=a_j^0\ot 1\mp\vac\ot \bar\vartheta_{ij}^{\pm,0}(z),\quad
  \bar\vartheta_i^{\pm}(z)a_j^{\epsilon}=a_j^{\epsilon}\ot \bar\vartheta_{ij}^{\pm,\epsilon}(z)\inv.\label{eq:sigma-a}
\end{align}
\end{lem}

\begin{proof}
View $\C((z))$ as a vertex algebra with vertex operator map
\begin{align*}
  Y(f(z),z_0)g(z)=f(z-z_0)g(z),
\end{align*}
and denote by $Y_\ot$ the tensor product nonlocal vertex algebra $F(\mathcal M_\vartheta)\ot \C((z))$.
Note that $\C((z))$ is a commutative vertex algebra.
By using Lemma \ref{lem:map-from-V}, we get nonlocal vertex algebra a homomorphism $\bar\vartheta_i^\pm(z):F(\mathcal M_\vartheta)\to F(\mathcal M_\vartheta)\ot \C((z))$
determined by \eqref{eq:sigma-a}.
Set $\C_\delta=\C[\delta]/\delta^2\C[\delta]$.
View $V\ot\C_\delta$ as a nonlocal vertex algebra over $\C_\delta$ for any nonlocal vertex algebra $V$ over $\C$.
By using Lemma \ref{lem:map-from-V} again, we get another
nonlocal vertex algebra homomorphism
\begin{align*}
  (1+\delta\bar\vartheta_i^0(z)):F(\mathcal M_\vartheta)\ot \C_\delta
  \to \big(F(\mathcal M_\vartheta) \ot \C_\delta\big)
  \ot_{\C_\delta}
  \big( \C((z))\ot \C_\delta\big)
\end{align*}
determined by
\begin{align*}
  &(1+\delta\bar\vartheta_i^0(z))a_j^0=a_j^0\ot 1+\vac\ot \delta \bar\vartheta_{ij}^{0,0}(z),\quad
  (1+\delta\bar\vartheta_i^0(z))a_j^\pm=a_j^\pm\ot (1
  \pm \delta \bar\vartheta_{ij}^{0,\pm}(z)).
\end{align*}
It follows from Proposition \ref{prop:pseudo} that
$\bar\vartheta_i^0(z)$ and $\bar\vartheta_i^\pm(z)$ are the
pseudo-derivations and pseudo-endomorphisms as desired.
\end{proof}

Let $H'$ be the symmetric algebra of the following vector space:
\begin{align*}
  \bigoplus_{b\in\{0,\pm\}}\bigoplus_{i\in J^a}\bigoplus_{n\in\N}\C\partial^n\wh a_i^b.
\end{align*}
Then $H'$ is a commutative and cocommutative bialgebra with $\Delta$ and $\varepsilon$ uniquely determined by
($i\in J$, $n\in\N$):
\begin{align*}
  &\Delta(\partial^n \wh a_i^0)=\partial^n \wh a_i^0\ot 1+1\ot \partial^n \wh a_i^0,\quad \varepsilon(\partial^n \wh a_i^0)=0,\\
  &\Delta(\partial^n \wh a_i^{\pm})=\sum_{k=0}^n\binom{n}{k} \partial^k \wh a_i^{\pm}\ot \partial^{n-k} \wh a_i^{\pm},\quad
  \varepsilon(\partial^n \wh a_i^{\pm})=\delta_{n,0}.
\end{align*}
Let $\partial$ be the derivation on $H'$ such that
\begin{align*}
  \partial (\partial^n \wh a_i^b)=\partial^{n+1}\wh a_i^b\quad\te{for }n\in\N,\,
  i\in J^a,\,b\in\{0,\pm\}.
\end{align*}
It is straightforward to see that $\Delta\circ\partial=(\partial\ot 1+1\ot\partial)\circ\Delta$ and $\varepsilon\circ \partial=0$.
From Remark \ref{rem:bialg-der} we have that $(H',\partial,\Delta,\varepsilon)$ carries a vertex bialgebra structure.
The following result defines an $H'$-comodule nonlocal vertex algebra structure on $F(\mathcal M_\vartheta)$.

\begin{lem}\label{lem:comod-nonlocalVA}
For each $\vartheta\in \mathfrak V$,
there is a unique nonlocal vertex algebra map $\rho:F(\mathcal M_\vartheta)\to F(\mathcal M_\vartheta) \ot H'$, such that
\begin{align}\label{eq:comod-nonlocalVA-rho}
  \rho(a_i^0)=a_i^0\ot 1+\vac\ot \wh a_i^0,\quad
  \rho(a_i^{\pm})=a_i^{\pm}\ot \wh a_i^{\pm}.
\end{align}
Moreover, $(F(\mathcal M_\vartheta),\rho)$ is an $H'$-comodule nonlocal vertex algebra.
\end{lem}

\begin{proof}
By using Lemma \ref{lem:map-from-V}, we get unique nonlocal vertex algebra homomorphisms $\rho:F(\mathcal M_\vartheta)\to F(\mathcal M_\vartheta)\ot H'$ determined by \eqref{eq:comod-nonlocalVA-rho},
and $\rho^{(2)}:F(\mathcal M_\vartheta)\to F(\mathcal M_\vartheta)\ot H'\ot H'$ determined by
\begin{align*}
  \rho^{(2)}(a_i^0)=a_i^0\ot 1\ot 1+\vac\ot \wh a_i^0\ot 1+ \vac\ot 1\ot \wh a_i^0,\quad
  \rho^{(2)}(a_i^\pm)=a_i^\pm\ot \wh a_i^\pm\ot \wh a_i^\pm.
\end{align*}
Note that $(1\ot \Delta)\circ\rho$, $(\rho\ot 1)\circ \rho$ are both nonlocal vertex algebra homomorphisms and
\begin{align*}
  &(1\ot \Delta)\circ\rho(a_i^0)=a_i^0\ot 1\ot 1+\vac\ot \wh a_i^0\ot 1+\vac\ot 1\ot \wh a_i^0
  =\rho^{(2)}(a_i^0)=(\rho\ot 1)\circ\rho(a_i^0),\\
  &(1\ot \Delta)\circ\rho(a_i^\pm)=a_i^\pm\ot \wh a_i^\pm\ot \wh a_i^\pm=\rho^{(2)}(a_i^\pm)=(\rho\ot 1)\circ\rho(a_i^\pm).
\end{align*}
Then the moreover statement follows immediate from the uniqueness.
\end{proof}

\begin{lem}\label{lem:deform-tri-nonlocalVA}
There is an $H'$-module nonlocal vertex algebra structure $\bar\vartheta$ on $F(\mathcal M_\vartheta)$ defined by
\begin{align*}
  \bar\vartheta(\wh a_i^a,z)=\bar\vartheta_i^a(z)\quad\te{for }i\in J^a,\,a\in\{0,\pm\}.
\end{align*}
Moreover, $(H',\rho,\bar\vartheta)$ becomes a deforming triple of $F(\mathcal M_\vartheta)$.
\end{lem}

\begin{proof}
Note that
\begin{align*}
  &\bar\vartheta_i^s(z)\in
    \Hom(F(\mathcal M_\vartheta),F(\mathcal M_\vartheta)\ot \C((z)))\subset
    \E(F(\mathcal M_\vartheta))\quad\te{for }i\in J^s,\,s\in\{0,\pm\},\\
  &\quad\te{and}\quad
  [\bar\vartheta_i^s(z_1),\bar\vartheta_j^t(z_2)]=0\quad\te{for }i\in J^s,\,j\in J^t,\,s,t\in\{0,\pm\}.
\end{align*}
From Theorem \ref{thm:nonlocal-VA-gen}, we have that
\begin{align*}
  U=\set{\bar\vartheta_i^s(z)}{i\in J^s,\,s\in\{0,\pm\}}
\end{align*}
generates a commutative vertex algebra.
By using Lemma \ref{lem:map-from-V}, we get a nonlocal vertex algebra homomorphism $f_{\bar\vartheta}:F(\mathcal M_\vartheta)\to F(\mathcal M_\vartheta)\ot \<U\>$ determined by
\begin{align*}
  f_{\bar\vartheta}(a_i^0)=a_i^0\ot 1+\vac\ot \bar\vartheta_i^0(z),\quad
  f_{\bar\vartheta}(a_i^\pm)=a_i^\pm\ot \bar\vartheta_i^\pm(z).
\end{align*}
The construction of $H'$ implies an algebra homomorphism $g:H'\to \<U\>$ determined by
\begin{align*}
  g(\partial^n\wh a_i^s,z)=\frac{d^n}{dz^n}
  \bar\vartheta_i^s(z)\quad\te{for }i\in J^s,\,s\in\{0,\pm\}.
\end{align*}
It is easy to see that
\begin{align*}
  g(\partial h,z)=\frac{d}{dz} g(h,z)\quad\te{for }h\in H'.
\end{align*}
Then $g$ is a vertex algebra homomorphism.
Hence, $F(\mathcal M_\vartheta)$ is an $H'$-module, we denote this module action by $\bar\vartheta(\cdot,z)$.

Let
\begin{align*}
  S=\set{a_i^s}{i\in I,\,s\in\{0,\pm\}}\uplus\{\vac\},\quad
  T=\set{\wh a_i^s}{i\in I,\,s\in\{0,\pm\}}\uplus\{1\}
\end{align*}
Since $\bar\vartheta_i^0(z)$ is a pseudo-derivation and $\bar\vartheta_i^\pm(z)$ are pseudo-endomorphisms,
we obtain
\begin{align*}
  &\bar\vartheta(h,z)\vac=\varepsilon(h)\vac,
  \quad
  \bar\vartheta(h,z_1)Y(u,z_2)v=\sum Y(\bar\vartheta(h_{(1)},z_1-z_2)u,z_2)\bar\vartheta(h_{(2)},z_1)v
\end{align*}
for all $u,v\in F(\mathcal M_\vartheta)$ and $h\in T$, where $\Delta(h)=\sum h_{(1)}\ot h_{(2)}$.
As $H'$ is generated by $T$, these relations extend to the whole space $H'$.
Consequently, $(F(\mathcal M_\vartheta),\bar\vartheta)$ forms an $H'$-module nonlocal vertex algebra.


Finally, we prove the compatibility of $\rho$ and $\bar\vartheta$. A straightforward verification shows that
\begin{align*}
  \bar\vartheta(h,z)v\in S\ot \C((z)),\quad
  \rho(\bar\vartheta(h,z)v)=(\bar\vartheta(h,z)\ot 1)\rho(v)
  \quad\te{for }h\in T,\,v\in S.
\end{align*}
The desired compatibility then follows from \cite[Lemma 2.28]{JKLT-Defom-va}.
\end{proof}

%

Applying Proposition \ref{prop:deforming-triple-twistor}
to Lemma \ref{lem:deform-tri-nonlocalVA}, we immediately get the following result by \eqref{eq:T-tau}.
\begin{lem}\label{lem:qva-twistor}
Let $\vartheta,\bar\vartheta\in\mathfrak V$.
Then there exists a unique twistor $T_{\bar\vartheta}(z)$ of $F(\mathcal M_\vartheta)$, such that \eqref{eq:qyb-com1} holds and
\begin{align*}
  &T_{\bar\vartheta}(z)(a_j^0\ot a_i^0)=a_j^0\ot a_i^0+\vac\ot\vac\ot\bar\vartheta_{ij}^{0,0}(-z),\\
  &T_{\bar\vartheta}(z)(a_j^0\ot a_i^{\pm})=a_j^0\ot a_i^{\pm}\mp \vac\ot a_i^{\pm}\ot {\bar\vartheta}_{ij}^{\pm,0}(-z),\\
  &T_{\bar\vartheta}(z)(a_j^{\pm}\ot a_i^0)= a_j^{\pm}\ot a_i^0\pm a_j^{\pm}\ot \vac\ot {\bar\vartheta}_{ij}^{0,\pm}(-z),\\
  &T_{\bar\vartheta}(z)(a_j^{\epsilon_1}\ot a_i^{\epsilon_2})=a_j^{\epsilon_1}\ot a_i^{\epsilon_2}
  \ot {\bar\vartheta}_{ij}^{\epsilon_2,\epsilon_1}(-z)\inv.
\end{align*}
\end{lem}

\begin{lem}\label{lem:deform-M-obj}
Let $\vartheta,\bar\vartheta\in \mathfrak V$.
Then
\begin{align*}
  (\mathfrak D_{T_{\bar\vartheta}}(F(\mathcal M_\vartheta)),
  \mathfrak D_{T_{\bar\vartheta}}(Y)(a_i^s,z))\in \obj\mathcal M_{\vartheta\ast\bar\vartheta}.
\end{align*}
\end{lem}

\begin{proof}
Note that
\begin{align*}
  &\mathfrak D_{T_{\bar\vartheta}}(Y)(a_i^0,z)=Y(a_i^0,z)+\bar\vartheta_i^0(z),\quad
   \mathfrak D_{T_{\bar\vartheta}}(Y)(a_i^{\pm},z)=Y(a_i^{\pm},z)
   \bar\vartheta_i^{\pm}(z).
\end{align*}
The lemma now follows by a direct verification.
\end{proof}

\begin{lem}\label{lem:F-J-gen}
$\mathfrak D_{T_{\bar\vartheta}}(F(\mathcal M_\vartheta))$ is generated by
$\set{a_i^s}{i\in J^s,\,s\in\{0,\pm\}}$.
\end{lem}

\begin{proof}
Let
\begin{align*}
  &S=\set{a_i^s}{i\in J^s,\,s\in\{0,\pm\}}\cup\{\vac\}\quad
  T=\set{\wh a_i^s}{i\in J^s,\,s\in\{0,\pm\}}\cup\{\vac\}.
\end{align*}
Then $F(\mathcal M_\vartheta)$ is generated by $S$ and $H'$ is generated by $T$.
Note that
\begin{align*}
  &\rho(S)\subset S\ot T,\quad \Delta(T)\subset T\ot T,\quad \bar\vartheta\inv(T,z)S\subset S\ot \C((z)).
\end{align*}
We complete the proof by using \cite[Lemma 3.7]{JKLT-Defom-va}.
\end{proof}

\begin{prop}\label{prop:F-deform}
We have that
\begin{align*}
  \mathfrak D_{T_{\bar\vartheta}}(F(\mathcal M_\vartheta))\cong F(\mathcal M_{\vartheta\ast\bar\vartheta}).
\end{align*}
\end{prop}

\begin{proof}
Combining Lemma \ref{lem:deform-M-obj} and Proposition \ref{prop:universal-nonlocal-va}, we get a nonlocal vertex algebra homomorphism $f_{\vartheta,\bar\vartheta}:F(\mathcal M_{\vartheta\ast\bar\vartheta})\to\mathfrak D_{T_{\bar\vartheta}}(F(\mathcal M_\vartheta))$ uniquely determined by
\begin{align*}
  f_{\vartheta,\bar\vartheta}(a_i^s)=a_i^s\quad\te{for }i\in J^s,\,s\in\{0,\pm\}.
\end{align*}
We deduce from Proposition \ref{prop:L-H-rho-V-comosition-pre} that
$(\mathfrak D_{T_{\bar\vartheta}}(F(\mathcal M_\vartheta)),\rho)$ is an $H'$-comodule nonlocal vertex algebra.
In addition, we get from Proposition \ref{prop:L-H-rho-V-compostition} that $(H',\rho,\bar\vartheta\inv)$ is a deforming triple of $\mathfrak D_{T_{\bar\vartheta}}(F(\mathcal M_\vartheta))$.
It is straightforward to check that
\begin{align*}
  &\rho\circ f_{\vartheta,\bar\vartheta}=(f_{\vartheta,\bar\vartheta}\ot 1)\circ \rho,
\end{align*}
Let
\begin{align*}
  S=\set{a_i^s}{i\in J^s,\,s\in\{0,\pm\}},\quad
  T=\set{\wh a_i^s}{i\in J^s,\,s\in\{0,\pm\}}.
\end{align*}
A straightforward verification shows that
\begin{align}\label{eq:f-bartheta-com}
  f_{\vartheta,\bar\vartheta}(\bar\vartheta\inv(h,z)v)=\bar\vartheta\inv(h,z)f_{\vartheta,\bar\vartheta}(v)
  \quad\te{for }h\in T,\,v\in S.
\end{align}
Since $F(\mathcal M_{\vartheta\ast\bar\vartheta})$ is generated by $S$ and $H'$ by $T$, the relations \eqref{eq:mod-va-for-vertex-bialg1-2},
\eqref{eq:mod-va-for-vertex-bialg3} together with the fact that
$f_{\vartheta,\bar\vartheta}$ is a nonlocal vertex algebra homomorphism, imply that \eqref{eq:f-bartheta-com} extends to all $v\in F(\mathcal M_{\vartheta\ast\bar\vartheta})$ and $h\in H'$.
Combined with
\cite[Remark 4.5]{K-Quantum-aff-va}, this yields that $f_{\vartheta,\bar\vartheta}$ is also a nonlocal vertex algebra homomorphism $$\mathfrak D_{T_{\bar\vartheta\inv}}(F(\mathcal M_{\vartheta\ast\bar\vartheta}))\to F(\mathcal M_\vartheta).$$
Replacing $\vartheta$ and $\bar\vartheta$ by $\vartheta\ast\bar\vartheta$
and $\bar\vartheta\inv$, respectively, we obtain a nonlocal vertex algebra homomorphism
\begin{align*}
  &f_{\vartheta\ast\bar\vartheta,\bar\vartheta\inv}:
  \mathfrak D_{T_{\bar\vartheta}}(F(\mathcal M_\vartheta))\to
  F(\mathcal M_{\vartheta\ast\bar\vartheta}).
\end{align*}
Since $f_{\vartheta\ast\bar\vartheta,\bar\vartheta\inv}\circ f_{\vartheta,\bar\vartheta}$
maps each generator $a_i^s$ to $a_i^s$ for $i\in J^s$, $s\in\{0,\pm\}$, we have
\begin{align*}
  &f_{\vartheta\ast\bar\vartheta,\bar\vartheta\inv}\circ f_{\vartheta,\bar\vartheta}
  =1_{F(\mathcal M_{\vartheta\ast\bar\vartheta})}.
\end{align*}
Hence $f_{\vartheta,\bar\vartheta}$ is injective.
By Lemma \ref{lem:F-J-gen}, we get that
$\mathfrak D_{T_{\bar\vartheta}}(F(\mathcal M_\vartheta))$ is generated by $S$, which implies that $f_{\vartheta,\bar\vartheta}$ is also surjective.
Therefore, $f_{\vartheta,\bar\vartheta}$ is an isomorphism, as desired.
\end{proof}

Since $\varepsilon$ is the identity of $\mathfrak V$,
we get the immediate consequence of Proposition \ref{prop:F-deform}.
\begin{coro}\label{coro:F-deform}
For any $\vartheta\in \mathfrak V$, we have that
\begin{align*}
  F(\mathcal M_\vartheta)\cong \mathfrak D_{T_\vartheta}(F(\mathcal M_\varepsilon)).
\end{align*}
\end{coro}
%
%
%

\begin{proof}[{\bf Proof of Proposition \ref{prop:V-M-quantumVA}}]
By using Lemma \ref{lem:qva-twistor}, a direct verification shows that the relations
\begin{align}\label{eq:V-M-quantumVA-temp}
  &A^{ij}(z_1)B^{ab}(z_2)=B^{ab}(z_2)A^{ij}(z_1),\\
  &\nonumber
  \quad \te{for }(i,j)\not\in\{ (a,b),\,(b,a)\}\,\,\te{and}\,\,A,B\in\{S_\varepsilon,T_{\vartheta}\}
  \,\,\te{with}\,\,(A,B)\ne (S_\varepsilon,S_\varepsilon)\nonumber
\end{align}
hold on the generating subset
\begin{align*}
  \set{a_i^s}{i\in J^s,\,s\in\{0,\pm\}}\subset F(\mathcal M_\varepsilon).
\end{align*}
Using \eqref{eq:qyb-hex1} and \eqref{eq:qyb-hex2}, these relations extend to the whole space $F(\mathcal M_\varepsilon)$.
Applying Theorem \ref{thm:qva-twistor} then completes the proof.
\end{proof}

\begin{coro}\label{coro:countable-dim}
For each $\vartheta\in \mathfrak V$, $F(\mathcal M_\vartheta)$ has countable dimension.
\end{coro}

\begin{proof}
It is obvious that the supercommutative algebra $F$ constructed in the proof of
Proposition \ref{prop:V-M-sp-epsilon} has countable dimension.
Then $F(\mathcal M_\varepsilon)\cong F$ also has countable dimension.
Applying Corollary \ref{coro:F-deform} then complete the proof.
\end{proof}

\subsection{Associated to $\mathcal R_\zeta^\ell(\wh\g)$}\label{subsec:construct-V}

In this subsection, we construct the quantum vertex algebra associated to $\mathcal R_\zeta^\ell(\wh\g)$.
Define the following integers
\begin{align*}
  &\Apsc{\ell}_{ijmn}^{0,0}
  =\vvp{[a_{ij}]_{q_i}[r\ell/r_j]_{q_j}q^{-r\ell+n-m}},\quad
  \Apsc{\ell}_{ijmn}^{2,2}=\vvp{q_i^{a_{ij}}q^{n-m}},\\
  &\Apsc{\ell}_{ijmn}^{0,1}
    =\vvp{[a_{ij}]_{q_i}(q^{-r\ell}-q^{r\ell})
    q^{-r\ell+n-m}},\quad
   \Apsc{\ell}_{ijmn}^{1,0}
    =\vvp{[a_{ji}]_{q_j}(q^{r\ell}-q^{-r\ell})
    q^{-r\ell+n-m}},\\
  &\Apsc{\ell}_{ijmn}^{0,2}
    =\vvp{[a_{ij}]_{q_i}q^{-r\ell+n-m}},\quad
  \Apsc{\ell}_{ijmn}^{2,0}
    =\vvp{[a_{ji}]_{q_j}q^{-r\ell+n-m}},\\
  &\Bpsc{\ell}_{ijmn}^{1,1,+}
    =\vvp{q_i^{a_{ij}}(1-q^{-2r\ell})q^{n-m}},\quad
  \Bpsc{\ell}_{ijmn}^{1,1,-}
    =\vvp{q_i^{-a_{ij}}(1-q^{-2r\ell})q^{n-m}},\\
  &\Bpsc{\ell}_{ijmn}^{1,2}
    =\vvp{q_i^{a_{ij}}q^{-r\ell+n-m}},\quad
  \Bpsc{\ell}_{ijmn}^{2,1}
    =\vvp{q_i^{-a_{ij}}q^{-r\ell+n-m}},\\
  &\Apsc{\ell}_{ijmn}^{1,1}
    =\Bpsc{\ell}_{ijmn}^{1,1,+}+\Bpsc{\ell}_{jinm}^{1,1,-}
  =-\vvp{ q_i^{a_{ij}}(q^{r\ell}-q^{-r\ell})^2q^{n-m} },\\
  &\Apsc{\ell}_{ijmn}^{2,1}=-\Apsc{\ell}_{ijmn}^{1,2}=\Bpsc{\ell}_{ijmn}^{1,2}-\Bpsc{\ell}_{jinm}^{2,1}
  =\vvp{q_i^{a_{ij}}(q^{-r\ell}-q^{r\ell})q^{n-m}},
\end{align*}
where $i,j\in I$ and $m,n\in\Z_\wp$.
Let $\mathfrak T$ be the set of tuples
\begin{align*}
  \tau=\left(\tau_{ijmn}^{a,b}(z)\right)_{i,j\in I,m,n\in\Z_\wp}^{a,b\in \{0,1,2\}}\in \C[[z]]^{9\wp^2|I|},
\end{align*}
such that $\tau_{ijmn}^{a,b}(z)\in \C[[z]]^\times$ for $a,b\in \{1,2\}$, for $i,j\in I$, $m,n\in\Z_\wp$,
\begin{align*}
  &\tau_{ijmn}^{0,1}(z)
    =\tau_{ijm,n-r\ell}^{0,2}(z)-\tau_{ijm,n+r\ell}^{0,2}(z),\quad
   \tau_{ijmn}^{1,0}(z)=\tau_{ij,m-r\ell,n}^{2,0}(z)-\tau_{ij,m+r\ell,n}^{2,0}(z),\\
  &\tau_{ijmn}^{1,1}(z)=\tau_{ijm,n-r\ell}^{1,2}(z)\tau_{ijm,n+r\ell}^{1,2}(z)\inv,\quad
   \tau_{ijmn}^{1,1}(z)=\tau_{ij,m-r\ell,n}^{2,1}(z)\tau_{ij,m+r\ell,n}^{2,1}(z)\inv,\\
  &\tau_{ijmn}^{2,1}(z)=\tau_{ijm,n-r\ell}^{2,2}(z)\tau_{ijm,n+r\ell}^{2,2}(z)\inv,\quad
   \tau_{ijmn}^{1,2}(z)=\tau_{ij,m-r\ell,n}^{2,2}(z)\tau_{ij,m+r\ell,n}^{2,2}(z)\inv,\\
  &\pd z \tau_{ijmn}^{0,1}(z)=\tau_{ijm,n+r_j}^{0,0}(z)-\tau_{ijm,n-r_j}^{0,0}(z),\quad
   \pd z \tau_{ijmn}^{1,0}(z)=\tau_{ij,m+r_i,n}^{0,0}(z)-\tau_{ij,m-r_i,n}^{0,0}(z),\\
  &\pd z \tau_{ijmn}^{a,1}(z)=\left(\tau_{ijm,n+r_j}^{a,0}(z)
        -\tau_{ijm,n-r_j}^{a,0}(z)\right)\tau_{ijmn}^{a,1}(z)
    \quad\te{for }a=1,2,\\
  &\pd z \tau_{ijmn}^{1,a}(z)=\left(\tau_{ij,m+r_i,n}^{0,a}(z)
    -\tau_{ij,m-r_i,n}^{0,a}(z)\right)\tau_{ijmn}^{1,a}(z)
    \quad\te{for }a=1,2.
\end{align*}

We note that $\mathfrak T$ carries an abelian group structure with the identity
$\varepsilon$ defined by
\begin{align*}
  \varepsilon_{ijmn}^{a,b}(z)=\begin{cases}
                                0, & \mbox{if }0\in\{a,b\},\\
                                1, & \mbox{otherwise},
                              \end{cases}
\end{align*}
the multiplication $\tau\ast \tau'$ defined by
\begin{align*}
  (\tau\ast \tau')_{ijmn}^{a,b}(z)=
  \begin{cases}
    \tau_{ijmn}^{a,b}(z)+\tau_{ijmn}'^{a,b}(z), & \mbox{if }0\in \{a,b\},\\
    \tau_{ijmn}^{a,b}(z)\tau_{ijmn}'^{a,b}(z), &\mbox{otherwise},
  \end{cases}
\end{align*}
and the inverse $\tau\inv$ defined by
\begin{align*}
  (\tau\inv)_{ijmn}^{a,b}(z)=
  \begin{cases}
    -\tau_{ijmn}^{a,b}(z), & \mbox{if }0\in \{a,b\},\\
    \tau_{ijmn}^{a,b}(z)\inv, &\mbox{otherwise}.
  \end{cases}
\end{align*}

\begin{de}\label{de:V-pre}
For $\tau\in\mathfrak T$ and $\ell\in\Z$ with $0\le \ell<\wp$, we define
$\mathcal M_{\wp,\tau}^\ell(\wh\g)$ to be the category, whose objects are vector spaces over $\C$ equipped with fields
\begin{align*}
  \xi_{i,m}^a(z)=\sum_{n\in\Z}\xi_{i,m}^a(n)z^{-n-1}\quad\te{for }i\in I,\,m\in\Z_\wp,\,a\in \{0,1^\pm,2^\pm\},
\end{align*}
satisfying the relations below
\begin{align}
  \label{tau1}\tag{$\tau$1}
  &[\xi_{i,m}^0(z_1),\xi_{j,n}^0(z_2)]
  =\frac{\Apsc{\ell}_{ijmn}^{0,0}}{(z_1-z_2)^2}
  -\frac{\Apsc{\ell}_{jinm}^{0,0}}{(z_2-z_1)^2}
  +\tau_{ijmn}^{0,0}(z_1-z_2)-\tau_{jinm}^{0,0}(z_2-z_1),\nonumber\\
  \label{tau2}\tag{$\tau$2}
  &[\xi_{i,m}^0(z_1),\xi_{j,n}^{1^\pm}(z_2)]=\pm \xi_{j,n}^{1^\pm}(z_2)
  \left(\frac{\Apsc{\ell}_{ijmn}^{0,1}}{z_1-z_2}+
  \frac{\Apsc{\ell}_{jinm}^{1,0}}{z_2-z_1}+\tau_{ijmn}^{0,1}(z_1-z_2)+\tau_{jinm}^{1,0}(z_2-z_1)\right),\\
  \label{tau3}\tag{$\tau$3}
  &[\xi_{i,m}^0(z_1),\xi_{j,n}^{2^\pm}(z_2)]=\pm \xi_{j,n}^{2^\pm}(z_2)
  \left(\frac{\Apsc{\ell}_{ijmn}^{0,2}}{z_1-z_2}+
  \frac{\Apsc{\ell}_{jinm}^{2,0}}{z_2-z_1}+\tau_{ijmn}^{0,2}(z_1-z_2)+\tau_{jinm}^{2,0}(z_2-z_1)\right),\\
  \label{tau4}\tag{$\tau$4}
  &
  (z_1-z_2)^{\epsilon_1\epsilon_2 \Bpsc{\ell}_{ijmn}^{1,1,+}}
  (z_2-z_1)^{-\epsilon_1\epsilon_2 \Bpsc{\ell}_{ijmn}^{1,1,-}}
  \tau_{ijmn}^{1,1}(z_1-z_2)^{\epsilon_1\epsilon_2 1}
  \xi_{i,m}^{1^{\epsilon_1}}(z_1)\xi_{j,n}^{1^{\epsilon_2}}(z_2)\\
  &\quad=
  (z_2-z_1)^{\epsilon_1\epsilon_2 \Bpsc{\ell}_{jinm}^{1,1,+}}
  (z_1-z_2)^{-\epsilon_1\epsilon_2 \Bpsc{\ell}_{jinm}^{1,1,-}}
  \tau_{jinm}^{1,1}(z_2-z_1)^{\epsilon_1\epsilon_21}
  \xi_{j,n}^{1^{\epsilon_2}}(z_2)\xi_{i,m}^{1^{\epsilon_1}}(z_1),\nonumber\\
  \label{tau5}\tag{$\tau$5}
  &
  (z_1-z_2)^{-\epsilon_1\epsilon_2 \Bpsc{\ell}_{ijmn}^{1,2}}
  (z_2-z_1)^{\epsilon_1\epsilon_2 \Bpsc{\ell}_{ijmn}^{2,1}}
  \tau_{ijmn}^{1,2}(z_1-z_2)^{\epsilon_1\epsilon_2 1}
  \xi_{i,m}^{1^{\epsilon_1}}(z_1)\xi_{j,n}^{2^{\epsilon_2}}(z_2)\\
  &\quad=
  (z_2-z_1)^{\epsilon_1\epsilon_2  \Bpsc{\ell}_{jinm}^{1,2}}
  (z_1-z_2)^{-\epsilon_1\epsilon_2 \Bpsc{\ell}_{jinm}^{2,1}}
  \tau_{jinm}^{2,1}(z_2-z_1)^{\epsilon_1\epsilon_2  1}
  \xi_{j,n}^{2^{\epsilon_2}}(z_2)\xi_{i,m}^{1^{\epsilon_1}}(z_1),\nonumber\\
  \label{tau6}\tag{$\tau$6}
  &(z_1-z_2)^{\Apsc{\ell}_{ijmn}^{2,2}}\tau_{ijmn}^{2,2}(z_1-z_2)
  \xi_{i,m}^{2^\pm}(z_1)\xi_{j,n}^{2^\pm}(z_2)\\
  &\quad=-(z_2-z_1)^{\Apsc{\ell}_{jinm}^{2,2}}
  \tau_{jinm}^{2,2}(z_2-z_1)
  \xi_{j,n}^{2^\pm}(z_2)\xi_{i,m}^{2^\pm}(z_1),\nonumber\\
  \label{tau7}\tag{$\tau$7}
  &(z_1-z_2)^{\delta_{ij}\max\{\delta_{mn},\delta_{m+2r\ell,n}\}}
  \xi_{i,m}^{2^+}(z_1)\xi_{j,n}^{2^-}(z_2)
  \\
  &\quad=
  -\iota_{z_2,z_1}\frac{(z_1-z_2)^{\delta_{ij}\max\{\delta_{mn},\delta_{m+2r\ell,n}\}}
    (z_1-z_2)^{\Apsc{\ell}_{ijnm}^{2,2}}
    \tau_{ijmn}^{2,2}(z_1-z_2)
    }{(z_2-z_1)^{\Apsc{\ell}_{jinm}^{2,2}} \tau_{jinm}^{2,2}(z_2-z_1)}
  \xi_{j,n}^{2^-}(z_2)\xi_{i,m}^{2^+}(z_1).\nonumber
\end{align}
And denote by $V_{\wp,\tau}'^\ell(\g)$ the quantum vertex algebra $F(\mathcal M_{\wp,\tau}^\ell(\g))$ obtained in Proposition \ref{prop:V-M-quantumVA}.
\end{de}

By utilizing Proposition \ref{prop:universal-nonlocal-va}, we immediately have the following result.
\begin{prop}\label{prop:universal-V}
Let $V$ be a nonlocal vertex algebra, and let $Y$ be the vertex operator map of $V$.
Suppose there exists $\bar \xi_{i,m}^a\in V$, such that
\begin{align*}
  (V,Y(\bar\xi_{i,m}^a,z))\in\obj\mathcal M_{\wp,\tau}^\ell(\wh\g).
\end{align*}
Then $\xi_{i,m}^a\mapsto \bar\xi_{i,m}^a$ defines a unique nonlocal vertex algebra homomorphism $V_{\wp,\tau}'^\ell(\g)\to V$.
\end{prop}

\begin{rem}\label{rem:S-wp-tau}
Denote by $S_{\wp,\tau}(z)$ the quantum Yang-Baxter operator of $V_{\wp,\tau}'^\ell(\g)$.
We have that
\begin{align}
  &S_{\wp,\tau}(z)(\xi_{j,n}^0\ot \xi_{i,m}^0)=\xi_{j,n}^0\ot \xi_{i,m}^0
  +\vac\ot\vac\ot
  \frac{\Apsc{\ell}_{ijmn}^{0,0}-\Apsc{\ell}_{jinm}^{0,0}}{z^2}\\
  &\quad\nonumber+\vac\ot\vac\ot(\tau_{ijmn}^{0,0}(-z)-\tau_{jinm}^{0,0}(z)),\\
  &S_{\wp,\tau}(z)(\xi_{j,n}^0\ot \xi_{i,m}^{a^\pm})=\xi_{j,n}^0\ot \xi_{i,m}^{a^\pm}
  \pm\vac\ot \xi_{i,m}^{a^\pm} \ot \frac{\Apsc{\ell}_{ijmn}^{a,0}-\Apsc{\ell}_{jinm}^{0,a}}{z}\\
  &\quad\nonumber\mp \vac\ot \xi_{i,m}^{a^\pm} \ot (\tau_{ijmn}^{a,0}(-z)+\tau_{jinm}^{0,a}(z)),\\
  &S_{\wp,\tau}(z)(\xi_{j,n}^{a^\pm}\ot \xi_{i,m}^0)=\xi_{j,n}^{a^\pm}\ot\xi_{i,m}^0
  \mp \xi_{j,n}^{a^\pm}\ot \vac\ot \frac{\Apsc{\ell}_{ijmn}^{0,a}-\Apsc{\ell}_{jinm}^{a,0}}{z}\\
  &\quad\nonumber\pm \xi_{j,n}^{a^\pm}\ot \vac\ot (\tau_{ijmn}^{0,a}(-z)+\tau_{jinm}^{a,0}(z)),\\
  &S_{\wp,\tau}(z)(\xi_{j,n}^{b^{\epsilon_2}}\ot \xi_{i,m}^{a^{\epsilon_1}})
  =(-1)^{\delta_{a,2}\delta_{b,2}}\xi_{j,n}^{b^{\epsilon_2}}\ot \xi_{i,m}^{a^{\epsilon_1}}
  \ot (-z)^{-\epsilon_1\epsilon_2 \Apsc{\ell}_{ijmn}^{a,b}}
  z^{\epsilon_1\epsilon_2 \Apsc{\ell}_{jinm}^{b,a}}\\
  &\quad\nonumber\times
  \tau_{ijmn}^{a,b}(-z)^{-\epsilon_1\epsilon_21}\tau_{jinm}^{b,a}(z)^{\epsilon_1\epsilon_21}.
\end{align}
\end{rem}

\begin{de}\label{de:V}
Define $V_{\wp,\tau}^\ell(\g)$ to be the quotient nonlocal vertex algebra of $V_{\wp,\tau}'^\ell(\g)$ modulo the ideal $R_{\wp,\tau}^\ell(\g)$ generated by
\begin{align}
  &\tag{$\tau$8}\xi_{i,m}^{2^+}(0)\xi_{j,n}^{2^-}-\delta_{ij}\delta_{mn}\vac
  +\delta_{ij}\delta_{m+2r\ell,n}\xi_{i,m+r\ell}^{1^+}\label{eq:va-e+e-},\\
  &\tag{$\tau$9}(\pm \xi_{i,m-r_i}^0(-1)\mp\xi_{i,m+r_i}^0(-1))\xi_{i,m}^{1^\pm}-\partial \xi_{i,m}^{1^\pm}\label{eq:va-rel-der}
  -(\tau_{ii,m-r_i,m}^{0,1}(0)-\tau_{ii,m+r_i,m}^{0,1}(0))\xi_{i,m}^{1^\pm} \\
  &\quad\te{for }i\in I,\,m\in\Z_\wp,\nonumber\\
  &\tag{$\tau$10}\xi_{i,m}^{1^\pm}\big(-1-\Bpsc{\ell}_{iimm}^{1,1,+}+\Bpsc{\ell}_{iimm}^{1,1,-}\big)
    \xi_{i,m}^{1^\mp}-
    (-1)^{\Bpsc{\ell}_{iimm}^{1,1,-}}\tau_{ii,m+r\ell,m}^{1,2}(0)
    \tau_{iim,m+r\ell}^{2,1}(0)\inv\vac,\label{eq:va-rel-varphi+-}\\
  &\tag{$\tau$11}\xi_{i,m-r_ia_{ij}}^{2^\pm}(0)\xi_{i,m-r_ia_{ij}-2r_i}^{2^\pm}(0)\cdots \xi_{i,m+r_ia_{ij}}^{2^\pm}(0)\xi_{j,m}^{2^\pm}\quad\te{for }m\in\Z_\wp,\,i,j\in I,\,\te{with }a_{ij}\le 0,\label{eq:va-rel-serre}\\
  &\tag{$\tau$12}\xi_{i,m+2r_i(\wp_i-1)}^{2^\pm}(0)\xi_{i,m+2r_i(\wp_i-2)}^{2^\pm}(0)\cdots \xi_{i,m+2r_i}^{2^\pm}(0)\xi_{i,m}^{2^\pm}\quad\te{for }i\in I,\,m\in\Z_\wp.
  \label{eq:va-rel-rtu}
\end{align}
\end{de}

\begin{rem}\label{rem:def-Y-V}
From Proposition \ref{prop:V-M-pre}, and the Definitions \ref{de:V-pre}, \ref{de:V}, we get that
the vertex operator map $Y$ of $V_{\wp,\tau}^\ell(\g)$
is determined by
\begin{align*}
  Y(\xi_{i,m}^a,z)=\xi_{i,m}^a(z)\quad\te{for }i\in I,\,m\in\Z_\wp,\,a\in\{0,1^\pm,2^\pm\}.
\end{align*}
\end{rem}

\begin{rem}\label{rem:countible-dim}
From Corollary \ref{coro:countable-dim}, it follows that
$V_{\wp,\tau}'^\ell(\g)$ has countable dimension.
Since $V_{\wp,\tau}^\ell(\g)$ is a quotient space of $V_{\wp,\tau}'^\ell(\g)$, we conclude that $V_{\wp,\tau}^\ell(\g)$ also has countable dimension.
\end{rem}

The proofs of the following five results are similar to the proof of \cite[Lemma 6.17]{K-Quantum-aff-va}.

\begin{lem}\label{lem:compatible-lower-bound}
Let $V$ be a nonlocal vertex algebra and let $u,v\in V$.
Suppose there exists $N\in\Z$, such that
\begin{align*}
  (z_1-z_2)^NY(u,z_1)Y(v,z_2)\in\Hom(V,V((z_1,z_2))).
\end{align*}
Then we have that
\begin{align*}
  &z^NY(u,z)v\in V[[z]]\quad\te{and}\quad
  Y(u_{N-1}v,z)=\lim_{z_1\to z}(z_1-z)^NY(u,z_1)Y(v,z).
\end{align*}
\end{lem}

\begin{proof}
From \eqref{eq:def-Y-E} and \eqref{eq:weak-asso}, we have that
\begin{align*}
  &z_0^NY(Y(u,z_0)v,z_2)=z_0^NY_\E(Y(u,z_2),z_0)Y(v,z_2)\\
  =&\left.\left((z_1-z_2)^NY(u,z_1)Y(v,z_2)\right)\right|_{z_1=z_2+z_0}
  \in \Hom(V,V((z_2))[[z_0]]).
\end{align*}
Consequently, $z^NY(u,z)v\in V[[z]]$, and
\begin{align*}
  &Y(u_{N-1}v,z)=\lim_{z_0\to 0}z_0^NY(Y(u,z_0)v,z)\\
  =&\lim_{z_0\to 0}\left.\left((z_1-z_2)^NY(u,z_1)Y(v,z_2)\right)\right|_{z_1=z_2+z_0}\\
  =&\lim_{z_1\to z}(z_1-z)^NY(u,z_1)Y(v,z).
\end{align*}
We complete the proof of lemma.
\end{proof}

\begin{lem}\label{lem:S-A}
Set
\begin{align*}
  &A_{i,m,n}(z)=Y(\xi_{i,m}^{2^+},z)^-\xi_{i,n}^{2^-}
  -\delta_{mn}\vac z\inv+\delta_{m+2r\ell,n}\xi_{i,m+r\ell}^{1^+}z\inv.
\end{align*}
Then
\begin{align*}
  &S_{\wp,\tau}(z)(A_{j,m,n}(z_1)\ot \xi_{i,k}^a)
  \equiv 0 \mod R_{\wp,\tau}^\ell(\g)\ot V_{\wp,\tau}'^\ell(\g)\ot \C((z))\\
  &S_{\wp,\tau}(z)(\xi_{j,k}^a\ot A_{i,m,n}(z_1))
  \equiv 0 \mod V_{\wp,\tau}'^\ell(\g)\ot R_{\wp,\tau}^\ell(\g)\ot \C((z))\\
  &\qquad\qquad\te{for }i,j\in I,\,m,n,k\in \Z_\wp,\,a\in\{0,1^\pm,2^\pm\}.
\end{align*}
\end{lem}

\begin{proof}
Applying Lemma \ref{lem:compatible-lower-bound} to \eqref{tau7}, we have that
\begin{align*}
  Y(\xi_{i,m}^{2^+},z)^-\xi_{i,m}^{2^-}\in z\inv V_{\wp,\tau}'^\ell(\g).
\end{align*}
Combining this with \eqref{eq:qyb-vac}, \eqref{eq:qyb-hex1}, \eqref{eq:qyb-hex2} and Remark \ref{rem:S-wp-tau}, the lemma follows by a straightforward verification.
\end{proof}

\begin{lem}\label{lem:S-B}
Set
\begin{align*}
  B_{i,m}^\pm=&(\pm \xi_{i,m-r_i}^0(-1)\mp\xi_{i,m+r_i}^0(-1))\xi_{i,m}^{1^\pm}-\partial \xi_{i,m}^{1^\pm}
  -(\tau_{ii,m-r_i,m}^{0,1}(0)-\tau_{ii,m+r_i,m}^{0,1}(0))\xi_{i,m}^{1^\pm}.
\end{align*}
Then
\begin{align*}
  &S_{\wp,\tau}(z)(B_{j,n}^\pm(z_1)\ot \xi_{i,m}^a)
  \equiv 0 \mod R_{\wp,\tau}^\ell(\g)\ot V_{\wp,\tau}'^\ell(\g)\ot \C((z))\\
  &S_{\wp,\tau}(z)(\xi_{j,n}^a\ot B_{i,m}^\pm(z_1))
  \equiv 0 \mod V_{\wp,\tau}'^\ell(\g)\ot R_{\wp,\tau}^\ell(\g)\ot \C((z))\\
  &\qquad\qquad\te{for }i,j\in I,\,m,n\in \Z_\wp,\,a\in\{0,1^\pm,2^\pm\}.
\end{align*}
\end{lem}

\begin{proof}
A direct verification using \eqref{eq:qyb-vac}, \eqref{eq:qyb-hex1}, \eqref{eq:qyb-hex2}, \eqref{eq:qyb-der-shift} and Remark \ref{rem:S-wp-tau} then establishes the lemma.
\end{proof}

\begin{lem}\label{lem:S-C}
Set
\begin{align*}
  C_{i,m}^\pm=\xi_{i,m}^{1^\pm}
  (-1-B_{iimm}^{1,1,+}+B_{iimm}^{1,1,-})\xi_{i,m}^{1^{\mp}}
  -(-1)^{\Bpsc{\ell}_{iimm}^{1,1,-}}\tau_{ii,m+r\ell,m}^{1,2}(0)
    \tau_{iim,m+r\ell}^{2,1}(0)\inv\vac.
\end{align*}
Then
\begin{align*}
  &S_{\wp,\tau}(z)(C_{j,n}^\pm(z_1)\ot \xi_{i,m}^a)
  \equiv 0 \mod R_{\wp,\tau}^\ell(\g)\ot V_{\wp,\tau}'^\ell(\g)\ot \C((z))\\
  &S_{\wp,\tau}(z)(\xi_{j,n}^a\ot C_{i,m}^\pm(z_1))
  \equiv 0 \mod V_{\wp,\tau}'^\ell(\g)\ot R_{\wp,\tau}^\ell(\g)\ot \C((z))\\
  &\qquad\qquad\te{for }i,j\in I,\,m,n\in \Z_\wp,\,a\in\{0,1^\pm,2^\pm\}.
\end{align*}
\end{lem}

\begin{proof}
Applying Lemma \ref{lem:compatible-lower-bound} to \eqref{tau4}, we have that
\begin{align*}
  z^{-\Bpsc{\ell}_{iimm}^{1,1,+}+\Bpsc{\ell}_{iimm}^{1,1,-}}
  Y(\xi_{i,m}^{1^\pm},z)\xi_{i,m}^{1^\mp}
  \in V_{\wp,\tau}'^\ell(\g)[[z]].
\end{align*}
Combining this with \eqref{eq:qyb-vac}, \eqref{eq:qyb-hex1}, \eqref{eq:qyb-hex2} and Remark \ref{rem:S-wp-tau}, the lemma follows by a straightforward verification.
\end{proof}

\begin{lem}\label{lem:S-D}
For $i,j\in I$ with $a_{ij}\le 0$ and $m\in\Z_\wp$, set
\begin{align*}
  &D_{i,j,m}^\pm=\xi_{i,m-r_ia_{ij}}^{2^\pm}(0)\xi_{i,m-r_ia_{ij}-2r_i}^{2^\pm}(0)\cdots \xi_{i,m+r_ia_{ij}}^{2^\pm}(0)\xi_{j,m}^{2^\pm}.
\end{align*}
Then
\begin{align*}
  &S_{\wp,\tau}(z)(D_{i,j,m}^\pm(z_1)\ot \xi_{k,n}^a)
  \equiv 0 \mod R_{\wp,\tau}^\ell(\g)\ot V_{\wp,\tau}'^\ell(\g)\ot \C((z))\\
  &S_{\wp,\tau}(z)(\xi_{k,n}^a\ot D_{i,j,m}^\pm(z_1))
  \equiv 0 \mod V_{\wp,\tau}'^\ell(\g)\ot R_{\wp,\tau}^\ell(\g)\ot \C((z))\\
  &\qquad\qquad\te{for }i,j,k\in I,\,m,n\in \Z_\wp,\,a\in\{0,1^\pm,2^\pm\}.
\end{align*}
\end{lem}

\begin{proof}
For each $-1\le N\in\Z$, we set
\begin{align*}
  &\wt D_{i,j,m,-1}^\pm=\xi_{j,m}^\pm,\quad
  \wt D_{i,j,m,N+1}^\pm= \xi_{i,m+r_ia_{ij}+2r_i(N+1)}^\pm(0)
  \wt D_{i,j,m,N}^\pm,\quad
  Y_{i,j,m,-1}^\pm(z)=\xi_{j,m}^\pm(z),\\
  &\te{and}\quad Y_{i,j,m,N+1}^\pm(z_1,\dots,z_{N+2},z)
  =\prod_{t=2}^{N+2}(z_1-z_t)\xi_{i,m+r_ia_{ij}+2r_i(N+1)}^\pm(z_1)
  Y_{i,j,m,N}^\pm(z_2,\dots, z_{N+2},z).
\end{align*}
From \eqref{tau6}, we have that
\begin{align}\label{eq:S-D-temp}
  Y_{i,j,m,N}^\pm(z_1,\dots,z_{N+1},z)\in \Hom(V_{\wp,\tau}'^\ell(\g),
    V_{\wp,\tau}'^\ell(\g)((z_1,\dots,z_{N+1},z))).
\end{align}
We first prove the following two results by using induction on $N$:
\begin{align}
  &Y(\xi_{i,m+r_ia_{ij}+2r_iN}^\pm,z)\wt D_{i,j,m,N-1}^\pm
  \in z\inv V_{\wp,\tau}'^\ell(\g)[[z]],\label{eq:S-D-cond1}\\
  &Y(\wt D_{i,j,m,N}^\pm,z)
  =\prod_{s=1}^{N+1}\lim_{z_s\to z}
  Y_{i,j,m,N}^\pm(z_1,\dots,z_{N+1},z)\label{eq:S-D-cond2}.
\end{align}
When $N=0$, \eqref{eq:S-D-cond1} and \eqref{eq:S-D-cond2} hold upon applying Lemma \ref{lem:compatible-lower-bound} to \eqref{tau6}.
Suppose that \eqref{eq:S-D-cond1} and \eqref{eq:S-D-cond2} hold  for $N$.
Combining the induction assumption with \eqref{eq:S-D-temp}, we have that
\begin{align*}
  (z_1-z_2)Y(\xi_{i,m,r_ia_{ij}+2r_i(N+1)}^\pm,z_1)
  Y(\wt D_{i,j,m,N},z_2)\in \Hom(V_{\wp,\tau}'^\ell(\g),
    V_{\wp,\tau}'^\ell(\g)((z_1,z_2))).
\end{align*}
By applying Lemma \ref{lem:compatible-lower-bound}, we complete the proof of \eqref{eq:S-D-cond1} and \eqref{eq:S-D-cond2} for $N+1$.

From \eqref{eq:S-D-cond1}, we have that
\begin{align*}
  &Y(\xi_{i,m-r_ia_{ij}}^{2^\pm},z_1)^-
  Y(\xi_{i,m-r_ia_{ij}-2r_i}^{2^\pm},z_2)^-\cdots Y(\xi_{i,m+r_ia_{ij}}^{2^\pm},z_{m_{ij}})^-\xi_{j,m}^{2^\pm}\\
  =&\xi_{i,m-r_ia_{ij}}^{2^\pm}(0)\xi_{i,m-r_ia_{ij}-2r_i}^{2^\pm}(0)\cdots \xi_{i,m+r_ia_{ij}}^{2^\pm}(0)\xi_{j,m}^{2^\pm}
  z_1\inv z_2\inv\cdots z_{m_{ij}}\inv.
\end{align*}
Combining this with \eqref{eq:qyb-vac}, \eqref{eq:qyb-hex1}, \eqref{eq:qyb-hex2} and Remark \ref{rem:S-wp-tau}, the lemma follows by a straightforward verification.
\end{proof}

A proof similar to that of Lemma \ref{lem:S-D} yields the following result.

\begin{lem}\label{lem:S-E}
Set
\begin{align*}
  &E_{i,m}^\pm=\xi_{i,m+2r_i(\wp_i-1)}^{2^\pm}(0)\xi_{i,m+2r_i(\wp_i-2)}^{2^\pm}(0)\cdots \xi_{i,m+2r_i}^{2^\pm}(0)\xi_{i,m}^{2^\pm}.
\end{align*}
Then
\begin{align*}
  &S_{\wp,\tau}(z)(E_{j,n}^\pm(z_1)\ot \xi_{i,m}^a)
  \equiv 0 \mod R_{\wp,\tau}^\ell(\g)\ot V_{\wp,\tau}'^\ell(\g)\ot \C((z))\\
  &S_{\wp,\tau}(z)(\xi_{j,n}^a\ot E_{i,m}^\pm(z_1))
  \equiv 0 \mod R_{\wp,\tau}^\ell(\g)\ot V_{\wp,\tau}'^\ell(\g)\ot \C((z))\\
  &\qquad\qquad\te{for }i,j\in I,\,m,n\in \Z_\wp,\,a\in\{0,1^\pm,2^\pm\}.
\end{align*}
\end{lem}

The following is the main result of this subsection.

\begin{thm}\label{thm:V-wp-tau-qva}
$V_{\wp,\tau}^\ell(\g)$ is a quantum vertex algebra with quantum Yang-Baxter operator induced by $S_{\wp,\tau}(z)$, which is still denoted by $S_{\wp,\tau}(z)$.
\end{thm}

\begin{proof}
Note that $V_{\wp,\tau}'^\ell(\g)$ is generated by
\begin{align*}
  \set{\xi_{i,m}^a}{i\in I,\,m\in\Z_\wp,\,a\in \{0,1^\pm,2^\pm\}}.
\end{align*}
Together with \eqref{eq:qyb-vac}, \eqref{eq:qyb-hex1}, \eqref{eq:qyb-hex2}, and Lemmas \ref{lem:S-A}--\ref{lem:S-E}, a straightforward verification yields
\begin{align*}
  &S_{\wp,\tau}(z)(R_{\wp,\tau}^\ell(\g)\ot V_{\wp,\tau}'^\ell(\g)),\quad
  S_{\wp,\tau}(z)(V_{\wp,\tau}'^\ell(\g)\ot R_{\wp,\tau}^\ell(\g))\\
  &\quad\subset R_{\wp,\tau}^\ell(\g)\ot V_{\wp,\tau}'^\ell(\g)\ot\C((z))+V_{\wp,\tau}'^\ell(\g)\ot R_{\wp,\tau}^\ell(\g)\ot\C((z)).
\end{align*}
This completes the proof of the theorem.
\end{proof}

\subsection{$\Z_\wp$-module structure}\label{subsec:Z-wp-mod}

In this subsection, we give a special $\tau\in \mathfrak T$,
and construct a $\Z_\wp$-module structure on $V_{\wp,\tau}^\ell(\g)$.
For $g(q)\in \Z[q,q\inv]$, we define
\begin{align}\label{eq:def-C-g-zeta}
  C(g(q))=\prod_{0\not\equiv s\in\Z_\wp}
  (1-\zeta^s)^{\vvp{g(q)q^{-s}}}\in\C^\times.
\end{align}
and define
\begin{align}\label{eq:def-E-g-zeta}
  E(z,g(q))=C(g(q))
  \prod_{s\in\Z_\wp}\exp\big( \vvp{g(q)q^{-s}}\vartheta_s(z) \big)\in\C[[z]]^\times,
\end{align}
where
\begin{align}\label{eq:def-vartheta}
  &\vartheta_s(z)=\begin{cases}
                    \log \left((e^{z/2}-e^{-z/2})/z\right),&\te{if }s\equiv 0 \mod \wp,\\
                    \log \left((e^{z/2}-\zeta^s e^{-z/2})/(1-\zeta^s)\right),&\te{if }s\not\equiv 0 \mod \wp.
                  \end{cases}
\end{align}

\begin{rem}
For $s\in\Z_\wp$, we have that
\begin{align*}
  &\vartheta_s(z)\in z\C[[z]],\quad \vartheta_s(-z)=\vartheta_{-s}(z),\\
  &\pd z\vartheta_0(z)=
  \frac{1+e^{-z}}{2-2e^{-z}}-\frac{1}{z}, \quad
  \pdiff{z}{2}\vartheta_0(z)=-\frac{e^{-z}}{(1-e^{-z})^2}+\frac{1}{z^2},\\
  &\pd z\vartheta_s(z)=  \frac{1+\zeta^s e^{-z}}{2-2\zeta^se^{-z}},\quad
  \pdiff{z}{2}\vartheta_s(z)=-\frac{\zeta^s e^{-z}}{(1-\zeta^s e^{-z})^2},\quad\mbox{if }s\not\equiv 0\mod \wp.
\end{align*}
\end{rem}

\begin{rem}
For $f(q),g(q)\in \Z[q,q\inv]$, we have that
\begin{align*}
  &C(g(q))\inv=C(-g(q)),\quad C(f(q)+g(q))=C(f(q))C(g(q)),\\
  &C(g(q\inv))=(-1)^{g(1)-\vvp{g(q)}}\zeta^{-dg(1)}C(g(q)),
  \quad\te{where }dg(q)=q\pd qg(q),\\
  &E(-z,g(q))=(-1)^{g(1)-\vvp{g(q)}}\zeta^{dg(1)}
  E(z,g(q\inv)),\\
  & E(z,g(q))\inv=E(z,-g(q)),\quad
  E(0,g(q))=C(g(q)),\\
  &E(z,f(q)+g(q))=E(z,f(q))E(z,g(q)),\\
  &\pd z \log E(z,g(q))=\sum_{s\in\Z_\wp}\vvp{g(q)q^{-s}}\pd z \vartheta_s(z),\\
  &\prod_{s\in\Z_\wp}(1-\zeta^se^{-z})^{\vvp{g(q)q^{-s}}}
  =e^{-g(1)z/2}z^{\vvp{g(q)}}
  E(z,g(q)).
\end{align*}
\end{rem}

Set
\begin{align}
  &\tau_{ijmn}^{0,0}(z)= -\sum_{s\in\Z_\wp}\vvp{  q_i^{-a_{ij}}[r\ell/r_j]_{q_j}[r\ell/r_i]_{q_i}
  q^{n-m-s} }
  \pdiff{z}{2}\vartheta_s(z),\\
  &\tau_{ijmn}^{0,1}(z)=\sum_{s\in\Z_\wp}\vvp{ q_i^{-a_{ij}}[r\ell/r_i]_{q_i}(q^{-r\ell}-q^{r\ell})
  q^{n-m-s} }
  \pd z \vartheta_s(z),\\
  &\tau_{ijmn}^{1,0}(z)= \sum_{s\in\Z_\wp}\vvp{ q_i^{-a_{ij}}[r\ell/r_j]_{q_j}(q^{r\ell}-q^{-r\ell})
  q^{n-m-s} }
  \pd z \vartheta_s(z),\\
  &\tau_{ijmn}^{0,2}(z)=\sum_{s\in\Z_\wp}\vvp{ q_i^{-a_{ij}}[r\ell/r_i]_{q_i}q^{n-m-s} }
  \pd z \vartheta_s(z),\\
  &\tau_{ijmn}^{2,0}(z)=\sum_{s\in\Z_\wp}\vvp{ q_i^{-a_{ij}}[r\ell/r_j]_{q_j}q^{n-m-s} }
  \pd z \vartheta_s(z),\\
  &\tau_{ijmn}^{1,1}(z)=
  (-1)^{\vvp{q_i^{a_{ij}}(q^{r\ell}-q^{-r\ell})^2
  q^{n-m}}}
  E(z,q_i^{-a_{ij}}(q^{r\ell}-q^{-r\ell})^2q^{n-m}),\\
  &\tau_{ijmn}^{1,2}(z)=(-1)^{\vvp{ q_i^{a_{ij}}(q^{-r\ell}-q^{r\ell})q^{n-m}} }
   E(z,q_i^{-a_{ij}} (q^{-r\ell}-q^{r\ell})q^{n-m}),\\
  &\tau_{ijmn}^{2,1}(z)=(-1)^{\vvp{ q_i^{a_{ij}}(q^{-r\ell}-q^{r\ell})q^{n-m}} }\zeta^{-2r\ell}
   E(z,q_i^{-a_{ij}}(q^{r\ell}-q^{-r\ell})q^{n-m}),\\
  &\tau_{ijmn}^{2,2}(z)=(-1)^{\vvp{q_i^{a_{ij}}q^{n-m}}}\zeta^{n}
   E(z,-q_i^{-a_{ij}}q^{n-m}).
\end{align}

It is straightforward to verify the following result.
\begin{lem}\label{lem:tau-alt}
The relations \eqref{tau1}-\eqref{eq:va-e+e-} and \eqref{eq:va-rel-varphi+-} are equivalent to the relations below
\begin{align}
  &\label{tau1-alt}\tag{$\tau1'$}
  \xi_{i,m}^0(z_1)\xi_{j,n}^0(z_2)
  -\sum_{s\in\Z_\wp}
  \vvp{[a_{ij}]_{q_i}[r\ell/r_j]_{q_j}q^{-r\ell+n-m-s}}
  \frac{\zeta^se^{-z_1+z_2}}{(1-\zeta^se^{-z_1+z_2})^2}\\
  &\quad\nonumber=\xi_{j,n}^0(z_2)\xi_{i,m}^0(z_1)
  -\sum_{s\in\Z_\wp}
  \vvp{[a_{ij}]_{q_i}[r\ell/r_j]_{q_j}q^{r\ell+n-m-s}}
  \frac{\zeta^se^{z_2-z_1}}{(1-\zeta^se^{z_2-z_1})^2},\\
  &\label{tau2-alt}\tag{$\tau2'$}
  \xi_{i,m}^0(z_1)\xi_{j,n}^{1^\pm}(z_2)
  \mp \xi_{j,n}^{1^\pm}(z_2)\sum_{s\in\Z_\wp}
  \vvp{[a_{ij}]_{q_i}(q^{-2r\ell}-1)q^{n-m-s}}
  \frac{1+\zeta^se^{-z_1+z_2}}{2-2\zeta^se^{-z_1+z_2}}\\
  &\quad\nonumber=
  \xi_{j,n}^{1^\pm}(z_2)\xi_{i,m}^0(z_1)
  \mp \xi_{j,n}^{1^\pm}(z_2)\sum_{s\in\Z_\wp}
  \vvp{[a_{ij}]_{q_i}(1-q^{2r\ell})q^{n-m-s}}
  \frac{1+\zeta^se^{z_2-z_1}}{2-2\zeta^se^{z_2-z_1}},\\
  &\label{tau3-alt}\tag{$\tau3'$}
  \prod_{s\in\Z_\wp}(1-\zeta^se^{-z_1+z_2})^{
  \epsilon_1\epsilon_2\vvp{
    (q_i^{a_{ij}}-q_i^{-a_{ij}})(1-q^{-2r\ell})q^{n-m-s}
  }}\xi_{i,m}^{1^{\epsilon_1}}(z_1)\xi_{j,n}^{1^{\epsilon_2}}(z_2)\\
  &\quad\nonumber=
  \prod_{s\in\Z_\wp}(1-\zeta^se^{z_2-z_1})^{
    \epsilon_1\epsilon_2\vvp{
    (q_i^{a_{ij}}-q_i^{-a_{ij}})(q^{2r\ell}-1)q^{n-m-s}
  }}\xi_{j,n}^{1^{\epsilon_2}}(z_2)\xi_{i,m}^{1^{\epsilon_1}}(z_1),\\
  &\label{tau4-alt}\tag{$\tau4'$}
  \lim_{z_1\to z_2}\prod_{s\in\Z_\wp}(1-\zeta^se^{-z_1+z_2})^{
    \vvp{
    (q_i^2-q_i^{-2})(q^{-2r\ell}-1)q^{-s}
  }}\xi_{i,m}^{1^\pm}(z_1)\xi_{j,m}^{1^\mp}(z_2)=1,\\
  &\label{tau5-alt}\tag{$\tau5'$}
  \xi_{i,m}^0(z_1)\xi_{j,n}^{2^\pm}(z_2)
  \mp \xi_{j,n}^{2^\pm}(z_2)\sum_{s\in\Z_\wp}
  \vvp{[a_{ij}]_{q_i}q^{-r\ell+n-m-s}}
  \frac{1+\zeta^se^{-z_1+z_2}}{2-2\zeta^se^{-z_1+z_2}}\\
  &\quad\nonumber=\xi_{j,n}^{2^\pm}(z_2)\xi_{i,m}^0(z_1)
  \mp \xi_{j,n}^{2^\pm}(z_2)\sum_{s\in\Z_\wp}
  \vvp{[a_{ij}]_{q_i}q^{r\ell+n-m-s}}
  \frac{1+\zeta^se^{z_2-z_1}}{2-2\zeta^se^{z_2-z_1}},\\
  &\label{tau6-alt}\tag{$\tau6'$}
  \prod_{s\in\Z_\wp}(1-\zeta^se^{-z_1+z_2})^{
  -\epsilon_1\epsilon_2\vvp{
    (q_i^{a_{ij}}-q_i^{-a_{ij}})q^{-r\ell+n-m-s}
  }}\xi_{i,m}^{1^{\epsilon_1}}(z_1)\xi_{j,n}^{2^{\epsilon_2}}(z_2)\\
  &\quad\nonumber=
  \prod_{s\in\Z_\wp}(1-\zeta^se^{z_2-z_1})^{
  -\epsilon_1\epsilon_2\vvp{
    (q_i^{a_{ij}}-q_i^{-a_{ij}})q^{r\ell+n-m-s}
  }}\xi_{j,n}^{2^{\epsilon_2}}(z_2)\xi_{i,m}^{1^{\epsilon_1}}(z_1),\\
  &\label{tau7-alt}\tag{$\tau7'$}
  \prod_{s\in\Z_\wp}(1-\zeta^se^{-z_1+z_2})^{
    \vvp{q_i^{a_{ij}}q^{n-m-s}}
  }\xi_{i,m}^{2^\pm}(z_1)\xi_{j,n}^{2^\pm}(z_2)\\
  &\quad\nonumber=\zeta_i^{a_{ij}}
  \prod_{s\in\Z_\wp}(1-\zeta^se^{z_2-z_1})^{
    \vvp{q_i^{-a_{ij}}q^{n-m-s}}
  }\xi_{j,n}^{2^\pm}(z_2)\xi_{i,m}^{2^\pm}(z_1),\\
  &\label{tau8-alt}\tag{$\tau8'$}
  \xi_{i,m}^{2^+}(z_1)\xi_{j,n}^{2^-}(z_2)
  -\zeta_i^{-a_{ij}}\prod_{s\in\Z_\wp}(1-\zeta^se^{z_2-z_1})^{
    \vvp{
    (q_i^{a_{ij}}-q_i^{-a_{ij}})q^{n-m-s}
  }}\xi_{j,n}^{2^-}(z_2)\xi_{i,m}^{2^+}(z_1)\\
  &\quad\nonumber=
  \delta_{ij}\delta_{mn}z_1\inv\delta\left(\frac{z_2}{z_1}\right)
  -\delta_{ij}\delta_{m+2r\ell,n}\xi_{j,m+r\ell}^{1^+}(z_2)
  z_1\inv\delta\left(\frac{z_2}{z_1}\right).
\end{align}
\end{lem}

The following result is an immediate consequence of Proposition \ref{prop:universal-V} and Lemmas \ref{lem:tau-alt}.

\begin{prop}\label{prop:universal-V-sp}
Let $V$ be a nonlocal vertex algebra, and let $Y$ be the vertex operator map of $V$. Suppose that there exists $\bar\zeta_{i,m}^a\in V$ ($i\in I$, $m\in\Z_\wp$, $a\in \{0,1^\pm,2^\pm\}$), such that
$Y(\bar\zeta_{i,m}^a,z)$ satisfy the relations \eqref{tau1-alt}-\eqref{tau8-alt}, \eqref{eq:va-rel-der},
\eqref{eq:va-rel-serre} and \eqref{eq:va-rel-rtu} with
\begin{align*}
  \xi_{i,m}^a(z)=Y(\bar \xi_{i,m}^a,z)\quad i\in I,\,m\in\Z_\wp,\,a\in \{0,1^\pm,2^\pm\}.
\end{align*}
Then there is a unique nonlocal vertex operator algebra homomorphism from $V_{\wp,\tau}^\ell(\g)\to V$ determined by $\zeta_{i,m}^a\mapsto\bar \zeta_{i,m}^a$ ($i\in I$, $m\in\Z_\wp$, $a\in \{0,1^\pm,2^\pm\}$).
\end{prop}

The following result defines a $\Z_\wp$-module structure on $V_{\wp,\tau}^\ell(\g)$.

\begin{prop}\label{prop:Z-module}
For each $s\in\Z_\wp$, there is a quantum vertex algebra homomorphism $R(s)$ on $V_{\wp,\tau}^\ell(\g)$ determined by
\begin{align*}
  R(s)\xi_{i,m}^a=\xi_{i,m+s}^a\quad\te{for }i\in I,\,m\in\Z_\wp,\,a\in \{0,1^\pm,2^\pm\}.
\end{align*}
Moreover, $(V_{\wp,\tau}^\ell(\g),R)$ is a $\Z_\wp$-module quantum vertex algebra.
\end{prop}

\begin{proof}
Note that the relations \eqref{tau1-alt}-\eqref{tau8-alt}, \eqref{eq:va-rel-der},
\eqref{eq:va-rel-serre} and \eqref{eq:va-rel-rtu} hold with
\begin{align*}
  \xi_{i,m}^a(z)=Y(\xi_{i,m+s}^a,z),\quad i\in I,\,a\in \{0,1^\pm,2^\pm\},\,m,s\in\Z_\wp.
\end{align*}
By using Proposition \ref{prop:universal-V-sp}, we get a
nonlocal vertex algebra homomorphism $R(s)$ on $V_{\wp,\tau}^\ell(\g)$ determined by
\begin{align}\label{eq:Z-module-R-s-def}
  \xi_{i,m}^a\mapsto\xi_{i,m+s}^a,\quad i\in I,\,m\in\Z_\wp,\,a\in\{0,1^\pm,2^\pm\}.
\end{align}
It is straightforward to verify that $R(s)$ preserves the quantum Yang-Baxter operator $S_{\wp,\tau}(z)$.
Note that
\begin{align*}
  R(s)R(t)\xi_{i,m}^a=\xi_{i,m+t+s}^a=R(s+t)\xi_{i,m}^a,
  \quad i\in I,\,a\in \{0,1^\pm,2^\pm\},\,m,s,t\in\Z_\wp.
\end{align*}
Then $R(s)R(t)=R(s+t)$ since $V_{\wp,\tau}^\ell(\g)$ generated by
\begin{align*}
  \set{\xi_{i,m}^a}{ i\in I,\,m\in\Z_\wp,\,a\in\{0,1^\pm,2^\pm\}}.
\end{align*}
From \eqref{eq:Z-module-R-s-def}, we have that
$R(0)=\id_{V_{\wp,\tau}^\ell(\g)}$.
Then $R(s)$ is invertible for each $s\in\Z_\wp$.
Therefore, $R$ is a group homomorphism from $\Z_\wp$ to the automorphism group of $V_{\wp,\tau}^\ell(\g)$,
and hence $(V_{\wp,\tau}^\ell(\g),R)$ is a $(\Z_\wp,\chi)$-module quantum vertex algebra.
\end{proof}

\section{$\phi$-coordinated modules}\label{sec:phi-mod}

This section is devoted to establish a connection between $\mathcal R_\zeta^\ell(\wh\g)$ and the category of $(\Z_\wp,\chi_\phi)$ equivariant $\phi$-coordinated quasi-modules for $V_{\wp,\tau}^\ell(\g)$.
Throughout this section, we fix an associate $\phi(x,z)=xe^z$ of $F_a(x,y)=x+y$.

\subsection{From $\mathcal R_\zeta^\ell(\wh\g)$ to $\phi$-coordinated quasi-modules}\label{subsec:rtu-mod-to-phi-mod}

Let $(W,H_i(z),\Psi_i^\pm(z),X_i^\pm(z))$ be an object in $\mathcal R_\zeta^\ell(\wh\g)$.
Recall the quasi compatible subset from \eqref{eq:U-W}
\begin{align*}
  U_W=\set{H_i(\zeta^sz),\,\Psi_i^\pm(\zeta^sz),\,X_i^\pm(\zeta^sz)}{i\in I,\,s\in\Z_\wp},
\end{align*}
and the $\Z_\wp$-module nonlocal vertex algebra $(\<U_W\>_\phi,R)$, the $(\Z_\wp,\chi_\phi)$-equivariant $\phi$-coordinated quasi $\<U_W\>_\phi$-module structure on $W$ obtained in Proposition
\ref{prop:qaff-mod-to-module-nva}.
The following result shows that $W$ is a $\phi$-coordinated quasi-module of $V_{\wp,\tau}^\ell(\g)$.

\begin{prop}\label{prop:V-U-W-hom}
There is a nonlocal vertex algebra homomorphism from $V_{\wp,\tau}^\ell(\g)$ to $\<U_W\>_\phi$ uniquely determined by
\begin{align*}
  \xi_{i,m}^0\mapsto H_i(\zeta^m z),\quad
  \xi_{i,m}^{1^\pm}\mapsto \Psi_i^\pm(\zeta^m z),\quad
  \xi_{i,m}^{2^\pm}\mapsto X_i^\pm(\zeta^m z).
\end{align*}
\end{prop}

In order to prove Proposition \ref{prop:V-U-W-hom} by utilizing Proposition \ref{prop:universal-V-sp}, we need to verify the relations \eqref{tau1-alt}-\eqref{tau8-alt}, \eqref{eq:va-rel-der}, \eqref{eq:va-rel-serre} and \eqref{eq:va-rel-rtu} with
\begin{align*}
  \xi_{i,m}^0(z)=Y_\E^\phi(H_i(\zeta^m z_1),z),\quad \xi_{i,m}^{1^\pm}(z)=Y_\E^\phi(\Psi_i^\pm(\zeta^m z_1),z), \quad
  \xi_{i,m}^{2^\pm}(z)=Y_\E^\phi(X_i^\pm(\zeta^m z_1),z).
\end{align*}
The relations \eqref{tau1-alt}-\eqref{tau8-alt} follow directly from Lemma \ref{lem:tau-alt}, Proposition \ref{prop:phi-mod-rel-f} and Theorem \ref{thm:presentation}; the relations \eqref{eq:va-rel-serre} and \eqref{eq:va-rel-rtu} follow immediately from \eqref{zeta-serre} and \eqref{zeta-res-form}. Thus it remains to verify the relation \eqref{eq:va-rel-der}. To this end, we first establish the following result.

\begin{lem}\label{lem:phi-neg-one}
Fix an associate $\phi(x,z)=\exp(zp(x)\pd{x})x\ne x$ of $F_a(x,y)$.
Let $W$ be a vector space, let
$U$ be a quasi-compatible subset of $\E(W)$,
and let $\al^\pm(z),\beta(z)\in U$, such that
\begin{align*}
  &[\al^\pm(z_1),\al^\pm(z_2)]=0,\quad \al^+(z)\in \Hom(W,W[[z]]),\quad\te{and }
   [\al^-(z_1),\beta(z_2)]=\beta(z_2)\iota_{z_1,z_2}f(z_1,z_2)
\end{align*}
for some $f(z_1,z_2)\in \C(z_1,z_2)$.
Then
\begin{align*}
  (\al^-(z)+\al^+(z))_{-1}^\phi\beta(z)=\al^+(z)\beta(z)+\beta(z)\al^-(z)+\beta(z)\Res_{z_0}z_0\inv \iota_{z,z_0}f(\phi(z,z_0),z).
\end{align*}
\end{lem}

\begin{proof}
Set $\al(z)=\al^-(z)+\al^+(z)$ and set
\begin{align*}
  &\:\al(z)\beta(w)\;=\al^+(z)\beta(w)+\beta(w)\al^-(z).
\end{align*}
Then we have that
\begin{align*}
  \al(z_1)\beta(z)=\:\al(z_1)\beta(z)\;+\beta(z)\iota_{z_1,z}f(z_1,z).
\end{align*}
Let $g(z_1,z_2)\in \C[z_1,z_2]$ such that $g(z_1,z_2)f(z_1,z_2)\in \C[z_1,z_2]$.
Then we have that
\begin{align*}
  &\al(z)_{-1}^\phi\beta(z)=\Res_{z_0}z_0\inv Y_\E^\phi(\al(z),z_0)\beta(z)\nonumber\\
  =&\Res_{z_0}z_0\inv \iota_{z,z_0}g(\phi(z,z_0),z)\inv \left( g(z_1,z) \al(z_1)\beta(z) \right)|_{z_1=\phi(z,z_0)}\nonumber\\
  =&\:\al(z)\beta(z)\;+\Res_{z_0}z_0\inv \iota_{z,z_0}g(\phi(z,z_0),z)\inv \left(g(\phi(z,z_0),z)f(\phi(z,z_0),z)\right)\beta(z)\nonumber\\
  =&\:\al(z)\beta(z)\;+\beta(z)\Res_{z_0}z_0\inv \iota_{z,z_0}f(\phi(z,z_0),z),
\end{align*}
which completes the proof of lemma.
\end{proof}

The relation \eqref{eq:va-rel-der} then follows immediately from the fact that $z\pd z$ is the canonical derivation of $\< U_W\>_\phi$ and the following result, completing the proof of Proposition \ref{prop:V-U-W-hom}.

\begin{lem}\label{lem:zeta10-to-tau-9}
For each $i\in I$ and $m\in\Z_\wp$, we have that
\begin{align*}
  &(\pm H_i(\zeta^{m-r_i}z)\mp H_i(\zeta^{m+r_i}z))_{-1}^\phi \Psi_i^\pm(\zeta^m z)\\
  =&z\pd z\Psi_i^\pm(\zeta^mz)
  +(\tau_{ii,m-r_i,m}^{0,1}(0)-\tau_{ii,m+r_i,m}^{1,0}(0))\Psi_i^\pm(\zeta^mz).
\end{align*}
\end{lem}

\begin{proof}
Set
\begin{align}\label{eq:def-wtH}
  \wt H_i^\pm(z)=\sum_{\mp m>0}H_i(m)(z\zeta_i\inv)^{-m}-
  \sum_{\mp m>0}H_i(m)(z\zeta_i)^{-m}.
\end{align}
We deduce from \eqref{zeta-Psi-def2} that
\begin{align*}
  &z\pd z \Psi_i^\pm(z)=\pm \wt H_i^+(z)\Psi_i^\pm(z)\pm \Psi_i^\pm(z)\wt H_i^-(z).
\end{align*}
Note that
\begin{align*}
  \iota_{z_1,z_2}\frac{1+z_2/z_1}{2-2z_2/z_1}=\frac 12+\sum_{m>0}\frac{z_2^m}{z_1^m},\quad
  \iota_{z_2,z_1}\frac{1+z_2/z_1}{2-2z_2/z_1}=-\frac 12-\sum_{m<0}\frac{z_2^m}{z_1^m}.
\end{align*}
From \eqref{zeta2}, we get that
\begin{align}
  [\wt H_i^-(z_1),&\Psi_i^\pm(z_2)]
  =
  \pm\Psi_i^\pm(z_2)\sum_{s\in\Z_\wp}\vvp{(q_i^2-q_i^{-2})
  (q^{-2r\ell}-1) q^{-s}}\frac{1+\zeta^sz_2/z_1}{2-2\zeta^sz_2/z_1}.
  \label{eq:wtH-Psi-rel+}
\end{align}
By using Lemma \ref{lem:phi-neg-one}, we get that
\begin{align*}
  &(\wt H_i^+(z)+\wt H_i^-(z))_{-1}^\phi\Psi_i^\pm(z)
  =\wt H_i^+(z)\Psi_i^\pm(z)+\Psi_i^\pm(z)\wt H_i^-(z)
  +c\Psi_i^\pm(z),
\end{align*}
where
\begin{align*}
  c=&\pm \Res_{z_0}z_0\inv \sum_{s\in\Z_\wp}\vvp{(q_i^2-q_i^{-2})(q^{-2r\ell}-1) q^{-s}}\left( \pd{z_0}\vartheta_s(z_0)+\delta_{s,0}\frac{1}{z_0} \right)\\
  =&\pm \Res_{z_0}z_0\inv \sum_{s\in\Z_\wp}\vvp{(q_i^2-q_i^{-2})(q^{-2r\ell}-1) q^{-s}} \pd{z_0}\vartheta_s(z_0)\\
  =&\pm \sum_{s\in\Z_\wp}\vvp{(q_i^2-q_i^{-2})(q^{-2r\ell}-1) q^{-s}} \lim_{z_0\to 0}\pd{z_0}\vartheta_s(z_0).
\end{align*}
A straightforward verification yields
\begin{align}
  &\tau_{ii,m-r_i,m}^{0,1}(0)-\tau_{ii,m+r_i,m}^{0,1}(0)
  =\sum_{s\in\Z_\wp}\vvp{
    (q_i^2-q_i^{-2})(q^{-2r\ell}-1)q^{-s}
  }\lim_{z\to 0}\pd z\vartheta_s(z).\label{eq:zeta10-to-tau-9-temp}
\end{align}
This completes the proof of the lemma.
\end{proof}

The following is the main result of this subsection.
\begin{thm}\label{thm:R-to-O}
For each $(W,H_i(z),\Psi_i^\pm(z),X_i^\pm(z))\in\obj\mathcal R_\zeta^\ell(\wh\g)$,
there is a $(\Z_\wp,\chi_\phi)$-equivariant $\phi$-coordinated quasi $V_{\wp,\tau}^\ell(\g)$-module structure $Y_W^\phi$ on $W$ uniquely determined by ($i\in I,\,m\in\Z_\wp$)
\begin{align*}
  &Y_W^\phi(\xi_{i,m}^0,z)=H_i(\zeta^m z),\quad
  Y_W^\phi(\xi_{i,m}^{1^\pm},z)=\Psi_i^\pm(\zeta^m z),\quad
  Y_W^\phi(\xi_{i,m}^{2^\pm},z)=X_i^\pm(\zeta^m z).
\end{align*}
%
\end{thm}

\begin{proof}
From Proposition \ref{prop:V-U-W-hom}, we get a nonlocal vertex algebra homomorphism $\varphi_W:V_{\wp,\tau}^\ell(\g)\to \<U_W\>_\phi$.
Note that $W$ is a $(\Z_\wp,\chi_\phi)$-equivariant $\phi$-coordinated quasi $\<U_W\>_\phi$-module.
To complete the proof, it suffices to show $\varphi_W$
commutes with the $\Z_\wp$-module action.

Let $i\in I$ and $m,s\in\Z_\wp$. Then we have
\begin{align*}
  &R(s)\varphi_W(\xi_{i,m}^0)=R(s)H_i(\zeta^m z)=H_i(\zeta^{m+s}z)=\varphi_W(\xi_{i,m+s}^0)=\varphi_W(R(s)\xi_{i,m}^0),\\
  &R(s)\varphi_W(\xi_{i,m}^{1^\pm})=R(s)\Psi_i^\pm(\zeta^m z)=\Psi_i^\pm(\zeta^{m+s} z)=\varphi_W(\xi_{i,m+s}^{1^\pm})=\varphi_W(R(s)\xi_{i,m}^{1^\pm}),\\
  &R(s)\varphi_W(\xi_{i,m}^{2^\pm})=R(s)X_i^\pm(\zeta^m z)=X_i^\pm(\zeta^{m+s} z)=\varphi_W(\xi_{i,m+s}^{2^\pm})=\varphi_W(R(s)\xi_{i,m}^{2^\pm}).
\end{align*}
Since $V_{\wp,\tau}^\ell(\g)$ generated by
\begin{align*}
  \set{\xi_{i,m}^a}{i\in I,\,m\in\Z_\wp,\,a\in\{0,1^\pm,2^\pm\}},
\end{align*}
this completes the proof of the theorem.
\end{proof}

The $(\Z_\wp,\chi_\phi)$-equivariant $\phi$-coordinated quasi $V_{\wp,\tau}^\ell(\g)$-module given by Theorem \ref{thm:R-to-O}
satisfies more conditions.

\begin{de}
Let $\mathcal O^\phi(V_{\wp,\tau}^\ell(\g))$ be the full subcategory of the category of $(\Z_\wp,\chi_\phi)$-equivariant $\phi$-coordinated quasi-modules $(W,Y_W^\phi)$ for $V_{\wp,\tau}^\ell(\g)$, such that
\begin{align}
  \label{O1}\tag{$\mathcal O1$}&Y_W^\phi(\xi_{i,0}^0,0)\,\,\te{acts semisimple on }W,\\
  \label{O2}\tag{$\mathcal O2$}&\exp\left(\sum_{m>0}Y_W^\phi(\xi_{i,-r_i}^0-\xi_{i,r_i}^0,m)
    \frac{z^{-m}}{m}\right)\in \Hom(W,W[z\inv]),
\end{align}
where
\begin{align*}
  Y_W^\phi(a,z)=\sum_{m\in\Z}Y_W^\phi(a,m)z^{-m}\quad\te{for }a\in V_{\wp,\tau}^\ell(\g).
\end{align*}
\end{de}

\begin{coro}\label{coro:R-to-O}
Theorem \ref{thm:R-to-O} defines a fully faithful functor $\mathcal Q:\mathcal R_\zeta^\ell(\wh\g)\to \mathcal O^\phi(V_{\wp,\tau}^\ell(\g))$.
\end{coro}

\begin{proof}
Let $(W,H_i(z),\Psi_i^\pm(z),X_i^\pm(z))\in \obj\mathcal R_\zeta^\ell(\wh\g)$, and let $Y_W^\phi$ be the $(\Z_\wp,\chi_\phi)$-equivariant $\phi$-coordinated quasi $V_{\wp,\tau}^\ell(\g)$-module structure given by Theorem \ref{thm:R-to-O}.
The condition \eqref{O1} follows from \eqref{zeta-weight-mod},
and the condition \eqref{O2} follows from \eqref{zeta-Psi-def2},
the condition $\Psi_i^\pm(z)\in \E(W)$ and the fact that
\begin{align*}
  Y_W^\phi(\xi_{i,-r_i}^0-\xi_{i,r_i}^0,m)=(\zeta_i^m-\zeta_i^{-m})
  Y_W^\phi(\xi_{i,0}^0,m)=(\zeta_i^m-\zeta_i^{-m})H_i(m),
\end{align*}
where $m\in\Z$.
\end{proof}

\subsection{From $\phi$-coordinated quasi-modules to $\mathcal R_\zeta^\ell(\wh\g)$}\label{subsec:phi-mod-to-rtu-mod}

We now turn to the converse of Theorem \ref{thm:R-to-O}.
Let $(W, Y_W^\phi)$ be a fixed $(\Z_\wp,\chi_\phi)$-equivariant $\phi$-coordinated quasi-module of $V_{\wp,\tau}^\ell(\g)$.
To apply Theorem \ref{thm:presentation}, we need to determine the conditions under which the relations \eqref{zeta1}-\eqref{zeta-res-form} hold with
\begin{align*}
  H_i(z)=Y_W^\phi(\xi_{i,0}^0,z),\quad
  \Psi_i^\pm(z)=Y_W^\phi(\xi_{i,0}^{1^\pm},z),\quad
  X_i^\pm(z)=Y_W^\phi(\xi_{i,0}^{2^\pm},z)\quad\te{for } i\in I.
\end{align*}
From Proposition \ref{prop:phi-mod-rel-f}, Remark \ref{rem:def-Y-V}, and Lemma \ref{lem:tau-alt}, it follows immediately that the relations \eqref{zeta1}–\eqref{zeta8}, \eqref{zeta-serre}, and \eqref{zeta-res-form} hold.
Moreover, condition \eqref{zeta-weight-mod} follows directly from condition \eqref{O1}.
Note that relation \eqref{zeta-Psi-def2} is equivalent to the following condition:
\begin{align}
  &\mho_{i,W}^{\phi,\pm}(z)=E_{i,W}^{\phi,+}(z)^{\mp 1}
   \zeta_i^{\pm 2 Y_W^\phi(\xi_{i,0}^0,0) }
   Y_W^\phi(\xi_{i,0}^{1^\pm},z)
   E_{i,W}^{\phi,-}(z)^{\mp 1}=1_W\quad\te{for }i\in I,
   \label{eq:zeta-Psi-def3}\tag{$\mathcal O3$}\\
  &\quad\te{where}\quad
  E_{i,W}^{\phi,\pm}(z)=\exp\left(\sum_{\mp m>0}
   Y_W^\phi(\xi_{i,r_i}^0-\xi_{i,-r_i}^0,m)\frac{z^{-m}}{ m}   \right).\nonumber
\end{align}
We immediately obtain the following converse of Theorem \ref{thm:R-to-O}.
\begin{thm}\label{thm:O-to-R}
For each $(W,Y_W^\phi)\in \obj \mathcal O^\phi(V_{\wp,\tau}^\ell(\g))$ satisfying the condition \eqref{eq:zeta-Psi-def3}, we have that
\begin{align*}
  (W,Y_W^\phi(\xi_{i,0}^0,z),Y_W^\phi(\xi_{i,0}^{1^\pm},z),
  Y_W^\phi(\xi_{i,0}^{2^\pm},z))\in \obj\mathcal R_\zeta^\ell(\wh\g).
\end{align*}
\end{thm}

The following result shows that the condition \eqref{eq:zeta-Psi-def3} is natural and not overly restrictive.
\begin{prop}\label{prop:mho}
Let $(W,Y_W^\phi)$ be a $(\Z_\wp,\chi_\phi)$-equivariant $\phi$-coordinated quasi module for $V_{\wp,\tau}^\ell(\g)$
satisfying the condition \eqref{O2}.
Then we have that
\begin{align}\label{eq:mho}
  &\mho_{i,W}^{\phi,\pm}(z)\in \End_\C(W)\quad\te{for }i\in I.
\end{align}
For brevity, we denote $\mho_{i,W}^{\phi,\pm}(z)$ simply by $\mho_{i,W}^{\phi,\pm}$.
Moreover, $\mho_{i,W}^{\phi,\pm}$ ($i\in I$) lie in the centroid of $Y_W^\phi(V_{\wp,\tau}^\ell(\g),z)$, that is,
\begin{align}\label{eq:centroid}
  [\mho_{i,W}^{\phi,\pm},Y_W^\phi(a,z)]=0\quad\te{for }i\in I,\,a\in V_{\wp,\tau}^\ell(\g).
\end{align}
Furthermore,
\begin{align}\label{eq:mho-invertible}
  \mho_{i,W}^{\phi,+}\mho_{i,W}^{\phi,-}=1_W.
\end{align}
\end{prop}

Before proving this proposition, we present some remarks.

\begin{rem}
From the definition of $\mho_{i,W}^{\phi,\pm}$, we conclude from \eqref{eq:centroid} that
\begin{align*}
  \big[\mho_{i,W}^{\phi,\epsilon_1},\mho_{j,W}^{\phi,\epsilon_2}\big]=0
  \quad\te{for all }i,j\in I,\,\epsilon_1,\epsilon_2\in\{\pm\}.
\end{align*}
\end{rem}

\begin{rem}
Let $(W,Y_W^\phi)$ be a simple $(\Z_\wp,\chi_\phi)$-equivariant $\phi$-coordinated quasi module for $V_{\wp,\tau}^\ell(\g)$ satisfying the condition \eqref{O2}.
It follows from Remark \ref{rem:countible-dim} that $V_{\wp,\tau}^\ell(\g)$ has countable dimension.
Since $W$ is simple, $W$ also has countable dimension.
Consequently, each $\mho_{i,W}^{\phi,\pm}$ must act as a nonzero scalar multiple on $W$.
\end{rem}

\begin{rem}
Let $(W,Y_W^\phi)$ be a $(\Z_\wp,\chi_\phi)$-equivariant $\phi$-coordinated quasi $V_{\wp,\tau}^\ell(\g)$-module of finite length satisfying the condition \eqref{O2}.
Then there exist polynomials $P_i(t)\in \C[t]$ with nonzero constant terms, such that $P_i(\mho_{i,W}^{\phi,\pm})=0$ on $W$ for all $i\in I$.
\end{rem}

In the remainder of this subsection, we prove the results of Proposition \ref{prop:mho} sequentially.
To prove \eqref{eq:mho}, we need the following two technical results.

\begin{lem}\label{lem:E-rel}
Suppose that the condition \eqref{O2} holds for $W$.
Then
\begin{align}
  &[Y_W^\phi(\xi_{i,0}^0,z_1),E_{j,W}^{\phi,+}(z_2)]
  =E_{j,W}^{\phi,+}(z_2)\sum_{s\in\Z_\wp}
    \vvp{ [a_{ij}]_{q_i}(q^{- r\ell}-q^{r\ell})q^{-r\ell-s}}
    \frac{\zeta^sz_2/z_1}{1-\zeta^sz_2/z_1},\label{eq:E-rel-1}\\
  &[Y_W^\phi(\xi_{i,0}^0,z_1),E_{j,W}^{\phi,-}(z_2)]
  =E_{j,W}^{\phi,-}(z_2)\sum_{s\in\Z_\wp}
    \vvp{ [a_{ij}]_{q_i}(q^{ -r\ell}-q^{r\ell})q^{r\ell-s}}
    \frac{1}{\zeta^sz_2/z_1-1},\label{eq:E-rel-2}\\
  &E_{i,W}^{\phi,-}(z_1)E_{j,W}^{\phi,+}(z_2)=E_{j,W}^{\phi,+}(z_2)
  E_{i,W}^{\phi,-}(z_1)
  \prod_{s\in\Z_\wp}(1-\zeta^sz_2/z_1)^{\vvp{
    (q_i^{a_{ij}}-q_i^{-a_{ij}})(q^{-r\ell}-q^{r\ell})q^{-r\ell-s}
  }}
  ,\label{eq:E-rel-E}\\
  &E_{i,W}^{\phi,-}(z_1)Y_W^\phi(\xi_{j,0}^{1^\pm},z_2)
  =Y_W^\phi(\xi_{j,0}^{1^\pm},z_2)E_{i,W}^{\phi,-}(z_1)
  \prod_{s\in\Z_\wp}(1-\zeta^sz_2/z_1)^{\pm \vvp{
    (q_i^{a_{ij}}-q_i^{-a_{ij}})(q^{-2r\ell}-1)q^{-s}
  }},\label{eq:E-rel-3}\\
  &E_{i,W}^{\phi,+}(z_1)Y_W^\phi(\xi_{j,0}^{1^\pm},z_2)
  =Y_W^\phi(\xi_{j,0}^{1^\pm},z_2)E_{i,W}^{\phi,+}(z_1)
  \prod_{s\in\Z_\wp}(1-\zeta^s z_1/z_2)^{\mp\vvp{
    (q_i^{a_{ij}}-q_i^{-a_{ij}})(q^{-2r\ell}-1)q^{-s}
  }},\label{eq:E-rel-4}\\
  &E_{i,W}^{\phi,-}(z_1)Y_W^\phi(\xi_{j,0}^{2^\pm},z_2)
  =Y_W^\phi(\xi_{j,0}^{2^\pm},z_2)E_{i,W}^{\phi,-}(z_1)
  \prod_{s\in\Z_\wp}(1-\zeta^sz_2/z_1)^{\pm\vvp{
    (q_i^{a_{ij}}-q_i^{-a_{ij}})q^{-r\ell-s}
  }},\label{eq:E-rel-5}\\
  &E_{i,W}^{\phi,+}(z_1)Y_W^\phi(\xi_{j,0}^{2^\pm},z_2)
  =Y_W^\phi(\xi_{j,0}^{2^\pm},z_2)E_{i,W}^{\phi,+}(z_1)
  \prod_{s\in\Z_\wp}(1-\zeta^s z_1/z_2)^{\pm\vvp{
    (q_i^{a_{ij}}-q_i^{-a_{ij}})q^{-r\ell-s}
  }}.\label{eq:E-rel-6}
\end{align}
\end{lem}

\begin{proof}
Similar to the definition of $\wt H_i^\pm(z)$ (see \eqref{eq:def-wtH}), we set
\begin{align*}
  Y_W^\phi(\xi_{i,-r_i}^0-\xi_{i,r_i}^0,z)^\pm
  =\sum_{\mp m>0}Y_W^\phi(\xi_{i,-r_i}^0-\xi_{i,r_i}^0,m)z^{-m}.
\end{align*}
Note that
\begin{align}
  &Y_W^\phi(\xi_{i,-r_i}^0- \xi_{i,r_i}^0,z)
  =Y_W^\phi(\xi_{i,-r_i}^0-\xi_{i,r_i}^0,z)^+
  +Y_W^\phi(\xi_{i,-r_i}^0-\xi_{i,r_i}^0,z)^-,\label{eq:def-wtH-1}\\
  &z\pd z \sum_{\mp m>0}Y_W^\phi(\xi_{i,r_i}^0-\xi_{i,-r_i}^0,m)\frac{z^{-m}}{m}
  =Y_W^\phi(\xi_{i,-r_i}^0-\xi_{i,r_i}^0,z)^\pm.\label{eq:def-wtH-1}
\end{align}
From \eqref{zeta1}, we have that
\begin{align*}
  &[Y_W^\phi(\xi_{i,0}^0,z_1),
    Y_W^\phi(\xi_{j,-r_j}^0-\xi_{j,r_j}^0,z_2)]\\
  =&[Y_W^\phi(\xi_{i,0}^0,z_1),
    Y_W^\phi(\xi_{j,0}^0,\zeta_j\inv z_2)
    -Y_W^\phi(\xi_{j,0}^0,\zeta_j z_2)]\\
  =&\sum_{s\in\Z_\wp}\Big(
    \vvp{ [a_{ij}]_{q_i}(q^{ -r\ell}-q^{r\ell})q^{-r\ell-s}}
    \iota_{z_1,z_2}
    -\vvp{ [a_{ij}]_{q_i}(q^{-r\ell}-q^{r\ell})q^{r\ell-s}}
    \iota_{z_2,z_1}
  \Big)\frac{\zeta^sz_2/z_1}{(1-\zeta^sz_2/z_1)^2},\\
  &[Y_W^\phi(\xi_{i,-r_i}^0-\xi_{i,r_i}^0,z_1),
  Y_W^\phi(\xi_{j,-r_j}^0-\xi_{j,r_j}^0,z_2)]\\
  =&[Y_W^\phi(\xi_{i,0}^0,\zeta_i\inv z_1)-Y_W^\phi(\xi_{i,0}^0,\zeta_i z_1),
  Y_W^\phi(\xi_{j,0}^0,\zeta_j\inv z_2)
    -Y_W^\phi(\xi_{j,0}^0,\zeta_j z_2)]\\
  =&\sum_{s\in\Z_\wp}\Big(
    \vvp{ (q_i^{a_{ij}}-q_i^{-a_{ij}})(q^{ -r\ell}-q^{r\ell})q^{-r\ell-s}}
    \iota_{z_1,z_2}\\
    &\quad
    -\vvp{ (q_i^{a_{ij}}-q_i^{-a_{ij}})(q^{-r\ell}-q^{r\ell})q^{r\ell-s}}
    \iota_{z_2,z_1}
  \Big)\frac{\zeta^sz_2/z_1}{(1-\zeta^sz_2/z_1)^2}
\end{align*}
Then
\begin{align*}
  &[Y_W^\phi(\xi_{i,0}^0,z_1),Y_W^\phi(\xi_{j,-r_j}^0-\xi_{j,r_j}^0,z_2)^+]
  =\sum_{s\in\Z_\wp}
    \vvp{ [a_{ij}]_{q_i}(q^{ -r\ell}-q^{r\ell})q^{-r\ell-s}}
    \frac{\zeta^sz_2/z_1}{(1-\zeta^sz_2/z_1)^2},\\
  &[Y_W^\phi(\xi_{i,0}^0,z_1),Y_W^\phi(\xi_{j,-r_j}^0-\xi_{j,r_j}^0,z_2)^-]
  =-\sum_{s\in\Z_\wp}
    \vvp{ [a_{ij}]_{q_i}(q^{-r\ell}-q^{r\ell})q^{r\ell-s}}
    \iota_{z_2,z_1}
  \frac{\zeta^sz_2/z_1}{(\zeta^sz_2/z_1-1)^2},\\
  &[Y_W^\phi(\xi_{i,-r_i}^0-\xi_{i,r_i}^0,z_1)^-,
  Y_W^\phi(\xi_{j,-r_j}^0-\xi_{j,r_j}^0,z_2)^+]\\
  &\qquad=\sum_{s\in\Z_\wp}
    \vvp{ (q_i^{a_{ij}}-q_i^{-a_{ij}})(q^{ -r\ell}-q^{r\ell})q^{-r\ell-s}}
    \frac{\zeta^sz_2/z_1}{(1-\zeta^sz_2/z_1)^2}.
\end{align*}
From \eqref{eq:def-wtH-1}, we complete the proof of \eqref{eq:E-rel-1}, \eqref{eq:E-rel-2} and \eqref{eq:E-rel-E}.

From \eqref{zeta2}, we have that
\begin{align*}
  &[Y_W^\phi(\xi_{i,-r_i}^0-\xi_{i,r_i}^0,z_1),
    Y_W^\phi(\xi_{j,0}^{1^\pm},z_2)]
  =[Y_W^\phi(\xi_{i,0}^0,\zeta_i\inv z_1)
    -Y_W^\phi(\xi_{i,0}^0,\zeta_i z_1),
    Y_W^\phi(\xi_{j,0}^{1^\pm},z_2)]\\
  =&\pm Y_W^\phi(\xi_{j,0}^{1^\pm},z_2)
  \sum_{s\in\Z_\wp}
    \vvp{(q_i^{a_{ij}}-q_i^{-a_{ij}})(q^{-2r\ell}-1)q^{-s}}
    \frac{1+\zeta^sz_2/z_1}{2-2\zeta^sz_2/z_1}\\
  &\pm Y_W^\phi(\xi_{j,0}^{1^\pm},z_2)
  \sum_{s\in\Z_\wp}
  \vvp{(q_i^{a_{ij}}-q_i^{-a_{ij}})(1-q^{2r\ell})q^{-s}}
  \frac{1+\zeta^sz_2/z_1}{2\zeta^sz_2/z_1-2}.
\end{align*}
Then
\begin{align}
  &[Y_W^\phi(\xi_{i,-r_i}^0-\xi_{i,r_i}^0,z_1)^-,
  Y_W^\phi(\xi_{j,0}^{1^\pm},z_2)]\label{eq:E-rel-temp}\\
  =&\pm Y_W^\phi(\xi_{j,0}^{1^\pm},z_2)
  \sum_{s\in\Z_\wp}
  \vvp{(q_i^{a_{ij}}-q_i^{-a_{ij}})(q^{-2r\ell}-1)q^{-s}}
    \frac{\zeta^sz_2/z_1}{1-\zeta^sz_2/z_1}\nonumber\\
  &[Y_W^\phi(\xi_{i,-r_i}^0-\xi_{i,r_i}^0,z_1)^+,
    Y_W^\phi(\xi_{j,0}^{1^\pm},z_2)]\nonumber\\
  =&\pm Y_W^\phi(\xi_{j,0}^{1^\pm},z_2)
  \sum_{s\in\Z_\wp}
  \vvp{(q_i^{a_{ij}}-q_i^{-a_{ij}})(1-q^{2r\ell})q^{-s}}
  \frac{1}{\zeta^sz_2/z_1-1}.\nonumber
\end{align}
From \eqref{eq:def-wtH-1}, we complete the proof of \eqref{eq:E-rel-3} and \eqref{eq:E-rel-4}.

From \eqref{zeta5}, we have that
\begin{align*}
  &[Y_W^\phi(\xi_{i,-r_i}^0-\xi_{i,r_i}^0,z_1),
  Y_W^\phi(\xi_{j,0}^{2^\pm},z_2)]
  =[Y_W^\phi(\xi_{i,0}^0,\zeta_i\inv z_1)-Y_W^\phi(\xi_{i,0}^0,\zeta_i z_1),Y_W^\phi(\xi_{j,0}^{2^\pm},z_2)]\\
  =&\pm Y_W^\phi(\xi_{j,0}^{2^\pm},z_2)
  \sum_{s\in\Z_\wp}
    \vvp{ (q_i^{a_{ij}}-q_i^{-a_{ij}})q^{-r\ell-s} }
  \frac{1+\zeta^sz_2/z_1}{2-2\zeta^sz_2/z_1}\\
  &\pm Y_W^\phi(\xi_{j,0}^{2^\pm},z_2)
  \sum_{s\in\Z_\wp}
    \vvp{ (q_i^{a_{ij}}-q_i^{-a_{ij}})q^{r\ell-s} }
    \frac{1+\zeta^sz_2/z_1}{2\zeta^sz_2/z_1-2}.
\end{align*}
Then
\begin{align*}
  &[Y_W^\phi(\xi_{i,-r_i}^0-\xi_{i,r_i}^0,z_1)^-,
  Y_W^\phi(\xi_{j,0}^{2^\pm},z_2)]\\
  =&\pm Y_W^\phi(\xi_{j,0}^{2^\pm},z_2)
  \sum_{s\in\Z_\wp}
    \vvp{ (q_i^{a_{ij}}-q_i^{-a_{ij}})q^{-r\ell-s} }
  \frac{\zeta^sz_2/z_1}{1-\zeta^sz_2/z_1},\\
  &[Y_W^\phi(\xi_{i,-r_i}^0-\xi_{i,r_i}^0,z_1)^+,
  Y_W^\phi(\xi_{j,0}^{2^\pm},z_2)]\\
  =&\pm Y_W^\phi(\xi_{j,0}^{2^\pm},z_2)
  \sum_{s\in\Z_\wp}
  \vvp{ (q_i^{a_{ij}}-q_i^{-a_{ij}})q^{r\ell-s} }
    \frac{1}{\zeta^sz_2/z_1-1}.
\end{align*}
From \eqref{eq:def-wtH-1}, we complete the proof of \eqref{eq:E-rel-5} and \eqref{eq:E-rel-6}.
\end{proof}

\begin{lem}\label{lem:phi-exp}
Fix an associate $\phi(x,z)=\exp(zp(x)\pd{x})x\ne x$ of $F_a(x,y)$.
Let $W$ be a vector space, let
$U$ be a quasi-compatible subset of $\E(W)$,
and let $\al^\pm(z),\beta(z)\in U$ such that
\begin{align*}
  &p(z)\inv\al^+(z)=\sum_{n<0}\al^+(n)z^{-n-1}\in \Hom(W,W[[z]]),\\
  &p(z)\inv \al^-(z)=\sum_{n\in\Z}\al^-(n)z^{-n-1}\in \Hom(W,W[z,z\inv]).
\end{align*}
Set $\al(z)=\al^-(z)+\al^+(z)$.
Suppose that
\begin{align*}
  &[\al^\pm(z_1),\al^\pm(z_2)]=0,\quad \exp\left(\mp\sum_{0\ne n\in\Z}\frac{\al^-(n)}{n}z^{-n}\right)z^{\pm\al^-(0)}\in \Hom(W,W[z,z\inv]),\\
  &[\al^-(z_1),\beta(z_2)]=\beta(z_2)\iota_{z_1,z_2}f(z_1,z_2)\quad\te{for some }f(z_1,z_2)\in \C(z_1,z_2),\\
  &\te{and}\quad \al(z)_{-1}^\phi\beta(z)=p(z)\pd{z}\beta(z)+\beta(z)\Res{z_0}z_0\inv \iota_{z,z_0}f(\phi(z,z_0),z).
\end{align*}
Then
\begin{align*}
  &\exp\left( \sum_{n<0}\frac{\al^+(n)}{n}z^{-n} \right)\beta(z)\exp\left( \sum_{0\ne n\in\Z}\frac{\al^-(n)}{n}z^{-n} \right)z^{-\al^-(0)}
  \in\End(W).
\end{align*}
\end{lem}

\begin{proof}
From Lemma \ref{lem:phi-neg-one}, we have that
\begin{align*}
  \al(z)_{-1}^\phi \beta(z)=\al^+(z)\beta(z)+\beta(z)\al^-(z)+\beta(z)\Res_{z_0}z_0\inv \iota_{z,z_0}f(\phi(z,z_0),z).
\end{align*}
It follows that
\begin{align*}
  &p(z)\pd z\beta(z)=\al^+(z)\beta(z)+\beta(z)\al^-(z).
\end{align*}
Then
\begin{align*}
  &p(z)\pd z \left(\exp\left( \sum_{n<0}\frac{\al^+(n)}{n}z^{-n} \right)\beta(z)\exp\left( \sum_{0\ne n\in\Z}\frac{\al^-(n)}{n}z^{-n} \right)z^{-\al^-(0)}\right)\\
  =&-\exp\left( \sum_{n<0}\frac{\al^+(n)}{n}z^{-n} \right)\al^+(z)\beta(z)\exp\left( \sum_{n>0}\frac{\al^-(n)}{n}z^{-n} \right)z^{-\al^-(0)}\\
  &+\exp\left( \sum_{n<0}\frac{\al^+(n)}{n}z^{-n} \right)(\al^+(z)\beta(z)+\beta(z)\al^-(z))\exp\left( \sum_{n>0}\frac{\al^-(n)}{n}z^{-n} \right)z^{-\al^-(0)}\\
  &-\exp\left( \sum_{n<0}\frac{\al^+(n)}{n}z^{-n} \right)\beta(z)\al^-(z)\exp\left( \sum_{n>0}\frac{\al^-(n)}{n}z^{-n} \right)z^{-\al^-(0)}\\
  =&0.
\end{align*}
Since $p(x)\ne 0$, we complete the proof of lemma.
\end{proof}

\begin{proof}[{\bf Proof of \eqref{eq:mho}}]
From \eqref{eq:va-rel-der}, we have that
\begin{align*}
  &Y_W^\phi(\pm\xi_{i,-r_i}^0\mp \xi_{i,r_i}^0,z)_{-1}^\phi Y_W^\phi(\xi_{i,0}^{1^\pm},z)
  =z\pd z Y_W^\phi(\xi_{i,0}^{1^\pm},z)
  +(\tau_{ii,-r_i,0}^{0,1}(0)-\tau_{ii,r_i,0}^{0,1}(0))Y_W^\phi(\xi_{i,0}^{1^\pm},z).
\end{align*}
From \eqref{eq:E-rel-temp}, we get that
\begin{align*}
  &[Y_W^\phi(\xi_{i,\mp r_i}^0-\xi_{i,\pm r_i}^0,z_1)^-,Y_W^\phi(\xi_{i,0}^{1^\pm},z_2)]\\
  =&Y_W^\phi(\xi_{i,0}^{1^\pm},z_2)
  \sum_{s\in\Z_\wp}\vvp{
  (q_i^2-q_i^{-2})(q^{-2r\ell}-1)q^{-s}}
  \frac{\zeta^sz_2/z_1}{1-\zeta^sz_2/z_1}.
\end{align*}
Note that
\begin{align*}
   &\Res_zz\inv\sum_{s\in\Z_\wp}
   \vvp{(q_i^2-q_i^{-2})(q^{-2r\ell}-1)q^{-s}}
   \frac{\zeta^se^{-z}}{1-\zeta^se^{-z}}\\
   =&\Res_zz\inv\sum_{s\in\Z_\wp}
   \vvp{(q_i^2-q_i^{-2})(q^{-2r\ell}-1)q^{-s}}
   \frac{1+\zeta^se^{-z}}{2-2\zeta^se^{-z}}\\
   &-\sum_{s\in\Z_\wp}
   \vvp{(q_i^2-q_i^{-2})(q^{-2r\ell}-1)q^{-s}}\\
   =&\Res_zz\inv\sum_{s\in\Z_\wp}
   \vvp{(q_i^2-q_i^{-2})(q^{-2r\ell}-1)q^{-s}}
   \frac{1+\zeta^se^{-z}}{2-2\zeta^se^{-z}}\\
   =&\sum_{s\in\Z_\wp}
   \vvp{(q_i^2-q_i^{-2})(q^{-2r\ell}-1)q^{-s}}
   \lim_{z\to 0}\pd z\vartheta_s(z)\\
   =&\tau_{ii,-r_i,0}^{0,1}(0)-\tau_{ii,r_i,0}^{0,1}(0),
\end{align*}
where the last equation follows from \eqref{eq:zeta10-to-tau-9-temp}.
By using Lemma \ref{lem:phi-exp}, we complete the proof.
\end{proof}

\begin{proof}[{\bf Proof of \eqref{eq:centroid}}]
Since $V_{\wp,\tau}^\ell(\g)$ is generated by
\begin{align*}
  \set{\xi_{i,m}^a}{i\in I,m\in\Z_\wp,a\in\{0,1^\pm,2^\pm\}},
\end{align*}
it sufficient to show that
\begin{align*}
  [Y_W^\phi(\zeta_{i,0}^a,z),\mho_{j,W}^{\phi,\pm}]=0\quad\te{for }i,j\in I,\,a\in\{0,1^\pm,2^\pm\}.
\end{align*}
From \eqref{zeta2},
\eqref{eq:E-rel-1} and \eqref{eq:E-rel-2}, we have that
\begin{align*}
  &[Y_W^\phi(\xi_{i,0}^0,z_1),\mho_{j,W}^{\phi,\pm}]\\
  =&[Y_W^\phi(\xi_{i,0}^0,z_1),E_{j,W}^{\phi,+}(z_2)^{\mp 1}
  \zeta_j^{\pm 2Y_W^\phi(\zeta_{i,0}^0,0)}
  Y_W^\phi(\xi_{j,0}^{1^\pm},z_2)E_{j,W}^{\phi,-}(z_2)^{\mp 1}]\\
  =&\mp \mho_{j,W}^{\phi,\pm}
  \sum_{s\in\Z_\wp}\vvp{ [a_{ij}]_{q_i}(q^{-2r\ell}-1)q^{-s} }
  \frac{\zeta^sz_2/z_1}{1-\zeta^sz_2/z_1}\\
  &\pm \mho_{j,W}^{\phi,\pm} \sum_{s\in\Z_\wp}
  \vvp{ [a_{ij}]_{q_i}(q^{-2r\ell}-1)q^{-s} }
  \frac{1+\zeta^sz_2/z_1}{2-2\zeta^sz_2/z_1}\\
  &\pm \mho_{j,W}^{\phi,\pm}\sum_{s\in\Z_\wp}
  \vvp{ [a_{ij}]_{q_i}(1-q^{2r\ell})q^{-s} }
  \frac{1+\zeta^sz_2/z_1}{2\zeta^sz_2/z_1-2}\\
  &\mp \mho_{j,W}^{\phi,\pm}\sum_{s\in\Z_\wp}
  \vvp{ [a_{ij}]_{q_i}(1-q^{2r\ell})q^{-s} }
  \frac{1}{\zeta^sz_2/z_1-1}\\
  =&0.
\end{align*}
From \eqref{zeta3}, \eqref{eq:E-rel-3} and \eqref{eq:E-rel-4}, we have that
\begin{align*}
  &Y_W^\phi(\xi_{i,0}^{1^{\epsilon_1}},z_1)\mho_{j,W}^{\phi,\epsilon_2}
  =Y_W^\phi(\xi_{i,0}^{1^{\epsilon_1}},z_1)
  E_{j,W}^{\phi,+}(z_2)^{-\epsilon_2}
  \zeta_j^{2\epsilon_2Y_W^\phi(\xi_{j,0}^0,0)}
  Y_W^\phi(\xi_{j,0}^{1^{\epsilon_2}},z_2)
  E_{j,W}^{\phi,-}(z_2)^{-\epsilon_2}\\
  \sim&E_{j,W}^{\phi,+}(z_2)^{-\epsilon_2}
  \zeta_j^{2\epsilon_2Y_W^\phi(\xi_{j,0}^0,0)}
  Y_W^\phi(\xi_{j,0}^{1^{\epsilon_2}},z_2)
  E_{j,W}^{\phi,-}(z_2)^{-\epsilon_2}
  Y_W^\phi(\xi_{i,0}^{1^{\epsilon_1}},z_1)\\
  &\times
  \prod_{s\in\Z_\wp}(1-\zeta^s z_2/z_1)^{-\epsilon_1\epsilon_2
  \vvp{
    (q_i^{a_{ij}}-q_i^{-a_{ij}})(q^{-2r\ell}-1)q^{-s}
  }}\\
  &\times
  \prod_{s\in\Z_\wp}
  (1-\zeta^sz_2/z_1)^{\epsilon_1\epsilon_2\vvp{
    (q_i^{a_{ij}}-q_i^{-a_{ij}})(q^{2r\ell}+q^{-2r\ell}-2)q^{-s}
  }}\\
  &\times
  \prod_{s\in\Z_\wp}
  (1-\zeta^sz_1/z_2)^{\epsilon_1\epsilon_2\vvp{
    (q_i^{a_{ij}}-q_i^{-a_{ij}})(q^{-2r\ell}-1)q^{-s}
  }}\\
  =&E_{j,W}^{\phi,+}(z_2)^{-\epsilon_2}
  \zeta_j^{2\epsilon_2Y_W^\phi(\xi_{j,0}^0,0)}
  Y_W^\phi(\xi_{j,0}^{1^{\epsilon_2}},z_2)
  E_{j,W}^{\phi,-}(z_2)^{-\epsilon_2}
  Y_W^\phi(\xi_{i,0}^{1^{\epsilon_1}},z_1)\\
  &\times
  \prod_{s\in\Z_\wp}(1-\zeta^s z_2/z_1)^{-\epsilon_1\epsilon_2
  \vvp{
    (q_i^{a_{ij}}-q_i^{-a_{ij}})(q^{-2r\ell}-1)q^{-s}
  }}\\
  &\times
  \prod_{s\in\Z_\wp}
  (1-\zeta^sz_2/z_1)^{\epsilon_1\epsilon_2\vvp{
    (q_i^{a_{ij}}-q_i^{-a_{ij}})(q^{2r\ell}+q^{-2r\ell}-2)q^{-s}
  }}\\
  &\times
  \prod_{s\in\Z_\wp}
  (1-\zeta^{-s}z_1/z_2)^{-\epsilon_1\epsilon_2\vvp{
    (q_i^{a_{ij}}-q_i^{-a_{ij}})(q^{2r\ell}-1)q^{-s}
  }}\\
  =&E_{j,W}^{\phi,+}(z_2)^{-\epsilon_2}
  \zeta_j^{2\epsilon_2Y_W^\phi(\xi_{j,0}^0,0)}
  Y_W^\phi(\xi_{j,0}^{1^{\epsilon_2}},z_2)
  E_{j,W}^{\phi,-}(z_2)^{-\epsilon_2}
  Y_W^\phi(\xi_{i,0}^{1^{\epsilon_1}},z_1)\\
  &\times
  \zeta^{\epsilon_1\epsilon_2 \sum_s s\vvp{
    (q_i^{a_{ij}}-q_i^{-a_{ij}})(q^{2r\ell}-1)q^{-s}
  }}\\
  =&\mho_{j,W}^{\phi,\epsilon_2}Y_W^\phi(\xi_{i,0}^{1^{\epsilon_1}},z_1).
\end{align*}
By using \eqref{eq:mho}, we have that
\begin{align*}
  [Y_W^\phi(\xi_{i,0}^{1^{\epsilon_1}},z),\mho_{j,W}^{\phi,\epsilon_2}]=0.
\end{align*}
From \eqref{zeta6}, \eqref{eq:E-rel-5} and \eqref{eq:E-rel-6}, we have that
\begin{align*}
  &Y_W^\phi(\xi_{i,0}^{2^{\epsilon_1}},z_1)\mho_{j,W}^{\phi,\epsilon_2}
  =Y_W^\phi(\xi_{i,0}^{2^{\epsilon_1}},z_1)
  E_{j,W}^{\phi,+}(z_2)^{-\epsilon_2}
  \zeta_j^{2\epsilon_2Y_W^\phi(\xi_{j,0}^0,0)}
  Y_W^\phi(\xi_{j,0}^{1^{\epsilon_2}},z_2)
  E_{j,W}^{\phi,-}(z_2)^{-\epsilon_2}\\
  \sim&E_{j,W}^{\phi,+}(z_2)^{-\epsilon_2}
  \zeta_j^{2\epsilon_2Y_W^\phi(\xi_{j,0}^0,0)}
  Y_W^\phi(\xi_{j,0}^{1^{\epsilon_2}},z_2)
  E_{j,W}^{\phi,-}(z_2)^{-\epsilon_2}
  Y_W^\phi(\xi_{i,0}^{2^{\epsilon_1}},z_1)\\
  &\times \prod_{s\in\Z_\wp}(1-\zeta^{-s}z_2/z_1)^{-\epsilon_1\epsilon_2\vvp{
    (q_i^{a_{ij}}-q_i^{-a_{ij}})q^{r\ell-s}
  }}\\
  &\times
  \prod_{s\in\Z_\wp}(1-\zeta^sz_1/z_2)^{\epsilon_1\epsilon_2\vvp{
    (q_i^{a_{ij}}-q_i^{-a_{ij}})(q^{r\ell}-q^{-r\ell})q^{-s}
  }}\\
  &\times
  \prod_{s\in\Z_\wp}
  (1-\zeta^sz_1/z_2)^{\epsilon_1\epsilon_2\vvp{
    (q_i^{a_{ij}}-q_i^{-a_{ij}})q^{-r\ell-s}
  }}
  \zeta_j^{-2\epsilon_1\epsilon_2a_{ji}}\\
  =&E_{j,W}^{\phi,+}(z_2)^{-\epsilon_2}
  \zeta_j^{2\epsilon_2Y_W^\phi(\xi_{j,0}^0,0)}
  Y_W^\phi(\xi_{j,0}^{1^{\epsilon_2}},z_2)
  E_{j,W}^{\phi,-}(z_2)^{-\epsilon_2}
  Y_W^\phi(\xi_{i,0}^{2^{\epsilon_1}},z_1)\\
  &\times \zeta^{\epsilon_1\epsilon_2\sum_ss\vvp{
    (q_i^{a_{ij}}-q_i^{-a_{ij}})q^{r\ell-s}
  }}
  \zeta^{-2\epsilon_1\epsilon_2r_ia_{ij}}\\
  =&\mho_{j,W}^{\phi,\epsilon_2}Y_W^\phi(\xi_{i,0}^{2^{\epsilon_1}},z_1).
\end{align*}
By using \eqref{eq:mho}, we have that
\begin{align*}
  [Y_W^\phi(\xi_{i,0}^{2^{\epsilon_1}},z),\mho_{j,W}^{\phi,\epsilon_2}]=0.
\end{align*}
We complete the proof.
\end{proof}

\begin{proof}[{\bf Proof of \eqref{eq:mho-invertible}}]
From \eqref{zeta1} and \eqref{zeta2}, we have that
\begin{align*}
  [Y_W^\phi(\xi_{i,0}^0,0),E_{i,W}^{\phi,\pm}(z)]=0=[Y_W^\phi(\xi_{i,0}^0,0),
  Y_W^\phi(\xi_{i,0}^{1^\pm},z)].
\end{align*}
Then we have that
\begin{align*}
  &\mho_i^+\mho_i^-
  =E_{i,W}^{\phi,+}(z_1)^{- 1}
   \zeta_i^{ 2 Y_W^\phi(\xi_{i,0}^0,0) }
   Y_W^\phi(\xi_{i,0}^{1^+},z_1)
   E_{i,W}^{\phi,-}(z_1)^{- 1}\\
   &\quad\times
   E_{i,W}^{\phi,+}(z_2)
   \zeta_i^{- 2 Y_W^\phi(\xi_{i,0}^0,0) }
   Y_W^\phi(\xi_{i,0}^{1^-},z_2)
   E_{i,W}^{\phi,-}(z_2)\\
   =&
   \prod_{s\in\Z_\wp}
   (1-\zeta^sz_2/z_1)^{\vvp{
    (q_i^2-q_i^{-2})(q^{-2r\ell}-1)q^{-s}
   }}
   E_{i,W}^{\phi,+}(z_1)^{- 1}E_{i,W}^{\phi,+}(z_2)\\
   &\quad\times
   Y_W^\phi(\xi_{i,0}^{1^+},z_1)
   Y_W^\phi(\xi_{i,0}^{1^-},z_2)
   E_{i,W}^{\phi,-}(z_1)^{- 1}   E_{i,W}^{\phi,-}(z_2)\\
   =&\lim_{z_1\to z_2}
   \prod_{s\in\Z_\wp}
   (1-\zeta^sz_2/z_1)^{\vvp{
    (q_i^2-q_i^{-2})(q^{-2r\ell}-1)q^{-s}
   }}
   Y_W^\phi(\xi_{i,0}^{1^+},z_1)
   Y_W^\phi(\xi_{i,0}^{1^-},z_2)\\
    =&1,
\end{align*}
where the second equation follows from \eqref{eq:E-rel-E},
\eqref{eq:E-rel-3} and \eqref{eq:E-rel-4},
the third equation follows from \eqref{eq:mho},
and the last equation follows from \eqref{zeta4}.
\end{proof}

\section{Structure of $V_{\wp,\tau}^\ell(\g)$}\label{sec:struct}

In this section, we study the structure of $V_{\wp,\tau}^\ell(\g)$. We first show that $V_{\wp,\tau}^\ell(\g)$ can be realized as a deformation of $V_{\wp,\varepsilon}^\ell(\g)$ via deforming triples.
Next, we decompose $V_{\wp,\varepsilon}^\ell(\g)$ into a Heisenberg vertex algebra and a quantum vertex algebra determined by a quiver.

\subsection{Realization of $V_{\wp,\tau}^\ell(\g)$ as deformation of $V_{\wp,\varepsilon}^\ell(\g)$}\label{subsec:V-deform-triple}

This subsection is devoted to constructing a deforming triple for $V_{\wp,\varepsilon}^\ell(\g)$ that allows us to realize
$V_{\wp,\tau}^\ell(\g)$ as a deformation of $V_{\wp,\varepsilon}^\ell(\g)$.

Recall the abelian group $\mathfrak T$ with identity $\varepsilon$ given in Section \ref{subsec:construct-V}.
Let $\tau,\bar\tau\in \mathfrak T$.
Recall from Section \ref{subsec:free-qva}, we obtain a deforming triple $(H',\rho,\bar\tau)$ on $V_{\wp,\tau}'^\ell(\g)$.
To be more precise, $H'$ is the commutative cocommutative vertex bialgebra determined by a bialgebra with derivation $\partial$.
As an associative algebra, $H'$ is the
symmetric algebra of the following vector space
\begin{align*}
  \bigoplus_{a\in\{0,1^\pm,2^\pm\}}\bigoplus_{i\in I}\bigoplus_{m\in\Z_\wp}\bigoplus_{n\in\N}\C\partial^n\wh \xi_{i,m}^a,
\end{align*}
the coproduct $\Delta$ and counit $\varepsilon$ is determined by
\begin{align*}
  &\Delta(\partial^n \wh \xi_{i,m}^0)=\partial^n \wh \xi_{i,m}^0\ot 1+1\ot \partial^n \wh \xi_{i,m}^0,\quad \varepsilon(\partial^n \wh \xi_{i,m}^0)=0,\\
  &\Delta(\partial^n \wh \xi_{i,m}^{a^\pm})=\sum_{k=0}^n\binom{n}{k} \partial^k \wh \xi_{i,m}^{a^\pm}\ot \partial^{n-k} \wh \xi_{i,m}^{a^\pm},\quad
  \varepsilon(\partial^n \wh \xi_{i,m}^{a^\pm})=\delta_{n,0}
  \quad\te{for }a=1,2,
\end{align*}
and the derivation $\partial$ is determined by
\begin{align*}
  \partial (\partial^n \wh \xi_{i,m}^a)=\partial^{n+1}\wh \xi_{i,m}^a\quad\te{for }n\in\N,\,
  i\in J^a,\,a\in\{0,1^\pm,2^\pm\}.
\end{align*}
The $H'$-comodule nonlocal vertex algebra structure $\rho$ on $V_{\wp,\tau}'^\ell(\g)$ is defined by
\begin{align*}
  \rho(\xi_{i,m}^0)=\xi_{i,m}^0\ot 1+\vac\ot \wh \xi_{i,m}^0,\quad
  \rho(\xi_{i,m}^{a^\pm})=\xi_{i,m}^{a^\pm}\ot \wh\xi_{i,m}^{a^\pm}\quad\te{for }a=1,2.
\end{align*}
The $H'$-module nonlocal vertex algebra structure $\bar\tau$ on $V_{\wp,\tau}'^\ell(\g)$ is defined by
\begin{align*}
  \bar\tau(\wh\xi_{i,m}^a,z)=\bar\tau_{i,m}^a(z)\quad \te{for }a\in\{0,1^\pm,2^\pm\},
\end{align*}
where $\bar\tau_{i,m}^0(z)$ is the pseudo-derivation on $V_{\wp,\tau}'^\ell(\g)$ defined by
\begin{align}
  &\bar\tau_{i,m}^0(z)\xi_{j,n}^0=\vac\ot \bar\tau_{ijmn}^{0,0}(z),\quad
  \bar\tau_{i,m}^0(z)\xi_{j,n}^{a^\pm}=\pm \xi_{j,n}^{a^\pm}\ot \bar\tau_{ijmn}^{0,a}(z),\label{eq:bartau-0}
\end{align}
and $\bar\tau_{i,m}^{a^\pm}(z)$ ($a=1,2$) is a pseudo-endomorphism on $V_{\wp,\tau}'^\ell(\g)$ defined by
\begin{align}
  &\bar\tau_{i,m}^{a^\pm}(z)\xi_{j,n}^0=\xi_{j,n}^0\ot 1\mp\vac\ot \bar\tau_{ijmn}^{a,0}(z),\quad
  \bar\tau_{i,m}^{a^\pm}(z)\xi_{j,n}^{b^\epsilon}=\xi_{j,n}^{b^\epsilon}\ot \bar\tau_{ijmn}^{a,b}(z)^{\mp\epsilon 1}\label{eq:bartau-1}
  \quad\te{for }b=1,2.
\end{align}
We now show that the deforming triple $(H',\rho,\bar\tau)$ of $V_{\wp,\tau}'^\ell(\g)$ induces a deforming triple of $V_{\wp,\tau}^\ell(\g)$.

It is straightforward to verify that the relations \eqref{eq:bartau-0} and \eqref{eq:bartau-1} are compatible with relations \eqref{eq:va-e+e-}, \eqref{eq:va-rel-der}, \eqref{eq:va-rel-varphi+-},
\eqref{eq:va-rel-serre} and \eqref{eq:va-rel-rtu}.
Then we immediately get the following result.

\begin{lem}
For each $i\in I$ and $m\in\Z_\wp$,
there exist a pseudo-derivation $\bar\tau_{i,m}^0(z)$ and pseudo-endomorphisms $\bar\tau_{i,m}^{1^\pm}(z)$, $\bar\tau_{i,m}^{2^\pm}(z)$ on $V_{\wp,\tau}^\ell(\g)$ determined by the relations \eqref{eq:bartau-0} and \eqref{eq:bartau-1}.
\end{lem}

\begin{de}
Let $H$ be the quotient algebra of $H'$ modulo the ideal generated by
\begin{align*}
  &\partial^n\left(\wh \xi_{i,m}^{2^+}\wh \xi_{i,m}^{2^-}-1\right),\quad
  \partial^n\left(\wh \xi_{i,m-r\ell}^{2^+}\wh \xi_{i,m+r\ell}^{2^-}-\wh\xi_{i,m}^{1^+}\right),\quad
  \partial^n\left(\wh \xi_{i,m}^{1^+}\wh\xi_{i,m}^{1^-}-1\right),\\
  &\partial^n\left(\partial\wh\xi_{i,m}^{1^\pm}
  \mp\left(\wh\xi_{i,m-r_i}^0-\wh\xi_{i,m+r_i}^0\right)\wh\xi_{i,m}^{1^\pm}\right)
  \quad\te{for }i\in I,\,m\in\Z_\wp,\,n\in\N.
\end{align*}
\end{de}


It is immediate to verify the following three results.

\begin{lem}
The coproduct $\Delta$, counit $\varepsilon$ and derivation $\partial$ on $H'$ naturally induce corresponding structures on $H$, which we continue to denote by $\Delta$, $\varepsilon$ and $\partial$, respectively.
Moreover, $(H,\Delta,\varepsilon,\partial)$ defines a commutative cocommutative vertex bialgebra structure on $H$.
\end{lem}
%
%
%

\begin{lem}
The $H'$-comodule nonlocal vertex algebra structure $\rho$ on $V_{\wp,\tau}'^\ell(\g)$ naturally induces an $H$-comodule nonlocal vertex algebra structure on $V_{\wp,\tau}^\ell(\g)$, which is still denoted by $\rho$.
\end{lem}

\begin{lem}
There is an $H$-module nonlocal vertex algebra structure $\bar\tau(\cdot,z)$ on $V_{\wp, \tau}^\ell(\g)$ uniquely defined by
\begin{align*}
  \bar\tau(\wh\xi_{i,m}^0,z)=\bar\tau_{i,m}^0(z),\quad \bar\tau(\wh\xi_{i,m}^{a^\pm},z)\tau_{i,m}^{a^\pm}(z).
\end{align*}
Moreover, $(H,\rho,\bar\tau)$ becomes a deforming triple of $V_{\wp,\tau}^\ell(\g)$.
\end{lem}

As an immediate consequence of Corollary \ref{coro:D-T-tau-Y}, we have the following result.
\begin{coro}
For $i\in I$, $m\in\Z_\wp$ and $a=1,2$,
\begin{align*}
  &\mathfrak D_{T_{\bar\tau}}(Y)(\xi_{i,m}^0,z)=Y(\xi_{i,m}^0,z)+\bar\tau_{i,m}^0(z),\quad
  \mathfrak D_{T_{\bar\tau}}(Y)(\xi_{i,m}^{a^\pm},z)=Y(\xi_{i,m}^{a^\pm},z)
    \bar\tau_{i,m}^{a^\pm}(z).
\end{align*}
\end{coro}

By using the Corollary above, one can straightforwardly verify the following result by using the fact that
\begin{align*}
  \bar\tau_{ijmn}^{st}(z)\in\C[[z]]\quad\te{for }i,j\in I,\,m,n\in\Z_\wp,\,s,t\in\{0,1^\pm,2^\pm\}.
\end{align*}

\begin{prop}
We have that
\begin{align*}
  V_{\wp,\tau\ast\bar\tau}^\ell(\g)\cong\mathfrak D_{T_{\bar\tau}}(V_{\wp,\tau}^\ell(\g))
\end{align*}
defined by
\begin{align*}
  &\xi_{i,m}^0\mapsto\xi_{i,m}^0,\quad
  \xi_{i,m}^{1^+}\mapsto
  \bar\tau_{iim,m+r\ell}^{2,1}(0)\inv
  \xi_{i,m}^{1^+},\quad
  \xi_{i,m}^{1^-}\mapsto
  \bar\tau_{ii,m+r\ell,m}^{1,2}(0)\bar\tau_{iimm}^{1,1}(0)\inv \xi_{i,m}^{1^-},\\
  &
  \xi_{i,m}^{2^+}\mapsto 
  \bar\tau_{iimm}^{2,2}(0)\inv
  \xi_{i,m}^{2^+},\quad
  \xi_{i,m}^{2^-}\mapsto \xi_{i,m}^{2^-}
\end{align*}
\end{prop}

Since $\varepsilon\ast\tau=\tau$, we get the following immediate consequence.

\begin{coro}
For $\tau\in\mathfrak T$, we have that
\begin{align*}
  V_{\wp,\tau}^\ell(\g)\cong\mathfrak D_{T_\tau}(V_{\wp,\varepsilon}^\ell(\g)).
\end{align*}
\end{coro}

Applying Proposition \ref{prop:L-H-rho-V-comosition-pre} to this, we immediately get the following corollary.
\begin{coro}
For $\tau\in\mathfrak T$, we have the quantum vertex algebra injection $V_{\wp,\tau}^\ell(\g)\to V_{\wp,\varepsilon}^\ell(\g)\sharp H$ defined by
\begin{align*}
  &\xi_{i,m}^0\mapsto\xi_{i,m}^0\ot\vac+\vac\ot\wh\xi_{i,m}^0,\quad
  \xi_{i,m}^{1^+}\mapsto
  \tau_{iim,m+r\ell}^{2,1}(0)\inv
  \xi_{i,m}^{1^+}\ot\wh\xi_{i,m}^{1^+},\\
  &\xi_{i,m}^{1^-}\mapsto
  \tau_{ii,m+r\ell,m}^{1,2}(0)\tau_{iimm}^{1,1}(0)\inv \xi_{i,m}^{1^-}\ot \wh\xi_{i,m}^{1^-},\\
  &
  \xi_{i,m}^{2^+}\mapsto 
  \tau_{iimm}^{2,2}(0)\inv
  \xi_{i,m}^{2^+}\ot \wh\xi_{i,m}^{2^+},\quad
  \xi_{i,m}^{2^-}\mapsto \xi_{i,m}^{2^-}\ot\wh\xi_{i,m}^{2^-}.
\end{align*}
\end{coro}

%

\subsection{Decomposition of $V_{\wp,\varepsilon}^\ell(\g)$}\label{subsec:structure}

In this subsection, we study the structure of $V_{\wp,\varepsilon}^\ell(\g)$.
We prove that $V_{\wp,\varepsilon}^\ell(\g)$ can be deformed from a tensor product quantum vertex algebra of a Heisenberg vertex algebra and a quantum vertex algebra
$V_\wp^\ell(Q,L)$ defined by a quiver $Q$ with a $\Z_\wp$-action, a subset $L$ of directed loops, and an integer $\ell$ (see Definition \ref{de:qva-quiver}).
Define
\begin{align*}
  h_{i,m}=\sum_{s\in\Z_\wp}\zeta^{ms}\xi_{i,s}^0,\quad
  \wt \xi_{i,m}^0=\xi_{i,m-r_i}^0-\xi_{i,m+r_i}^0
  \quad\te{for }i\in I,\,m\in\Z_\wp.
\end{align*}
The following result rewrite the relations \eqref{tau1}-\eqref{tau3}, \eqref{tau9} in terms of $h_{i,\wp_im}(z)$ and $\wt\xi_{i,m}^0(z)$.
\begin{lem}\label{lem:equiv-rels}
Let $W$ be a vector space equipped with fields $\xi_{i,m}^a(z)\in\E(W)$, $i\in I$, $m\in\Z_\wp$, $a\in\{0,1^\pm,2^\pm\}$.
Set
\begin{align*}
  h_{i,m}(z)=\sum_{s\in\Z_\wp}\zeta^{ms}\xi_{i,s}^0(z),\quad
  \wt \xi_{i,m}^0(z)=\xi_{i,m-r_i}^0(z)-\xi_{i,m+r_i}^0(z).
\end{align*}
Then the relation \eqref{tau1} holds if and only if
\begin{align}
  &\label{eq:h-h-rel}
  [h_{i,\wp_im}(z_1),h_{j,\wp_jn}(z_2)]
  =\delta_{\wp_im+\wp_jn,0}\wp [a_{ij}]_{\zeta_i^{\wp_im}}[r\ell/r_j]_{\zeta_j^{\wp_jn}}
  \zeta^{\wp_i mr\ell}\pd{z_2}z_1\inv\delta\left(\frac{z_2}{z_1}\right),\\
  &[h_{i,\wp_im}(z_1),\wt\xi_{j,n}^0(z_2)]
  =0,\nonumber\\
  &[\wt \xi_{i,m}^0(z_1),\wt \xi_{j,n}^0(z_2)]
  =\frac{\vvp{ (q_i^{a_{ij}}-q_i^{-a_{ij}})(q^{-2r\ell}-1)q^{n-m} }}{(z_1-z_2)^2}
  -\frac{\vvp{
  (q_i^{a_{ij}}-q_i^{-a_{ij}})(1-q^{2r\ell})q^{n-m}
  }}{(z_2-z_1)^2}.\nonumber
\end{align}
The relation \eqref{tau2} holds if and only if
\begin{align*}
  &[h_{i,\wp_im}(z_1),\xi_{j,n}^{1^\pm}(z_2)]=0,\\
  &[\wt\xi_{i,m}^0(z_1),\xi_{j,n}^{1^\pm}(z_2)]
  =\pm\xi_{j,n}^{1^\pm}(z_2)\frac{\vvp{
    (q_i^{a_{ij}}-q_i^{-a_{ij}})(q^{-2r\ell}-1)q^{n-m}
  }}{z_1-z_2}\\
  &\qquad\pm\xi_{j,n}^{1^\pm}(z_2)\frac{\vvp{
    (q_i^{a_{ij}}-q_i^{-a_{ij}})(1-q^{2r\ell})q^{n-m}
  }}{z_2-z_1}.
\end{align*}
The relation \eqref{tau3} holds if and only if
\begin{align*}
  &[h_{i,\wp_im}(z_1),\xi_{j,n}^{2^\pm}(z_2)]
  =\pm \xi_{j,n}^{2^\pm}(z_2)[a_{ij}]_{\zeta_i^{\wp_im}}
  \zeta^{\wp_im(n+r\ell)}
  z_1\inv\delta\left(\frac{z_2}{z_1}\right),\\
  &[\wt\xi_{i,m}^0(z_1),\xi_{j,n}^{2^\pm}(z_2)]
  =\pm \xi_{j,n}^{2^\pm}(z_2)\left(\frac{\vvp{
    (q_i^{a_{ij}}-q_i^{-a_{ij}})q^{-r\ell+n-m}
  }}{z_1-z_2}
  +\frac{\vvp{
    (q_i^{a_{ij}}-q_i^{-a_{ij}})q^{r\ell+n-m}
  }}{z_2-z_1}
  \right).
\end{align*}
Moreover, if $(W,\xi_{i,m}^a(z))\in \obj\mathcal M_{\wp,\varepsilon}^\ell(\wh\g)$, then
the relation \eqref{eq:va-rel-der} holds if and only if
\begin{align*}
  &\pm \wt \xi_{i,m}^0(-1)\xi_{i,m}^{1^\pm}=\partial \xi_{i,m}^{1^\pm}+(\tau_{ii,m-r_i,m}^{0,1}(0)
    -\tau_{ii,m+r_i,m}^{0,1}(0))\xi_{i,m}^{1^\pm}.
\end{align*}
\end{lem}

\begin{proof}
Note that
\begin{align*}
  \xi_{i,m}^0(z)=\frac{1}{\wp}\sum_{s=0}^{\wp/\wp_i-1}
  \zeta^{-\wp_ims}h_{i,s\wp_i}(z)
  +\frac{1}{\wp_i}\sum_{s=0}^{\wp_i-2}\sum_{t=0}^s
  \wt\xi_{i,m+(2t+1)r_i}^0(z).
\end{align*}
Then the lemma follows from a straightforward verification and the fact that
$\zeta^{r\wp_i}=\zeta^{-r\wp_i}$.
\end{proof}

From Remark \ref{rem:S-wp-tau}, one can straightforwardly verify the following result.
\begin{lem}\label{lem:S-h}
For $i\in I$ and $m\in\Z_\wp$, we have that
\begin{align*}
  S_{\wp,\varepsilon}(z)(h_{i,\wp_im}\ot u)=h_{i,\wp_im}\ot u\quad\te{for }u\in V_{\wp,\varepsilon}^\ell(\g).
\end{align*}
\end{lem}

Let $\mathcal H$ and $v_{\wp,\varepsilon}^\ell(\g)$ be the nonlocal vertex subalgebras of $V_{\wp,\varepsilon}^\ell(\g)$ generated by
\begin{align*}
  \set{h_{i,m}}{m\in\wp_i\Z_\wp,\,i\in I}
  \quad\te{and}\quad \set{\wt \xi_{i,m}^0,\,\xi_{i,m}^{1^\pm},\,\xi_{i,m}^{2^\pm}}{i\in I,\,m\in\Z_\wp},\,\,\te{respectively}.
\end{align*}
We will give another description of $\mathcal H$ and $v_{\wp,\varepsilon}^\ell(\g)$.
Let
\begin{align*}
  \h=\bigoplus_{i\in I}\bigoplus_{m\in \wp_i\Z_\wp}\C \al^\vee_{i,m}
\end{align*}
be a finite-dimensional abelian Lie algebra with a symmetric bilinear form
\begin{align*}
  (\al^\vee_{i,\wp_im}| \al^\vee_{j,\wp_jn})=\delta_{\wp_im+\wp_jn,0}
  \wp [a_{ij}]_{\zeta_i^{\wp_im}}[r\ell/r_j]_{\zeta_j^{\wp_jn}}
  \zeta^{\wp_imr\ell}.
\end{align*}
Form the corresponding affine Lie algebra $\wh{\h}$ as the vector space
\begin{align*}
  \wh\h=\h\ot\C[t,t\inv]\oplus\C c
\end{align*}
where $c$ is central and the commutation relations are given by
\begin{align*}
  [\al^\vee_{i,\wp_is}(m),\al^\vee_{j,\wp_jt}(n)]=m\delta_{m+n,0}
  (\al^\vee_{i,\wp_is}|\al^\vee_{j,\wp_jt})c
  \quad\te{for }i,j\in I,\,s,t\in\Z_\wp.
\end{align*}
with $\al^\vee_{i,\wp_is}(m)$ denoting $\al^\vee_{i,\wp_is} \otimes t^m$.
Define
\begin{align*}
  \wh\h^+=\bigoplus_{i\in I}\bigoplus_{s\in\wp_i\Z_\wp}\bigoplus_{m\ge 0}\C \al^\vee_{i,s}(m).
\end{align*}
For $a\in\C$, let $\C_a = \C$ be a $\wh\h^+ \oplus \C c$-module where $\wh{\h}^+$ acts trivially and $c$ acts as the scalar $a$. Then we set
\begin{align*}
  V_{\wh\h}^a=U(\wh\h)\ot_{U(\wh\h^+\oplus\C c)}\C_a.
\end{align*}
Denote the vacuum vector by $\vac = 1 \ot 1 \in V_{\wh\h}^a$, and identify $\al^\vee_{i,\wp_is}$ with $\al^\vee_{i,\wp_is}(-1)\vac$. Then $V_{\wh\h}^a$ carries a vertex algebra structure with vacuum $\vac$ and vertex operator map $Y$ uniquely determined by
\begin{align*}
  Y(\al^\vee_{i,\wp_is},z)=\al^\vee_{i,\wp_is}(z)
  =\sum_{m\in\Z}\al^\vee_{i,\wp_is}(m)z^{-m-1}\quad\te{for }i\in I,\,s\in\Z_\wp.
\end{align*}
From the commutation relations \eqref{eq:h-h-rel}, we immediately obtain the following lemma.
\begin{lem}\label{lem:Heisenberg-hom}
There is a vertex algebra epimorphism from $V_{\wh\h}^1$ to $\mathcal H$ given by $\al^\vee_{i,\wp_is}\mapsto h_{i,\wp_is}$.
\end{lem}

Next, we show that $v_{\wp,\varepsilon}^\ell(\g)$ admits a definition depending on a quiver.
Let $Q=(Q_0,Q_1)$ be a quiver with vertex set $Q_0$ and arrow set $Q_1$.
Suppose that there exists a $\Z_\wp$-action on $Q$.
Let $L$ be a set of directed loops.
For $u,v\in Q_0$, define
\begin{align*}
  |u\to v|=
  \begin{cases}
    1, & \mbox{if }u\to v\in Q_1,\\
    0, & \mbox{otherwise}.
  \end{cases}
\end{align*}

\begin{de}\label{de:qva-quiver}
Let $\ell\in\Z$, and let $\mathcal M_\wp^\ell(Q,L)$ be the category whose objects are vector spaces $W$ equipped with fields
\begin{align*}
  &y_i^a(z)=\sum_{n\in\Z}y_i^a(n)z^{-n-1}\quad i\in Q_0,\,a\in\{0,1^\pm,2^\pm\},
\end{align*}
satisfying the relations below
\begin{align}
  &\label{Qv1}\tag{Qv1}
  [y_i^0(z_1),y_j^0(z_2)]
  =\left( |2\ell\cdot i\to j|-|j\to 2\ell\cdot i|
  +|j\to i|-|i\to j| \right)\frac{1}{(z_1-z_2)^2}\\
  &\quad \nonumber
  -\left( |2\ell\cdot j\to i|-|i\to 2\ell\cdot j|
  +|i\to j|-|j\to i| \right)\frac{1}{(z_2-z_1)^2},\\
  &\label{Qv2}\tag{Qv2}
  [y_i^0(z_1),y_j^{1^\pm}(z_2)]
  =\pm y_j^{1^\pm}(z_2)
  \left( |2\ell\cdot i\to j|-|j\to 2\ell\cdot i|
  +|j\to i|-|i\to j| \right)\frac{1}{z_1-z_2}\\
  &\quad \nonumber
  \pm y_j^{1^\pm}(z_2)
  \left( |2\ell\cdot j\to i|-|i\to 2\ell\cdot j|
  +|i\to j|-|j\to i| \right)\frac{1}{z_2-z_1},\\
  &\label{Qv3}\tag{Qv3}
  [y_i^0(z_1),y_j^{2^\pm}(z_2)]
  =\pm y_j^{2^\pm}(z_2)
  \left( |\ell\cdot i\to j|-|j\to \ell\cdot i| \right)\frac{1}{z_1-z_2}\\
  &\quad\nonumber
  \pm y_j^{2^\pm}(z_2)
  \left( |i\to \ell\cdot j|-|\ell\cdot j\to i| \right)
  \frac{1}{z_2-z_1},\\
  &\label{Qv4}\tag{Qv4}
  (z_1-z_2)^{\epsilon_1\epsilon_2\left( |i\to j|-|2\ell\cdot i\to j| \right)}
  (z_2-z_1)^{\epsilon_1\epsilon_2\left( |j\to 2\ell\cdot i|-|j\to i| \right)}y_i^{1^{\epsilon_1}}(z_1)y_j^{1^{\epsilon_2}}(z_2)\\
  &\quad\nonumber=
  (z_2-z_1)^{\epsilon_1\epsilon_2\left( |j\to i|-|2\ell\cdot j\to i| \right)}
  (z_1-z_2)^{\epsilon_1\epsilon_2\left( |i\to 2\ell\cdot j|-|i\to j| \right)}y_j^{1^{\epsilon_2}}(z_2)y_i^{1^{\epsilon_1}}(z_1),\\
  &\label{Qv5}\tag{Qv5}
  (z_1-z_2)^{-\epsilon_1\epsilon_2|\ell\cdot i\to j|}
  (z_2-z_1)^{\epsilon_1\epsilon_2|j\to \ell\cdot i|}
  y_i^{1^{\epsilon_1}}(z_1)y_j^{2^{\epsilon_2}}(z_2)\\
  &\quad\nonumber
  =(z_2-z_1)^{\epsilon_1\epsilon_2|\ell\cdot j\to i|}
  (z_1-z_2)^{-\epsilon_1\epsilon_2|i\to \ell\cdot j|}
  y_j^{2^{\epsilon_2}}(z_2)y_i^{1^{\epsilon_1}}(z_1),\\
  &\label{Qv6}\tag{Qv6}
  (z_1-z_2)^{|i\to j|}y_i^{2^\pm}(z_1)y_j^{2^\pm}(z_2)
  =-(z_2-z_1)^{|j\to i|}y_j^{2^\pm}(z_2)y_i^{2^\pm}(z_1),\\
  &\label{Qv7}\tag{Qv7}
  (z_1-z_2)^{\max\{\delta_{ij}, \delta_{2\ell\cdot i,j}\}}
  y_i^{2^+}(z_1)y_j^{2^-}(z_2)\\
  &\quad\nonumber
  =-(z_1-z_2)^{\max\{\delta_{ij}, \delta_{2\ell\cdot i,j}\}}
  \frac{(z_1-z_2)^{|i\to j|}}{(z_2-z_1)^{|j\to i|}}
  y_j^{2^-}(z_2)y_i^{2^+}(z_1).
\end{align}
Denote by $V_\wp'^\ell(Q,L)$ the quantum vertex algebra $F(\mathcal M_\wp^\ell(Q,L))$ obtained in Proposition \ref{prop:V-M-quantumVA}.
Define $V_\wp^\ell(Q,L)$ to be the quotient nonlocal vertex algebra of $V_\wp^\ell(Q,L)$ modulo the ideal generated by
the following elements
\begin{align}
  &\label{Qv8}\tag{Qv8}
  y_i^{2^+}(0)y_j^{2^-}-\delta_{ij}\vac+\delta_{2\ell\cdot i,j}y_{\ell\cdot i}^{1^+},\\
  &\label{Qv9}\tag{Qv9}
  \pm y_i^0(-1)y_i^{1^\pm}-\partial y_i^{1^\pm},\\
  &\label{Qv10}\tag{Qv10}
  y_i^{1^\pm}(-1-|2\ell\cdot i\to i|+|i\to 2\ell\cdot i|)y_i^{1^\mp}-(-1)^{|2\ell\cdot i\to i|}\vac,\\
  &\label{Qv11}\tag{Qv11}
  y_{i_1}^{2^\pm}(0)y_{i_2}^{2^\pm}(0)\cdots y_{i_{m-1}}^{2^\pm}(0)y_{i_m}^{2^\pm}
  \quad\te{for
  \begin{tikzcd}[column sep=1.5em, ampersand replacement=\&]
    i_1\arrow[r]\&i_2\arrow[r]\&\cdots\arrow[r]\&i_m\in L.
    \arrow[lll, bend left]
  \end{tikzcd}}
\end{align}
\end{de}

The proof follows by the same argument as in Theorem \ref{thm:V-wp-tau-qva}.
\begin{thm}\label{thm:V-Q-L-qva}
$V_\wp^\ell(Q,L)$ is a quantum vertex algebra with quantum Yang-Baxter operator $S_Q(z)$ determined by
\begin{align*}
  S_Q(z)(y_j^0\ot y_i^0)=&y_j^0\ot y_i^0+
  \vac\ot \vac\ot \big( |2\ell\cdot i\to j|-|j\to 2\ell\cdot i|-|2\ell\cdot j\to i|\\
  &\nonumber+|i\to 2\ell\cdot j|-2|i\to j|+2|j\to i| \big)z^{-2},\\
  S_Q(z)(y_j^0\ot y_i^{1^\pm})=&y_j^0\ot y_i^{1^\pm}
  \pm\vac\ot y_i^{1^\pm}\ot \big(
    |i\to 2\ell\cdot j|-|2\ell\cdot j\to i|
    -|j\to 2\ell\cdot i|\\
   &\nonumber  +|2\ell\cdot i\to j|
   -2|i\to j|+2|j\to i|
  \big)z\inv,\\
  S_Q(z)(y_j^0\ot y_i^{2^\pm})=&y_j^0\ot y_i^{2^\pm}
  \pm\vac\ot y_i^{2^\pm}\ot \big(
    |i\to \ell\cdot j|-|\ell\cdot j\to i|\\
  &\nonumber  +|j\to \ell\cdot i|-|\ell\cdot i\to j|
  \big)z\inv,\\
  S_Q(z)(y_j^{1^{\epsilon_2}}\ot y_i^{1^{\epsilon_1}})
  =&y_j^{1^{\epsilon_2}}\ot y_i^{1^{\epsilon_1}}
  \ot (-z)^{\epsilon_1\epsilon_2\left( |i\to 2\ell\cdot j|+|2\ell\cdot i\to j|-2|i\to j| \right)}\\
  &\nonumber\times
  z^{-\epsilon_1\epsilon_2\left( |j\to 2\ell\cdot i|+|2\ell\cdot j\to i|-2|j\to i| \right)},\\
  S_Q(z)(y_j^{2^{\epsilon_2}}\ot y_i^{1^{\epsilon_1}})
  =&y_j^{2^{\epsilon_2}}\ot y_i^{1^{\epsilon_1}}
  \ot (-z)^{\epsilon_1\epsilon_2\left(|\ell\cdot i\to j|-|i\to\ell\cdot j|\right)}
  z^{\epsilon_1\epsilon_2\left( |\ell\cdot j\to i|-|j\to\ell\cdot i| \right)},\\
  S_Q(z)(y_j^{2^{\epsilon_2}}\ot y_i^{2^{\epsilon_1}})
  =&-y_j^{2^{\epsilon_2}}\ot y_i^{2^{\epsilon_1}}
  \ot (-z)^{-\epsilon_1\epsilon_2|i\to j|}z^{\epsilon_1\epsilon_2|j\to i|}.
\end{align*}
\end{thm}

Define $Q=(Q_0,Q_1)$ to be the quiver with vertex set
\begin{align*}
  Q_0=\set{ p_{i,m}}{i\in I,m\in\Z_\wp}
\end{align*}
and arrow set
\begin{align*}
  Q_1=\set{p_{i,m}\to p_{j,n}}{\vvp{q_i^{a_{ij}}q^{n-m}}=1}.
\end{align*}
There is a $\Z_\wp$-action on $Q$ defined by
\begin{align*}
  s\cdot p_{i,m}=p_{i,m+s}\quad\te{for }m,s\in\Z_\wp,\,i\in I.
\end{align*}
Let $L$ be the set consisting of the following directed loops
\begin{align*}
  &\begin{tikzcd}[column sep=2em, ampersand replacement=\&]
    p_{i,m-r_ia_{ij}}\arrow[r]\&p_{i,m-r_ia_{ij}-2r_i}
    \arrow[r]\&\cdots \arrow[r]\& p_{i,m+r_ia_{ij}}
    \arrow[lld]\\
    \&p_{j,m}\arrow[lu]\&\&
  \end{tikzcd}&&\te{for }i,j\in I\,\,\te{with}\,\,a_{ij}< 0,\\
  &\quad\,\,\begin{tikzcd}[column sep=2em, ampersand replacement=\&]
    p_{i,m}\arrow[rrr, bend right]\& p_{i,m+2r_i}\arrow[l]\&\cdots\arrow[l]\&
    p_{i,m+2r_i(\wp_i-1)}\arrow[l]
  \end{tikzcd}&&\te{for }i\in I.
\end{align*}
It is straightforward to verify that $v_{\wp,\varepsilon}^\ell(\g)$ satisfies the relations \eqref{Qv1}-\eqref{Qv11}.
Combining this with Proposition \ref{prop:universal-nonlocal-va}, we immediately get the following result.
\begin{lem}\label{lem:small-map}
There exists a quantum vertex algebra epimorphism from $V_\wp^{r\ell}(Q,L)$ to $v_{\wp,\varepsilon}^\ell(\g)$ determined by
\begin{align*}
  y_{p_{i,m}}^0\mapsto \wt \xi_{i,m}^0,\quad
  y_{p_{i,m}}^a\mapsto \xi_{i,m}^a\quad\te{for }i\in I,\,m\in\Z_\wp,\,a\in\{1^\pm,2^\pm\}.
\end{align*}
\end{lem}

In the rest of this subsection, we construct a twistor of the tensor product quantum vertex algebra $V_{\wh\h}^1\ot V_\wp^{r\ell}(Q,L)$, and prove that the quantum vertex algebra homomorphisms $V_{\wh\h}^1\to \mathcal H$ and $V_\wp^{r\ell}(Q,L)\to v_{\wp,\varepsilon}^\ell(\g)$ given by Lemmas \ref{lem:Heisenberg-hom} and \ref{lem:small-map}, lift
to a quantum vertex algebra isomorphism from the deformed quantum vertex algebra of $V_{\wh\h}^1\ot V_\wp^{r\ell}(Q,L)$ to $V_{\wp,\varepsilon}^\ell(\g)$.
It is straightforward to verify the following result.

\begin{lem}\label{lem:Heisenberg-comod}
$V_{\wh\h}^0$ carries a cocommutative commutative vertex bialgebra structure with coproduct $\Delta$ determined by
\begin{align*}
  \Delta(\al^\vee_{i,\wp_is})=\al^\vee_{i,\wp_is}\ot \vac+\vac\ot \al^\vee_{i,\wp_is}\quad\te{for }i\in I,\,s\in\Z_\wp,
\end{align*}
and counit $\varepsilon$ determined by
\begin{align*}
  \varepsilon(\al^\vee_{i,s\wp_i})=0\quad\te{for }i\in I,\,s\in\wp_i\Z_\wp.
\end{align*}
Moreover, there exists a $V_{\wh\h}^0$-comodule nonlocal vertex algebra structure $\rho$ on $V_{\wh\h}^1$ determined by
\begin{align*}
  \rho(\al^\vee_{i,\wp_is})=\al^\vee_{i,\wp_is}\ot \vac+\vac\ot \al^\vee_{i,\wp_is}\quad\te{for }i\in I,\,s\in\Z_\wp.
\end{align*}
\end{lem}

\begin{lem}
There exist pseudo-derivations $\mu_{i,\wp_is}^\pm(z)$ on $V_\wp^{r\ell}(Q,L)$ defined by ($\epsilon\in\{\pm\}$)
\begin{align}\label{eq:sigma-pesudo-derivation}
  &\mu_{i,\wp_is}^\epsilon(z)y_{p_{j,n}}^0=0
  =\mu_{i,\wp_is}^\epsilon(z)y_{p_{j,n}}^{1^\pm},\quad
   \mu_{i,\wp_is}^\epsilon(z)y_{p_{j,n}}^{2^\pm}
   =\pm \epsilon [a_{ij}]_{\zeta_i^{\wp_is}}\zeta^{\wp_is(n+r\ell)}y_{p_{j,n}}^{2^\pm}\ot z\inv.
\end{align}
\end{lem}

\begin{proof}
Similar to the proof of Lemma \ref{lem:pseudos}, we get pseudo-derivations $\mu_{i,\wp_is}^\epsilon(z)$ on $V_\wp'^{r\ell}(Q,L)$
satisfying conditions \eqref{eq:sigma-pesudo-derivation}.
It is easy to verify that $\mu_{i,\wp_is}^\epsilon(z)$ is compatible with \eqref{Qv8}-\eqref{Qv11}.
Therefore $\mu_{i,\wp_is}^\epsilon(z)$ induces pseudo-derivations on $V_\wp^{r\ell}(Q,L)$ as desired.
\end{proof}

\begin{lem}\label{lem:Heisenberg-mod}
There exists a unique invertible $V_{\wh\h}^0$-module nonlocal vertex algebra structure $\mu(\cdot,z)$ on $V_\wp^{r\ell}(Q,L)$ defined by
\begin{align*}
  \mu(h_{i,\wp_im},z)=\mu_{i,\wp_im}^+(z)\quad\te{for }i\in I,\,m\in\Z_\wp.
\end{align*}
\end{lem}

\begin{proof}
The uniqueness follows immediate from the fact that
$V_{\wh\h}^0$ is generated by $$T=\set{h_{i,\wp_im}}{i\in I,m\in\Z_\wp}.$$
Similar to the proof of Lemma \ref{lem:deform-tri-nonlocalVA}, we get two $V_{\wh\h}^0$-module nonlocal vertex algebra structure $\mu^\pm(\cdot,z)$ on $V_\wp^{r\ell}(Q,L)$ determined by
\begin{align*}
  \mu^\pm(h_{i,\wp_is},z)=\mu_{i,\wp_is}^\pm(z)\quad\te{for }i\in I,\,s\in\Z_\wp.
\end{align*}
Then
\begin{align}\label{eq:Heisenberg-mod-inv}
  &(\mu^+\ast\mu^-)(h_{i,\wp_im},z)y_{p_{j,n}}^a
  =\mu^+(h_{i,\wp_im},z)y_{p_{j,n}}^a+\mu^-(h_{i,\wp_im},z)y_{p_{j,n}}^a
  =0=\varepsilon(h_{i,\wp_im},z)y_{p_{j,n}}^a
\end{align}
for $i,j\in I$ and $m,n\in\Z_\wp$.
Since $V_{\wh\h}^0$ is generated by $T$ and $V_\wp^{r\ell}(Q,L)$ is generated by
\begin{align*}
  S=\set{y_{p_{i,m}}^a}{i\in I,\,m\in\Z_\wp,\,a\in\{0,1^\pm,2^\pm\}},
\end{align*}
the relation \eqref{eq:Heisenberg-mod-inv} can be extend to the whole space $V_{\wh\h}^0$ and $V_\wp^{r\ell}(Q,L)$.
It follows that $\mu^+\ast\mu^-=\varepsilon$,
which means that $\mu^+$ is invertible.
Take $\mu=\mu^+$. We complete the proof.
\end{proof}

\begin{lem}\label{lem:sigma-S-com}
We have
\begin{align}
  &(\mu(h,z_1)\ot 1)S(z_2)(u\ot v)=S(z_2)(\mu(h,z_1)u\ot v),\label{eq:sigma-S-com1}\\
  &(1\ot \mu(h,z_1))S(z_2)(u\ot v)=S(z_2)(u\ot \mu(h,z_1)v)\label{eq:sigma-S-com2}
\end{align}
for any $h\in V_{\wh\h}^0$ and $u,v\in V_\wp^{r\ell}(Q,L)$.
\end{lem}

\begin{proof}
Set $T=\set{h_{i,\wp_is}}{i\in I,\,s\in\Z_\wp}\cup \{\vac\}$ and set
\begin{align*}
  S=\set{y_{p_{i,m}}^a}{i\in I,\,m\in\Z_\wp,\,a\in\{0,1^\pm,2^\pm\}}\cup\{\vac\}.
\end{align*}
It is straightforward to check that \eqref{eq:sigma-S-com1} and \eqref{eq:sigma-S-com2} hold for any $h\in T$ and $u,v\in S$.
Note that $V_{\wh\h}^0$ is generated by $T$, $V_\wp^{r\ell}(Q,L)$ is generated by $S$, and
\begin{align*}
  \Delta(T)\subset T\ot T,\quad
  S(z)(S\ot S)\subset S\ot S\ot \C((z)).
\end{align*}
Then the relations \eqref{eq:sigma-S-com1} and \eqref{eq:sigma-S-com2} can be extended to the whole space $V_{\wh\h}^0$ and $V_\wp^{r\ell}(Q,L)$, which completes the proof of lemma.
\end{proof}

Applying Lemma \ref{lem:def-twisting-op} to Lemmas \ref{lem:Heisenberg-comod} and \ref{lem:Heisenberg-mod},
we get an invertible twisting operator $T_\mu(z)$ for the ordered pair $(V_\wp^{r\ell}(Q,L),V_{\wh\h}^1)$ (see \eqref{eq:T-tau} for the definition of $T_\mu(z)$).
The follow result follows immediate from the definition of $T_\mu(z)$.

\begin{lem}\label{lem:T-sigma}
For $i,j\in I$, $n\in\Z_\wp$ and $a\in \{0,1^\pm,2^\pm\}$, we have that
\begin{align*}
  &T_\mu(z)(y_{p_{j,n}}^a\ot h_{i,\wp_im})=y_{p_{j,n}}^a\ot \al^\vee_{i,\wp_im}
  -( \delta_{a,2^+}-\delta_{a,2^-}) [a_{ij}]_{\zeta_i^{\wp_im}}\zeta^{\wp_im(n+r\ell)} y_{p_{j,n}}^{2^\pm}\ot \vac \ot z\inv.
\end{align*}
\end{lem}

Since $V_{\wh\h}^0$ is commutative,
we have that
\begin{align*}
  [\mu(h,z_1),\mu(k,z_2)]=0\quad\te{for }h,k\in V_{\wh\h}^0.
\end{align*}
Note that $V_{\wh\h}^0$ is cocommutative and the quantum Yang-Baxter of $V_{\wh\h}^1$ is trivial.
By applying Propositions \ref{prop:qva-double-twisted} 
and \ref{prop:double-twisted-vba} to these and Lemma \ref{lem:sigma-S-com}, we get a twistor $T^\mu(z)=T_\mu^{14}(z)T_\mu^{32}(-z)$ of the tensor product quantum vertex algebra $V_\wp^{r\ell}(Q,L)\ot V_{\wh\h}^1$, and the deformed nonlocal vertex algebra
$\mathfrak D_{T^\mu}(V_\wp^{r\ell}(Q,L)\ot V_{\wh\h}^1)$
is a quantum vertex algebra with quantum Yang-Baxter operator $S_Q^{24}(z)$.
In this subsection, we denote this quantum vertex algebra by
\begin{align*}
  V_\wp^{r\ell}(Q,L)\rtimes_\mu V_{\wh\h}^1.
\end{align*}
From Remark \ref{rem:double-twisted} we see that
$V_\wp^{r\ell}(Q,L)\rtimes_\mu V_{\wh\h}^1$
  contains $V_{\wh\h}^1$ and $V_\wp^{r\ell}(Q,L)$ as quantum vertex subalgebras.

\begin{prop}\label{prop:decomp}
The vertex algebra epimorphism $V_{\wh\h}^1\to \mathcal H$ given in Lemma \ref{lem:Heisenberg-hom} and
the quantum vertex algebra epimorphism $V_\wp^{r\ell}(Q,L)\to v_{\wp,\varepsilon}^\ell(\g)$ given in Lemma \ref{lem:small-map}
can be extend to a quantum vertex algebra homomorphism
\begin{align*}
  V_{\wp,\varepsilon}^\ell(\g)\longrightarrow V_\wp^{r\ell}(Q,L)\rtimes_\mu V_{\wh\h}^1.
\end{align*}
\end{prop}

\begin{proof}
Denote by $Y$ the vertex operator map of $V_\wp^{r\ell}(Q,L)\rtimes_\mu V_{\wh\h}^1$.
From Theorem \ref{thm:nonlocal-VA-twistor-deform} and \eqref{eq:T-tau}, we have that
\begin{align*}
  &\Sing_zY(\vac\ot \al^\vee_{i,\wp_im},z)(y_{p_{j,n}}^a\ot \vac)\\
  =&\Sing_zY^{12}(z)Y^{34}(z)\sigma^{23}T_\mu^{32}(-z)T_\mu^{14}(z)
  (\vac\ot \al^\vee_{i,\wp_im}\ot y_{p_{j,n}}^a\ot\vac)\\
  =&\Sing_zY^{12}(z)Y^{34}(z)\sigma^{23}
  \big(\vac\ot \al^\vee_{i,\wp_im}\ot y_{p_{j,n}}^a\ot\vac\\
  &\quad+(\delta_{a,2^+}-\delta_{a,2^-})[a_{ij}]_{\zeta_i^{\wp_im}}
  \zeta^{\wp_im(n+r\ell)}
  \vac\ot \vac\ot y_{p_{j,n}}^{2^\pm}\ot \vac z\inv \big)\\
  =&(\delta_{a,2^+}-\delta_{a,2^-})[a_{ij}]_{\zeta_i^{\wp_im}}
  \zeta^{\wp_im(n+r\ell)} y_{p_{j,n}}^{2^\pm}\ot\vac z\inv.
\end{align*}
Combining this with the fact that the quantum Yang-Baxter operator of $V_\wp^{r\ell}(Q,L)\rtimes_\mu V_{\wh\h}^1$ is $S_Q^{24}(z)$, we have that
\begin{align*}
  &[Y(\vac\ot \al^\vee_{i,m\wp_i},z_1),Y(y_{p_{j,n}}^a\ot \vac,z_2)]\\
  =&(\delta_{a,2^+}-\delta_{a,2^-})
  [a_{ij}]_{\zeta_i^{\wp_im}}
  \zeta^{\wp_im(n+r\ell)}
  Y(y_{p_{j,n}}^a\ot \vac,z_2)z_1\inv\delta\left(\frac{z_2}{z_1}\right).
\end{align*}
From Proposition \ref{prop:universal-V} and Lemma \ref{lem:equiv-rels}, we get a nonlocal vertex algebra homomorphism
$$f:V_{\wp,\varepsilon}'^\ell(\g)\longrightarrow V_\wp^{r\ell}(Q,L)\rtimes_\mu V_{\wh\h}^1$$
defined by
\begin{align*}
  &h_{i,\wp_im}\mapsto \vac\ot\al_{i,\wp_im}^\vee,\quad
  \wt\xi_{i,m}^0\mapsto y_{i,m}^0\ot\vac,\quad
  \xi_{i,m}^a\mapsto y_{i,m}^a\ot \vac\quad\te{for }i\in I,\,m\in\Z_\wp,\,a\in\{1^\pm,2^\pm\}.
\end{align*}
By comparing \eqref{eq:va-e+e-}-\eqref{eq:va-rel-rtu} with
\eqref{Qv8}-\eqref{Qv11}, we get that $f$ factor through $V_{\wp,\varepsilon}^\ell(\g)$, as desired.
\end{proof}

Next, we construct the following
quantum vertex algebra homomorphism
\begin{align*}
  V_\wp^{r\ell}(Q,L)\rtimes_\mu V_{\wh\h}^1\longrightarrow V_{\wp,\varepsilon}^\ell(\g).
\end{align*}

\begin{lem}\label{lem:V-wp-T}
There exists a unique invertible twistor $T(z)$ of $V_{\wp,\varepsilon}^\ell(\g)$, such that \eqref{eq:qyb-com1} holds for the pair $(T(z),T(z))$ and
\begin{align*}
  &T(z)(b_{j,n}\ot a_{i,m})=b_{j,n}\ot a_{i,m}\quad\te{for }a,b\in\{\wt\xi^0,\xi^{1^\pm},\xi^{2^\pm}\},\\
  &T(z)(h_{j,\wp_jn}\ot \xi_{i,m}^a)
  =h_{j,\wp_jn}\ot \xi_{i,m}^a
  +(\delta_{a,2^+}-\delta_{a,2^-}) \vac\ot \xi_{i,m}^{2^\pm}\ot [a_{ij}]_{\zeta_i^{\wp_im}}
  \zeta^{\wp_im(n+r\ell)}z\inv,\\
  &T(z)(\xi_{j,n}^a\ot h_{i,\wp_im})
  =\xi_{j,n}^a\ot h_{i,\wp_im}
  -(\delta_{a,2^+}-\delta_{a,2^-})\xi_{j,n}^a\ot\vac
  \ot [a_{ji}]_{\zeta_j^{\wp_jn}}
  \zeta^{\wp_jn(m+r\ell)}z\inv.
\end{align*}
\end{lem}

\begin{proof}
By using Lemma \ref{lem:qva-twistor}, we get the unique twistor $T(z)$ of $V_{\wp,\varepsilon}'^\ell(\g)$ satisfying the desired conditions.
Recall the ideal $R_{\wp,\varepsilon}^\ell(\g)$ of $V_{\wp,\varepsilon}'^\ell(\g)$ defined in Definition \ref{de:V}.
It is straightforward to verify that $T(z)$ preserves
\begin{align*}
  R_{\wp,\varepsilon}^\ell(\g)\ot V_{\wp,\varepsilon}'^\ell(\g)+
  V_{\wp,\varepsilon}'^\ell(\g)\ot R_{\wp,\varepsilon}^\ell(\g).
\end{align*}
Then $T(z)$ induces the twistor of $V_{\wp,\varepsilon}^\ell(\g)$, as desired.
\end{proof}

\begin{lem}\label{lem:V-wp-T-inv}
$T(z)\inv$ is also a twistor of $V_{\wp,\varepsilon}^\ell(\g)$.
Moreover,
\begin{align*}
  &T(z)\inv(b_{j,n}\ot a_{i,m})=b_{j,n}\ot a_{i,m}\quad\te{for }a,b\in\{\wt\xi^0,\xi^{1^\pm},\xi^{2^\pm}\},\\
  &T(z)\inv(h_{j,\wp_jn}\ot \xi_{i,m}^a)
  =h_{j,\wp_jn}\ot \xi_{i,m}^a
  -(\delta_{a,2^+}-\delta_{a,2^-}) \vac\ot \xi_{i,m}^{2^\pm}\ot [a_{ij}]_{\zeta_i^{\wp_im}}
  \zeta^{\wp_im(n+r\ell)}z\inv,\\
  &T(z)\inv(\xi_{j,n}^a\ot h_{i,\wp_im})
  =\xi_{j,n}^a\ot h_{i,\wp_im}
  +(\delta_{a,2^+}-\delta_{a,2^-})\xi_{j,n}^a\ot\vac
  \ot [a_{ji}]_{\zeta_j^{\wp_jn}}
  \zeta^{\wp_jn(m+r\ell)}z\inv
\end{align*}
for $i,j\in I$ and $m,n\in\Z_\wp$.
\end{lem}

\begin{proof}
The moreover statement is clear.
By Remark \ref{rem:twistor-inv}, we get another twistor $\bar T(z)=T^{21}(-z)\inv$.
Since the relations \eqref{eq:twistor-com} and \eqref{eq:qyb-com1} hold for the pair $(T(z),T(z))$,
they also hold for the pair $(T(z),\bar T(z))$.
Remark \ref{rem:prod-twistor} yields that $T(z)\bar T(z)$ is also a twistor.
Note that
\begin{align*}
  T(z)\bar T(z)(b\ot a)=b\ot a
\end{align*}
for $a,b$ lie in the following generating subset of $V_{\wp,\varepsilon}^\ell(\g)$:
\begin{align*}
  \set{h_{i,\wp_im},\,\wt\xi_{i,m}^0,\,\xi_{i,m}^a}{i\in I,\,m\in\Z_\wp,\,a\in\{1^\pm,2^\pm\}}.
\end{align*}
Therefore, $T(z)\bar T(z)$ is the trivial twistor of $V_{\wp,\varepsilon}^\ell(\g)$, and hence, $T(z)\inv=\bar T(z)$ is a twistor.
\end{proof}

\begin{lem}\label{lem:T-S-wp-epsilon-com}
The relations \eqref{eq:qyb-com1} and \eqref{eq:twistor-com} hold for $(T(z),S_{\wp,\varepsilon}(z))$.
\end{lem}

\begin{proof}
It is straightforward to verify that the following relations
\begin{align}\label{eq:T-S-wp-epsilon-com}
  T^{ij}(z_1)S_{\wp,\varepsilon}^{ab}(z_2)(u\ot v\ot w)
  =S_{\wp,\varepsilon}^{ab}(z_2)T^{ij}(z_1)(u\ot v\ot w)
\end{align}
hold for $(i,j)\not\in\{(a,b),(b,a)\}$ and
\begin{align*}
  u,v,w\in S=\set{h_{i,\wp_im},\wt\xi_{i,m}^0,\xi_{i,m}^a}{i\in I,\,m\in\Z_\wp,\,a\in\{1^\pm,2^\pm\}}.
\end{align*}
Since $V_{\wp,\varepsilon}^\ell(\g)$ is generated by $S$,
the relation \eqref{eq:T-S-wp-epsilon-com} can be extended to the whole space $V_{\wp,\varepsilon}^\ell(\g)$,
which proves the lemma.
\end{proof}

\begin{lem}\label{lem:tensor-prod-hom}
There is a unique nonlocal vertex algebra homomorphism
from the tensor product nonlocal vertex algebra  $V_\wp^{r\ell}(Q,L)\ot V_{\wh\h}^1$ to $\mathfrak D_T(V_{\wp,\varepsilon}^\ell(\g))$ such that
\begin{align*}
  \vac\ot\al_{i,\wp_im}^\vee\mapsto h_{i,\wp_im},\quad
  y_{i,m}^0\ot\vac\mapsto \wt \xi_{i,m}^0,\quad
  y_{i,m}^a\ot\vac\mapsto \xi_{i,m}^a
  \quad\te{for }i\in I,\,m\in\Z_\wp,\,a\in\{1^\pm,2^\pm\}.
\end{align*}
\end{lem}

\begin{proof}
From Lemma \ref{lem:V-wp-T}, we have that
$T(z)(b\ot a)=b\ot a$ for $a,b\in \mathcal H$ or $a,b\in v_{\wp,\varepsilon}^\ell(\g)$.
Then $\mathcal H$ and $v_{\wp,\varepsilon}^\ell(\g)$ are also nonlocal vertex subalgebras of $\mathfrak D_T(V_{\wp,\varepsilon}^\ell(\g))$.
Let $g_1$ and $g_2$ be the following two composition nonlocal vertex algebra homomorphisms
\begin{align*}
  g_1:\begin{tikzcd}[column sep=1.5em, ampersand replacement=\&]
    V_{\wh\h}^1\ar[r,"\pi_2"]\&\mathcal H\ar[r,hook]
    \&\mathfrak D_T(V_{\wp,\varepsilon}^\ell(\g)),
  \end{tikzcd}
  &&
  g_2:
  \begin{tikzcd}[column sep=1.5em, ampersand replacement=\&]
    V_\wp^{r\ell}(Q,L)\ar[r,"\pi_1"]\&v_{\wp,\varepsilon}^\ell(\g)\ar[r,hook]
    \&\mathfrak D_T(V_{\wp,\varepsilon}^\ell(\g)),
  \end{tikzcd}
\end{align*}
where $\pi_1$ is given in Lemma \ref{lem:Heisenberg-hom}
and $\pi_2$ is given in Lemma \ref{lem:small-map}.
To show that $g_1,g_2$ can be extended to the desired nonlocal vertex algebra homomorphism, it is sufficient to prove the following result
\begin{align}\label{eq:tensor-prod-hom}
  [\mathfrak D_T(Y)(u,z_1),\mathfrak D_T(Y)(v,z_2)]=0
  \quad\te{for }u\in \mathcal H,\,v\in v_{\wp,\varepsilon}^\ell(\g).
\end{align}

Recall the definition of $\mathfrak D_T(Y)$ (see Theorem \ref{thm:nonlocal-VA-twistor-deform}).
We have that
\begin{align*}
  &\Sing_z \mathfrak D_T(Y)(h_{i,\wp_im},z)\xi_{j,n}^a
  =\Sing_z Y(z)T^{21}(-z)(h_{i,\wp_im}\ot \xi_{j,n}^a)\\
  =&\Sing_zY(h_{i,\wp_im},z) \xi_{j,n}^a
  -(\delta_{a,2^+}-\delta_{a,2^-})\Sing_zY(\vac,z) \xi_{j,n}^a
  \ot [a_{ji}]_{\zeta_j^{\wp_jn}}\zeta^{\wp_jn(m+r\ell)}z\inv \\
  =&0.
\end{align*}
Applying Theorem \ref{thm:qva-twistor} to Lemmas \ref{lem:V-wp-T} and \ref{lem:T-S-wp-epsilon-com}, we have that $\mathfrak D_T(V_{\wp,\varepsilon}^\ell(\g))$ is a quantum vertex algebra with quantum Yang-Baxter operator $T^{21}(-z)\inv S(z)T(z)$.
It is straightforward to verify that
\begin{align*}
  T^{21}(-z)\inv S(z)T(z)(h_{i,\wp_im}\ot \xi_{j,n}^a)=h_{i,\wp_im}\ot \xi_{j,n}^a.
\end{align*}
The commutator formula yields that the relation \eqref{eq:tensor-prod-hom} holds for
\begin{align*}
  u\in A=\set{h_{i,\wp_im}}{i\in I,\,m\in\Z_\wp},\quad
  v\in B=\set{\wt\xi_{i,m}^0,\,\xi_{i,m}^a}{i\in I,\,m\in\Z_\wp,\,a\in\{1^\pm,2^\pm\}}.
\end{align*}
Since $\mathcal H$ is generated by $A$ and $v_{\wp,\varepsilon}^\ell(\g)$ is generated by $B$,
the relation \eqref{eq:tensor-prod-hom} holds for all $u\in \mathcal H$ and $v\in v_{\wp,\varepsilon}^\ell(\g)$, which completes the proof of lemma.
\end{proof}

\begin{prop}\label{prop:decomp-reverse}
The map obtained in Lemma \ref{lem:tensor-prod-hom} is also a nonlocal vertex algebra homomorphism from
$V_\wp^{r\ell}(Q,L)\rtimes_\mu V_{\wh\h}^1$ to $V_{\wp,\varepsilon}^\ell(\g)$.
\end{prop}

\begin{proof}
Denote by $g$ the nonlocal vertex algebra homomorphism obtained in Lemma \ref{lem:tensor-prod-hom}.
Recall the twistor $T^\mu(z)$ of the tensor product nonlocal vertex algebra $V_\wp^{r\ell}(Q,L)\ot V_{\wh\h}^1$.
From Lemma \ref{lem:V-wp-T-inv}, we get a twistor $T(z)\inv$ of $V_{\wp,\varepsilon}^\ell(\g)$.
Let
\begin{align*}
  &A=\set{\vac\ot\al_{i,\wp_im}^\vee,\,y_{i,m}^a\ot \vac}{i\in I,\,m\in\Z_\wp,\,a\in\{0,1^\pm,2^\pm\}}.
\end{align*}
Note that
\begin{align}\label{eq:decomp-reverse}
  T(z)(g(v)\ot g(u))=(g\ot g)(T^\mu(z)(v\ot u))\quad\te{for }u,v\in A.
\end{align}
Since $V_\wp^{r\ell}(Q,L)\ot V_{\wh\h}^1$ is generated by $A$,
the relation \eqref{eq:decomp-reverse} holds for the whole space $V_\wp^{r\ell}(Q,L)\ot V_{\wh\h}^1$.
Therefore, $g$ becomes a nonlocal vertex algebra homomorphism
from $V_\wp^{r\ell}(Q,L)\rtimes_\mu V_{\wh\h}^1$ to $V_{\wp,\varepsilon}^\ell(\g)$ as desired.
\end{proof}

Combining Propositions \ref{prop:decomp} with \ref{prop:decomp-reverse}, we immediately get the main result of this subsection.
\begin{thm}\label{thm:decomp}
The vertex algebra epimorphism $V_{\wh\h}^1\to \mathcal H$ given in Lemma \ref{lem:Heisenberg-hom} and
the quantum vertex algebra epimorphism $V_\wp^{r\ell}(Q,L)\to v_{\wp,\varepsilon}^\ell(\g)$ given in Lemma \ref{lem:small-map}
extend to a quantum vertex algebra isomorphism
\begin{align*}
  V_{\wp,\varepsilon}^\ell(\g)\cong V_\wp^{r\ell}(Q,L)\rtimes_\mu V_{\wh\h}^1.
\end{align*}
\end{thm}

\bibliographystyle{unsrt}



\end{document}